\documentclass[11pt]{amsart}
\usepackage{amssymb,latexsym, color
}
\usepackage[dvips,centering,
includehead,width=14.05cm,height=21.45cm,
paperheight=23.8cm,paperwidth=17cm, voffset=-0.6cm
]{geometry}
\usepackage{enumitem}
\usepackage{hyperref}
\usepackage{pict2e}
\usepackage{graphicx}
\usepackage{rotating}
\usepackage{tikz-cd}
\usetikzlibrary{arrows}

\usepackage{ccaption}
\changecaptionwidth
\captionwidth{5.54in}

\renewcommand{\ref}{\hyperref}

\newcommand{\twoheaduparrow}{\rotatebox[origin=c]{90}{$\twoheadrightarrow$
}}

\newcommand{\YB}{{\,$\backslash$\hspace{-7pt}$\slash$\hspace{-7.2pt}\smash{\raisebox{-2pt}{$-$}}\hspace{-0pt}}}

\def\ZZ{{\mathbb Z}}
\def\RR{{\mathbb R}}

\def\CC{{\mathbb C}}
\def\KK{{\mathbf K}}
\def\PP{{\mathbf P}}
\def\ww{{\mathtt{w}}}
\def\bww{{\overline\ww}}

\newcommand{\imbr}{{\operatorname{ImBr}}}

\newcommand{\Ball}{{\mathbf{B}}}
\newcommand{\Disk}{{\mathbf{D}}}
\newcommand{\AG}{{\rm A$\Gamma$}}
\newcommand{\circledplus}{{\oplus}}
\newcommand{\simplus}{\stackrel{+}{\sim}}
\DeclareMathOperator{\selflink}{sl}
\DeclareMathOperator{\link}{lk}

\newtheorem{theorem}{Theorem}[section]
\newtheorem{proposition}[theorem]{Proposition}
\newtheorem{conjecture}[theorem]{Conjecture}
\newtheorem{corollary}[theorem]{Corollary}
\newtheorem{lemma}[theorem]{Lemma}

\theoremstyle{definition}

\newtheorem{remark}[theorem]{Remark}

\newtheorem{example}[theorem]{Example}
\newtheorem{definition}[theorem]{Definition}
\newtheorem{problem}[theorem]{Problem}
\newtheorem{restr}[theorem]{Restriction}

\numberwithin{equation}{section}

\newcommand{\red}[1]{{\color{red} #1}}

\newcommand{\eps}{{\varepsilon}}

\begin{document}

\title[Morsifications and mutations]
{Morsifications and mutations}

\author{Sergey Fomin}
\address{Department of Mathematics, University of Michigan, 
Ann Arbor, MI 48109, USA}
\email{fomin@umich.edu}

\author{Pavlo Pylyavskyy}
\address{
School of Mathematics, University of Minnesota,
Minneapolis, MN 55414, USA}
\email{ppylyavs@umn.edu}

\author{Eugenii Shustin}
\address{School of Mathematical Sciences, Tel Aviv
University, 
Tel Aviv 69978, 
Israel}
\email{shustin@tauex.tau.ac.il}

\author{Dylan Thurston}
\address{Department of Mathematics, Indiana University, Bloomington, IN 47405, USA}
\email{dpthurst@indiana.edu}

\thanks
{\emph{2020 Mathematics Subject Classification}
Primary 13F60, 
Secondary  
20F36, 
57K10, 
58K65. 
}

\thanks{Partially supported by NSF grants 
DMS-1664722 and DMS-2054231 (S.~F.), 
DMS-1148634, DMS-1351590, DMS-1745638 and DMS-1949896 and  (P.~P.), 
DMS-1507244 and DMS-2110143 (D.~T.), 
a Simons Fellowship (S.~F.), 
a Sloan Fellowship (P.~P.), 
the ISF grants 176/15 and  501/18 (E.~S.), 
and the Bauer-Neuman Chair in Real and Complex Geometry (E.~S.).}

\date{February 18, 2020. Revised November 1, 2021}

\begin{abstract}
We describe and investigate a connection between 
the topology of iso\-lated singularities 
of plane curves and 
the mutation equivalence, in the sense of cluster algebra theory, 
of the quivers associated with 
their morsifications. 
\end{abstract}

\keywords{Plane curve singularity, morsification, A'Campo--Guse\u{\i}n-Zade diagram, 
quiver mutation, positive braid, plabic graph, oriented divide, link equivalence.} 


\maketitle

\tableofcontents

\thispagestyle{empty}



\section*{Introduction}

We present and explore a remarkable connection between two
seemingly unrelated subjects: the combinatorics of quiver
mutations (which originated in the theory of cluster algebras) 
and the topology of plane curve singularities. \linebreak[3]
Our constructions build on the elegant approach 
to the latter subfield of singularity theory that was 
pioneered in the 1970s by
N.~A'Campo~\cite{acampo} and 
S.~Guse\u{\i}n-Zade~\cite{gusein-zade-1,gusein-zade-2}.
Given a real form of an isolated plane curve singularity, 
one begins by finding its real \emph{morsification}, 
a real nodal local deformation that has 
the maximal possible number
of real hyperbolic nodes. 
From the combinatorial topology of 
the morsification (more precisely, of its \emph{divide},
the set of real points of the deformed curve in the vicinity of the
singularity, viewed up to isotopy), one constructs the associated 
\emph{A'Campo--Guse\u{\i}n-Zade} 
\emph{diagram}  (or \AG-diagram), 
a certain tricolored planar graph.    
The \AG-diagram 
can be used to explicitly compute the monodromy and the
intersection form in the vanishing homology of the singularity. 
In fact, more is true: the \AG-diagram 
uniquely determines the complex topological type of the underlying singularity; 
see 
\cite{balke-kaenders}
(for totally real singularities) and~\cite{leviant-shustin} (in full generality).  

A given complex singularity may have many distinct \emph{real forms}.
These are real plane singular curves which, 
when viewed over the complex numbers, 
are locally homeomorphic to each other---but over the reals, 
they are not.
Their real morsifications are also different from each other, 
and so are the associated \AG-diagrams. 
How, then, can we tell, looking at two morsifications, 
whether we are dealing with the same complex singularity or not? 

One answer to this question was given by N.~A'Campo~\cite{AC1} in the late 1990s,
in~terms of a certain link that can be constructed from the divide of a given morsification. 
In~this paper, we propose an alternative answer, which comes from the theory of \emph{cluster
  algebras}~\cite{ca1, fwz},
specifically from the combinatorics of \emph{quiver mutations}.  

Our approach starts with a small but important change of perspective:
we replace \AG-diagrams by closely related \emph{quivers},
removing the marking of the vertices but introducing orientations of the edges. 
We show (see Theorem~\ref{th:quiver-determines-topology}) that the quiver constructed from a morsification determines the topological type of the singularity. 
The question then becomes: how can we tell, by looking at two quivers
coming from morsifications of two plane curve singularities, whether these singularities are topologically equivalent or not? 

The answer comes from the theory of \emph{quiver mutations}. 
(A~quiver can be mutated in different ways, depending on the choice of a vertex.
One then applies a mutation to the resulting quiver, and so on.
The set of quivers obtained in this way defines a cluster algebra.) 
We~conjecture that two singularities are topologically equivalent
if and only if the quivers associated with their respective morsifications 
are mutation equivalent to each other, 
i.e., if and only if one quiver can be transformed into another by iterated mutations.  
Thus, different real forms of the same complex
singularity---and different morsifications of these real forms---should 
give rise to mutation equivalent quivers. 
Conversely, 
topologically distinct singularities are expected to produce quivers
of different mutation~type.  
Succinctly put, plane curve singularities are classified
by the cluster algebras defined by their morsifications. 

Our main results 
establish this relationship between the topology of 
plane curve singularities and the mutation equivalence of associated quivers
modulo some technical assumptions, 
which we optimistically expect to be redundant. 
More concretely, we obtain the following results. 

First, we establish the direction ``combinatorial equivalence implies topological equivalence''  
in the version where on the combinatorial side,  mutation equivalence of quivers associated with morsifications 
is replaced by the move equivalence of the corresponding \emph{plabic} (i.e., planar bicolored) graphs, 
in the sense of A.~Postnikov~\cite{postnikov}. 
The connection between mutation equivalence and move equivalence is well known,
cf.\ Proposition~\ref{pr:plabic-vs-quivers}; in the context of morsifications,
we expect these notions to be interchangeable, see  Conjecture~\ref{conj:shapiro-algebraic}.  
Our first main result, Corollary~\ref{cor:p=>d}, 
demonstrates that move equivalence of plabic graphs constructed from different morsifications 
implies the topological equivalence of the corresponding singularities. 

The key ingredient of the proof of this result is the construction of a link asso\-ciated 
with an arbitrary plabic graph. This link is invariant under Postnikov's local moves,
see Corollary~\ref{cor:plabic-isotopic}. 
Furthermore, the links associated with plabic graphs are naturally transverse,
and local moves translate into transverse isotopies, 
see Corollaries~\ref{cor:plabic-transverse} and~\ref{cor:divide-link-transverse}. 
(For a plabic graph of a divide, one recovers \linebreak[3]
its A'Campo link.)  

Our second main result concerns the opposite direction:  
``topological equivalence implies combinatorial equivalence.''  
In~Corollary~\ref{th:malleable+positive-isotopic}, 
we show that different morsifications of the same singularity produce 
mutation-equivalent quivers. 
This result is predicated on the following two assumptions: 
(i)~the existence of sequences of Yang-Baxter moves transforming each of 
the given divides into a \emph{scannable} form; 
and \linebreak[3]
(ii)~the existence of a \emph{positive isotopy} relating the positive braids
associated with the respective scannable divides. 
We expect the assumptions (i)--(ii) to be redundant, see Conjectures~\ref{conj:positive-isotopy}
and~\ref{conj:alg-is-malleable}. 
These assumptions are satisfied in all examples of real morsifications known to us,
cf.\ Remarks~\ref{rem:isotopy-in-known-constructions}. 
and~\ref{rem:evidence-via-overlays}. 


The link between morsifications and quiver mutations revealed in this paper 
is suggestive of a deep intrinsic relationship between 
singularities and cluster algebras. \linebreak[3]
To give one example, a quasihomogeneous singularity $x^a+y^b=0$ is
described by the same mutation class of quivers as the standard
cluster structure on the homogeneous coordinate ring of the Grassmannian
$\operatorname{Gr}_{a,a+b}(\CC)$, see Remark~\ref{rem:lissajous-grassmannian}. 
The~underlying reasons for these combinatorial coincidences 
are yet to be uncovered. 


Our investigations naturally lead us to a number of questions 
concerning the topology of plane curve singularities, 
their real forms, morsifications, divides, braids, quivers, and plabic graphs; 
see in particular Conjectures \ref{conj:plabic-vs-quivers}, 
\ref{conj:positive-isotopy}, and~\ref{conj:alg-is-malleable}
and Problems~\ref{problem:divide-links}, \ref{prob:plabic-transverse}, 
\ref{problem:triangle-moves-and-algebraicity}, and~\ref{problem:yb-equivalence-of morsif}. 
The recent paper~\cite{leviant-shustin} 
by P.~Leviant and the third \linebreak[3]
author was motivated by this work; cf.\ in particular Conjecture~\ref{conj:leviant-shustin}
and Theorem~\ref{th:ag-diagram-determines-topology}. 


The initial impetus for this project came from the desire to
understand the reasons behind the common appearance of the \textit{ADE}
classification, in its version involving quivers, 
in two ostensibly unrelated contexts:  
V.~Arnold's celebrated classification~\cite{arnold-classification} 
of simple singularities
and the much more recent classification of cluster algebras of finite type~\cite{ca2}. 
Indeed, we prove---unconditionally---that our main conjecture holds true for simple singularities 
(resp., quivers of finite type),  see Theorem~\ref{th:simple-singularities}. 

\pagebreak

\begin{center}
\textsc{Organization of the paper}
\end{center}


This paper has at least two intended audiences:
the readers whose primary interests lie either in singularity theory or in 
the theory of cluster algebras. 
Bearing~this in mind, we aimed to make our presentation as accessible as possible
to the mathematicians who might be unfamiliar with one of the two subjects. 
Additional details from singularity theory can be found in the textbooks 
\cite[Sections 2.1--2.2]{aglv1}
\cite[Section~4.1]{agzv2} and in the papers \cite{acampo,
  balke-kaenders,gusein-zade-1,gusein-zade-2, leviant-shustin}. 
For a thorough exposition of the fundamentals~of quiver mutations,
and their relations with cluster algebras, the reader is referred to~\cite{fwz}. 

In Sections~\ref{sec:singularities-and-morsifications}--\ref{sec:ag-diagrams}, 
we review the requisite singularity theory background: 
isolated singularities of complex and real plane curves,  
and their morsifications (Section~\ref{sec:singularities-and-morsifications}); 
divides and their role in the study of plane curve singularities (Section~\ref{sec:divides}); 
and the A'Campo--Guse\u{\i}n-Zade diagrams (Section~\ref{sec:ag-diagrams}). 

In Section~\ref{sec:quivers}, we introduce the quivers of divides, and
show that the topology of a~complex singularity can be recovered from the quiver of its morsification. 
In Section~\ref{sec:main-conjecture}, we define quiver mutations, 
and formulate the first version of our main conjecture (Conjecture~\ref{conj:morsif=mut}), describing 
the putative correspondence between topological equivalence of singularities and mutation equivalence of associated quivers.
 
In Sections~\ref{sec:plabic-graphs} and~\ref{sec:links-of-divides}, 
we reformulate the problem
on both sides of the conjectural correspondence. 
Section~\ref{sec:plabic-graphs} is devoted to the combinatorics of plabic graphs
and local moves on them, in a version slightly different from Postnikov's
original treatment~\cite{postnikov}. We explain that move equivalent plabic graphs
produce mutation equivalent quivers.
(The converse implication is false, but can conjecturally be fixed by
allowing an additional class of transformations called ``switches.'') 
We also explain how divides give rise to plabic graphs, 
sharing the same quivers, up to mutation equivalence.

In Section~\ref{sec:links-of-divides}, we turn to topology.
Following A'Campo, we define the links of divides, and review their main properties. 
We then recast the topological equivalence of singularities 
in terms of divide links: 
as shown by A'Campo, two singularities are equivalent if and only if the links of divides coming from their morsifications are isotopic to each other. 
This enables us to reformulate the main conjecture 
in the language of divide links and plabic graphs,
see Conjectures~\ref{conj:link-equivalence=move-equivalence} and~\ref{conj:morsif=mut-via-divides}.

Sections~\ref{sec:oriented-divides-and-their-links} and~\ref{sec:links-from-plabic} 
are dedicated to the proof of our first main result, Corollary~\ref{cor:p=>d}
(``move equivalence implies link equivalence'').  
In Section~\ref{sec:oriented-divides-and-their-links}, 
we introduce and study oriented divides and associated links,
in particular showing that local moves on oriented divides 
translate into smooth isotopies of their links. 
In Section~\ref{sec:links-from-plabic}, we define the links of plabic graphs, 
and show that Postnikov moves result in link isotopy. 

Section~\ref{sec:transverse}, though optional as far as the main results are concerned, 
clarifies the relationships between the links considered in this paper 
and several important classes of links studied in the literature. 
In particular, we show that all links associated with plabic graphs
(or divides) are both quasipositive and transverse;
that local moves result in transverse isotopies; 
and that for divide links, ordinary isotopy and transverse isotopy are equivalent. 
We note that a related (but different) construction due to 
V.~Shende, D.~Treumann, H.~Williams, and E.~Zaslow \cite{stwz} 
provided a link between cluster algebras and Legendrian 
(rather than transverse) links; cf.\ Remark~\ref{rem:transverse-vs-legendrian}. 

Sections~\ref{sec:scannable-divides}--\ref{sec:yang-baxter-transformations}
are devoted to the reverse direction of the main correspondence
(``link equivalence implies move equivalence'').  
In Section~\ref{sec:scannable-divides}, we discuss   
scannable divides and associated positive braids.
In particular, we recall a beautiful palindromic rule, 
due to O.~Couture and B.~Perron~\cite{couture-perron}, 
for constructing the A'Campo link of a scannable divide 
as the closure of a certain positive braid. 

A distinguished choice of a 
plabic graph attached to a scannable divide, which we call a \emph{plabic fence},
is introduced in Section~\ref{sec:plabic-fences}. 
These fences are closely related to the corresponding braids. 
In~Section~\ref{sec:positive-braid-isotopy}, we introduce the notion of
\emph{positive braid isotopy}, and relate it to move equivalence of plabic fences,
see Theorem~\ref{th:main-conjecture-scannable-1}.
This establishes the desired implication in the case of scannable divides,
modulo the assumption of positive isotopy, cf.\ condition~(ii) above. 

In Section~\ref{sec:yang-baxter-transformations}, we review the relevant properties
of triangle (or ``Yang-Baxter'') moves on divides.  
These moves preserve the associated links, 
and they can be emulated by local moves on plabic graphs. 
This enables us to extend the aforementioned result from the scannable case 
to the generality of arbitrary divides
which can be converted into a scannable form via a sequence of 
triangle moves, cf.\ assumption (i) above. 
This class conjecturally includes all divides coming from morsifications. 

Sections~\ref{sec:overlays} and~\ref{sec:lissajous} are devoted to 
concrete constructions producing scannable real morsifications of plane curve singularities. 
Here we describe a large class of examples coming
from transversal overlays of quasihomogeneous singularities,  
and their morsifications built from \emph{Lissajous divides} and \emph{wiring diagrams}. 

Section~\ref{sec:plabic-graphs-and-links} presents an alternative approach to 
the construction of divide links which employs a particular kind of orientations of plabic graphs. 
These orientations, which are closely related to Postnikov's notion of \emph{perfect} orientations,
always exist in the scannable case. 
Such an orientation can be used to describe the A'Campo link of a divide (or plabic graph)
by a simple local combinatorial rule.
Moreover these links are preserved under Postnikov moves, 
allowing for an elegant and elementary development of the theory. 
Unfortunately, this approach has limited applicability, since there exist divides
(including some which come from morsifications) 
whose associated plabic graphs do not possess an orientation with requisite properties. 

The final Section~\ref{sec:simple-singularities} is devoted to the special case of
simple singularities. In this case, everything works precisely as intended: 
we show that any morsification of a simple singularity gives rise to a quiver of finite type,
and conversely, if a morsification produces a quiver of finite type, then the singularity is simple. 
Moreover the $ADE$ type of a simple singularity matches the cluster type of the
corresponding quiver. 


\vspace{.2in}

\begin{center}
\textsc{Acknowledgments}
\end{center}

\smallskip\nopagebreak

We benefited from stimulating interactions with 
Ian Agol, Sergei Chmutov, John Etnyre, Matthew Hedden, Diana Hubbard, 
Ilia Itenberg, Peter Leviant, Stepan Orevkov, Michael Shapiro, and
Harold Williams. 

\newpage

\section{Singularities and morsifications}
\label{sec:singularities-and-morsifications}


Throughout this paper, the term \emph{singularity}
means a germ $(C,z)\subset\CC^2$ of a reduced  
analytic curve~$C$ in the complex plane~$\CC^2$
at a singular point~$z\in\CC^2$.
We can postulate, without loss of generality, that $z=(0,0)$. 

We shall always assume that our singularity is \emph{isolated}: 
there exists~a closed ball $\Ball=\Ball_{C,z}\subset\CC^2$ centered at~$z$ 
such that $z$ is the only singular point of $C$~in~$\Ball$. \linebreak[3]
Moreover we can assume that any sphere centered at~$z$ and contained
in~$\Ball$ intersects our curve~$C$ transversally;
we then call $\Ball$ the \emph{Milnor ball} at~$z$. 

The simplest example of a singularity is a \emph{node}, i.e., 
a transversal intersection of two locally smooth branches. 

Isolated singularities of plane complex curves can be
studied up to different types of equivalence.  
Here we focus on the \emph{topological} theory, 
which considers singularities up to homeo\-mor\-phisms of 
a neighborhood of an isolated singular point. 
We note that this point of view is substantially different from treating singularities
up to diffeomorphisms, cf.\ Example~\ref{example:4-transversal-branches}. 

\begin{example}
\label{example:4-transversal-branches}
All singularities consisting of four smooth branches transversally crossing at
a point $z\in\CC^2$ are 
topologically equivalent to each other.
On the other hand, any diffeomorphism of a neighborhood of~$z$ 
preserves the cross-ratio of the tangent lines (at~$z$) 
to the four branches, so configurations with different cross-ratios
are not equivalent to each other in the smooth category. 
\end{example}

By a theorem of Weierstrass (see, e.g., \cite[Theorems I.1.6 and ~I.1.8]{GLS} 
or \cite[Section~III.8.2]{brieskorn-knorrer}),
any locally convergent power series $f(x,y)$ splits into a product of
irreducible factors that are also locally convergent. 
In the case under consideration (an isolated singularity $f(x,y)=0$) 
we can choose~$\Ball$ so that everything converges there.
The factors are determined uniquely up to permutation, 
and up to multiplication by a unit 
(i.e., by a nonvanishing function $\Ball\to\CC$).
Since we assume that the curve is reduced, 
the factors are pairwise distinct: no two of them differ by a unit. 
These factors correspond to
the \emph{local branches} of the singularity. 

\begin{definition}
\label{def:quasi-homogeneous}
A singularity $(C,z)$ is called \emph{quasihomogeneous}
of type~$(a,b)$ 
(for $a\ge b\ge 2$) if, in suitable local coordinates, $(C,z)$ can be given by an equation of the form
\begin{equation}
\label{eq:quasihom-ab}
f(x,y)=\sum_{\substack{bi+aj=ab\\ i,j\ge0}} c_{ij}x^iy^j=0, 
\end{equation}
with $z\!=\!(0,0)$; here $f(x,y)$ must not contain multiple irreducible~factors. 
Any quasihomogeneous singularity of type~$(a,b)$ 
is topologically equivalent to the singularity 
\begin{equation}
\label{eq:x^a+y^b=0}
x^a+y^b=0. 
\end{equation}
The number of (complex) branches of the singularity \eqref{eq:x^a+y^b=0} 
is $\operatorname{gcd}(a,b)$.  
\end{definition}

\begin{remark}
Suppose a singularity $(C,z)$ is given (in some local coordinates) by an equation of the form 
\[
\sum_{\substack{bi+aj\ge ab\\ i,j\ge0}} c_{ij}x^iy^j=0, 
\]
with $z=(0,0)$. Assume that the corresponding equation \eqref{eq:quasihom-ab} 
defines a quasiho\-mo\-geneous singularity of type $(a,b)$.
Then $(C,z)$ is topologically equivalent to the quasihomogeneous singularity~\eqref{eq:x^a+y^b=0}. This follows by combining \cite[Theorem in Section~12.2, page 194]{agzv1} and \cite[Theorem~2.1]{le-ramanujam}.
\end{remark}

\begin{definition}
One important topological invariant of a singularity is its 
\emph{Milnor number}~\cite[\S~7]{Milnor}. 
Let $(C,z)$ be a singularity given by an equation 
$f(x,y)=0$. 
Algebraically, the Milnor number $\mu(C,z)$ is given by 
\[
\mu(C,z)=\dim_\CC\left(\CC[[x,y]]/
\langle \textstyle\frac{\partial f}{\partial x}, \textstyle\frac{\partial f}{\partial y}\rangle\right), 
\]
the dimension of the quotient of the algebra of power series in $x$ and~$y$
by its Jacobian ideal. 
%
To illustrate, the Milnor number of the quasihomogeneous 
singularity \eqref{eq:x^a+y^b=0} is equal to $(a-1)(b-1)$. 
\end{definition}

The Milnor number can also be defined as the maximal number of critical points 
(in the vicinity of~$z$) that a small deformation of~$f$ may have. 
Informally, the Milnor number measures the complexity of the singular point 
considered as the critical point of the defining function.
See, e.g., \cite[Section~2.1]{GLS} for further discussion. 

Any deformation of 
an isolated plane curve singularity that keeps the Milnor number constant yields topological equivalence, see~\cite{le-ramanujam}. 


\begin{definition}
A plane curve singularity $(C,z)$ is called \emph{real} if $C\subset\CC^2$ is an analytic curve 
invariant under complex conjugation,  
and $z\!\in\! C$ its real singular~point. 
Equivalently, $C$~is given by an equation $f(x,y)=0$ 
where all coefficients in the power series expansion of~$f$ 
(at~$z$) are real. 
\end{definition}

The simplest example of a real singularity is a real node of a real
plane curve. 
Such a node can be either \emph{hyperbolic} or
\emph{elliptic}, 
i.e., analytically equivalent over~$\RR$ to $x^2-y^2=0$
or to $x^2+y^2=0$, respectively.

\begin{definition}
A~(real) singularity $(C,z)$ is called \emph{totally real} 
if all its local branches are real. 
For example, a hyperbolic node is totally real, but an elliptic one is not. 
\end{definition}

\begin{theorem}[{\cite[Theorem~3]{gusein-zade-2}
\cite[Theorem~1.1]{balke-kaenders}}]
\label{th:totally-real-form-exists}
Every complex plane curve singularity is topologically equivalent to a
totally real one.
\end{theorem}

A real singularity topologically equivalent to a given complex singularity 
is called a \emph{real form} of the latter. 
By Theorem~\ref{th:totally-real-form-exists},
any complex plane curve singularity has a real form.  
There are typically several distinct real forms, 
up to conjugation-equivariant 
topological equivalence. 
For example, a complex node has two essentially different real forms: 
hyperbolic and elliptic. 
An irreducible complex singularity (i.e., one that has a single branch) 
has only one real form. 

\pagebreak[3]


One of our implicit goals is to better understand the relations between different real 
forms of the same complex singularity. 

\begin{definition}
A \emph{nodal deformation} of a singularity $(C,z)$ 
inside the Milnor ball~$\Ball$ is an analytic
family of curves $C_t\cap \Ball$ such that 
\begin{itemize}[leftmargin=.3in]
\item 
the complex parameter~$t$ varies 
in a (small) disk centered at $0\in\CC$; 
\item
for $t=0$, we recover the original curve: $C_0=C$; 
\item
each curve 
$C_t$ is smooth along $\partial \Ball$, and intersects 
$\partial \Ball$ transversally; 
\item
for any $t\ne0$, the curve $C_t$ has only ordinary nodes inside
$\Ball$; 
\item 
the number of these nodes does not depend on~$t$.
\end{itemize}
The maximal number of nodes in a nodal deformation 
is~$\delta(C,z)$,   the \emph{$\delta$-invariant} of the singularity; 
see, e.g., \cite[\S10]{Milnor}.
\end{definition}

\begin{definition}
A \emph{real nodal deformation} of a real singularity $(C,z)$ 
is obtained by taking a nodal deformation $(C_t\cap \Ball)$ 
which is equivariant with respect to complex conjugation, 
and restricting the parameter~$t$ to a (small) interval
$[0,\tau)\subset\RR$. 
\end{definition}

\begin{definition}
A \emph{real morsification} of a real singularity $(C,z)$ is a real
nodal deformation $C_t=\{f_t(x,y)=0\}$ as above
such that 
\begin{itemize}[leftmargin=.3in]
\item 
all critical points of $f_t$ are real and Morse (i.e., with nondegenerate Hessian);  
\item 
all saddle points of $f_t$ are at the zero level (i.e., lie
on~$C_t$). 
\end{itemize}
\end{definition}

Real morsifications of totally real singularities
have been successfully used to compute the monodromies and the
intersection forms of plane curve singularities
\cite{acampo,AC1,gusein-zade-1,GZ1}. 

\begin{proposition}[{\cite[Lemma~2]{leviant-shustin}}]
\label{pr:number-of-real-nodes}
The number of real hyperbolic nodes in 
any real nodal deformation of a real singularity $(C,z)$ is at most 
\begin{equation}
\label{eq:delta-R}
\delta_\RR(C,z)\stackrel{\rm def}{=}
\delta(C,z)-\imbr(C,z), 
\end{equation}
where $\delta(C,z)$ is the $\delta$-invariant of the singularity,
and $\,\imbr(C,z)$ denotes the number of pairs of distinct complex
  conjugate local branches of~$C$ centered at $z$.
Moreover it is equal to $\delta_\RR(C,z)$ if and only if this real
nodal deformation is a real morsification. 
\end{proposition}

Thus a real morsification is a real
nodal deformation that has 
$
\delta_\RR(C,z)$ real hyperbolic nodes, 
the maximal possible number. 
Cf.~Figure~\ref{fig:morsification-or-not}. 

\begin{figure}[ht]
\begin{center}
\begin{tabular}{ccccc}
\setlength{\unitlength}{0.8pt}
\begin{picture}(120,60)(0,-30)
\thinlines
\put(7.5,10){\red{\qbezier(0,15)(20,-10)(60,-10)}}
\put(7.5,-10){\red{\qbezier(0,-15)(20,10)(60,10)}}
\put(-7.5,10){\qbezier(135,15)(115,-10)(75,-10)}
\put(-7.5,-10){\qbezier(135,-15)(115,10)(75,10)}
\thicklines
\put(67.5,-42.5){\makebox(0,0){singularity}} 
\end{picture}
&\qquad  &
\setlength{\unitlength}{0.8pt}
\begin{picture}(120,60)(-10,-30)
\thinlines
\put(0,0){\red{\qbezier(0,20)(40,-10)(60,-10)}}
\put(0,0){\red{\qbezier(60,-10)(80,-10)(80,0)}}
\put(0,0){\red{\qbezier(60,10)(80,10)(80,0)}}
\put(0,0){\red{\qbezier(0,-20)(40,10)(60,10)}}
\put(0,0){\qbezier(135,20)(95,-10)(75,-10)}
\put(0,0){\qbezier(75,-10)(55,-10)(55,0)}
\put(0,0){\qbezier(75,10)(55,10)(55,0)}
\put(0,0){\qbezier(135,-20)(95,10)(75,10)}
\thicklines
\put(67.5,-40){\makebox(0,0){not a morsification}} 
\end{picture}
&\qquad\qquad &
\setlength{\unitlength}{0.8pt}
\begin{picture}(80,60)(0,-30)
\thinlines
\put(0,0){\red{\qbezier(0,20)(40,-10)(60,-10)}}
\put(0,0){\red{\qbezier(60,-10)(80,-10)(80,0)}}
\put(0,0){\red{\qbezier(60,10)(80,10)(80,0)}}
\put(0,0){\red{\qbezier(0,-20)(40,10)(60,10)}}
\put(-40,0){\qbezier(135,20)(95,-10)(75,-10)}
\put(-40,0){\qbezier(75,-10)(55,-10)(55,0)}
\put(-40,0){\qbezier(75,10)(55,10)(55,0)}
\put(-40,0){\qbezier(135,-20)(95,10)(75,10)}
\thicklines
\put(47.5,-40){\makebox(0,0){morsification}} 
\end{picture}
\end{tabular}
\quad{\ }
\vspace{.1in}
\end{center}
\caption{A quasihomogeneous plane curve singularity of type $(6,4)$ (two cusps sharing a tangent) and its two real nodal deformations. The first deformation is not a morsification, as it only has 4~real nodes. The second one has 8 real nodes, and is a morsification. 
In this example, $\delta(C,z)=\delta_\RR(C,z)=8$. 
}
\label{fig:morsification-or-not}
\end{figure}

\begin{conjecture} 
\label{conj:leviant-shustin}
Any real plane curve singularity possesses a real morsification.
\end{conjecture}

The totally real case of Conjecture~\ref{conj:leviant-shustin} 
was settled long time ago in \cite[Theorem~1]{acampo} and \cite[Theorem~4]{gusein-zade-2}. 
Much more recently, Conjecture~\ref{conj:leviant-shustin}  was 
established in \cite[Theorem~1]{leviant-shustin}
for a wide class of singularities that in particular includes 
all singularities that can be represented as a union of a totally real singularity 
with semi-quasihomogeneous singularities having distinct non-real tangents.

\pagebreak[3]

\begin{example}
\label{example:4-lines}
Consider the complex singularity with four
smooth branches intersecting transversally at the point $z=(0,0)$,
cf.\ Example~\ref{example:4-transversal-branches}. 
Its three essentially distinct real forms,
and their respective morsifications, are shown
in Figure~\ref{fig:4-lines}. 
\end{example}


\begin{figure}[ht]
\begin{center}
\begin{tabular}{c|cc}
\hline
\text{real singularity}
& \hfill morsifications \qquad\  &  
\\
\hline
\hline
\begin{tabular}{c} $x^3y-xy^3=0$ \\ (four real
  branches) \end{tabular}
& $xy(x-y+t)(x+y-2t)=0$ & \setlength{\unitlength}{.5pt}
\begin{picture}(80,75)(5,65)
\thinlines

\put(0,60){\line(1,0){100}}
\put(40,20){\line(0,1){100}}
\put(0,40){\line(1,1){80}}
\put(20,120){\line(1,-1){80}}

\end{picture}
\\[.35in]
\hline
\\[-.45in]
\begin{tabular}{c}\ \\[.3in]
$x^4-y^4=0$ \\ (two real branches,\\ 
two complex\\
  conjugate branches)
\end{tabular}
& 
\begin{tabular}{c}
\ \\[.45in]
$(x^2-y^2)(x^2+y^2-t^2)=0$ \\[.45in]
$(x^2-(y-1.2t)^2)(x^2+y^2-t^2)=0$
\end{tabular}
&
\begin{tabular}{c}
\setlength{\unitlength}{0.5pt}
\begin{picture}(80,100)(25,50)
\thinlines
\put(20,20){\line(1,1){80}}
\put(20,100){\line(1,-1){80}}
\put(60,60){\circle{56.5}}
\end{picture}
\\[.05in]
\setlength{\unitlength}{0.5pt}
\begin{picture}(80,80)(25,70)
\thinlines
\put(20,53){\line(1,1){60}}
\put(100,53){\line(-1,1){60}}
\put(60,60){\circle{56.5}}
\end{picture}
\end{tabular}
\\[.9in]
\hline
\\[-.45in]
 \begin{tabular}{c} $(x^2+4y^2)(4x^2+y^2)=0$ \\ (two pairs of complex\\
  conjugate branches) \end{tabular}
& $(x^2+4y^2-t^2)(4x^2+y^2-t^2)=0 $ &
\setlength{\unitlength}{0.5pt}
\begin{picture}(80,100)(25,50)
\thinlines
\qbezier(20,60)(20,82)(60,82)
\qbezier(20,60)(20,38)(60,38)
\qbezier(100,60)(100,82)(60,82)
\qbezier(100,60)(100,38)(60,38)
\qbezier(60,20)(38,20)(38,60)
\qbezier(60,100)(38,100)(38,60)
\qbezier(60,20)(82,20)(82,60)
\qbezier(60,100)(82,100)(82,60)
\end{picture}
\\[.23in]
\hline
\end{tabular}
\end{center}
\caption{Three real forms of the singularity from
  Example~\ref{example:4-lines}, and their morsifications.
}
\label{fig:4-lines}
\end{figure}

\vspace{-.15in} 

\begin{example}
Two morsifications of (different real forms of)
the quasihomogeneous singularity of type $(4,2)$
(cf.\ Definition~\ref{def:quasi-homogeneous}) 
are given by $y^2+x^4=t x^2$ 
(a~lemniscate)  and $(x^2-t)^2=y^2$ (two parabolas).
\end{example}

\begin{remark}
The topology of complex singularities of the kind considered in this paper is completely characterized 
by certain combinatorial invariants defined either in terms
of Puiseux expansions (or ``resolution trees,'' with multiplicities) or
equivalently in terms of the topology of a certain link (the Burau-Zariski 
Theorem~\cite{burau, zariski}, see \cite[Chapter~8]{brieskorn-knorrer}); cf.\ also Proposition~\ref{pr:link-determines-singularity} below.
\end{remark}

\clearpage

\newpage

\section{Divides}
\label{sec:divides}

In this section, we recall and discuss the concept of a divide, 
introduced and extensively studied by
N.~A'Campo, see \cite{acampo-2000, acampo-2003, ishikawa} and
references therein. 
There are several versions of this notion in the existing literature;
we will use the following~one. 

\begin{definition}
\label{def:divide}
Loosely speaking, a \emph{divide}~$D$ in a closed disk $\Disk\subset \RR^2$ is the image of a generic relative immersion of a finite set of intervals and circles into~$\Disk$. 
More precisely, the images of immersed intervals and circles,  
collectively called the \emph{branches} of~$D$, 
must satisfy the conditions (D1)--(D6) below. In particular: 
\begin{itemize}[leftmargin=.4in]
\item[(D1)]
the immersed circles do not intersect the boundary~$\partial\Disk$; 
\item[(D2)]
the immersed intervals have pairwise distinct endpoints which lie on~$\partial\Disk$;
moreover these 
immersed intervals intersect $\partial\Disk$ transversally; 
\item[(D3)]
all intersections and self-intersections of the branches 
are transversal; 
\item[(D4)]
no triple (self-)intersections are allowed. 
\end{itemize}
We are only interested in the topology of a divide. 
That is, we do not distinguish between divides related by a
diffeomorphism between their respective ambient disks. 

The connected components of the complement $\Disk\setminus D$ 
which are disjoint from $\partial\Disk$ are the
\emph{regions} of~$D$.
The closure of the union of all regions and all singular points of~$D$
(its \emph{nodes}) 
is called the \emph{body} of the divide, denoted~$I(D)$. 
We require that
\begin{itemize}[leftmargin=.4in]
\item[(D5)]
the body of the divide is connected, as is the union of its branches; 
\item[(D6)]
each region is homeomorphic to an open disk. 
\end{itemize}
In what follows, we don't always draw the boundary of the ambient disk~$\Disk$. 
\end{definition}

\begin{definition}
\label{def:divide-of-a-morsification}
Any real morsification $(C_t)_{t\in[0,\tau)}$ of a real plane curve singularity $(C,z)$ defines a
divide in the following natural way. 
The sets $\RR C_t$ of real points of the deformed curves~$C_t$, 
for $0<t<\tau$, are all isotopic to each other in the ``Milnor disk'' 
$\Disk=\RR\Ball\subset\RR^2$
consisting of the real points of the Milnor ball~$\Ball$. \linebreak[3]
Each real curve $\RR C_t\cap \Disk$, viewed up to
isotopy,  
defines the divide associated with the morsification. 
Conditions (D1)--(D4) and~(D6) of Definition~\ref{def:divide} are readily
checked. 
Condition~(D5) follows from the connectedness of the Dynkin diagram 
of a singularity~\cite{gabrielov} 
and from Guse\u{\i}n-Zade's algorithm~\cite{gusein-zade-2} 
that constructs this diagram from a divide, cf.\ Section~\ref{sec:ag-diagrams}.

A simple example is given in Figure~\ref{fig:x^4-y^4=0}. 
\end{definition}


\vspace{-.1in}

\begin{figure}[ht]
\begin{center}
\begin{tabular}{ccc}
\setlength{\unitlength}{0.48pt}
\begin{picture}(120,100)(0,20)
\thicklines
\put(20,20){\line(1,1){80}}
\put(20,100){\line(1,-1){80}}
\put(60,60){\circle{56.5}}
\thinlines
\put(60,60){\circle{113}}
\put(60,60){\circle{107}}
\end{picture}
&\qquad\qquad &
\setlength{\unitlength}{0.48pt}
\begin{picture}(120,100)(0,20)
\thicklines
\put(80,113){\line(-1,-1){73}}
\put(40,113){\line(1,-1){73}}
\put(60,60){\circle{56.5}}
\thinlines
\put(60,60){\circle{113}}
\put(60,60){\circle{107}}
\end{picture}
\end{tabular}
\end{center}
\caption{Divides associated with two real morsifications 
of the real singularity $x^4-y^4=0$ shown in Figure~\ref{fig:4-lines}
(second row). The circle~$\partial\Disk$ is represented by double lines. 
In each case, the two real branches $x\pm y=0$ get deformed
into~two immersed segments, 
and the two complex conjugate branches $x\pm iy=0$ get deformed into
an immersed circle.
Each divide 
has $4$~regions and $5$~nodes. 
}
\label{fig:x^4-y^4=0}
\end{figure}

Several examples of divides associated with morsifications of 
(various real forms of) quasihomogeneous singularities
$x^a+y^b=0$
are shown in Figure~\ref{fig:divides-quasihom}. 

\newsavebox{\crossing}
\setlength{\unitlength}{1.5pt} 
\savebox{\crossing}(10,10)[bl]{
\thicklines 
\qbezier(5,5)(7,10)(10,10)
\qbezier(5,5)(3,0)(0,0)
\qbezier(5,5)(3,10)(0,10)
\qbezier(5,5)(7,0)(10,0)
}

\newsavebox{\closing}
\setlength{\unitlength}{1.5pt} 
\savebox{\closing}(10,10)[bl]{
\thicklines 
\qbezier(0,0)(5,0)(5,5)
\qbezier(0,10)(5,10)(5,5)
}

\newsavebox{\cclosing}
\setlength{\unitlength}{1.5pt} 
\savebox{\cclosing}(10,30)[bl]{
\thicklines 
\qbezier(0,0)(10,0)(10,15)
\qbezier(0,30)(10,30)(10,15)
}

\newsavebox{\opening}
\setlength{\unitlength}{1.5pt} 
\savebox{\opening}(10,10)[bl]{
\thicklines 
\qbezier(0,0)(-5,0)(-5,5)
\qbezier(0,10)(-5,10)(-5,5)

}

\newsavebox{\ssone}
\setlength{\unitlength}{1.5pt} 
\savebox{\ssone}(10,20)[bl]{
\thicklines 
\qbezier(5,5)(7,10)(10,10)
\qbezier(5,5)(3,0)(0,0)
\qbezier(5,5)(3,10)(0,10)
\qbezier(5,5)(7,0)(10,0)
\put(0,20){\line(1,0){10}} 
}

\newsavebox{\sstwo}
\setlength{\unitlength}{1.5pt} 
\savebox{\sstwo}(10,20)[bl]{
\thicklines 
\qbezier(5,15)(7,20)(10,20)
\qbezier(5,15)(3,10)(0,10)
\qbezier(5,15)(3,20)(0,20)
\qbezier(5,15)(7,10)(10,10)
\put(0,0){\line(1,0){10}} 
}

\newsavebox{\sssone}
\setlength{\unitlength}{1.5pt} 
\savebox{\sssone}(10,30)[bl]{
\thicklines 
\qbezier(5,5)(7,10)(10,10)
\qbezier(5,5)(3,0)(0,0)
\qbezier(5,5)(3,10)(0,10)
\qbezier(5,5)(7,0)(10,0)
\put(0,20){\line(1,0){10}} 
\put(0,30){\line(1,0){10}} 
}

\newsavebox{\ssstwo}
\setlength{\unitlength}{1.5pt} 
\savebox{\ssstwo}(10,30)[bl]{
\thicklines 
\qbezier(5,15)(7,20)(10,20)
\qbezier(5,15)(3,10)(0,10)
\qbezier(5,15)(3,20)(0,20)
\qbezier(5,15)(7,10)(10,10)
\put(0,0){\line(1,0){10}} 
\put(0,30){\line(1,0){10}} 
}

\newsavebox{\sssthree}
\setlength{\unitlength}{1.5pt} 
\savebox{\sssthree}(10,30)[bl]{
\thicklines 
\qbezier(5,25)(7,30)(10,30)
\qbezier(5,25)(3,20)(0,20)
\qbezier(5,25)(3,30)(0,30)
\qbezier(5,25)(7,20)(10,20)
\put(0,0){\line(1,0){10}} 
\put(0,10){\line(1,0){10}} 
}

\newsavebox{\sssonethree}
\setlength{\unitlength}{1.5pt} 
\savebox{\sssonethree}(10,30)[bl]{
\thicklines 
\qbezier(5,25)(7,30)(10,30)
\qbezier(5,25)(3,20)(0,20)
\qbezier(5,25)(3,30)(0,30)
\qbezier(5,25)(7,20)(10,20)
\qbezier(5,5)(7,10)(10,10)
\qbezier(5,5)(3,0)(0,0)
\qbezier(5,5)(3,10)(0,10)
\qbezier(5,5)(7,0)(10,0)
}

\begin{figure}[htbp] 
\begin{center} 
\vspace{-.1in}
\begin{tabular}{c|c|c|c|c|c}
& $a\!=\!2$ & $a\!=\!3$ & $a\!=\!4$ & $a\!=\!5$ & $a\!=\!6$
\\[.0in]
\hline
&&&&&\\[-.18in]
&\setlength{\unitlength}{1.4pt} 
\begin{picture}(10,10)(-4,0) 
\put(0,0){\makebox(0,0){\usebox{\crossing}}} 
\end{picture} 
&
\begin{picture}(15,10)(-5,0)
\put(0,0){\makebox(0,0){\usebox{\crossing}}} 
\put(10,0){\makebox(0,0){\usebox{\closing}}} 
\end{picture} 
&
\begin{picture}(20,10)(-5,0)
\put(0,0){\makebox(0,0){\usebox{\crossing}}} 
\put(10,0){\makebox(0,0){\usebox{\crossing}}} 
\end{picture} 
&
\begin{picture}(25,10)(-5,0)
\put(0,0){\makebox(0,0){\usebox{\crossing}}} 
\put(10,0){\makebox(0,0){\usebox{\crossing}}} 
\put(20,0){\makebox(0,0){\usebox{\closing}}} 
\end{picture} 
&
\begin{picture}(30,10)(-5,0)
\put(0,0){\makebox(0,0){\usebox{\crossing}}} 
\put(10,0){\makebox(0,0){\usebox{\crossing}}} 
\put(20,0){\makebox(0,0){\usebox{\crossing}}} 
\end{picture} 
\\[.05in]
\rotatebox[origin=r]{90}{$b\!=\!2$}&\begin{picture}(10,10)(-10,0) 
\put(0,0){\makebox(0,0){\usebox{\opening}}} 
\put(0,0){\makebox(0,0){\usebox{\closing}}} 
\end{picture} 
& &
\begin{picture}(20,10)(-10,0)
\put(0,0){\makebox(0,0){\usebox{\opening}}} 
\put(0,0){\makebox(0,0){\usebox{\crossing}}} 
\put(10,0){\makebox(0,0){\usebox{\closing}}} 
\end{picture} 
& & 
\begin{picture}(30,10)(-10,0)
\put(0,0){\makebox(0,0){\usebox{\opening}}} 
\put(0,0){\makebox(0,0){\usebox{\crossing}}} 
\put(10,0){\makebox(0,0){\usebox{\crossing}}} 
\put(20,0){\makebox(0,0){\usebox{\closing}}} 
\end{picture} 
\\[-0.05in]
&$A_1$ & $A_2$ & $A_3$ & $A_4$ & $A_5$
\\[.0in]
& node & cusp & tacnode & & 
order~$3$ tangency \\[.05in]
\hline
&&&&&\\[-.25in]
&& \begin{picture}(30,20)(-5,0)
\put(0,0){\makebox(0,0){\usebox{\ssone}}} 
\put(10,0){\makebox(0,0){\usebox{\sstwo}}} 
\put(20,0){\makebox(0,0){\usebox{\ssone}}} 
\end{picture} 
& \begin{picture}(40,20)(-10,0)
\put(0,5){\makebox(0,0){\usebox{\opening}}} 
\put(30,5){\makebox(0,0){\usebox{\closing}}} 
\put(0,0){\makebox(0,0){\usebox{\ssone}}} 
\put(10,0){\makebox(0,0){\usebox{\sstwo}}} 
\put(20,0){\makebox(0,0){\usebox{\ssone}}} 
\end{picture} 
& \begin{picture}(50,20)(-10,0)
\put(0,5){\makebox(0,0){\usebox{\opening}}} 
\put(0,0){\makebox(0,0){\usebox{\ssone}}} 
\put(10,0){\makebox(0,0){\usebox{\sstwo}}} 
\put(20,0){\makebox(0,0){\usebox{\ssone}}} 
\put(30,0){\makebox(0,0){\usebox{\sstwo}}} 
\put(40,-5){\makebox(0,0){\usebox{\closing}}} 
\end{picture} 
& \begin{picture}(60,20)(-10,0)
\put(0,5){\makebox(0,0){\usebox{\opening}}} 
\put(0,0){\makebox(0,0){\usebox{\ssone}}} 
\put(10,0){\makebox(0,0){\usebox{\sstwo}}} 
\put(20,0){\makebox(0,0){\usebox{\ssone}}} 
\put(30,0){\makebox(0,0){\usebox{\sstwo}}} 
\put(40,0){\makebox(0,0){\usebox{\ssone}}} 
\put(50,5){\makebox(0,0){\usebox{\closing}}} 
\end{picture} 
\\[.1in]
\rotatebox[origin=r]{90}{$b\!=\!3$\qquad}
&&\begin{picture}(30,20)(-10,0)
\put(0,0){\makebox(0,0){\usebox{\ssone}}} 
\put(10,0){\makebox(0,0){\usebox{\sstwo}}} 
\put(20,-5){\makebox(0,0){\usebox{\closing}}} 
\put(0,5){\makebox(0,0){\usebox{\opening}}} 
\end{picture} 
&\begin{picture}(40,20)(-10,0)
\put(0,5){\makebox(0,0){\usebox{\opening}}} 
\put(30,5){\makebox(0,0){\usebox{\closing}}} 
\put(0,0){\makebox(0,0){\usebox{\sstwo}}} 
\put(10,0){\makebox(0,0){\usebox{\ssone}}} 
\put(20,0){\makebox(0,0){\usebox{\sstwo}}} 
\end{picture} 
&\begin{picture}(50,20)(-10,0)
\put(0,5){\makebox(0,0){\usebox{\opening}}} 
\put(0,0){\makebox(0,0){\usebox{\sstwo}}} 
\put(10,0){\makebox(0,0){\usebox{\ssone}}} 
\put(20,0){\makebox(0,0){\usebox{\sstwo}}} 
\put(30,0){\makebox(0,0){\usebox{\sstwo}}} 
\put(40,-5){\makebox(0,0){\usebox{\closing}}} 
\end{picture} 
&\begin{picture}(60,20)(-10,0)
\put(0,5){\makebox(0,0){\usebox{\opening}}} 
\put(0,0){\makebox(0,0){\usebox{\ssone}}} 
\put(10,0){\makebox(0,0){\usebox{\ssone}}} 
\put(20,0){\makebox(0,0){\usebox{\sstwo}}} 
\put(30,0){\makebox(0,0){\usebox{\ssone}}} 
\put(40,0){\makebox(0,0){\usebox{\ssone}}} 
\put(50,5){\makebox(0,0){\usebox{\closing}}} 
\end{picture} 
\\[-.45in]
&&&&
\begin{picture}(50,20)(-10,0)
\put(0,5){\makebox(0,0){\usebox{\opening}}} 
\put(0,0){\makebox(0,0){\usebox{\sstwo}}} 
\put(10,0){\makebox(0,0){\usebox{\sstwo}}} 
\put(20,0){\makebox(0,0){\usebox{\ssone}}} 
\put(30,0){\makebox(0,0){\usebox{\sstwo}}} 
\put(40,5){\makebox(0,0){\usebox{\closing}}} 
\end{picture} 
&
\begin{picture}(60,20)(-5,0)
\put(0,0){\makebox(0,0){\usebox{\ssone}}} 
\put(10,0){\makebox(0,0){\usebox{\sstwo}}} 
\put(20,0){\makebox(0,0){\usebox{\ssone}}} 
\put(30,0){\makebox(0,0){\usebox{\sstwo}}} 
\put(40,0){\makebox(0,0){\usebox{\ssone}}} 
\put(50,0){\makebox(0,0){\usebox{\sstwo}}} 
\end{picture} 
\\[.25in]
&& $D_4$ 
& $E_6$ 
& $E_8$ 
& 
$E_8^{(1,1)}$
\\[.0in]
&& $3$ lines& & & $3$ tangent branches 
\\[.05in]
\hline
&&&&& \\[-.35in]
&&&
\begin{picture}(60,30)(-10,0)
\put(0,0){\makebox(0,0){\usebox{\sssone}}} 
\put(10,0){\makebox(0,0){\usebox{\ssstwo}}} 
\put(20,0){\makebox(0,0){\usebox{\sssonethree}}} 
\put(30,0){\makebox(0,0){\usebox{\ssstwo}}} 
\put(40,0){\makebox(0,0){\usebox{\sssone}}} 
\end{picture}
 &
\begin{picture}(50,30)(-10,0)
\put(0,0){\makebox(0,0){\usebox{\opening}}} 
\put(0,0){\makebox(0,0){\usebox{\sssonethree}}} 
\put(10,0){\makebox(0,0){\usebox{\ssstwo}}} 
\put(20,0){\makebox(0,0){\usebox{\sssonethree}}} 
\put(30,0){\makebox(0,0){\usebox{\ssstwo}}} 
\put(40,-10){\makebox(0,0){\usebox{\closing}}} 
\put(40,10){\makebox(0,0){\usebox{\closing}}} 
\end{picture}
&
\begin{picture}(60,30)(-10,0)
\put(0,-10){\makebox(0,0){\usebox{\opening}}} 
\put(0,10){\makebox(0,0){\usebox{\opening}}} 
\put(0,0){\makebox(0,0){\usebox{\ssstwo}}} 
\put(10,0){\makebox(0,0){\usebox{\sssonethree}}} 
\put(20,0){\makebox(0,0){\usebox{\ssstwo}}} 
\put(30,0){\makebox(0,0){\usebox{\sssonethree}}} 
\put(40,0){\makebox(0,0){\usebox{\ssstwo}}} 
\put(50,-10){\makebox(0,0){\usebox{\closing}}} 
\put(50,10){\makebox(0,0){\usebox{\closing}}} 
\end{picture}
\\[.1in]
&&&
\begin{picture}(40,30)(-10,0)
\put(0,0){\makebox(0,0){\usebox{\sssonethree}}} 
\put(10,0){\makebox(0,0){\usebox{\ssstwo}}} 
\put(20,0){\makebox(0,0){\usebox{\sssonethree}}} 
\put(30,0){\makebox(0,0){\usebox{\closing}}} 
\put(0,0){\makebox(0,0){\usebox{\opening}}} 
\end{picture} &&
\begin{picture}(60,30)(-10,0)
\put(0,0){\makebox(0,0){\usebox{\opening}}} 
\put(0,0){\makebox(0,0){\usebox{\sssonethree}}} 
\put(10,0){\makebox(0,0){\usebox{\ssstwo}}} 
\put(20,0){\makebox(0,0){\usebox{\sssonethree}}} 
\put(30,0){\makebox(0,0){\usebox{\ssstwo}}} 
\put(40,0){\makebox(0,0){\usebox{\sssonethree}}} 
\put(50,0){\makebox(0,0){\usebox{\closing}}} 
\end{picture}
\\[.1in]
\rotatebox[origin=c]{90}{$b\!=\!4$}
&&&\begin{picture}(60,30)(-10,0)
\put(0,0){\makebox(0,0){\usebox{\sssone}}} 
\put(10,0){\makebox(0,0){\usebox{\ssstwo}}} 
\put(20,0){\makebox(0,0){\usebox{\sssthree}}} 
\put(30,0){\makebox(0,0){\usebox{\ssstwo}}} 
\put(40,0){\makebox(0,0){\usebox{\sssone}}} 
\put(0,0){\makebox(0,0){\usebox{\opening}}} 
\put(50,0){\makebox(0,0){\usebox{\closing}}} 
\end{picture} 
&&\begin{picture}(60,30)(-5,0)
\put(0,0){\makebox(0,0){\usebox{\sssonethree}}} 
\put(10,0){\makebox(0,0){\usebox{\sssonethree}}} 
\put(20,0){\makebox(0,0){\usebox{\ssstwo}}} 
\put(30,0){\makebox(0,0){\usebox{\sssonethree}}} 
\put(40,0){\makebox(0,0){\usebox{\ssstwo}}} 
\put(50,0){\makebox(0,0){\usebox{\closing}}} 
\put(50,0){\makebox(0,0){\usebox{\cclosing}}} 
\end{picture} 
\\[.1in]
&&&
\begin{picture}(40,30)(0,0)
\put(10,0){\makebox(0,0){\usebox{\ssstwo}}} 
\put(20,0){\makebox(0,0){\usebox{\sssonethree}}} 
\put(30,0){\makebox(0,0){\usebox{\ssstwo}}} 
\put(10,-10){\makebox(0,0){\usebox{\opening}}} 
\put(10,10){\makebox(0,0){\usebox{\opening}}} 
\put(40,-10){\makebox(0,0){\usebox{\closing}}} 
\put(40,10){\makebox(0,0){\usebox{\closing}}} 
\end{picture}&&
\\[.35in]
&&&$
E_7^{(1,1)}$ 
& 
& $2$ cusps with 
\\[.0in]
&&& $4$ lines (cf.~Fig.~\ref{fig:4-lines})&& the same tangent
\\[-.1in]
\end{tabular}
\end{center} 
\caption{Divides associated with real morsifications of (different real forms of) 
  quasihomogeneous singularities $x^a+y^b=0$, for $2\le b\le a\le 6$ and $b\le 4$. 
} 
\label{fig:divides-quasihom} 
\end{figure} 


A few additional examples are given in Figures~\ref{fig:divides-D5-D6-E7}--\ref{fig:divides-two-transversal cusps}. 

\vspace{.1in}

\begin{figure}[ht] 
\begin{center} 
\begin{tabular}{c|c|c}
\begin{picture}(30,20)(0,-10)
\put(0,-5){\makebox(0,0){\usebox{\opening}}} 
\put(0,0){\makebox(0,0){\usebox{\sstwo}}} 
\put(10,0){\makebox(0,0){\usebox{\ssone}}} 
\put(20,0){\makebox(0,0){\usebox{\sstwo}}} 
\end{picture} 
&
\begin{picture}(40,20)(-10,-10)
\put(0,-5){\makebox(0,0){\usebox{\opening}}} 
\put(0,0){\makebox(0,0){\usebox{\sstwo}}} 
\put(10,0){\makebox(0,0){\usebox{\ssone}}} 
\put(20,0){\makebox(0,0){\usebox{\sstwo}}} 
\put(30,5){\makebox(0,0){\usebox{\closing}}} 
\end{picture} 
&
\begin{picture}(50,20)(-10,-10)
\put(0,-5){\makebox(0,0){\usebox{\opening}}} 
\put(0,0){\makebox(0,0){\usebox{\sstwo}}} 
\put(10,0){\makebox(0,0){\usebox{\sstwo}}} 
\put(20,0){\makebox(0,0){\usebox{\ssone}}} 
\put(30,0){\makebox(0,0){\usebox{\sstwo}}} 
\end{picture} 
\\[.1in]
\begin{picture}(30,20)(0,-10)
\put(0,-5){\makebox(0,0){\usebox{\opening}}} 
\put(0,0){\makebox(0,0){\usebox{\ssone}}} 
\put(10,0){\makebox(0,0){\usebox{\sstwo}}} 
\put(20,0){\makebox(0,0){\usebox{\ssone}}} 
\end{picture} 
&
\begin{picture}(40,20)(-5,-10)
\put(0,0){\makebox(0,0){\usebox{\ssone}}} 
\put(10,0){\makebox(0,0){\usebox{\sstwo}}} 
\put(20,0){\makebox(0,0){\usebox{\ssone}}} 
\put(30,0){\makebox(0,0){\usebox{\sstwo}}} 
\end{picture} 
&
\begin{picture}(40,20)(-5,-10)
\put(0,-5){\makebox(0,0){\usebox{\opening}}} 
\put(0,0){\makebox(0,0){\usebox{\ssone}}} 
\put(10,0){\makebox(0,0){\usebox{\sstwo}}} 
\put(20,0){\makebox(0,0){\usebox{\ssone}}} 
\put(30,0){\makebox(0,0){\usebox{\ssone}}} 
\end{picture} 
\\[.1in]
&
\begin{picture}(40,20)(-5,-10)
\put(0,0){\makebox(0,0){\usebox{\ssone}}} 
\put(10,0){\makebox(0,0){\usebox{\ssone}}} 
\put(20,0){\makebox(0,0){\usebox{\sstwo}}} 
\put(30,0){\makebox(0,0){\usebox{\ssone}}} 
\end{picture} 
\\[.05in]
$D_5$ & $D_6$ & $E_7$
\end{tabular}
\end{center} 
\caption{Divides associated with 
singularities of types~$D_5$ (a cusp and a transversal line), 
$D_6$~(a tacnode and a transversal line), and $E_7$~(a cusp and its cuspidal tangent). 
} 
\label{fig:divides-D5-D6-E7} 
\end{figure} 

\begin{figure}[ht] 
\vspace{-.2in}
\begin{center} 
\begin{picture}(80,30)(0,-12)
\put(0,-10){\makebox(0,0){\usebox{\opening}}} 
\put(0,10){\makebox(0,0){\usebox{\opening}}} 
\put(0,0){\makebox(0,0){\usebox{\ssstwo}}} 
\put(10,0){\makebox(0,0){\usebox{\sssonethree}}} 
\put(20,0){\makebox(0,0){\usebox{\ssstwo}}} 
\put(30,0){\makebox(0,0){\usebox{\sssonethree}}} 
\put(40,0){\makebox(0,0){\usebox{\ssstwo}}} 
\put(50,-10){\makebox(0,0){\usebox{\crossing}}} 
\put(50,10){\makebox(0,0){\usebox{\closing}}} 
\end{picture} 
\end{center} 
\caption{A divide associated with the non-quasihomogeneous 
singularity defined by the Puiseux parametrization $y=x^{3/2}+x^{7/4}$, 
see \cite[Figure~31]{couture-perron}.
} 
\label{fig:divides-non-quasihom} 
\end{figure}

\begin{figure}[ht] 
\vspace{-.1in}
\begin{center} 
\begin{picture}(40,30)(-5,-15)
\put(0,-10){\makebox(0,0){\usebox{\opening}}} 
\put(0,10){\makebox(0,0){\usebox{\opening}}} 
\put(0,0){\makebox(0,0){\usebox{\ssstwo}}} 
\put(10,0){\makebox(0,0){\usebox{\sssonethree}}} 
\put(20,0){\makebox(0,0){\usebox{\ssstwo}}} 
\put(30,0){\makebox(0,0){\usebox{\sssonethree}}} 
\end{picture} 
\qquad\qquad
\begin{picture}(60,30)(-5,-15)
\put(0,-10){\makebox(0,0){\usebox{\opening}}} 
\put(0,10){\makebox(0,0){\usebox{\opening}}} 
\put(0,0){\makebox(0,0){\usebox{\ssstwo}}} 
\put(10,0){\makebox(0,0){\usebox{\sssone}}} 
\put(20,0){\makebox(0,0){\usebox{\sssthree}}} 
\put(30,0){\makebox(0,0){\usebox{\sssone}}} 
\put(40,0){\makebox(0,0){\usebox{\ssstwo}}} 
\put(50,-10){\makebox(0,0){\usebox{\closing}}} 
\put(50,10){\makebox(0,0){\usebox{\closing}}} 
\end{picture} 
\end{center} 
\caption{Two divides associated with two different real forms of the
  singularity $(y^2+x^3)(x^2+y^3)=0$ (two transversal cusps). 
} 
\label{fig:divides-two-transversal cusps} 
\end{figure} 

\vspace{-.1in}

\begin{definition}
The divides arising via the construction of Definition~\ref{def:divide-of-a-morsification} are called \emph{algebraic}. 
Thus, an algebraic divide is a divide that comes from a 
real morsification of (a~real form~of) some~complex isolated plane curve 
singularity.
\end{definition}

Any divide~$D$ in which some proper subset of branches does not form a divide is not algebraic. 
In particular, if $D$ contains two branches which are disjoint, then $D$ is non-algebraic. 

\begin{remark}
\label{rem:which-divides-come-from-morsifications}
We are not aware of any (efficiently testable) 
necessary and sufficient conditions---even conjectural ones---ensuring
that a given divide~$D$ represents 
\begin{itemize}[leftmargin=.2in]
\item
a real morsification of a given real singularity; 
or \item
a real morsification of some real form of a given complex
singularity; 
or \item
a real morsification of a real form of some complex
singularity (i.e., $D$ is algebraic).
\end{itemize}
\end{remark}

\pagebreak[3]

\pagebreak[3]

While a given real singularity typically 
has several inequivalent real morsifications,
giving rise to distinct divides
(cf., e.g., Figure~\ref{fig:x^4-y^4=0}), 
some of the basic features of the resulting divide are uniquely determined by 
the real singularity at hand, see 
Propositions~\ref{pr:divide-properties-1}--\ref{pr:divide-properties-2-3}
below. 
Proofs, further details, and references can be found~in~\cite{leviant-shustin}. 

\begin{proposition}
\label{pr:divide-properties-1}
The branches of a divide associated with a real morsification
are 
obtained by deforming the local
branches of the original real singularity. 
Each real local branch of the singularity deforms 
into 
an immersed interval with endpoints on the boundary of the Milnor disk. 
Each pair of distinct complex conjugate local branches deforms into an 
immersed circle in the interior of the Milnor disk. 
\end{proposition}

In particular, among 
algebraic divides, 
the ones corresponding to \emph{totally real} singularities are precisely those
which contain no closed curves. 

\begin{proposition}
\label{pr:divide-properties-2-3}
Given a real plane curve singularity, 
the following collections~of
numbers do not depend on the choice of its morsification
(or the associated divide): 
\begin{itemize}[leftmargin=.3in]
\item
the numbers of self-intersections of the individual branches of the divide; 
\item
the numbers of intersections of the pairs of branches of the divide;
\item
the total number of regions in a divide. 
\end{itemize}
Specifically, the number of regions is equal to $\mu(C,z)-\delta_\RR(C,z)$, 
where $\mu(C,z)$ is the Milnor number of the singularity;
and the aforementioned intersection numbers are determined by 
the $\delta$-invariants and the intersection and self-intersection numbers of the local
branches.  
\end{proposition}

We note that while the numbers appearing in Proposition~\ref{pr:divide-properties-2-3} do not depend on the choice of a morsification (or divide), they do depend on the choice of a real form of a particular complex singularity. 

The importance of divides in the context of singularity theory 
stems from the fact that 
an algebraic divide 
completely determines the topological type of the
underlying complex singularity. 
(It also contains some information 
concerning the real form at hand.)  
See~\cite{leviant-shustin} and references therein, 
as well as 
Remark~\ref{rem:link-determines-singularity} below. 


\pagebreak[3]

Let $D$ be a divide in a disk~$\Disk$, and $I(D)$ its body, 
cf.\ Definition~\ref{def:divide}. 
Since we assumed the regions to be homeomorphic to open disks, 
the body~$I(D)$ has a natural structure of a cell complex: 
\pagebreak[3]
\begin{itemize}[leftmargin=.3in]
\item
the nodes of~$D$ are the $0$-cells;  
\item
the components of the set of nonsingular points of~$D$ which are 
disjoint from~$\partial\Disk$ are the $1$-cells;
\item
the regions are the $2$-cells. 
\end{itemize} 
If $D$ is a hyperbolic node (i.e., two embedded segments with a single
transverse intersection), 
then $I(D)$ is a single point.
Otherwise $I(D)$ is a connected (cf.~(D5)) and simply-connected
$2$-dimensional cell complex. 

Propositions~\ref{pr:number-of-real-nodes} and~\ref{pr:divide-properties-2-3} imply the following statement. 

\begin{proposition}
\label{pr:milnor-number}
Let $D$ be an algebraic divide. 
The sum of the number of $0$-cells and the number of $2$-cells of the cell complex~$I(D)$
is equal to the Milnor number of the associated singularity.
In particular, this number does not depend on the choice of morsification, 
nor on the choice of the real form of the given complex singularity.
\end{proposition}

\begin{example}
The three divides in the lower-right corner of Figure~\ref{fig:divides-quasihom}
correspond to morsifications of the following real forms of the same complex singularity:
\begin{itemize}[leftmargin=.3in]
\item
two complex conjugate cusps with the common tangent; 
\item
two real cusps with the common tangent and opposite orientation; 
\item
two co-oriented real cusps with the common tangent, cf.\ Figure~\ref{fig:non-partitions}(a). 
\end{itemize}
In each of the three cases, the combined number of nodes and regions is equal to~$15$,
matching the Milnor number of the singularity. 
\end{example}

\begin{remark}
\label{rem:non-regular-divides}
For $D$ an algebraic divide, 
the cell complex $I(D)$ is not necessarily \emph{regular}:
the closure of a $d$-cell does not have to be a closed $d$-ball. 
Even if $I(D)$ is regular, 
the intersection of the closures of two $d$-cells 
may be disconnected.
Figure~\ref{fig:non-partitions}
(borrowed from~\cite{leviant-shustin}) illustrates 
each of these possibilities,  for both $d=1$ and $d=2$. 
\end{remark}
%
%
\begin{figure}[ht]
\begin{center}
\includegraphics[width=95mm]{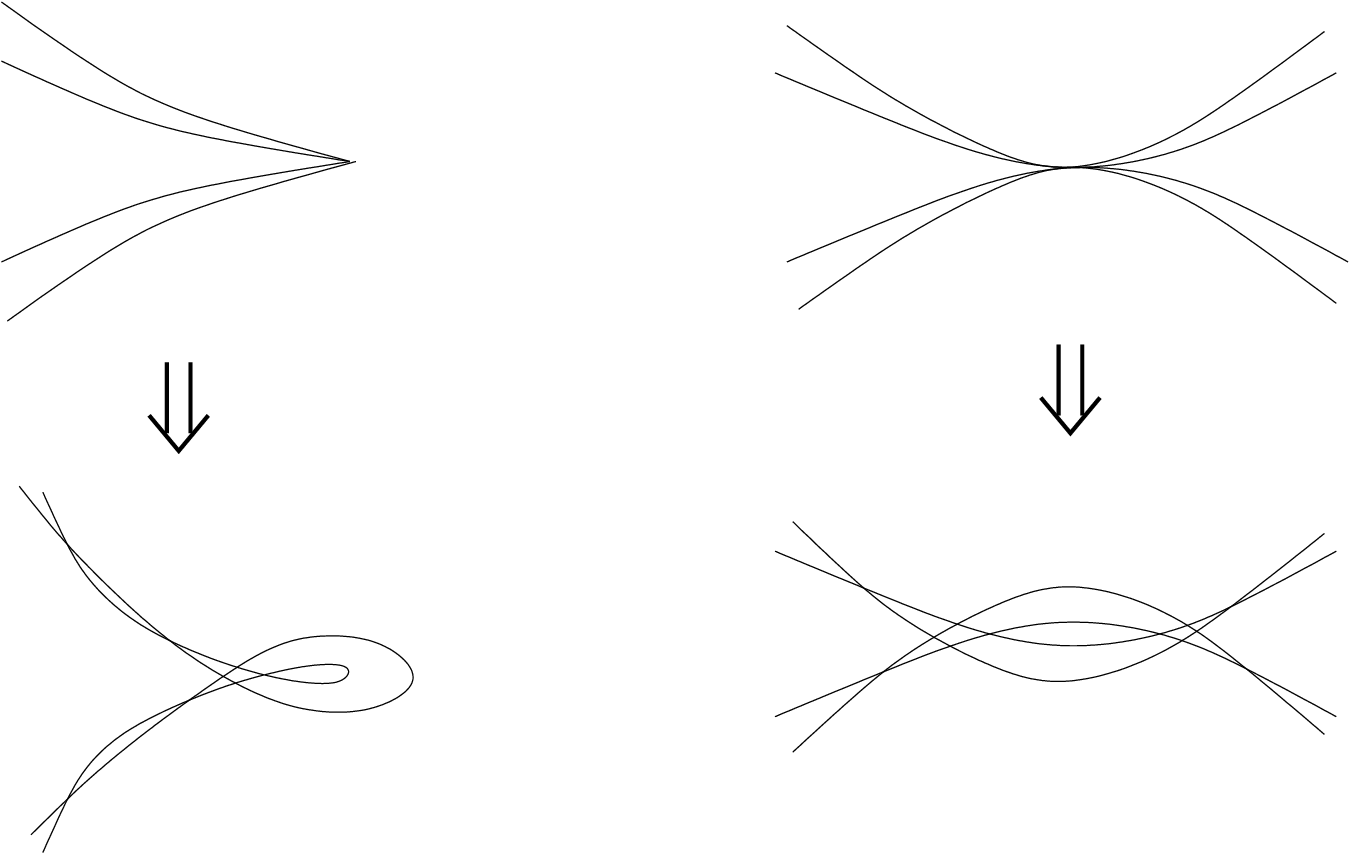}
\\
\hspace{-0.0in}
\begin{picture}(60,30)(-5,0)
\put(0,0){\makebox(0,0){\usebox{\sssonethree}}} 
\put(10,0){\makebox(0,0){\usebox{\sssonethree}}} 
\put(20,0){\makebox(0,0){\usebox{\ssstwo}}} 
\put(30,0){\makebox(0,0){\usebox{\sssonethree}}} 
\put(40,0){\makebox(0,0){\usebox{\ssstwo}}} 
\put(50,0){\makebox(0,0){\usebox{\closing}}} 
\put(50,0){\makebox(0,0){\usebox{\cclosing}}} 
\end{picture} 
\hspace{0.99in}
\begin{picture}(80,20)(0,0)
\put(0,0){\makebox(0,0){\usebox{\sssonethree}}} 
\put(10,0){\makebox(0,0){\usebox{\ssstwo}}} 
\put(20,0){\makebox(0,0){\usebox{\sssonethree}}} 
\put(30,0){\makebox(0,0){\usebox{\ssstwo}}} 
\put(40,0){\makebox(0,0){\usebox{\ssstwo}}} 
\put(50,0){\makebox(0,0){\usebox{\sssonethree}}} 
\put(60,0){\makebox(0,0){\usebox{\ssstwo}}} 
\put(70,0){\makebox(0,0){\usebox{\sssonethree}}} 
\end{picture} 
\\[.45in]
\hspace{-.35in} (a) \hspace{2.25in} (b)
\end{center}
\vspace{-.1in}
\caption{(a) A real morsification of the
singularity \hbox{$(y^2+x^3)(y^2+2x^3)=0$} 
(two cooriented real cuspidal branches with the
common cuspidal tangent) defined by 
\hbox{$(y^2+x^2(x-\eps_1))(y^2+2(x-\eps_2)^2(x-\eps_3))=0$},
with \hbox{$0<\eps_2<\eps_3\ll\eps_1\ll1$}\,, and the corresponding divide 
(cf.\ Figure~\ref{fig:divides-quasihom}, $a=6$, $b=4$, at the bottom). 
Here we~see that the closure of a cell in $I(D)$ \hbox{does not have to be
simply connected.} \\
%
(b)~A~real morsification of the real quasihomogeneous 
singularity of type~$(8,4)$ given by $(y^2-x^4)(y^2-2x^4)=0$ 
(four~real smooth branches quadratically tangent to each other), and its divide.
Here we see that the intersection of two $d$-cells may be
disconnected, for $d=1, 2$.} 
\label{fig:non-partitions}
\end{figure}

\clearpage

\newpage

\section{A'Campo-Gusein-Zade diagrams}
\label{sec:ag-diagrams}

In this section, we review the basics of \AG-diagrams, originally introduced by 
N.~A'Campo~\cite{acampo} and S.~Guse\u\i n-Zade~\cite{gusein-zade-1}. 
These diagrams also appeared in the literature 
under other names: Coxeter-Dynkin diagrams of singularities, $R$-diagrams,~etc. 
See~\cite{balke-kaenders} for another overview of this construction,
and for additional references. 

Two regions of a divide are called \emph{adjacent} if the intersection of 
their closures contains a $1$-cell (which is said to \emph{separate} these two regions). 

\begin{definition}
\label{def:ag-diagram}
Given a divide~$D$ as in Definition~\ref{def:divide},
its A'Campo-Guse\u{\i}n-Zade diagram \AG$(D)$ 
(\emph{\AG-diagram} for short)  
is a vertex-colored graph 
constructed as follows: 
\begin{itemize}[leftmargin=.3in]
\item
place a vertex at each node of~$D$, and color it black;
\item
place one vertex into each region of~$D$;
color these vertices $\circledplus$ or $\circleddash$ 
so that adjacent regions receive different colors (signs),
and non-adjacent regions sharing a node receive the same color;
\item
for each $1$-cell separating two regions, draw an edge connecting the vertices located
inside these regions;
\item
for each region~$R$, say bounded by $k$ one-dimensional cells, 
draw $k$~edges connecting the nodes on the boundary of~$R$ 
to the vertex located inside~$R$; 
these edges correspond to the $k$ distinct (up to isotopy) ways to draw a simple curve 
contained in~$R$ (except for one of the endpoints) 
connecting the interior vertex to a boundary node.  
\end{itemize}
\end{definition}

Figures~\ref{fig:A3-ACGZ} and~\ref{fig:E6-ACGZ} show 
\AG-diagrams of divides associated with 
different real morsifications of 
real singularities of types~$A_3$ and~$E_6$, respectively. 

\vspace{-.1in}

\begin{figure}[ht]
\begin{center}
\setlength{\unitlength}{1.5pt}
\begin{picture}(50,20)(0,0)
\thinlines
\qbezier(0,0)(30,60)(60,0)
\qbezier(0,30)(30,-30)(60,30)

\linethickness{1.2pt}

\put(8.8,15){\circle*{3}}
\put(51.2,15){\circle*{3}}
\put(30,15){\makebox(0,0){$\circledplus$}}
\put(11,15){\red
{\line(1,0){15}}}
\put(49,15){\red
{\line(-1,0){15}}}
\end{picture}
\hspace{1in}
\begin{picture}(50,40)(0,0)
\thinlines
\qbezier(0,15)(0,40)(30,15)
\qbezier(0,15)(0,-10)(30,15)
\qbezier(60,15)(60,40)(30,15)
\qbezier(60,15)(60,-10)(30,15)

\linethickness{1.2pt}

\put(30,15){\circle*{3}}
\put(51.2,15){\makebox(0,0){$\circleddash$}}
\put(8.8,15){\makebox(0,0){$\circleddash$}}
\put(12.7,15){\red
{\line(1,0){15}}}
\put(47.3,15){\red
{\line(-1,0){15}}}
\end{picture}
\end{center}
\caption{Two divides of type~$A_3$,
and their associated \AG-diagrams.}
\label{fig:A3-ACGZ}
\end{figure}

\vspace{-.3in}

\begin{figure}[ht]
\begin{center}
\setlength{\unitlength}{2pt}
\begin{picture}(50,45)(0,0)
\thinlines
\qbezier(0,0)(60,60)(60,20)
\qbezier(60,0)(0,60)(0,20)
\qbezier(60,20)(60,0)(30,0)
\qbezier(0,20)(0,0)(30,0)

\linethickness{1.2pt}

\put(6.2,6){\circle*{2.5}}
\put(53.8,6){\circle*{2.5}}
\put(30,26.5){\circle*{2.5}}
\put(30,6){\makebox(0,0){{$\circledplus$}}}
\put(6.2,26.5){\makebox(0,0){$\circleddash$}}
\put(53.8,26.5){\makebox(0,0){$\circleddash$}}
\put(9,6){\red{\line(1,0){18}}}
\put(51,6){\red{\line(-1,0){18}}}
\put(9,26.5){\red{\line(1,0){18}}}
\put(51,26.5){\red{\line(-1,0){18}}}
\put(6.2,24){\red{\line(0,-1){15.5}}}
\put(30,24){\red{\line(0,-1){15.5}}}
\put(53.8,24){\red{\line(0,-1){15.5}}}
\put(28,8){\red{\line(-8,7){19.5}}}
\put(32,8){\red{\line(8,7){19.5}}}
\end{picture}
\hspace{1in}
\begin{picture}(50,40)(0,5)
\thinlines
\qbezier(25,0)(60,70)(60,40)
\qbezier(35,0)(0,70)(0,40)
\qbezier(60,40)(60,30)(30,30)
\qbezier(0,40)(0,30)(30,30)
\linethickness{1.2pt}
\put(30,10){\circle*{2.5}}
\put(18.5,30.5){\circle*{2.5}}
\put(41.5,30.5){\circle*{2.5}}

\put(30,22){\makebox(0,0){$\circleddash$}}
\put(7.5,39.5){\makebox(0,0){$\circleddash$}}
\put(52.5,39.5){\makebox(0,0){$\circleddash$}}

\put(9.2,38){\red{\line(5,-4){7.4}}}
\put(50.8,38){\red{\line(-5,-4){7.4}}}

\put(32,23.5){\red{\line(5,4){7.2}}}
\put(28,23.5){\red{\line(-5,4){7.2}}}
\put(30,19.5){\red{\line(0,-1){7.2}}}

\end{picture}
\end{center}
\caption{Two divides of 
type~$E_6$,
and their associated \AG-diagrams.}
\label{fig:E6-ACGZ}
\end{figure}


\begin{remark}
The last rule in Definition~\ref{def:ag-diagram}
allows for the possibility of double edges in case 
the closure of~$R$ is not simply connected.
For example, this situation arises in the \AG-diagram associated with 
the morsification in Figure~\ref{fig:non-partitions}(a),
see Figure~\ref{fig:AG-multiple-edges}. 
\end{remark}

\begin{figure}[ht]
\begin{center}
\setlength{\unitlength}{2pt}
\begin{picture}(140,45)(0,0)

\linethickness{1.2pt}

\multiput(0,0)(40,0){3}{\makebox(0,0){\circle*{2.5}}}
\multiput(0,40)(40,0){3}{\makebox(0,0){\circle*{2.5}}}
\multiput(20,0)(40,0){2}{\makebox(0,0){$\circleddash$}}
\multiput(20,40)(40,0){2}{\makebox(0,0){$\circleddash$}}
\put(80,20){\makebox(0,0){{$\circledplus$}}}
\put(120,20){\makebox(0,0){{$\circledplus$}}}
\put(140,20){\makebox(0,0){{$\circleddash$}}}
\put(60,20){\makebox(0,0){\circle*{2.5}}}
\put(100,20){\makebox(0,0){\circle*{2.5}}}

\multiput(2,0)(40,0){2}{\line(1,0){15}}
\multiput(2,40)(40,0){2}{\line(1,0){15}}
\multiput(62,20)(40,0){2}{\line(1,0){15}}
\multiput(98,20)(40,0){1}{\line(-1,0){15}}
\multiput(137,20)(40,0){1}{\line(-1,0){14}}
\multiput(38,0)(40,0){2}{\line(-1,0){15}}
\multiput(38,40)(40,0){2}{\line(-1,0){15}}

\put(60,22){\line(0,1){15}}
\put(80,2){\line(0,1){15}}
\put(60,18){\line(0,-1){15}}
\put(80,38){\line(0,-1){15}}
\put(62,2){\line(1,1){16}}
\put(62,38){\line(1,-1){16}}

\qbezier(101.5,18.5)(120,12)(137.5,18.5)
\qbezier(101.5,21.5)(120,28)(137.5,21.5)

\qbezier(82.5,21.5)(110,40)(138,23)
\qbezier(82.5,18.5)(110,0)(138,17)

\qbezier(82,40)(130,40)(139.5,23.5)
\qbezier(82,0)(130,0)(139.5,16.5)
\end{picture}
\end{center}
\caption{The \AG-diagram for the divide/morsification shown in Figure~\ref{fig:non-partitions}(a).}
\label{fig:AG-multiple-edges}
\end{figure}

Definition~\ref{def:ag-diagram} specifies the coloring 
of the vertices in the \AG-diagram  
up to a global change of sign.
This coloring is \emph{proper}: every edge in \AG$(D)$
connects vertices of different~color. 
Thus \AG$(D)$ is a \emph{tripartite} graph. 

\begin{remark}
\label{rem:AG-planar-embeddings}
Any \AG-diagram is a (vertex-colored) \emph{planar} graph. 
Although its construction given in Definition~\ref{def:ag-diagram} supplies an embedding of this graph into the real plane,
the notion of an \AG-diagram
does \emph{not} include a choice of a planar embedding. 
Moreover a given \AG-diagram can have two non-homeomorphic planar embeddings, 
and can correspond to two topologically distinct divides, 
see, e.g.,~\cite[Figure~4]{balke-kaenders}. 
We~do not know whether this can happen for algebraic divides. 
\end{remark}

For an algebraic divide $D$ coming from a real morsification of a real singularity,
the vertices of the \AG-diagram \AG$(D)$ correspond to the critical
points of the morsified curve $C_t=\{f_t(x,y)=0\}$.
Furthermore, one can choose the coloring so that 
\begin{itemize}[leftmargin=.3in]
 \item 
 the vertices colored $\circledplus$ are located in the regions
 where $f_t>0$, and correspond to the local maxima of~$f_t$; 
 \item 
the vertices colored $\circleddash$ are located in the regions
 where $f_t<0$, and correspond to the local minima of~$f_t$; 
 \item 
the black vertices are located on the curve $f_t=0$, 
and correspond to the saddle points of~$f_t$. 
\end{itemize}
By Proposition~\ref{pr:milnor-number},
the number of vertices in \AG$(D)$ is equal to the Milnor number of
the singularity.

\begin{theorem}[{\cite{leviant-shustin}}]
\label{th:ag-diagram-determines-topology}
The \AG-diagram of a real morsification of a real isolated plane curve singularity 
determines the complex topological type of the singularity. 
\end{theorem}

We emphasize that Theorem~\ref{th:ag-diagram-determines-topology} does not require the knowledge of a
specific planar embedding of the \AG-diagram, cf.\ Remark~\ref{rem:AG-planar-embeddings}. 

In the case of totally real singularities,
Theorem~\ref{th:ag-diagram-determines-topology} 
was proved by L.~Balke and R.~Kaenders~\cite[Theorem~2.5 and Corollary~2.6]{balke-kaenders} 
under an additional assumption concerning the topology of 
the intersections of cell closures in~$I(D)$; 
this is related to the discussion in Remark~\ref{rem:non-regular-divides}.

\newpage

\section{Quivers}
\label{sec:quivers}

\begin{definition}
\label{def:quivers}
A \emph{quiver} is a finite directed graph.
Oriented cycles of length~$1$ or~$2$ are not allowed.
In other words, there must be no loops, 
and all arrows between a given pair of vertices must have the same direction. 
We do not distinguish between quivers (on the same vertex set)
which differ by simultaneous reversal of the direction of all arrows. 
\end{definition}

We will not need a more general notion of quivers with ``frozen'' vertices, just the simple setup described in Definition~\ref{def:quivers}. 

Throughout this paper, we use the standard Dynkin diagram nomenclature, 
along with K.~Saito's notation for the extended affine exceptional types (cf.\ \cite[Section~12]{cats1}), 
to assign names to some of the quivers
appearing in various examples, cf.\ in particular Figure~\ref{fig:divides-quasihom}. 


\begin{definition}
\label{def:quiver-of-divide}
Given a divide~$D$, 
its associated \emph{quiver} $Q(D)$ 
is constructed from the \AG-diagram \AG$(D)$ 
as follows: 
\begin{itemize}[leftmargin=.3in]
\item
first, orient the edges of~\AG$(D)$ using the rule
$\bullet \,\red{\to}\oplus\red{\to}\circleddash\red{\to}\,\bullet\,$;
\item
then remove the marking of the vertices. 
\end{itemize}
Since we consider quivers up to global reversal of arrows,
the choice of signs in the \AG-diagram does not matter. 
\end{definition}

Examples of quivers associated with divides coming from morsifications
can be found in Figure~\ref{fig:E6-morsifications}
(compare with Figure~\ref{fig:E6-ACGZ})
and Figure~\ref{fig:4-lines-quivers} (cf.\ Figure~\ref{fig:4-lines}). 

\begin{figure}[ht]
\begin{center}
\vspace{.2in}
\setlength{\unitlength}{2pt}
\begin{picture}(50,40)(0,0)
\thinlines
\qbezier(0,0)(60,60)(60,20)
\qbezier(60,0)(0,60)(0,20)
\qbezier(60,20)(60,0)(30,0)
\qbezier(0,20)(0,0)(30,0)
\linethickness{1.2pt}

\put(6.2,6){\circle*{2.5}}
\put(53.8,6){\circle*{2.5}}
\put(30,26.5){\circle*{2.5}}

\put(30,6){\circle*{2.5}}
\put(6.2,26.5){\circle*{2.5}}
\put(53.8,26.5){\circle*{2.5}}

\put(9,6){\red{\vector(1,0){18}}}
\put(51,6){\red{\vector(-1,0){18}}}
\put(9,26.5){\red{\vector(1,0){18}}}
\put(51,26.5){\red{\vector(-1,0){18}}}
\put(6.2,24){\red{\vector(0,-1){15.5}}}
\put(30,24){\red{\vector(0,-1){15.5}}}
\put(53.8,24){\red{\vector(0,-1){15.5}}}
\put(28,8){\red{\vector(-8,7){19.5}}}
\put(32,8){\red{\vector(8,7){19.5}}}
\end{picture}
\hspace{1in}
\begin{picture}(50,43)(0,5)
\thinlines
\qbezier(25,0)(60,70)(60,40)
\qbezier(35,0)(0,70)(0,40)
\qbezier(60,40)(60,30)(30,30)
\qbezier(0,40)(0,30)(30,30)

\linethickness{1.2pt}

\put(30,10){\circle*{2.5}}
\put(18.5,30.5){\circle*{2.5}}
\put(41.5,30.5){\circle*{2.5}}

\put(30,22){\circle*{2.5}}
\put(7.5,39.5){\circle*{2.5}}
\put(52.5,39.5){\circle*{2.5}}

\put(9.2,38){\red{\vector(5,-4){7.4}}}
\put(50.8,38){\red{\vector(-5,-4){7.4}}}

\put(32,23.5){\red{\vector(5,4){7.2}}}
\put(28,23.5){\red{\vector(-5,4){7.2}}}
\put(30,19.5){\red{\vector(0,-1){7.2}}}

\end{picture}
\end{center}
\caption{Quivers associated with divides of type~$E_6$.}
\label{fig:E6-morsifications}
\end{figure}


\begin{remark}
While the quiver $Q(D)$ is very closely related to the \AG-diagram \AG$(D)$, some information is lost in the transition from  \AG$(D)$ to~$Q(D)$.
For example, although the \AG-diagrams shown in Figure~\ref{fig:A3-ACGZ}
are different from each other, the corresponding quivers are both isomorphic to the quiver $\bullet\,\red{\to}\bullet\red{\leftarrow}\,\bullet\,$.  
\end{remark}

In the case of algebraic divides,
Theorem~\ref{th:ag-diagram-determines-topology} can be used to establish the following result. 


\begin{theorem}
\label{th:quiver-determines-topology}
Let $Q=Q(D)$ be a quiver constructed from an algebraic divide~$D$ 
corresponding to a real morsification of a real isolated plane curve singularity. 
Then $Q$ uniquely determines the complex topological type of the singularity. 
\end{theorem}

To prove this theorem, we will need the following lemma.

\begin{lemma}
\label{lem:euler-D}
Let $D$ be any divide (not necessarily algebraic).
Denote 
\begin{align*}
\rho &= \text{\rm number of regions in~$D$}, \\
\nu &= \text{\rm number of nodes in~$D$, and}\\
\iota &=\text{\rm number of interval branches in~$D$}.
\end{align*}
Then $\nu=\rho+\iota-1$. 
In particular, $\nu\ge\rho-1$. 
\end{lemma}

\begin{proof}
Let $K(D)=\overline{D\cup I(D)}$, the closure of the union of~$D$ and its body~$I(D)$, see Definition~\ref{def:divide}.
Since $K(D)$ is contractible, its Euler characteristic is equal to~1. 
Let $\varepsilon$ denote the number of $1$-cells in~$K(D)$. 
Then $2\varepsilon=4\nu+2\iota$, so $\varepsilon=2\nu+\iota$. 
Hence 
\[
1=\chi(K(D))=(\nu+2\iota)-(2\nu+\iota)+\rho=\iota-\nu+\rho, 
\]
as desired. 
\end{proof}

\begin{proof}[Proof of Theorem~\ref{th:quiver-determines-topology}]
In light of Theorem~\ref{th:ag-diagram-determines-topology}, 
all we need to do is reconstruct the $\{\bullet, \circledplus,\circleddash\}$-coloring of the vertices 
of the quiver~$Q=Q(D)$. 
More precisely, we aim to reconstruct the coloring up to a global color switch  $\circledplus\leftrightarrow\circleddash$,
which corresponds to changing the sign of the morsified function~$f_t$ (see Section~\ref{sec:ag-diagrams})
and consequently does not affect the topological type. 
As we shall demonstrate, this reconstruction can be accomplished in all but a few exceptional cases;
treating each of those cases separately, we will show that the quiver~$Q$ determines the topology of the singularity. 

The requisite coloring of the vertices of~$Q$ must obey 
the cyclic orientation rule $\bullet \,\red{\to}\oplus\red{\to}\circleddash\red{\to}\,\bullet\,$ 
(or $\bullet \,\red{\to}\circleddash\red{\to}\oplus\red{\to}\,\bullet\,$ after a global reversal of arrows). 
This enables us to determine the coloring of the vertices into three colors $\{1,2,3\}$,
without specifying a bijective identification $\{1,2,3\}\leftrightarrow\{\bullet, \circledplus,\circleddash\}$. 
Namely, assign color~1 to some vertex~$v$, then 
propagate the coloring away from~$v$ following the cyclic rule $1\to 2\to 3\to 1$. 

At this stage, we would like to determine which of the three colors $1,2,3$ is black, 
i.e., corresponds to the~nodes of~$D$; 
the other two colors would correspond to $\circledplus$ and~$\circleddash$ 
(the regions of the divide). 

Let $c_1$, $c_2$, and $c_3$ denote the number of vertices in~$Q$ which have color $1$, $2$, or~$3$,
respectively. 
Without loss of generality, we may assume that $c_1\le c_2\le c_3$. 

\emph{Case~1}: $c_1=0$. That is, at most two colors are present. One of them must correspond to the~nodes. 
We conclude that all regions of~$D$ must have the same color; 
every node is adjacent to at most two regions; 
and the \AG-diagram cannot have cycles. Thus $Q$ is an oriented tree.
If $Q$ has a vertex~$v$ of degree~$\ge3$, then $v$ is not a node, and we are done. 
Otherwise, $Q$~is a chain (with an alternating orientation), so we are dealing with a singularity of type~$A$,
cf.\ Figure~\ref{fig:divides-quasihom}, top row ($b=2$). 

\emph{Case~2}: $c_1\ge 2$. 
The inequality $\nu\ge \rho-1$ (see Lemma~\ref{lem:euler-D}) 
then implies that the number of nodes~$\nu$ is strictly greater than the number of regions of either color,
so we can identify which color is black. 

\emph{Case~3}: $c_1=1$. That is, each color is present, and there is a unique vertex~$v$ of color~1. 
This case splits into two subcases. 

\emph{Case~3A}: $c_1=1$ and $c_3>c_2$.
Then color~3 must be black (otherwise $\nu\le c_2\le c_3-1\le\rho-2$, 
contradicting Lemma~\ref{lem:euler-D}), and we are done. 

\emph{Case~3B}: $c_1=1$ and $c_2=c_3$. 
Let $Q_v$ denote the induced subquiver of~$Q$ obtained by removing~$v$ together with all incident arrows. 
Within each connected component~$Q'$ of~$Q_v$, the vertices connected to~$v$ in~$Q$ form a connected
subquiver~$Q''\subset Q'$ with all degrees~$\le 2$. 
Now, there are two possibilities. 

\emph{Case~3B(i)}: 
Each subquiver~$Q''$ as above is a chain. Then its endpoints must be nodes, and we are done. 

\emph{Case~3B(ii)}: 
$Q_v$ is connected, and $Q''\subset Q'=Q_v$ is a cycle, necessarily with alternating orientation. 
If there is a vertex $u$ in~$Q''$ connected to a vertex not belonging to the ``wheel'' $Q''\cup\{v\}$, then $u$ is a node. 
Otherwise, $Q$ is the wheel quiver on $2m+1$ vertices, for some $m\in\ZZ_{>0}$;
cf.\ Figure~\ref{fig:4-lines-quivers} (lower right), with $m=4$. 
Although in this case, we cannot uniquely determine which color on the perifery of the wheel is black, 
the two choices produce isomorphic \AG-diagrams, so Theorem~\ref{th:ag-diagram-determines-topology} applies. 
\end{proof}

Theorem~\ref{th:quiver-determines-topology} implies 
that any topological invariant of an isolated plane curve singularity is uniquely determined by the quiver of an arbitrary real morsification. While for particular invariants, this may be a very challenging task (cf.\ also Remark~\ref{rem:libgober-williams} below), there are a couple of cases where the answer is relatively easy. 
To discuss them, we need to recall the following standard notion. 

\begin{definition}
Let $Q$ be a quiver with the vertex set~$V$ of size~$n$. The \emph{skew-symmetric matrix associated with~$Q$} is the $n\times n$ matrix $B(Q)=(b_{ij})_{i,j\in V}$ defined~by
\[
b_{ij}=\left\{\begin{array}{cl}
\text{number of arrows $i\to j$ in~$Q$} & \text{if $Q$ contains such arrows;} \\[.1in]
-(\text{number of arrows $j\to i$ in~$Q$}) & \text{if $Q$ contains such arrows;} \\[.1in]
0 &\text{otherwise}. 
\end{array}
\right.
\]
\end{definition}

\begin{proposition}
\label{pr:top-inv-from-B}
Given a real morsification of an isolated plane curve singularity~$(C,z)$, 
let $Q$ be its quiver, and let $B=B(Q)$ the corresponding skew-symmetric matrix. Then: 
\begin{itemize}[leftmargin=.3in]
\item
the Milnor number $n=\mu(C,z)$ is the size of the matrix~$B$; 
\item
the number $r$ of complex local branches is given by 
$r=n-\operatorname{rank}(B)+1$;
\item
the $\delta$-invariant of the singularity is given by
$\delta=n-\frac12 \operatorname{rank}(B)$. 
\end{itemize}
\end{proposition}

\begin{proof}
The first statement follows from Proposition~\ref{pr:milnor-number}.

The $n\times n$ skew-symmetric matrix $B(Q)$ is the intersection matrix of
1-cycles in the Milnor fiber~$M$. The latter is a surface with $r$ holes, where $r$ is the number of complex local branches.
More precisely: given a pair of 1-cycles
$a,b\in H_1(M)$, map one of them, say~$b$, to
the relative group $H_1(M,\partial M)$; then apply
the (non-degenerate) Poincare pairing
$H_1(M) \times H_1(M,\partial M) \to \mathbb{Z}$.
The defect in the rank comes from the fact that
the homomorphism $H_1(M)\to H_1(M,\partial M)$
has $(r-1)$-dimensional kernel, which can be seen from the exact homology
sequence
\[
H_2(M,\partial M)\to H_1(\partial M)\to H_1(M)\to H_1(M,\partial M). 
\] 
Finally, by the Milnor formula~\cite{Milnor}, 
$n=2\delta-r+1$, and the last claim~follows. 
\end{proof}


\newpage

\section{Main conjecture}
\label{sec:main-conjecture}

Whereas by Theorem~\ref{th:quiver-determines-topology}, 
all information about the topology of a complex plane curve singularity  
is encoded in the quiver constructed from its real morsification, this does not yield a satisfactory topological classification of such singularities. 
One reason is that we do not know which quivers (or which divides) can arise from such morsifications---this problem is wide open, and likely hopeless. 
Another, more practical problem has to do with deciding whether two
singularities are isomorphic to each other or not. 
The same singularity will typically have many real forms;  
each of them will have many topologically different morsifications,  
each with its own quiver. 
What do all these quivers have in common? 
For example, Figure~\ref{fig:4-lines-quivers} shows four quivers arising from
morsifications of different real forms of the same singularity,
namely the quasihomogeneous singularity of type~$(4,4)$. 
What features set these four quivers apart from all other quivers
arising from similar constructions? 

To rephrase, how can we tell, looking at two quivers associated
to morsifications of two isolated plane curve singularities, 
whether these singularities are topologically
the same? In light of Theorem~\ref{th:quiver-determines-topology}, 
this should in principle be possible. 

\begin{figure}[ht]
\begin{center}
\vspace{-.1in}
\setlength{\unitlength}{1.7pt}
\begin{picture}(80,80)(15,25)
\thinlines

\put(10,25){\line(2,1){100}}
\put(110,25){\line(-2,1){100}}

\put(13.3,20){\line(2,3){54}}
\put(106.7,20){\line(-2,3){54}}

\linethickness{1.2pt}
\put(40,60){\circle*{2.5}}
\put(60,50){\circle*{2.5}}
\put(80,60){\circle*{2.5}}
\put(20,30){\circle*{2.5}}
\put(100,30){\circle*{2.5}}
\put(60,90){\circle*{2.5}}
\put(60,70){\circle*{2.5}}
\put(35,45){\circle*{2.5}}
\put(85,45){\circle*{2.5}}

\put(64,49.2){\red{\vector(5,-1){18}}}
\put(97,33){\red{\vector(-1,1){10}}}
\put(82,48){\red{\vector(-1,1){18}}}
\put(81,57){\red{\vector(1,-3){3}}}

\put(56,49.2){\red{\vector(-5,-1){18}}}
\put(23,33){\red{\vector(1,1){10}}}
\put(38,48){\red{\vector(1,1){18}}}
\put(39,57){\red{\vector(-1,-3){3}}}

\put(60,73){\red{\vector(0,1){14}}}
\put(60,67){\red{\vector(0,-1){14}}}

\put(63,68.5){\red{\vector(2,-1){15}}}
\put(57,68.5){\red{\vector(-2,-1){15}}}

\end{picture}
\qquad\qquad
\begin{picture}(80,80)(15,25)
\thinlines

\put(20,20){\line(1,1){80}}
\put(20,100){\line(1,-1){80}}

\put(60,60){\circle{56.5}}

\linethickness{1.2pt}
\put(40,40){\circle*{2.5}}
\put(40,80){\circle*{2.5}}
\put(80,40){\circle*{2.5}}
\put(80,80){\circle*{2.5}}
\put(60,60){\circle*{2.5}}

\put(40,60){\circle*{2.5}}
\put(60,80){\circle*{2.5}}
\put(60,40){\circle*{2.5}}
\put(80,60){\circle*{2.5}}

\put(43,40){\red{\vector(1,0){14}}}
\put(43,60){\red{\vector(1,0){14}}}
\put(43,80){\red{\vector(1,0){14}}}
\put(77,40){\red{\vector(-1,0){14}}}
\put(77,60){\red{\vector(-1,0){14}}}
\put(77,80){\red{\vector(-1,0){14}}}
\put(40,57){\red{\vector(0,-1){14}}}
\put(60,57){\red{\vector(0,-1){14}}}
\put(80,57){\red{\vector(0,-1){14}}}
\put(40,63){\red{\vector(0,1){14}}}
\put(60,63){\red{\vector(0,1){14}}}
\put(80,63){\red{\vector(0,1){14}}}

\put(58,42){\red{\vector(-1,1){16}}}
\put(58,78){\red{\vector(-1,-1){16}}}
\put(62,42){\red{\vector(1,1){16}}}
\put(62,78){\red{\vector(1,-1){16}}}

\end{picture}
\\[.3in]
\begin{picture}(80,100)(15,30)
\thinlines
\qbezier(20,60)(20,82)(60,82)
\qbezier(20,60)(20,38)(60,38)
\qbezier(100,60)(100,82)(60,82)
\qbezier(100,60)(100,38)(60,38)

\put(37.5,30){\line(1,4){25}}
\put(82.5,30){\line(-1,4){25}}

\linethickness{1.2pt}
\put(50.4,81.6){\circle*{2.5}}
\put(69.6,81.6){\circle*{2.5}}
\put(40,40){\circle*{2.5}}
\put(80,40){\circle*{2.5}}
\put(60,120){\circle*{2.5}}
\put(60,60){\circle*{2.5}}
\put(30,60){\circle*{2.5}}
\put(60,90){\circle*{2.5}}
\put(90,60){\circle*{2.5}}

\put(33,60){\red{\vector(1,0){24}}}
\put(87,60){\red{\vector(-1,0){24}}}
\put(60,87){\red{\vector(0,-1){24}}}
\put(60,93){\red{\vector(0,1){22}}}

\put(61,62.5){\red{\vector(2,5){7}}}
\put(59,62.5){\red{\vector(-2,5){7}}}

\put(62,58){\red{\vector(1,-1){16}}}

\put(58,58){\red{\vector(-1,-1){16}}}

\put(38,44){\red{\vector(-1,2){6.5}}}
\put(82,44){\red{\vector(1,2){6.5}}}

\put(52,83){\red{\vector(1,1){6}}}
\put(68,83){\red{\vector(-1,1){6}}}
\put(48,79){\red{\vector(-1,-1){16.5}}}
\put(72,79){\red{\vector(1,-1){16.5}}}

\end{picture}
\qquad\qquad
\begin{picture}(80,100)(15,10)
\thinlines
\qbezier(20,60)(20,82)(60,82)
\qbezier(20,60)(20,38)(60,38)
\qbezier(100,60)(100,82)(60,82)
\qbezier(100,60)(100,38)(60,38)

\qbezier(60,20)(38,20)(38,60)
\qbezier(60,100)(38,100)(38,60)
\qbezier(60,20)(82,20)(82,60)
\qbezier(60,100)(82,100)(82,60)

\linethickness{1.2pt}

\put(40,40){\circle*{2.5}}
\put(40,80){\circle*{2.5}}
\put(80,40){\circle*{2.5}}
\put(80,80){\circle*{2.5}}
\put(60,60){\circle*{2.5}}
\put(30,60){\circle*{2.5}}
\put(60,90){\circle*{2.5}}
\put(60,30){\circle*{2.5}}
\put(90,60){\circle*{2.5}}

\put(33,60){\red{\vector(1,0){24}}}
\put(87,60){\red{\vector(-1,0){24}}}
\put(60,33){\red{\vector(0,1){24}}}
\put(60,87){\red{\vector(0,-1){24}}}

\put(62,62){\red{\vector(1,1){16}}}
\put(62,58){\red{\vector(1,-1){16}}}
\put(58,62){\red{\vector(-1,1){16}}}
\put(58,58){\red{\vector(-1,-1){16}}}

\put(38,76){\red{\vector(-1,-2){6.5}}}
\put(38,44){\red{\vector(-1,2){6.5}}}
\put(82,76){\red{\vector(1,-2){6.5}}}
\put(82,44){\red{\vector(1,2){6.5}}}
\put(76,38){\red{\vector(-2,-1){13}}}
\put(44,38){\red{\vector(2,-1){13}}}
\put(76,82){\red{\vector(-2,1){13}}}
\put(44,82){\red{\vector(2,1){13}}}
\end{picture}
\vspace{-.1in}
\end{center}
\caption{
The quivers associated with
  morsifications from Figure~\ref{fig:4-lines}.}
\label{fig:4-lines-quivers}
\end{figure}

The conjectural answer to the last question 
(see Conjecture~\ref{conj:morsif=mut} below) 
comes from the concept of quiver mutation, which we shall now recall.
While quiver mutations play a fundamental role in the theory of cluster algebras, 
we will not rely on any results from this theory,
such as for example the Laurent Phenomenon~\cite{ca1} or the finite type classification~\cite{ca2}. 
We~refer the interested reader to~\cite{fwz, l.williams} for further details, examples, and motivation.

\begin{definition}
\label{def:mutations}
Given a vertex~$z$ in a quiver~$Q$, 
the \emph{quiver mutation} at~$z$ is a transformation of~$Q$ into a new quiver 
$Q'=\mu_z(Q)$ constructed in three steps: \\
\textbf{1.}
For each 
path $x\!\to\! z\!\to\! y$ of length~2 
passing through~$z$, introduce a new edge~$x\!\to\! y$.\\
\textbf{2.}
Reverse the direction of all edges incident
to~$z$.\\
\textbf{3.}
Remove oriented 2-cycles, one by one.

\end{definition}

\begin{figure}[ht]
\begin{center}
\setlength{\unitlength}{2pt} 
\begin{picture}(20,30)(0,-10) 

\put( 0,0){\circle*{3}}
\put(20,0){\circle*{3}}

\put(0,20){\circle*{3}}
\put(20,20){\makebox(0,0){$z$}}

\put(10,-7){\makebox(0,0){$Q$}}

\thicklines 

\put(16,0){\red{\vector(-1,0){12}}}
\put(16,20){\red{\vector(-1,0){12}}}
\put(20,16){\red{\vector(0,-1){12}}}
\put(0,4){\red{\vector(0,1){12}}}
\put(3,3){\red{\vector(1,1){14}}}

\end{picture}
\begin{picture}(30,30)(0,-10) 
\put(15,-5){\makebox(0,0){$\stackrel{\displaystyle\mu_z}{\longleftrightarrow}$}}
\end{picture}
\begin{picture}(20,30)(0,-10) 

\put( 0,0){\circle*{3}}
\put(20,0){\circle*{3}}
\put(0,20){\circle*{3}}
\put(20,20){\makebox(0,0){$z$}}
\put(10,-7){\makebox(0,0){$Q'$}}

\thicklines 

\put(4,20){\red{\vector(1,0){12}}}
\put(-1.3,4){\red{\vector(0,1){12}}}
\put(1.3,4){\red{\vector(0,1){12}}}
\put(20,4){\red{\vector(0,1){12}}}
\put(17,17){\red{\vector(-1,-1){14}}}

\end{picture}
\vspace{-.2in}
\end{center}
\caption{Two quivers related by a quiver mutation at the vertex~$z$.}
\label{fig:quiver-mutation}
\end{figure}

\vspace{-.2in}

\begin{definition}
Two quivers $Q$ and~$Q'$ are called \emph{mutation equivalent} 
if $Q$ can be transformed into a quiver isomorphic to~$Q'$ 
by a sequence of mutations. 
It is easy to see that quiver mutation is involutive (i.e., $\mu_{z}(\mu_z(Q))=Q$),  
and consequently mutation equivalence is indeed an equivalence relation. 
\end{definition}

\begin{example}
The two quivers in Figure~\ref{fig:E6-morsifications}
are mutation equivalent to each other. 
This is an instance of a general phenomenon discussed in 
Section~\ref{sec:yang-baxter-transformations} below:
divides related via triangle moves
have mutation equivalent quivers. 
\end{example} 

\begin{remark}
\label{rem:mut-equivalence-difficult}
The problem of deciding whether two given quivers are mutation equivalent or not
is notoriously difficult. 
Furthermore, there is a dearth of known invariants of quiver mutation,
even though experimental evidence strongly suggests that many independent 
invariants must exist. 
\end{remark}


\begin{conjecture}[\underbar{Main conjecture}]
\label{conj:morsif=mut}
Given two real morsifications of real isolated plane curve
singularities, the following are equivalent:
\begin{itemize}[leftmargin=.3in]
\item
the two singularities 
have the same complex topological type; 
\item
the quivers associated with the two morsifications are
mutation equivalent. 
\end{itemize}
\end{conjecture}

To rephrase, Conjecture~\ref{conj:morsif=mut} asserts that 
isolated plane curve singularities are topologically 
classified by the mutation classes of associated quivers. 
Put another way:
\begin{itemize}[leftmargin=.3in]
\item
different morsifications of (different real forms of) 
the same complex plane curve singularity have mutation equivalent quivers; 
\item
morsifications of (real forms of) topologically different complex plane curve 
singularities have quivers 
which are not mutation equivalent to each~other. 
\end{itemize}

Figure~\ref{fig:main-conjecture} illustrates 
the essence of Conjecture~\ref{conj:morsif=mut}, 
and the relationships between its various ingredients. 
 
\begin{figure}
\begin{center}
\setlength{\unitlength}{1pt}
\begin{picture}(180,490)(0,0)
\qbezier(-32,90)(-80,90)(-80,256)
\qbezier(-80,256)(-80,422)(-46,422)
\put(-36,90){\vector(1,0){5}}
\put(-48,422){\vector(1,0){5}}
\put(-95,287){\rotatebox[origin=r]{90}{Conjecture~\ref{conj:morsif=mut}}}
\put(-60,240){
\begin{tabular}{|c|c|}
\hline
& \\[-.13in]
general concept & illustration \\[.05in]
\hline
& \\[-.1in]
\quad complex singularity \quad\quad & 
\begin{tabular}{c}
four smooth branches\\
intersecting transversally\\
at the point $(0,0)\in\CC^2$
\end{tabular}
\\[-.04in]
$\twoheaduparrow$ & \\[.04in]
real singularity & \begin{tabular}{c}
$x^4-y^4=0$ \\
(two real branches and two \\
complex conjugate branches)
\end{tabular}
\\[-.04in]
$\uparrow$ & \\[.04in]
morsification & \begin{tabular}{c}
\ \\
$(x^2-y^2)(x^2+y^2-t^2)=0$\\
{\ }
\end{tabular}
 \\[-.01in]
$\downarrow$ & \\[.01in]
\setlength{\unitlength}{0.3pt}
\begin{picture}(120,100)(0,20)
\put(60,60){\makebox(0,0){divide}} 
\end{picture}
& \setlength{\unitlength}{0.3pt}
\begin{picture}(120,100)(0,20)
\linethickness{1pt}
\put(20,20){\line(1,1){80}}
\put(20,100){\line(1,-1){80}}
\put(60,60){\circle{56.5}}
\thinlines
\put(60,60){\circle{113}}
\put(60,60){\circle{103}}
\end{picture}\\[.01in]
$\downarrow$ & \\[.09in]
\setlength{\unitlength}{1pt}
\begin{picture}(40,40)(40,40)
\put(60,60){\makebox(0,0){\AG-diagram}} 
\end{picture}
& \setlength{\unitlength}{1pt}
\begin{picture}(40,40)(40,40)
\thinlines


\linethickness{1pt}
\put(40,40){\circle*{2.5}}
\put(40,80){\circle*{2.5}}
\put(80,40){\circle*{2.5}}
\put(80,80){\circle*{2.5}}
\put(60,60){\circle*{2.5}}

\put(40,60){\makebox(0,0){{$\scriptstyle\circleddash$}}}
\put(60,80){\makebox(0,0){{$\scriptstyle\circledplus$}}}
\put(60,40){\makebox(0,0){{$\scriptstyle\circledplus$}}}
\put(80,60){\makebox(0,0){{$\scriptstyle\circleddash$}}}

\put(43,40){\red{\line(1,0){13}}}
\put(44,60){\red{\line(1,0){13}}}
\put(43,80){\red{\line(1,0){13}}}
\put(77,40){\red{\line(-1,0){13}}}
\put(76,60){\red{\line(-1,0){13}}}
\put(77,80){\red{\line(-1,0){13}}}
\put(40,56){\red{\line(0,-1){13}}}
\put(60,57){\red{\line(0,-1){13}}}
\put(80,56){\red{\line(0,-1){13}}}
\put(40,64){\red{\line(0,1){13}}}
\put(60,63){\red{\line(0,1){13}}}
\put(80,64){\red{\line(0,1){13}}}

\put(57,43){\red{\line(-1,1){14}}}
\put(57,77){\red{\line(-1,-1){14}}}
\put(63,43){\red{\line(1,1){14}}}
\put(63,77){\red{\line(1,-1){14}}}

\end{picture}\\[.01in]
$\downarrow$ & \\[.07in]
\setlength{\unitlength}{1pt}
\begin{picture}(40,40)(40,40)
\put(60,60){\makebox(0,0){quiver}} 
\end{picture} & \setlength{\unitlength}{1pt}
\begin{picture}(40,40)(40,40)
\thinlines


\linethickness{1pt}
\put(40,40){\circle*{2.5}}
\put(40,80){\circle*{2.5}}
\put(80,40){\circle*{2.5}}
\put(80,80){\circle*{2.5}}
\put(60,60){\circle*{2.5}}

\put(40,60){\circle*{2.5}}
\put(60,80){\circle*{2.5}}
\put(60,40){\circle*{2.5}}
\put(80,60){\circle*{2.5}}

\put(43,40){\red{\vector(1,0){14}}}
\put(43,60){\red{\vector(1,0){14}}}
\put(43,80){\red{\vector(1,0){14}}}
\put(77,40){\red{\vector(-1,0){14}}}
\put(77,60){\red{\vector(-1,0){14}}}
\put(77,80){\red{\vector(-1,0){14}}}
\put(40,57){\red{\vector(0,-1){14}}}
\put(60,57){\red{\vector(0,-1){14}}}
\put(80,57){\red{\vector(0,-1){14}}}
\put(40,63){\red{\vector(0,1){14}}}
\put(60,63){\red{\vector(0,1){14}}}
\put(80,63){\red{\vector(0,1){14}}}

\put(58,42){\red{\vector(-1,1){16}}}
\put(58,78){\red{\vector(-1,-1){16}}}
\put(62,42){\red{\vector(1,1){16}}}
\put(62,78){\red{\vector(1,-1){16}}}

\end{picture}\\[-.01in]
$\downarrow$ & \\[.01in]
mutation class & \begin{tabular}{c}
\ \\
$
E_7^{(1,1)}$ \\
{\ }
\end{tabular}
\\[-.01in]
$\downarrow$ & \\[.01in]
cluster algebra& \begin{tabular}{c}
\ \\
Grassmannian $\operatorname{Gr}_{4,8}(\CC)$ \\
{\ }
\end{tabular}
\\
\hline
\end{tabular}
}
\end{picture}
\end{center}
\caption{\emph{Unpacking Conjecture~\ref{conj:morsif=mut}.} 
A~complex plane curve singularity has~at least one real form.
According to Conjecture~\ref{conj:leviant-shustin}, 
each of these real singularities has a real \emph{morsification}.
A~morsification defines a \emph{divide}. 
A~divide has the associated \AG-diagram. 
The~\AG-diagram produces a \emph{quiver}.
The quiver determines a mutation equivalence~class
(which can in turn be used to define a cluster algebra or category).  
Conjecture~\ref{conj:morsif=mut} 
asserts that this mutation class and 
the topology of the original complex singularity uniquely determine each other.
}
\label{fig:main-conjecture}
\end{figure}

\pagebreak[3]

\begin{example}
\label{example:4-lines-quivers}
Recall that Figure~\ref{fig:4-lines-quivers} shows the quivers associated with four
different morsifications of the same complex singularity, the quasihomogeneous singularity of type~$(4,4)$, cf.\ Figure~\ref{fig:4-lines}.
It is not hard to verify (with the help of one of the widely available software tools
for quiver mutations) that these four quivers are mutation equivalent to each other,
in agreement with Conjecture~\ref{conj:morsif=mut}. 
Moreover it can be shown that any real morsification of a complex singularity 
that is topologically inequivalent to the one referenced above gives rise to a quiver which 
is not mutation equivalent to these four quivers---again agreeing with Conjecture~\ref{conj:morsif=mut}. 
\end{example}

\begin{example}
For each cell $(a,b)$ of the table in Figure~\ref{fig:divides-quasihom},
the quivers associated to the divides shown therein 
are mutation equivalent to each other. 
Moreover the quivers appearing in different cells of the table are not mutation equivalent to each other. 
\end{example}

\begin{remark}
Conjecture~\ref{conj:morsif=mut} only applies to quivers which are already known
to have come from morsifications. 
We note that for a typical singularity,  the corresponding mutation equivalence class contains infinitely many pairwise
non-isomorphic quivers; among them, the quivers associated with real morsifications of a given singularity form a finite subset. 
Even if the mutation equivalence class is finite, most quivers appearing in it do not arise from morsifications. For example, the mutation class of a quiver of type~$A_n$ 
(i.e., an arbitrary orientation of a Dynkin diagram of type~$A_n$) 
contains exponentially many (as a function of~$n$) pairwise nonisomorphic quivers, see, e.g.,~\cite{torkildsen}; 
among them, at most two come from morsifications of a type~$A_n$ singularity, 
cf.\  Proposition~\ref{pr:ade-morsifications} below or the top ($b=2$) row of Figure~\ref{fig:divides-quasihom}. 
\end{remark}

\begin{remark}
Conjecture~\ref{conj:morsif=mut}, once established, would imply that 
any topological invariant~$\alpha=\alpha(C,z)$ of an isolated plane curve singularity $(C,z)$ is uniquely determined by the mutation equivalence class of the quiver~$Q$ of some (equivalently, any) real morsification of the singularity~$(C,z)$. 
Viewing $\alpha$ as a function of~$Q$, we conclude that this function must take the same value
at all quivers in a given mutation equivalence class
which are known to come from a morsification. 
Put differently, $\alpha(Q)$ should be (a restriction of) a mutation-invariant function of quivers. 
It would be very interesting to understand the combinatorial meaning of $\alpha(Q)$
for various well-studied topological invariants~$\alpha$. 

To illustrate, let us recall that Proposition~\ref{pr:top-inv-from-B} provided direct descriptions 
of three topological invariants of a singularity in terms of a quiver~$Q$ constructed from its real morsification. 
These three invariants (the Milnor number, the number of complex local branches,
and the $\delta$-invariant) are all expressed as functions of the number of vertices in~$Q$
(which is obviously a mutation invariant) and the rank of~$B(Q)$, the skew-symmetric matrix 
associated with the quiver~$Q$. 
It is well known (see \cite[Lemma 3.2]{ca3}) that the rank of $B(Q)$ is invariant under mutations of~$Q$, so the formulas of Proposition~\ref{pr:top-inv-from-B} are mutation invariant, as expected. 
\end{remark}

\newpage

\section{Plabic graphs 
}
\label{sec:plabic-graphs}

Plabic graphs were introduced by A.~Postnikov
\cite[Section~12]{postnikov}, who used them to describe 
parametrizations of cells in totally nonnegative Grassmannians. 
We~review the basic notions of this construction below, 
adapting it for our current purposes. 
The differences between our setting and Postnikov's 
are discussed in Remark~\ref{rem:us-vs-postnikov}. 

\begin{definition}
\label{def:plabic}
A finite connected planar graph~$P$ properly embedded
into a disk~$\Disk$ (as a $1$-dimensional cell complex) is called 
a \emph{plabic} 
\emph{graph} if 
\begin{itemize}[leftmargin=.3in]
\item
each vertex of~$P$ is colored in one of the two colors, either black or white; 
the~coloring does not have to be proper; 
\item
each vertex of $P$ lying in the interior of~$\Disk$ is trivalent 
(i.e., has degree~$3$);
\item 
each vertex of $P$ lying on the boundary~$\partial\Disk$ is univalent 
(i.e., has degree~$1$);
\item 
each internal face of~$P$ (i.e., a face not adjacent to~$\partial\Disk$)
is separated from at least one other internal face by an edge whose endpoints 
have different colors. (This condition does not apply if $P$ has 
a single internal face.)
\end{itemize}
We view plabic graphs up to isotopy, and up to \emph{color reversal},
which switches the color of all vertices in~$P$. 
Examples are shown in Figure~\ref{fig:plabic-graphs}. 
\end{definition}


\begin{figure}[ht]
\begin{center}
\setlength{\unitlength}{1pt}
\begin{picture}(60,70)(0,-5)
\thinlines
\put(30,30){\circle{70.5}}
\put(30,30){\circle{73.5}}
\put(0,10){\circle*{5}}\put(0,50){\circle*{5}}\put(60,10){\circle*{5}}\put(60,50){\circle*{5}}
\thicklines
\put(0,10){\line(1,1){8.5}}
\put(60,50){\line(-1,-1){8.5}}
\put(0,50){\line(1,-1){8.5}}
\put(60,10){\line(-1,1){8.5}}
\put(10,22.5){\line(0,1){15}}
\put(30,22.5){\line(0,1){15}}
\put(50,22.5){\line(0,1){15}}
\put(12.5,20){\line(1,0){15}}
\put(12.5,40){\line(1,0){15}}
\put(32.5,20){\line(1,0){15}}
\put(32.5,40){\line(1,0){15}}

\put(10,20){\circle{5}}
\put(10,40){\circle*{5}}
\put(30,20){\circle*{5}}
\put(30,40){\circle{5}}
\put(50,20){\circle{5}}
\put(50,40){\circle*{5}}
\end{picture}
\begin{picture}(40,40)(0,0)
\end{picture}
\begin{picture}(60,70)(0,-5)
\thinlines
\put(30,30){\circle{70.5}}
\put(30,30){\circle{73.5}}
\put(0,10){\circle*{5}}\put(0,50){\circle*{5}}\put(60,10){\circle*{5}}\put(60,50){\circle*{5}}
\thicklines
\put(0,10){\line(1,1){8.5}}
\put(60,50){\line(-1,-1){8.5}}
\put(0,50){\line(1,-1){8.5}}
\put(60,10){\line(-1,1){8.5}}
\put(10,22.5){\line(0,1){15}}
\put(30,22.5){\line(0,1){15}}
\put(50,22.5){\line(0,1){15}}
\put(12.5,20){\line(1,0){15}}
\put(12.5,40){\line(1,0){15}}
\put(32.5,20){\line(1,0){15}}
\put(32.5,40){\line(1,0){15}}
\put(10,20){\circle*{5}}
\put(10,40){\circle{5}}
\put(30,20){\circle{5}}
\put(30,40){\circle*{5}}
\put(50,20){\circle{5}}
\put(50,40){\circle*{5}}
\end{picture}
\begin{picture}(40,40)(0,0)
\end{picture}
\begin{picture}(60,70)(0,-5)
\thinlines
\put(30,30){\circle{70.5}}
\put(30,30){\circle{73.5}}
\put(0,10){\circle*{5}}\put(0,50){\circle*{5}}\put(60,10){\circle*{5}}\put(60,50){\circle*{5}}
\thicklines
\put(0,10){\line(1,1){8.5}}
\put(60,50){\line(-4,-1){17.5}}
\put(0,50){\line(1,-1){8.5}}
\put(60,10){\line(-1,1){8.5}}
\put(10,22.5){\line(0,1){15}}
\put(40,32.5){\line(0,1){10}}
\put(12.5,41){\line(6,1){25}}
\put(12.5,20){\line(1,0){15}}
\put(32.5,20){\line(1,0){15}}
\qbezier(37.5,30)(30,27.5)(30,22.5)
\qbezier(42.5,30)(50,27.5)(50,22.5)
\put(10,20){\circle*{5}}
\put(10,40){\circle{5}}
\put(30,20){\circle{5}}
\put(50,20){\circle{5}}
\put(40,45){\circle*{5}}
\put(40,30){\circle*{5}}
\end{picture}
\begin{picture}(40,40)(0,0)
\end{picture}
\begin{picture}(60,70)(0,-5)
\thinlines
\put(30,30){\circle{70.5}}
\put(30,30){\circle{73.5}}
\put(0,10){\circle*{5}}\put(0,50){\circle*{5}}\put(60,50){\circle*{5}}
\thicklines
\put(0,10){\line(1,1){8.5}}
\put(60,50){\line(-4,-1){17.5}}
\put(0,50){\line(1,-1){8.5}}
\put(10,22.5){\line(0,1){15}}
\put(40,32.5){\line(0,1){10}}
\put(12.5,41){\line(6,1){25}}
\put(12.5,20){\line(1,0){15}}
\qbezier(40,27.5)(40,20)(32.5,20)
\qbezier(37.5,30)(30,27.5)(30,22.5)
\put(10,20){\circle*{5}}
\put(10,40){\circle{5}}
\put(30,20){\circle{5}}
\put(40,45){\circle*{5}}
\put(40,30){\circle*{5}}
\end{picture}
\end{center}
\caption{Plabic graphs. The first two graphs are related by a
  square~move;~the second and the third by a flip move; the third and the fourth by tail removal.}
\label{fig:plabic-graphs}
\end{figure}

There are several types of transformations of plabic graphs 
which play an important role in this theory. 
First, there are three types of local moves :  

\begin{definition}
\label{def:moves}
Local \emph{moves} on plabic graphs are defined as follows,  
\hbox{cf.\ Figure~\ref{fig:plabic-moves}:} 
\begin{itemize}[leftmargin=.3in]
\item The \emph{flip} move replaces two adjacent trivalent vertices of
  the same color with two other vertices of the same color, connected
  in a different way.
\item The \emph{square} move switches the colors on a 4-cycle of
  vertices of alternating colors, subject to
  Restriction~\ref{restr:square-move-technical} below. 
  (This restriction is rarely relevant, so a casual reader may skip this technical detail.) 
\item 
The \emph{tail removal} move removes an edge~$e$ (a~``tail'')
  connected at one end to~$\partial\Disk$ and at the other end to a
  trivalent vertex~$v$. After removing~$e$, we also remove~$v$, and merge
  the two remaining edges which were incident to~$v$. The reverse move, called \emph{tail
    attachment}, inserts a vertex~$v$ (of any color)
  into an edge bordering a region~$R$ adjacent to~$\partial\Disk$, and adds a new edge~$e$ connecting $v$ across~$R$ to a vertex (of any color) 
  lying on~$\partial\Disk$.
\end{itemize}
Two plabic graphs related via a sequence of local moves are called
\emph{move equivalent}. 
See Figure~\ref{fig:plabic-graphs}. 
\end{definition}

\begin{figure}[ht]
\begin{tabular}{ll}
\setlength{\unitlength}{0.9pt}
\begin{picture}(60,40)(0,0)
\put(20,20){\makebox(0,0){(a)}}
\end{picture}
&
\setlength{\unitlength}{0.9pt}
\begin{picture}(40,40)(0,0)
\thicklines
\put(0,0){\line(2,1){20}}
\put(20,30){\line(2,1){20}}
\put(0,40){\line(2,-1){20}}
\put(20,10){\line(2,-1){20}}
\put(20,10){\line(0,1){20}}
\put(20,10){\circle*{5}}
\put(20,30){\circle*{5}}
\end{picture}
\begin{picture}(40,40)(0,0)
\put(20,20){\makebox(0,0){$\longleftrightarrow$}}
\end{picture}
\begin{picture}(40,40)(0,0)
\thicklines
\put(0,0){\line(1,2){10}}
\put(30,20){\line(1,2){10}}
\put(40,0){\line(-1,2){10}}
\put(10,20){\line(-1,2){10}}
\put(10,20){\line(1,0){20}}
\put(10,20){\circle*{5}}
\put(30,20){\circle*{5}}
\end{picture}
\begin{picture}(60,40)(0,0)
\end{picture}
\begin{picture}(40,40)(0,0)
\thicklines
\put(0,0){\line(2,1){18}}
\put(40,40){\line(-2,-1){18}}
\put(0,40){\line(2,-1){18}}
\put(40,0){\line(-2,1){18}}
\put(20,12.5){\line(0,1){15}}
\put(20,10){\circle{5}}
\put(20,30){\circle{5}}
\end{picture}
\begin{picture}(40,40)(0,0)
\thicklines
\put(20,20){\makebox(0,0){$\longleftrightarrow$}}
\end{picture}
\begin{picture}(40,40)(0,0)
\thicklines
\put(0,0){\line(1,2){9}}
\put(40,40){\line(-1,-2){9}}
\put(40,0){\line(-1,2){9}}
\put(0,40){\line(1,-2){9}}
\put(12.5,20){\line(1,0){15}}
\put(10,20){\circle{5}}
\put(30,20){\circle{5}}
\end{picture}
\\[.2in]
\setlength{\unitlength}{0.9pt}
\begin{picture}(60,40)(0,0)
\put(20,20){\makebox(0,0){(b)}}
\end{picture}
&
\setlength{\unitlength}{0.9pt}
\begin{picture}(40,40)(0,0)
\thicklines
\put(0,0){\line(1,1){8.5}}
\put(40,40){\line(-1,-1){8.5}}
\put(0,40){\line(1,-1){8.5}}
\put(40,0){\line(-1,1){8.5}}
\put(10,12.5){\line(0,1){15}}
\put(30,12.5){\line(0,1){15}}
\put(12.5,10){\line(1,0){15}}
\put(12.5,30){\line(1,0){15}}
\put(10,10){\circle*{5}}
\put(10,30){\circle{5}}
\put(30,10){\circle{5}}
\put(30,30){\circle*{5}}
\end{picture}
\begin{picture}(40,40)(0,0)
\thicklines
\put(20,20){\makebox(0,0){$\longleftrightarrow$}}
\end{picture}
\begin{picture}(40,40)(0,0)
\thicklines
\put(0,0){\line(1,1){8.5}}
\put(40,40){\line(-1,-1){8.5}}
\put(0,40){\line(1,-1){8.5}}
\put(40,0){\line(-1,1){8.5}}
\put(10,12.5){\line(0,1){15}}
\put(30,12.5){\line(0,1){15}}
\put(12.5,10){\line(1,0){15}}
\put(12.5,30){\line(1,0){15}}
\put(10,10){\circle{5}}
\put(10,30){\circle*{5}}
\put(30,10){\circle*{5}}
\put(30,30){\circle{5}}
\end{picture}
\\[0in]
\setlength{\unitlength}{0.9pt}
\begin{picture}(60,40)(0,0)
\put(20,20){\makebox(0,0){(c)}}
\end{picture}
&
\setlength{\unitlength}{0.9pt}
\begin{picture}(40,30)(0,0)
\thinlines
\put(0,6){\line(1,0){18}}
\put(0,4){\line(1,0){18}}
\put(40,6){\line(-1,0){18}}
\put(40,4){\line(-1,0){18}}
\thicklines

\put(0,25){\line(1,0){40}}
\put(20,25){\line(0,-1){17.5}}
\put(20,25){\circle*{5}}
\put(20,5){\circle{5}}
\put(-12,5){\makebox(0,0){$\partial\Disk$}}
\end{picture}
\begin{picture}(40,30)(0,0)
\thicklines
\put(20,15){\makebox(0,0){$\longleftrightarrow$}}
\end{picture}
\begin{picture}(40,30)(0,0)
\thinlines
\put(0,6){\line(1,0){40}}
\put(0,4){\line(1,0){40}}
\thicklines
\put(0,25){\line(1,0){40}}
\end{picture}
\begin{picture}(40,30)(0,0)
\thicklines
\put(20,15){\makebox(0,0){$\longleftrightarrow$}}
\end{picture}
\begin{picture}(40,30)(0,0)
\thinlines
\put(0,6){\line(1,0){40}}
\put(0,4){\line(1,0){40}}
\put(20,5){\circle*{5}}
\thicklines
\put(0,25){\line(1,0){17.5}}
\put(40,25){\line(-1,0){17.5}}
\put(20,5){\line(0,1){17.5}}
\put(20,25){\circle{5}}
\put(52,5){\makebox(0,0){$\partial\Disk$}}
\end{picture}
\end{tabular}
\vspace{-.05in}
\caption{Local moves in plabic graphs. 
(a) The flip move (two versions). 
(b)~The square move.
(c)~The tail attachment/removal moves. 
For this last type of move, the colors of the two vertices involved 
can be arbitrary.}
\vspace{-.25in}
\label{fig:plabic-moves}
\end{figure}

\begin{restr}
\label{restr:square-move-technical}
We impose a restriction on the square move of
Figure~\ref{fig:plabic-moves}(b): 
among the four faces surrounding the square, 
the opposite ones are allowed to coincide, but the
consecutive ones must be distinct. 
See Figure~\ref{fig:square-move-tricky}.
\end{restr}

\begin{figure}[ht]
\begin{center}
\setlength{\unitlength}{0.9pt}
\begin{picture}(50,25)(-30,15)
\thinlines
\thicklines
\put(10,22.5){\line(0,1){15}}
\put(30,22.5){\line(0,1){15}}
\qbezier(32.5,20)(50,20)(50,30)
\qbezier(32.5,40)(50,40)(50,30)
\put(12.5,20){\line(1,0){15}}

\put(8,39){\line(-2,-1){16}}
\put(12.5,40){\line(1,0){15}}
\put(-27.5,30){\line(1,0){15}}
\put(-87.5,30){\line(1,0){15}}
\put(8,21){\line(-2,1){16}}
\put(-52,21){\line(-2,1){16}}
\put(-32,31){\line(-2,1){16}}
\put(-32,29){\line(-2,-1){16}}
\put(-52,39){\line(-2,-1){16}}
\put(-50,42.5){\line(0,1){10}}
\put(-50,17.5){\line(0,-1){10}}

\put(10,20){\circle*{5}}
\put(10,40){\circle{5}}
\put(30,20){\circle{5}}
\put(30,40){\circle*{5}}
\put(20,30){\makebox(0,0){$A$}}
\put(-50,30){\makebox(0,0){$B$}}
\put(65,30){\makebox(0,0){$C$}}
\put(-10,30){\circle*{5}}
\put(-30,30){\circle{5}}
\put(-70,30){\circle{5}}
\put(-50,40){\circle*{5}}
\put(-50,20){\circle*{5}}
\end{picture}
\end{center}
\caption{A fragment of a plabic graph. The square move is allowed at~$A$,
  but not at~$B$, because face~$C$ is adjacent to two consecutive
  sides of~$B$.} 
\label{fig:square-move-tricky}
\end{figure}

\vspace{-.2in}

\begin{remark}
Using a tail removal followed by a tail attachment, one can 
change the color of any boundary vertex (or the vertex connected to it).
For this reason, when drawing a plabic graph, 
we sometimes do not show the boundary of the ambient disk, 
and accordingly do not specify the colors of boundary vertices. 
\end{remark}

\begin{remark}
Applying repeated tail removals, 
any plabic graph can be transformed into a trivalent 
one, with no vertices on the boundary~$\partial\Disk$. 
Note however that restricting the setup to trivalent plabic graphs
would have resulted in a different 
equivalence relation among them, for the reasons explained in 
Figure~\ref{fig:tail-attachment-removal}. 
\end{remark}


\begin{figure}[ht]
\begin{center}
\setlength{\unitlength}{0.9pt}
\begin{picture}(50,55)(0,-5)
\thicklines
\qbezier(4,11.8)(-10,42)(25,50)
\qbezier(46,11.8)(60,42)(25,50)
\qbezier(6.5,8.5)(25,-9)(43.5,8.5)
\put(25,27.5){\line(0,1){22.5}}
\put(23,23.5){\line(-4,-3){16}}
\put(27,23.5){\line(4,-3){16}}
\put(5,10){\circle*{5}}
\put(45,10){\circle*{5}}
\put(25,50){\circle*{5}}
\put(25,25){\circle{5}}
\end{picture}
\qquad
\begin{picture}(50,50)(0,-5)
\thicklines
\qbezier(6.5,8.5)(14,0.2)(22.5,-0.2)
\qbezier(43.5,8.5)(36,0.2)(27.5,-0.2)
\qbezier(4,11.8)(-10,42)(25,50)
\qbezier(46,11.8)(60,42)(25,50)
\put(25,27.5){\line(0,1){22.5}}
\put(23,23.5){\line(-4,-3){16}}
\put(27,23.5){\line(4,-3){16}}
\put(5,10){\circle*{5}}
\put(45,10){\circle*{5}}
\put(25,50){\circle*{5}}
\put(25,25){\circle{5}}
\put(25,-2.5){\line(0,-1){7.5}}
\put(25,0){\circle{5}}
\end{picture}
\qquad
\begin{picture}(50,50)(0,-5)
\thicklines
\qbezier(4,11.8)(-10,42)(25,50)
\qbezier(46,11.8)(60,42)(25,50)
\qbezier(6.5,8.5)(25,-9)(43.5,8.5)
\put(25,27.5){\line(0,1){22.5}}
\put(23,23.5){\line(-4,-3){16}}
\put(27,23.5){\line(4,-3){16}}
\put(5,10){\circle{5}}
\put(45,10){\circle{5}}
\put(25,50){\circle*{5}}
\put(25,25){\circle*{5}}
\put(25,-2.5){\line(0,-1){7.5}}
\put(25,0){\circle*{5}}
\end{picture}
\qquad
\begin{picture}(50,50)(0,-5)
\thicklines
\qbezier(4,11.8)(-10,42)(25,50)
\qbezier(46,11.8)(60,42)(25,50)
\qbezier(6.5,8.5)(25,-9)(43.5,8.5)
\put(25,27.5){\line(0,1){22.5}}
\put(23,23.5){\line(-4,-3){16}}
\put(27,23.5){\line(4,-3){16}}
\put(5,10){\circle{5}}
\put(45,10){\circle{5}}
\put(25,50){\circle*{5}}
\put(25,25){\circle*{5}}
\end{picture}
\end{center}
\caption{The trivalent plabic graphs on the left 
and on the far right are related by a sequence of three moves: 
a tail attachment, a square move, and a tail removal. 
These two trivalent graphs are not related via flip and
square moves alone, 
since such moves do not change the number of vertices of each~color.}
\label{fig:tail-attachment-removal}
\end{figure}

\vspace{-.15in}

\begin{remark}
\label{rem:us-vs-postnikov}
As explained by Postnikov~\cite{postnikov}, Definitions~\ref{def:plabic} 
and~\ref{def:moves} naturally
extend to arbitrary planar graphs embedded in a disk. 
We find it easier, for our current purposes, to work in the
restricted generality of trivalent-univalent graphs. 
We also require that for each internal face, there is a black-and-white
edge separating it from another internal face. 
This condition, which propagates under all types of moves, 
ensures that the moves do not create monogons, nor
digons with vertices of the same color. 

A more significant difference between our setting and Postnikov's
is the introduction of the tail attachment/removal moves,
which were not present in~\cite{postnikov}. 
\end{remark}

\begin{remark}
A plabic graph defines a dual triangulation of the disk~$\Disk$, 
with each triangle colored black or white. 
A flip
move in a plabic graph corresponds to a flip in the 
dual triangulation (hence the terminology), 
i.e., to removing an interior arc~$\alpha$ 
and replacing it by another ``diagonal'' of the quadrilateral region 
formed by the two triangles separated by~$\alpha$.
Note that we are only allowed to do this when the triangles are of the
same color. 
Incidentally, this process will never create self-folded triangles 
(in the terminology of~\cite{cats1})
since those correspond to monogons in a plabic graph. 
\end{remark}

It is well known that local moves on plabic graphs are a special case
of quiver mutation. 
To see this, one needs the following definition. 

\begin{definition}
\label{def:Q(P)}
The quiver~$Q(P)$ associated with a plabic graph~$P$ is constructed as follows.
Place a vertex of~$Q(P)$ into each internal face of~$P$.
For each edge~$e$ in~$P$ such that
\begin{itemize}[leftmargin=.3in]
\item
the endpoints of $e$ are of different color, and 
\item
the faces $F_1,F_2$ on the two sides of~$e$ are internal and distinct,  
\end{itemize}
draw an arrow of~$Q(P)$ across~$e$ connecting the vertices of~$Q(P)$ located inside the faces $F_1$ and~$F_2$, 
and orient this arrow so that the black endpoint of~$e$ appears on its right as one moves
in the chosen direction. 
If this construction produces oriented cycles of length~2,
i.e., pairs of arrows connecting the same vertices but going in
opposite directions, then remove such pairs, one by one. 
See Figure~\ref{fig:quivers-plabic-graphs}.
\end{definition}

We note that the colors of boundary vertices do not affect the quiver. 

\begin{figure}[ht]
\begin{center}
\vspace{0.1in}
\setlength{\unitlength}{1.2pt}
\begin{picture}(100,40)(0,-10)
\thicklines
\put(10,0){\circle*{5}}
\put(30,0){\circle{5}}
\put(70,0){\circle*{5}}
\put(90,0){\circle{5}}
\put(10,20){\circle{5}}
\put(30,20){\circle*{5}}
\put(50,20){\circle*{5}}
\put(70,20){\circle{5}}
\put(10,40){\circle*{5}}
\put(50,40){\circle{5}}
\put(70,40){\circle*{5}}
\put(90,40){\circle{5}}

\put(0,0){\line(1,0){27.5}}
\put(32.5,0){\line(1,0){55}}
\put(92.5,0){\line(1,0){7.5}}
\put(12.5,20){\line(1,0){55}}
\put(0,40){\line(1,0){47.5}}
\put(52.5,40){\line(1,0){35}}
\put(92.5,40){\line(1,0){7.5}}

\put(10,2.5){\line(0,1){15}}
\put(10,22.5){\line(0,1){15}}
\put(30,17.5){\line(0,-1){15}}
\put(50,22.5){\line(0,1){15}}
\put(70,2.5){\line(0,1){15}}
\put(70,22.5){\line(0,1){15}}
\put(90,2.5){\line(0,1){35}}

\put(20,10){\red{\circle*{5}}}
\put(61,10){\red{\circle*{5}}}

\put(20,30){\red{\circle*{5}}}
\put(61,30){\red{\circle*{5}}}
\put(83,20){\red{\circle*{5}}}

\put(24,30){\red{\vector(1,0){33}}}
\put(56,10){\red{\vector(-1,0){33}}}
\put(79,22){\red{\vector(-2.3,1){15}}}
\put(20,14){\red{\vector(0,1){13}}}
\put(61,26){\red{\vector(0,-1){13}}}
\put(64,12){\red{\vector(2.3,1){15}}}

\end{picture}
\qquad
\setlength{\unitlength}{1.2pt}
\begin{picture}(120,58)(-50,0)
\thinlines
\thicklines
\put(10,22.5){\line(0,1){15}}
\put(10,2.5){\line(0,1){15}}
\put(10,42.5){\line(0,1){15}}
\put(30,22.5){\line(0,1){15}}
\qbezier(32.5,20)(50,20)(50,30)
\qbezier(32.5,40)(50,40)(50,30)
\qbezier(12.5,0)(70,0)(70,30)
\qbezier(7.5,0)(-20,0)(-20,30)
\qbezier(7.5,60)(-20,60)(-20,30)
\qbezier(12.5,60)(70,60)(70,30)
\put(-20,30){\line(-1,0){20}}
\put(70,30){\line(1,0){20}}

\put(12.5,20){\line(1,0){15}}
\put(12.5,40){\line(1,0){15}}
\put(10,0){\circle{5}}
\put(10,60){\circle*{5}}
\put(10,20){\circle*{5}}
\put(10,40){\circle{5}}
\put(30,20){\circle{5}}
\put(30,40){\circle*{5}}
\put(-20,30){\circle*{5}}
\put(70,30){\circle*{5}}

\put(-5,40){\red{\circle*{5}}}
\put(19,30){\red{\circle*{5}}}
\put(40,30){\red{\circle*{5}}}
\put(30,50){\red{\circle*{5}}}

\put(36,30){\red{\vector(-1,0){13}}}
\put(32,47){\red{\vector(1,-2){7}}}
\put(22,33){\red{\vector(1,2){7}}}
\put(20,34){\red{\vector(1,2){7}}}
\put(-1,38){\red{\vector(3,-1){18}}}
\put(27,49){\red{\vector(-7,-2){29}}}
\put(26,51){\red{\vector(-7,-2){29}}}

\end{picture}
\end{center}
\caption{Quivers associated with plabic graphs.
The double arrows in the right quiver correspond to the instances where a pair of faces of the
plabic graph share two disconnected boundary segments.} 
\label{fig:quivers-plabic-graphs}
\end{figure}


The following observation is implicit in Postnikov's original work~\cite{postnikov}:  

\begin{proposition}
\label{pr:plabic-vs-quivers}
If two plabic graphs are move equivalent to each other, 
then their associated quivers are mutation equivalent. 
\end{proposition}

\begin{proof}
It is straightforward to check that a square move in a plabic graph
translates into a quiver mutation, 
and that the quiver associated with a plabic graph does not change under 
a flip move, or a tail attachment/removal. 
See Figure~\ref{fig:plabic-moves-are-mutations}. 
\end{proof}


\begin{figure}[ht]
\begin{center}
\setlength{\unitlength}{1.1pt}
\begin{picture}(100,40)(0,0)
\thicklines
\put(10,0){\circle*{5}}
\put(30,0){\circle{5}}
\put(70,0){\circle*{5}}
\put(90,0){\circle{5}}
\put(10,20){\circle{5}}
\put(30,20){\circle*{5}}
\put(50,20){\circle*{5}}
\put(70,20){\circle{5}}
\put(10,40){\circle*{5}}
\put(50,40){\circle{5}}
\put(70,40){\circle*{5}}
\put(90,40){\circle{5}}

\put(0,0){\line(1,0){27.5}}
\put(32.5,0){\line(1,0){55}}
\put(92.5,0){\line(1,0){7.5}}
\put(12.5,20){\line(1,0){55}}
\put(0,40){\line(1,0){47.5}}
\put(52.5,40){\line(1,0){35}}
\put(92.5,40){\line(1,0){7.5}}

\put(10,2.5){\line(0,1){15}}
\put(10,22.5){\line(0,1){15}}
\put(30,17.5){\line(0,-1){15}}
\put(50,22.5){\line(0,1){15}}
\put(70,2.5){\line(0,1){15}}
\put(70,22.5){\line(0,1){15}}
\put(90,2.5){\line(0,1){35}}

\put(20,10){\red{\circle*{5}}}
\put(61,10){\red{\circle*{5}}}

\put(20,30){\red{\circle*{5}}}
\put(61,30){\red{\circle*{5}}}
\put(83,20){\red{\circle*{5}}}

\put(24,30){\red{\vector(1,0){33}}}
\put(56,10){\red{\vector(-1,0){33}}}
\put(79,22){\red{\vector(-2.3,1){15}}}
\put(20,14){\red{\vector(0,1){13}}}
\put(61,26){\red{\vector(0,-1){13}}}
\put(64,12){\red{\vector(2.3,1){15}}}

\end{picture}
\qquad
\begin{picture}(100,40)(0,0)
\thicklines
\put(10,0){\circle*{5}}
\put(30,0){\circle{5}}
\put(70,0){\circle*{5}}
\put(90,0){\circle{5}}
\put(10,20){\circle{5}}
\put(30,20){\circle*{5}}
\put(50,20){\circle{5}}
\put(70,20){\circle*{5}}
\put(10,40){\circle*{5}}
\put(50,40){\circle*{5}}
\put(70,40){\circle{5}}
\put(90,40){\circle{5}}

\put(0,0){\line(1,0){27.5}}
\put(32.5,0){\line(1,0){55}}
\put(92.5,0){\line(1,0){7.5}}
\put(12.5,20){\line(1,0){35}}
\put(52.5,20){\line(1,0){15}}
\put(0,40){\line(1,0){47.5}}
\put(52.5,40){\line(1,0){15}}
\put(72.5,40){\line(1,0){15}}
\put(92.5,40){\line(1,0){7.5}}

\put(10,2.5){\line(0,1){15}}
\put(10,22.5){\line(0,1){15}}
\put(30,17.5){\line(0,-1){15}}
\put(50,22.5){\line(0,1){15}}
\put(70,2.5){\line(0,1){15}}
\put(70,22.5){\line(0,1){15}}
\put(90,2.5){\line(0,1){35}}

\put(20,10){\red{\circle*{5}}}
\put(61,10){\red{\circle*{5}}}

\put(20,30){\red{\circle*{5}}}
\put(61,30){\red{\circle*{5}}}
\put(83,20){\red{\circle*{5}}}

\put(56,30){\red{\vector(-1,0){33}}}
\put(56,10){\red{\vector(-1,0){33}}}
\put(20,14){\red{\vector(0,1){13}}}
\put(61,14){\red{\vector(0,1){13}}}
\put(64,28){\red{\vector(2.3,-1){16}}}
\put(24,28){\red{\vector(2,-1){33}}}

\end{picture}
\qquad
\begin{picture}(100,40)(0,0)
\thicklines
\put(10,0){\circle*{5}}
\put(30,0){\circle{5}}
\put(70,0){\circle*{5}}
\put(90,0){\circle{5}}
\put(10,20){\circle{5}}
\put(30,20){\circle*{5}}
\put(50,20){\circle{5}}
\put(50,0){\circle*{5}}
\put(10,40){\circle*{5}}
\put(50,40){\circle*{5}}
\put(70,40){\circle{5}}
\put(90,40){\circle{5}}

\put(0,0){\line(1,0){27.5}}
\put(32.5,0){\line(1,0){55}}
\put(92.5,0){\line(1,0){7.5}}
\put(12.5,20){\line(1,0){35}}
\
\put(0,40){\line(1,0){47.5}}
\put(52.5,40){\line(1,0){15}}
\put(72.5,40){\line(1,0){15}}
\put(92.5,40){\line(1,0){7.5}}

\put(10,2.5){\line(0,1){15}}
\put(10,22.5){\line(0,1){15}}
\put(30,17.5){\line(0,-1){15}}
\put(50,22.5){\line(0,1){15}}
\put(50,2.5){\line(0,1){15}}
\put(70,2.5){\line(0,1){35}}
\put(90,2.5){\line(0,1){35}}

\put(20,10){\red{\circle*{5}}}
\put(40,10){\red{\circle*{5}}}

\put(30,30){\red{\circle*{5}}}
\put(60,20){\red{\circle*{5}}}
\put(83,20){\red{\circle*{5}}}

\put(56,22){\red{\vector(-3,1){22}}}
\put(36,10){\red{\vector(-1,0){13}}}
\put(22,14){\red{\vector(1,2){6}}}
\put(44,12){\red{\vector(2,1){13}}}

\put(64,20){\red{\vector(1,0){15}}}
\put(32,26){\red{\vector(1,-2){6}}}

\end{picture}
\vspace{.1in}
\end{center}
\caption{The first two plabic graphs are related by a square move; their quivers are obtained from each other by a single mutation.
The second and the third plabic graphs are related by a flip move, and have isomorphic quivers.
}
\label{fig:plabic-moves-are-mutations}
\end{figure}

\begin{remark}
\label{rem:q-not=>p}
The converse to Proposition~\ref{pr:plabic-vs-quivers} is unfortunately false:
there exist plabic graphs which are not move equivalent even though 
their quivers are isomorphic (hence mutation equivalent). 
An example is shown in Figure~\ref{fig:need-switch}. 
See also Example~\ref{ex:using-links-to show-move-inequivalence}/Figure~\ref{fig:isthmus-move}. 
\end{remark}

\begin{figure}[ht]
\begin{center}
\setlength{\unitlength}{2.1pt}
\begin{picture}(40,90)(0,0)
\thicklines
\put(20,0){\circle{3}}
\put(20,10){\circle*{3}}
\put(20,20){\circle{3}}
\put(20,30){\circle*{3}}
\put(20,60){\circle{3}}
\put(20,70){\circle*{3}}
\put(20,80){\circle{3}}
\put(20,90){\circle*{3}}
\put(10,40){\circle{3}}
\put(10,50){\circle*{3}}
\put(30,40){\circle{3}}
\put(30,50){\circle*{3}}

\put(20,1.5){\line(0,1){7}}
\put(20,21.5){\line(0,1){7}}
\put(20,61.5){\line(0,1){7}}
\put(20,81.5){\line(0,1){7}}
\qbezier(21,1)(80,45)(21,89)
\qbezier(19,1)(-40,45)(19,89)

\qbezier(21,11)(26,15)(21,19)
\qbezier(19,11)(14,15)(19,19)
\qbezier(21,71)(26,75)(21,79)
\qbezier(19,71)(14,75)(19,79)

\qbezier(11,41)(16,45)(11,49)
\qbezier(9,41)(4,45)(9,49)
\qbezier(31,41)(36,45)(31,49)
\qbezier(29,41)(24,45)(29,49)

\put(19,31){\line(-1,1){8}}
\put(21,31){\line(1,1){8}}
\put(19,59){\line(-1,-1){8}}
\put(21,59){\line(1,-1){8}}
\end{picture}
\qquad\qquad\qquad\qquad
\begin{picture}(40,90)(0,0)
\thicklines
\put(20,0){\circle{3}}
\put(20,70){\circle*{3}}
\put(20,80){\circle{3}}
\put(20,10){\circle*{3}}
\put(20,40){\circle{3}}
\put(20,50){\circle*{3}}
\put(20,60){\circle{3}}
\put(20,90){\circle*{3}}
\put(10,20){\circle{3}}
\put(10,30){\circle*{3}}
\put(30,20){\circle{3}}
\put(30,30){\circle*{3}}

\put(20,1.5){\line(0,1){7}}
\put(20,41.5){\line(0,1){7}}
\put(20,61.5){\line(0,1){7}}
\put(20,81.5){\line(0,1){7}}
\qbezier(21,1)(80,45)(21,89)
\qbezier(19,1)(-40,45)(19,89)

\qbezier(21,51)(26,55)(21,59)
\qbezier(19,51)(14,55)(19,59)

\qbezier(21,71)(26,75)(21,79)
\qbezier(19,71)(14,75)(19,79)

\qbezier(11,21)(16,25)(11,29)
\qbezier(9,21)(4,25)(9,29)
\qbezier(31,21)(36,25)(31,29)
\qbezier(29,21)(24,25)(29,29)

\put(19,11){\line(-1,1){8}}
\put(21,11){\line(1,1){8}}
\put(19,39){\line(-1,-1){8}}
\put(21,39){\line(1,-1){8}}
\end{picture}
\end{center}
\caption{Two plabic graphs whose quivers are isomorphic but which are not related to each other by local moves. In fact, the only moves that can be applied to either graph are tail attachments/removals.} 
\label{fig:need-switch}
\end{figure}



We next relate plabic graphs to divides. 

\begin{definition} 
\label{def:plabic-divide}
The set $\PP(D)$ of \emph{plabic graphs attached to a divide}~$D$ is defined as follows.
Replace each node of $D$ by a ``roundabout''
involving four trivalent vertices connected into a square,
and colored alternately black and white, as shown below:
\[
\setlength{\unitlength}{1.8pt}
\begin{picture}(20,25)(0,-2)
\thinlines
\put(0,0){\line(1,1){20}}
\put(0,20){\line(1,-1){20}}
\end{picture}
\begin{picture}(20,25)(0,-2)
\put(10,10){\makebox(0,0){$\longrightarrow$}}
\end{picture}
\begin{picture}(20,25)(0,-2)
\thicklines
\put(0,0){\line(1,1){5}}
\put(15,15){\line(1,1){5}}
\put(0,20){\line(1,-1){4.2}}
\put(20,0){\line(-1,1){4.2}}
\put(15,15){\line(1,1){5}}
\put(6,5){\line(1,0){8}}
\put(6,15){\line(1,0){8}}
\put(5,6){\line(0,1){8}}
\put(15,6){\line(0,1){8}}
\put(5,5){\circle*{2.5}}
\put(15,15){\circle*{2.5}}
\put(5,15){\circle{2.5}}
\put(15,5){\circle{2.5}}
\end{picture}
\]
There are two choices of coloring at each node, 
related to each other by a square~move.  
We then color the endpoints in $D\cap\partial\Disk$ in an arbitrary way. 
The set $\PP(D)$ consists of the plabic graphs which can be obtained 
from the divide~$D$ via this procedure. 
All plabic graphs in $\PP(D)$ are obviously move equivalent to each other. 
\end{definition} 

An example is shown in Figure~\ref{fig:E6-plabic}. 

\begin{figure}[ht]
\begin{center}
\setlength{\unitlength}{1.8pt}
\begin{picture}(50,45)(0,-5)
\thinlines
\qbezier(0,0)(60,60)(60,20)
\qbezier(60,0)(0,60)(0,20)
\qbezier(60,20)(60,0)(30,0)
\qbezier(0,20)(0,0)(30,0)
\end{picture}
\qquad\qquad\qquad
\begin{picture}(50,35)(0,-5)
\thicklines
\put(21.5,20){{\line(1,0){15.75}}}
\put(22.75,32){{\line(1,0){15.75}}}
\put(21.5,20){{\line(0,1){10.75}}}
\put(38.5,21.25){{\line(0,1){10.75}}}

\put(9.25,10){{\line(-1,0){7}}}
\put(10.5,3.5){{\line(-1,0){5.75}}}
\put(10.5,3.5){{\line(0,1){5.25}}}
\put(2.5,10){{\line(1,-6.5){.84}}}

\put(49.5,10){{\line(1,0){6.75}}}
\put(56.5,3.5){{\line(-1,0){5.75}}}
\put(49.5,10){{\line(0,-1){5.25}}}
\put(56.5,3.5){{\line(1,6.5){.84}}}

\put(2.5,10){\circle*{2.5}}
\put(10.5,3.5){\circle*{2.5}}
\put(21.5,20){\circle*{2.5}}
\put(38.5,32){\circle*{2.5}}
\put(49.5,10){\circle*{2.5}}
\put(56.5,3.5){\circle*{2.5}}
\put(3.5,3.5){\circle{2.5}}
\put(10.5,10){\circle{2.5}}
\put(57.5,10){\circle{2.5}}
\put(49.5,3.5){\circle{2.5}}
\put(21.5,32){\circle{2.5}}
\put(38.5,20){\circle{2.5}}

\put(0,0){\circle*{2.5}}
\put(60,0){\circle{2.5}}

\qbezier(20.3,32.5)(0,44.5)(0,20)
\qbezier(39.7,32.5)(60,44.5)(60,20)
\qbezier(60,20)(60,15)(58,11)
\qbezier(0,20)(0,15)(2,11)
\qbezier(11.6,3)(30,-4)(48.4,3)

\put(0,0){{\line(1,1){2.7}}}
\put(56.5,3.5){{\line(1,-1){2.7}}}
\put(11.3,10.8){{\line(11,10){9.4}}}
\put(48.7,10.8){{\line(-11,10){9.4}}}
\end{picture}

\end{center}
\caption{A divide coming from a morsification of 
a type~$E_6$ singularity, and one of the plabic graphs attached to it.}
\label{fig:E6-plabic}
\end{figure}


Definition~\ref{def:plabic-divide} is justified by the following simple but important observation. 

\begin{proposition}
\label{pr:Q(D)=Q(P(D))}
For any divide~$D$ and any plabic graph $P\in \mathbf{P}(D)$ attached to~$D$, the quivers $Q(D)$ and $Q(P)$ are mutation equivalent to each other. 

In fact, there is always a choice of $P\in \mathbf{P}(D)$ such that $Q(D)=Q(P)$. 
\end{proposition}

(See Definitions~\ref{def:quiver-of-divide}, \ref{def:Q(P)}, and~\ref{def:plabic-divide} for the explanations of the notations involved.)

In other words, the quiver of a plabic graph attached to a divide~$D$ is the same (up to mutation equivalence) as the quiver associated with~$D$ (i.e., the oriented \AG-diagram of~$D$). 
See Figure~\ref{fig:E6-quiver-via-plabic}.

\begin{figure}[ht]
\begin{center}
\setlength{\unitlength}{1.8pt}
\begin{picture}(50,38)(0,0)
\thinlines
\qbezier(0,0)(60,60)(60,20)
\qbezier(60,0)(0,60)(0,20)
\qbezier(60,20)(60,0)(30,0)
\qbezier(0,20)(0,0)(30,0)
\thicklines

\put(6.2,6){\circle*{2.5}}
\put(53.8,6){\circle*{2.5}}
\put(30,26.5){\circle*{2.5}}

\put(30,6){\circle*{2.5}}
\put(6.2,26.5){\circle*{2.5}}
\put(53.8,26.5){\circle*{2.5}}

\put(9,6){\red{\vector(1,0){18}}}
\put(51,6){\red{\vector(-1,0){18}}}
\put(9,26.5){\red{\vector(1,0){18}}}
\put(51,26.5){\red{\vector(-1,0){18}}}
\put(6.2,24){\red{\vector(0,-1){15.5}}}
\put(30,24){\red{\vector(0,-1){15.5}}}
\put(53.8,24){\red{\vector(0,-1){15.5}}}
\put(28,8){\red{\vector(-8,7){19.5}}}
\put(32,8){\red{\vector(8,7){19.5}}}
\end{picture}
\qquad\qquad\qquad
\setlength{\unitlength}{1.8pt}
\begin{picture}(50,38)(0,0)
\thicklines
\put(2.5,10){\circle*{2.5}}
\put(10.5,3.5){\circle*{2.5}}
\put(21.5,20){\circle*{2.5}}
\put(38.5,32){\circle*{2.5}}
\put(49.5,10){\circle*{2.5}}
\put(56.5,3.5){\circle*{2.5}}

\put(6.2,6){\red{\circle*{2.5}}}
\put(53.8,6){\red{\circle*{2.5}}}
\put(30,26.5){\red{\circle*{2.5}}}

\put(30,6){\red{\circle*{2.5}}}
\put(6.2,26.5){\red{\circle*{2.5}}}
\put(53.8,26.5){\red{\circle*{2.5}}}

\put(9,6){\red{\vector(1,0){18}}}
\put(51,6){\red{\vector(-1,0){18}}}
\put(9,26.5){\red{\vector(1,0){18}}}
\put(51,26.5){\red{\vector(-1,0){18}}}
\put(6.2,24){\red{\vector(0,-1){15.5}}}
\put(30,24){\red{\vector(0,-1){15.5}}}
\put(53.8,24){\red{\vector(0,-1){15.5}}}
\put(28,8){\red{\vector(-8,7){19.5}}}
\put(32,8){\red{\vector(8,7){19.5}}}

\put(21.5,20){{\line(1,0){15.75}}}
\put(22.75,32){{\line(1,0){15.75}}}
\put(21.5,20){{\line(0,1){10.75}}}
\put(38.5,21.25){{\line(0,1){10.75}}}

\put(9.25,10){{\line(-1,0){7}}}
\put(10.5,3.5){{\line(-1,0){5.75}}}
\put(10.5,3.5){{\line(0,1){5.25}}}
\put(2.5,10){{\line(1,-6.5){.84}}}

\put(49.5,10){{\line(1,0){6.75}}}
\put(56.5,3.5){{\line(-1,0){5.75}}}
\put(49.5,10){{\line(0,-1){5.25}}}
\put(56.5,3.5){{\line(1,6.5){.84}}}

\put(2.5,10){\circle*{2.5}}
\put(10.5,3.5){\circle*{2.5}}
\put(21.5,20){\circle*{2.5}}
\put(38.5,32){\circle*{2.5}}
\put(49.5,10){\circle*{2.5}}
\put(56.5,3.5){\circle*{2.5}}
\put(3.5,3.5){\circle{2.5}}
\put(10.5,10){\circle{2.5}}
\put(57.5,10){\circle{2.5}}
\put(49.5,3.5){\circle{2.5}}
\put(21.5,32){\circle{2.5}}
\put(38.5,20){\circle{2.5}}

\put(0,0){\circle*{2.5}}
\put(60,0){\circle{2.5}}

\qbezier(20.3,32.5)(0,44.5)(0,20)
\qbezier(39.7,32.5)(60,44.5)(60,20)
\qbezier(60,20)(60,15)(58,11)
\qbezier(0,20)(0,15)(2,11)
\qbezier(11.6,3)(30,-4)(48.4,3)

\put(0,0){{\line(1,1){2.7}}}
\put(56.5,3.5){{\line(1,-1){2.7}}}
\put(11.3,10.8){{\line(11,10){9.4}}}
\put(48.7,10.8){{\line(-11,10){9.4}}}
\end{picture}
\end{center}
\caption{Left: the quiver obtained from the \AG-diagram 
(cf.\ Figure~\ref{fig:E6-morsifications}).
Right: the quiver obtained from the plabic graph (cf.\ Figure~\ref{fig:E6-plabic}).}
\label{fig:E6-quiver-via-plabic}
\end{figure}


Experimental evidence suggests that in the case of plabic graphs attached to algebraic divides, the converse to Proposition~\ref{pr:plabic-vs-quivers} holds (cf.\ Remark~\ref{rem:q-not=>p}): 

\begin{conjecture}
\label{conj:shapiro-algebraic}
Plabic graphs attached to algebraic divides are move-equivalent if and only if the corresponding quivers are mutation equivalent. 
\end{conjecture}

\begin{remark}
By Proposition~\ref{pr:Q(D)=Q(P(D))}, it does not matter whether the words ``the corresponding quivers'' appearing in Conjecture~\ref{conj:shapiro-algebraic} are interpreted as ``the quivers associated with the divides'' or as ``the quivers associated with the plabic graphs.'' 
\end{remark}

We conclude this section by describing an equivalence relation on (arbitrary) plabic graphs that conjecturally corresponds to mutation equivalence of their quivers. The readers not interested in this digression may proceed directly to the next section.

The key idea is to complement Postnikov's local moves by certain non-local transformations which do not change the quiver associated with a plabic graph. These transformations are closely related to H.~Whitney's \emph{2-switching} operations which relate different planar embeddings of a given planar graph 
(see, e.g.,~\cite[Section~2.6]{Mohar-Thomassen}). 
Cf.\ also Remark~\ref{rem:AG-planar-embeddings}.


\begin{definition}
\label{def:switch}
We say that two plabic graphs $P$ and $P'$ are related to each other by a \emph{switch} if  $P'$ can be obtained from $P$ in the following way. 
Suppose a closed simple curve~$\mathcal{C}$ in the interior of
the disk~$\Disk$ intersects 
(the drawing of) $P$ exactly twice, at two different edges. 
Since we consider our plabic graphs up to isotopies of the disk,
we may assume, without loss of generality, that $\mathcal{C}$
encloses a rectangle~$R$, and moreover $\mathcal{C}$ intersects~$P$ 
at two points located at the top and the bottom sides of~$R$,
respectively, precisely opposite each other. 

Let $P_{\text{in}}$ denote the portion of (the drawing 
of) $P$ contained inside~$\mathcal{C}$. 
To obtain~$P'$, we flip~$P_{\text{in}}$  upside down
(i.e., replace it by its mirror image with respect to the horizontal
axis of symmetry of~$R$), and reverse the colors
of all vertices in~$P_{\text{in}}$.  The remaining portion of~$P$ is kept intact. 

It is easy to see that applying the same transformation to $P'$ recovers~$P$. 
  
Two plabic graphs related to each other via a sequence of local moves
(see Definition~\ref{def:moves}) and/or switches 
are called \emph{move-and-switch equivalent}.
  
An example is shown in Figure~\ref{fig:isthmus-move}. 
\end{definition}

\begin{figure}[ht]
\begin{center}
\bigskip
\setlength{\unitlength}{1.2pt}
\begin{picture}(120,60)(-5,-8)
\thicklines
\qbezier(40,-2.5)(40,-15)(50,-15)
\qbezier(60,-2.5)(60,-15)(50,-15)
\put(62.5,0){\line(1,0){35}}
\put(2.5,20){\line(1,0){15}}
\put(-17.5,20){\line(1,0){15}}
\put(22.5,20){\line(1,0){15}}
\put(62.5,20){\line(1,0){15}}
\put(82.5,20){\line(1,0){15}}
\put(102.5,20){\line(1,0){15}}
\put(22.5,40){\line(1,0){15}}
\put(-2.5,40){\line(1,0){20}}
\qbezier(40,42.5)(40,55)(50,55)
\qbezier(60,42.5)(60,55)(50,55)
\put(62.5,40){\line(1,0){15}}

\put(40,2.5){\line(0,1){15}}
\put(0,2.5){\line(0,1){15}}
\put(60,2.5){\line(0,1){15}}
\put(100,2.5){\line(0,1){15}}
\put(20,22.5){\line(0,1){15}}
\put(40,22.5){\line(0,1){15}}
\put(60,22.5){\line(0,1){15}}
\put(80,22.5){\line(0,1){15}}

\qbezier(-20,22.5)(-20,40)(-2.5,40)
\qbezier(-20,17.5)(-20,0)(-2.5,0)
\put(2.5,0){\line(1,0){35}}
\qbezier(102.5,0)(120,0)(120,17.5)
\qbezier(102.5,40)(120,40)(120,22.5)
\put(82.5,40){\line(1,0){20}}

\put(0,0){\circle*{5}}
\put(40,0){\circle*{5}}
\put(60,0){\circle*{5}}
\put(100,0){\circle*{5}}

\put(-20,20){\circle{5}}
\put(0,20){\circle{5}}
\put(20,20){\circle*{5}}
\put(40,20){\circle{5}}
\put(60,20){\circle{5}}
\put(80,20){\circle*{5}}
\put(100,20){\circle{5}}
\put(120,20){\circle{5}}

\put(20,40){\circle{5}}
\put(40,40){\circle*{5}}
\put(60,40){\circle*{5}}
\put(80,40){\circle{5}}

\thinlines
\multiput(50,-7)(0,4){14}{\line(0,1){1.5}}
\multiput(127.5,-7)(0,4){14}{\line(0,1){1.5}}
\multiput(50,-7)(4,0){20}{\line(1,0){1.5}}
\multiput(50,46.5)(4,0){20}{\line(1,0){1.5}}

\put(80,9){\makebox(0,0){$P_\textup{in}$}}
\end{picture}
\qquad\qquad\qquad
\begin{picture}(120,60)(-5,-8)
\thicklines
\qbezier(40,-2.5)(40,-15)(50,-15)
\qbezier(60,-2.5)(60,-15)(50,-15)
\put(62.5,0){\line(1,0){15}}

\put(2.5,20){\line(1,0){15}}
\put(-17.5,20){\line(1,0){15}}
\put(22.5,20){\line(1,0){15}}
\put(62.5,20){\line(1,0){15}}
\put(82.5,20){\line(1,0){15}}
\put(102.5,20){\line(1,0){15}}

\put(22.5,40){\line(1,0){15}}
\put(-2.5,40){\line(1,0){20}}
\qbezier(40,42.5)(40,55)(50,55)
\qbezier(60,42.5)(60,55)(50,55)
\put(62.5,40){\line(1,0){35}}

\put(40,2.5){\line(0,1){15}}
\put(0,2.5){\line(0,1){15}}
\put(60,2.5){\line(0,1){15}}
\put(100,22.5){\line(0,1){15}}
\put(20,22.5){\line(0,1){15}}
\put(40,22.5){\line(0,1){15}}
\put(60,22.5){\line(0,1){15}}
\put(80,17.5){\line(0,-1){15}}

\qbezier(-20,22.5)(-20,40)(-2.5,40)
\qbezier(-20,17.5)(-20,0)(-2.5,0)
\put(2.5,0){\line(1,0){35}}
\qbezier(102.5,0)(120,0)(120,17.5)
\qbezier(102.5,40)(120,40)(120,22.5)
\put(82.5,0){\line(1,0){20}}

\put(0,0){\circle*{5}}
\put(40,0){\circle*{5}}
\put(60,0){\circle{5}}
\put(80,0){\circle*{5}}

\put(-20,20){\circle{5}}
\put(0,20){\circle{5}}
\put(20,20){\circle*{5}}
\put(40,20){\circle{5}}
\put(60,20){\circle*{5}}
\put(80,20){\circle{5}}
\put(100,20){\circle*{5}}
\put(120,20){\circle*{5}}

\put(20,40){\circle{5}}
\put(40,40){\circle*{5}}
\put(60,40){\circle{5}}
\put(100,40){\circle{5}}

\thinlines
\multiput(50,-7)(0,4){14}{\line(0,1){1.5}}
\multiput(127.5,-7)(0,4){14}{\line(0,1){1.5}}
\multiput(50,-7)(4,0){20}{\line(1,0){1.5}}
\multiput(50,46.5)(4,0){20}{\line(1,0){1.5}}
\end{picture}

\end{center}
\vspace{.1in}
\caption{Plabic graphs related by a switch.
The dotted line represents a simple closed curve~$\mathcal{C}$. 
The portions outside~$\mathcal{C}$ are the same.
The portions inside~$\mathcal{C}$ are related by
flipping upside down and changing the colors of all vertices.}
\label{fig:isthmus-move}
\end{figure}

Proposition~\ref{pr:plabic-vs-quivers} can be strengthened as follows. 

\begin{proposition}
\label{prop:Shapiro-easy}
If two plabic graphs are move-and-switch equivalent to each other, 
then their associated quivers are mutation equivalent. 
\end{proposition}

\begin{proof}
Let $P$ and $P'$ be two plabic graphs related by a switch, as in Definition~\ref{def:switch}. 
Let us verify that the corresponding quivers $Q(P)$ and $Q(P')$
are isomorphic to each other. Indeed, for every edge of~$P$ contained entirely inside~$P_{\text{in}}$, the flipping of~$P_{\text{in}}$ reverses the direction of the corresponding arrow in the quiver; the subsequent reversal of colors restores the original direction. It remains to examine the edges of the plabic graph which cross the boundary of~$P_{\text{in}}$ (denoted by~$\mathcal{C}$ in Definition~\ref{def:switch}). A~case-by-case inspection shows that the combined contribution of the corresponding arrows remains unchanged under a switch. 

The statement now follows by Proposition~\ref{pr:plabic-vs-quivers}. 
\end{proof}

It seems reasonable to expect that the converse to Proposition~\ref{prop:Shapiro-easy} holds as well. 
The following conjecture is inspired by our communications with Michael Shapiro. 

\begin{conjecture}[M.~Shapiro's conjecture]
\label{conj:plabic-vs-quivers}
Two plabic graphs are move-and-switch equivalent 
if and only if their associated quivers are mutation equivalent. 
\end{conjecture}

Recall that according to Conjecture~\ref{conj:shapiro-algebraic}, 
in the case of algebraic divides the switch transformations are not required. 

\begin{remark}
We cannot resist stating a closely related version of Shapiro's conjecture which can be formulated entirely in terms of quivers,
without any mention of plabic graphs. 
This version asserts that if two mutation equivalent quivers $Q$ and $Q'$ are both planar (i.e., each of them can be drawn on the plane without crossings), then $Q$ can be transformed into $Q'$ by a sequence of mutations in which each intermediate quiver is planar. 

It is important to note that in the course of these mutations, it may be necessary to alter the topology of a planar embedding of (a portion of) the quiver at hand. 
To illustrate, the quivers associated with the plabic graphs shown in Figure~\ref{fig:isthmus-move} are isomorphic to each other (so no mutations are necessary)---but their respective planar embeddings naturally associated with these drawings are different. 
\end{remark}


\section{Links from divides
}
\label{sec:links-of-divides}

As mentioned in Remark~\ref{rem:which-divides-come-from-morsifications}, 
it is very difficult to distinguish algebraic divides, 
i.e., those associated with real morsifications, from
the divides which do not arise in this way. 
Luckily, this problem can be circumvented using an elegant construction 
introduced by N.~A'Campo~\cite{AC1}, which we recall in Definition~\ref{def:acampo-link} below. 
For surveys of some of the related research, 
see~\cite[Sections 1 and~6]{ishikawa}
and \cite[Sections~4--5]{rudolph-handbook}. 

The main idea is to extend the equivalence of divides based on the topology of the associated
singularity (which can only be defined for algebraic divides)
to a more general equivalence relation---defined for all divides---based on 
the topology of a certain link constructed from a given divide. 

\begin{definition} 
\label{def:acampo-link}
Let $D$ be a divide in the unit disk~$\Disk=\{x^2+y^2\le 1\}\subset\RR^2$. 
The \emph{(A'Campo) link} $L(D)$ of~$D$ is constructed inside the unit $3$-sphere 
\[
\mathbf{S}^3=\{(x,y,u,v)\in\RR^4 \mid x^2+y^2+u^2+v^2=1\}, 
\] 
as follows. 
Assume that $D$ is given by a smooth immersion of a collection of intervals and circles.  
For each regular (resp., nodal) point $(x,y)\in D$ in the interior of~$\Disk$, find 
the two (resp., four) different points $(x,y,u,v)\in\mathbf{S}^3$ 
such that $(u,v)$ is a tangent vector to~$D$ at~$(x,y)$.
The link $L(D)$ is defined as the set of all such points $(x,y,u,v)$,
together with the points $(x,y,0,0)$ for $(x,y)\in D\cap\partial\Disk$. 
We can view $L(D)$ as a subset of~$\CC^2$ via the identification 
$(x,y,u,v)\simeq(x+\sqrt{-1}\,u,y+\sqrt{-1}\,v)$. 

Two divides are called \emph{link equivalent} if their associated links are isotopic. 
\end{definition}

\begin{remark}
While the original construction in~\cite{AC1} was for divides without closed branches, 
it can be extended to full generality, cf.~\cite{acampo-ihes, couture-perron, kawamura-quasi}. 
\end{remark}

\begin{remark}
All links appearing in this paper are naturally \emph{oriented}. 
Accordingly, the term ``link'' will generally mean ``oriented link'' (with the natural orientation).     
\end{remark}

We next review the relationship between A'Campo's construction presented in Definition~\ref{def:acampo-link} and the classical notion of the link of an isolated singularity. 

\begin{definition}
\label{def:link-of-singularity}
The link $L(C,z)$ associated with an isolated plane curve singularity $(C,z)$
(as in Section~\ref{sec:singularities-and-morsifications}) 
is the intersection of the curve~$C$ with a small sphere centered~at~$z$. 
\end{definition}

The importance of this construction comes from the following fundamental fact 
(see \cite{Epple, Neumann} for historical background): 

\begin{proposition}
\label{pr:link-determines-singularity}
The link $L(C,z)$ completely determines---and is determined~by---the 
local topology of a given singular complex plane curve $(C,z)$. 
\end{proposition}

The crucial property established by N.~A'Campo is that 
the constructions of Definitions~\ref{def:acampo-link} and~\ref{def:link-of-singularity} produce the same link. 
More precisely: 

\begin{theorem}[N.~A'Campo]
\label{th:L(D)=L(C,z)}
For an algebraic divide~$D$ arising from a real morsification of 
an isolated plane curve singularity~$(C,z)$, 
the links~$L(D)$ and $L(C,z)$ are isotopic to each other
inside~$\mathbf{S}^3$. 
\end{theorem}

Combining Proposition~\ref{pr:link-determines-singularity} with Theorem~\ref{th:L(D)=L(C,z)}, we obtain the following statement. 

\begin{corollary}
\label{cor:link-equivalence=top-equivalence}
Algebraic divides are link equivalent if and only if corresponding singularities are  topologically equivalent. 
\end{corollary}

\begin{remark}
\label{rem:link-determines-singularity}
Proposition~\ref{pr:link-determines-singularity} and Theorem~\ref{th:L(D)=L(C,z)} imply that for a divide~$D$ coming \linebreak[3]
from a real morsification, the link~$L(D)$ (hence the divide~$D$) determines the topological type of the underlying singularity. 
This does not however imply Theorems~\ref{th:ag-diagram-determines-topology} and/or~\ref{th:quiver-determines-topology} because the same \AG-diagram/quiver may potentially come from several distinct divides (either coming from morsifications or not), cf.\ Figure~\ref{fig:A3-ACGZ}. 
\end{remark}

\pagebreak[3]

\begin{remark}
\label{rem:libgober-williams}
By Corollary~\ref{cor:link-equivalence=top-equivalence},
any topological invariant of a plane curve singularity 
can be in principle recovered from the A'Campo link $L(D)$ 
of a divide~$D$ coming \linebreak[3]
from 
a real morsification. 
In practice, extracting such invariants~from~$L(D)$ 
can be challenging.
For example, the \emph{multiplicity} of a singularity 
is equal to the \emph{braid index} of its link~\cite{libgober,r.williams},
i.e., the smallest number of strands in a braid defining~it. 
However, computing the braid index of a link is, in general, a very difficult problem.
\end{remark}

We propose the following conjecture. 

\begin{conjecture}
\label{conj:link-equivalence=move-equivalence}
Algebraic divides are link equivalent if and only if the plabic graphs attached to them are move equivalent. 
\end{conjecture}

Since all plabic graphs attached to a given divide are move equivalent to each~other, the particular choices of attached plabic graphs in Conjecture~\ref{conj:link-equivalence=move-equivalence} are immaterial. 

\medskip

Conjectures~\ref{conj:morsif=mut}, \ref{conj:shapiro-algebraic} and~\ref{conj:link-equivalence=move-equivalence} and Corollary~\ref{cor:link-equivalence=top-equivalence} 
are subsumed within the following statement, 
which is diagrammatically represented in 
Figure~\ref{fig:relationships-between-conjectures}. 

\begin{conjecture}[\underbar{Main conjecture, expanded}]
\label{conj:morsif=mut-via-divides}
Let $D_1$ and $D_2$ be algebraic divides. 
Let $Q_1=Q(D_1)$ and $Q_2=Q(D_2)$ be their quivers. 
Let $P_1\in\mathbf{P}(D_1)$ and \linebreak[3]
$P_2\in\mathbf{P}(D_2)$ 
be plabic graphs attached to $D_1$ and~$D_2$.
Then the following are equivalent:
\begin{itemize}[leftmargin=.3in]
\item[{\rm\bf (s)}] the singularities giving rise
to the divides $D_1$ and $D_2$ are topologically equivalent;  
\item[{\rm\bf (d)}] the divides $D_1$ and $D_2$ are link equivalent;
\item[{\rm\bf (q)}] the quivers $Q_1$ and $Q_2$ are mutation equivalent;
\item[{\rm\bf (p)}] the plabic graphs $P_1$ and $P_2$ are move equivalent. 
\end{itemize}
\end{conjecture}

\begin{figure}[ht]
\[
\begin{array}{|c|}
\hline
\\[-.15in]
\hspace{-.2in}
\begin{tikzcd}[row sep=large] 
\begin{array}{c}\text{\textbf{(s)} topological equivalence}\\ 
\text{\quad of singularities}
\end{array}
\arrow[d, Leftrightarrow, "\text{Corollary~\ref{cor:link-equivalence=top-equivalence}\ }"']   
\arrow[rr, Leftrightarrow, "\text{\textbf{?}}", "\text{Conjecture~\ref{conj:morsif=mut}}"'] && 
\begin{array}{l}
\text{\textbf{(q)} mutation equivalence of}\\
\, \qquad \text{quivers of singularities} 
\arrow[d, Leftarrow, xshift=-.1in, "\text{Proposition~\ref{pr:plabic-vs-quivers}}\ "']
\arrow[d, Rightarrow, xshift=.25in, 
"\text{\ Conjecture~\ref{conj:shapiro-algebraic}}", 
"\text{\textbf{?}}"']
\end{array} \\
\begin{array}{c}
\quad\ \text{\textbf{(d)} link equivalence of}\quad \ \\
\text{\qquad algebraic divides}  
\end{array}
\arrow[rr, Leftarrow, yshift=.12in, "\text{Corollary~\ref{cor:p=>d}}"]
\arrow[rr, Rightarrow, yshift=-.12in, "\text{\textbf{?}}", 
"\text{Conjecture~\ref{conj:link-equivalence=move-equivalence}}"
'] 
&& 
\begin{array}{c}
\text{\textbf{(p)} move equivalence of plabic graphs}\\
\quad \text{attached to algebraic divides}  
\end{array}
\end{tikzcd}\hspace{-.1in}
\\[-.1in]
\\
\hline
\end{array}
\]
\caption{Unpacking Conjecture~\ref{conj:morsif=mut-via-divides}.
} 
\label{fig:relationships-between-conjectures} 
\end{figure} 

The key part of Conjecture~\ref{conj:morsif=mut-via-divides} is Conjecture~\ref{conj:link-equivalence=move-equivalence} (the equivalence \textbf{(d)}$\Leftrightarrow$\textbf{(p)}); it is arguably more important than the original Conjecture~\ref{conj:morsif=mut} (the equivalence~\textbf{(s)}$\Leftrightarrow$\textbf{(q)}). 
On the singularity theory side, replacing topological equivalence of singularities
by the link equivalence of divides 
makes the issue at hand more tractable computationally, 
and might allow extensions to non-algebraic divides and their 
links, cf.\ Problem~\ref{problem:divide-links}
and Remark~\ref{rem:transverse-non-simplicity} below. 
On~the cluster side, replacing mutation equivalence of quivers by 
the move equivalence of plabic graphs makes even more sense:
in light of Remark~\ref{rem:mut-equivalence-difficult}, 
it seems reasonable to restrict the mutation dynamics 
to a manageable subset of allowed directions.
%

In the rest of the paper, we focus on Conjecture~\ref{conj:link-equivalence=move-equivalence} (the equivalence \textbf{(d)}$\Leftrightarrow$\textbf{(p)}). 
In Section~\ref{sec:links-from-plabic}, we prove the implication \textbf{(p)}$\Rightarrow$\textbf{(d)}, see Corollary~\ref{cor:p=>d}.
In subsequent sections, we make partial progress towards the converse implication \textbf{(d)}$\Rightarrow$\textbf{(p)}. 

\medskip

It is tempting to extend Conjecture~\ref{conj:morsif=mut-via-divides} to a larger generality: 

\begin{problem}
\label{problem:divide-links}
Identify a class of divides---as broad as possible---within which the various equivalences in Conjecture~\ref{conj:morsif=mut-via-divides} hold. 
\end{problem}

\begin{remark}
It may well be that \textbf{(d)}$\Rightarrow$\textbf{(q)} for arbitrary divides.
It~is even possible that \textbf{(d)}$\Rightarrow$\textbf{(p)} for arbitrary plabic graphs, provided one uses transverse equivalence, cf.\ Problem~\ref{prob:plabic-transverse} below. 
\end{remark}

\begin{remark}
For general plabic graphs, \textbf{(q)}~does not imply \textbf{(p)}, see Remark~\ref{rem:q-not=>p}. 
Likewise, for general (non-algebraic) divides,  \textbf{(q)}~does not imply~\textbf{(d)};
a counterexample, borrowed from \cite[Figure~4]{balke-kaenders},
is shown in Figure~\ref{fig:counterexample-balke-kaenders}. 
\end{remark}

\begin{figure}[ht]
\begin{center}
\setlength{\unitlength}{0.75pt}
\begin{picture}(80,200)(0,-100)
\thicklines
\qbezier(40,0)(20,30)(20,50)
\qbezier(50,80)(20,80)(20,50)
\qbezier(40,0)(60,30)(60,100)
\qbezier(50,80)(80,80)(80,60)
\qbezier(0,30)(80,30)(80,60)

\qbezier(40,0)(20,-30)(20,-50)
\qbezier(50,-80)(20,-80)(20,-50)
\qbezier(40,0)(60,-30)(60,-100)
\qbezier(50,-80)(80,-80)(80,-60)
\qbezier(0,-30)(80,-30)(80,-60)
\end{picture}
\hspace{1in}
\begin{picture}(80,200)(0,-100)
\thicklines
\qbezier(40,0)(20,30)(20,50)
\qbezier(50,80)(20,80)(20,50)
\qbezier(40,0)(60,30)(60,100)
\qbezier(50,80)(80,80)(80,60)
\qbezier(0,30)(80,30)(80,60)

\qbezier(40,0)(60,-30)(60,-50)
\qbezier(30,-80)(60,-80)(60,-50)
\qbezier(40,0)(20,-30)(20,-100)
\qbezier(30,-80)(0,-80)(0,-60)
\qbezier(80,-30)(0,-30)(0,-60)
\end{picture}
\vspace{-.05in}
\end{center}
\caption{Two non-algebraic divides $D_1$ and~$D_2$. 
The quivers $Q(D_1)$ and $Q(D_2)$ are isomorphic. 
On the other hand, the links $L(D_1)$ and $L(D_2)$ are not isotopic: 
the link $L(D_2)$ has an unknotted component but the link $L(D_1)$ has not. 
}
\vspace{-.15in}
\label{fig:counterexample-balke-kaenders}
\end{figure}

\section{Oriented divides and their links}
\label{sec:oriented-divides-and-their-links}

Our proof of the implication \textbf{(p)}$\Rightarrow$\textbf{(d)} in Conjecture~\ref{conj:morsif=mut-via-divides} will rely on a construction that associates a link to an arbitrary plabic graph. 
But first, we need to discuss a more flexible notion of oriented divides
(and their associated links). 
This notion, due to
W.~Gibson and M.~Ishikawa~\cite{gibson, gibson-ishikawa}, is a variation of Arnold's
description of links associated to plane curves~\cite{arnold-curves}.

\begin{definition}
  \label{def:oriented-divide}
  An \emph{oriented divide} $\vec{D}$ in a disk~$\Disk$ is an immersion into~$\Disk$ of a finite set of
  \emph{oriented} circles satisfying conditions
  (D1), (D3), and (D4) of
  Definition~\ref{def:divide}.
  (Condition~(D2) is unnecessary, since there are no intervals. The
  connectivity restrictions (D5) and (D6) are not required for
  our purposes here.) 
  
  The \emph{link} $L(\vec{D})$ of an
  oriented divide is defined as in Definition~\ref{def:acampo-link},
  except that we only take the vectors $(x,y,u,v)$ where $(u,v)$
  points in the direction of the orientation of~$\vec D$.
  As before, we consider oriented divides up to isotopy inside~$\Disk$. 
\end{definition}

\begin{definition}
  \label{def:oriented-divide-moves}
  Two oriented divides are called \emph{move equivalent} if they can be 
  related to each other by a sequence of local moves of the following three kinds: 
  \begin{itemize}[leftmargin=.3in]
  \item \emph{triangle moves}, as in Figure~\ref{fig:oriented-yb-move} (with
    any orientations); 
  \item \emph{safe tangency moves} with oppositely oriented strands, as in
    Figure~\ref{fig:oriented-reidemeister2-move}; and/or 
  \item \emph{U-turn moves}, only allowed near the boundary of the disk~$\Disk$, as in
    Figure~\ref{fig:oriented-uturn-move}.
  \end{itemize}
\end{definition}

\begin{figure}[ht]
\begin{center}
\setlength{\unitlength}{1.4pt}
\begin{picture}(33,24)(0,0)
\thicklines
\qbezier(2,12)(16.5,5)(31,12)
\qbezier(10,0)(12,14)(23,24)
\qbezier(10,24)(21,14)(23,0)
\put(31,12){\vector(14.5,7){1}}
\put(10,24){\vector(-11,10){1}}
\put(10,0){\vector(-2,-14){0.15}}
\end{picture}
\begin{picture}(25,24)(0,0)
\put(12.5,12){\makebox(0,0){$\longleftrightarrow$}}
\end{picture}
\begin{picture}(33,24)(0,0)
\thicklines
\qbezier(2,12)(16.5,19)(31,12)
\qbezier(10,0)(21,10)(23,24)
\qbezier(10,24)(12,10)(23,0)
\put(31,12){\vector(14.5,-7){1}}
\put(10,0){\vector(-11,-10){1}}
\put(10,24){\vector(-2,14){0.15}}
\end{picture}
\qquad\qquad\qquad
\begin{picture}(33,24)(0,0)
\thicklines
\qbezier(2,12)(16.5,5)(31,12)
\qbezier(10,0)(12,14)(23,24)
\qbezier(10,24)(21,14)(23,0)
\put(31,12){\vector(14.5,7){1}}
\put(23,0){\vector(2,-14){0.15}}
\put(10,0){\vector(-2,-14){0.15}}
\end{picture}
\begin{picture}(25,24)(0,0)
\put(12.5,12){\makebox(0,0){$\longleftrightarrow$}}
\end{picture}
\begin{picture}(33,24)(0,0)
\thicklines
\qbezier(2,12)(16.5,19)(31,12)
\qbezier(10,0)(21,10)(23,24)
\qbezier(10,24)(12,10)(23,0)
\put(31,12){\vector(14.5,-7){1}}
\put(10,0){\vector(-11,-10){1}}
\put(23,0){\vector(11,-10){1}}
\end{picture}
\end{center}
\caption{Triangle moves on oriented divides.
}
\label{fig:oriented-yb-move}
\end{figure}

\begin{figure}[ht]
\begin{center}
\setlength{\unitlength}{1.4pt}
\begin{picture}(33,10)(0,8)
\thicklines
\qbezier(2,16)(16.5,4)(31,16)
\qbezier(2,8)(16.5,20)(31,8)
\put(2,16){\vector(-14.5,12){1}}
\put(31,8){\vector(14.5,-12){1}}
\end{picture}
\begin{picture}(25,10)(0,8)
\put(12.5,12){\makebox(0,0){$\longleftrightarrow$}}
\end{picture}
\begin{picture}(33,10)(0,8)
\thicklines
\put(31,16){\vector(-1,0){29}}
\put(2,8){\vector(1,0){29}}
\end{picture}
\qquad\qquad\qquad
\begin{picture}(33,10)(0,8)
\thicklines
\qbezier(2,16)(16.5,4)(31,16)
\qbezier(2,8)(16.5,20)(31,8)
\put(2,8){\vector(-14.5,-12){1}}
\put(31,16){\vector(14.5,12){1}}
\end{picture}
\begin{picture}(25,10)(0,8)
\put(12.5,12){\makebox(0,0){$\longleftrightarrow$}}
\end{picture}
\begin{picture}(33,10)(0,8)
\thicklines
\put(31,8){\vector(-1,0){29}}
\put(2,16){\vector(1,0){29}}
\end{picture}

\end{center}
\caption{Safe tangency moves on oriented divides.
}
\label{fig:oriented-reidemeister2-move}
\end{figure}

\begin{figure}[ht]
\begin{center}
\setlength{\unitlength}{1.4pt}
\begin{picture}(33,29)(0,-6)
\thicklines
\qbezier(13,4)(12,14)(23,24)
\qbezier(10,24)(21,14)(20,4)
\qbezier(13,4)(13,-2)(16.5,-2)
\qbezier(20,4)(20,-2)(16.5,-2)
\put(10,24){\vector(-11,10){1}}
\thinlines
\put(5,-6){\line(1,0){23}}
\put(5,-7){\line(1,0){23}}
\end{picture}
\begin{picture}(25,26)(0,-6)
\put(12.5,12){\makebox(0,0){$\longleftrightarrow$}}
\end{picture}
\begin{picture}(33,26)(0,-6)
\thicklines
\qbezier(20,4)(20,18)(23,24)
\qbezier(10,24)(13,18)(13,4)
\qbezier(13,4)(13,-2)(16.5,-2)
\qbezier(20,4)(20,-2)(16.5,-2)
\put(10,24){\vector(-3,6){1}}
\thinlines
\put(5,-6){\line(1,0){23}}
\put(5,-7){\line(1,0){23}}
\end{picture}
\qquad\qquad\qquad
\begin{picture}(33,26)(0,-6)
\thicklines
\qbezier(13,4)(12,14)(23,24)
\qbezier(10,24)(21,14)(20,4)
\qbezier(13,4)(13,-2)(16.5,-2)
\qbezier(20,4)(20,-2)(16.5,-2)
\put(23,24){\vector(11,10){1}}
\thinlines
\put(5,-6){\line(1,0){23}}
\put(5,-7){\line(1,0){23}}
\end{picture}
\begin{picture}(25,26)(0,-6)
\put(12.5,12){\makebox(0,0){$\longleftrightarrow$}}
\end{picture}
\begin{picture}(33,26)(0,-6)
\thicklines
\qbezier(20,4)(20,18)(23,24)
\qbezier(10,24)(13,18)(13,4)
\qbezier(13,4)(13,-2)(16.5,-2)
\qbezier(20,4)(20,-2)(16.5,-2)
\put(23,24){\vector(3,6){1}}
\thinlines
\put(5,-6){\line(1,0){23}}
\put(5,-7){\line(1,0){23}}
\end{picture}
\end{center}
\caption{(Boundary) U-turn moves on oriented divides.
The double horizontal lines at the bottom represent
the boundary of the ambient disk~$\Disk$.
}
\label{fig:oriented-uturn-move}
\vspace{-.15in}
\end{figure}

\begin{remark}
  The U-turn move is not an explicit move in the works of Gibson and
  Ishikawa, but appears implicitly in \cite[Lemma
  2.5]{gibson-ishikawa} and \cite[Proposition 4.2]{gibson}, which
  assert that adding a loop on an outside arc does not change the
  isotopy class of the link. In addition to our U-turn move, they
  allow another move adding an inward-pointing rather than
  outward-pointing loop. We do not include inward-pointing loops
  because, on the one hand, it is not needed for the cluster algebra
  applications, and, on the other hand, it changes the transverse
  isotopy class of the link (Definition~\ref{def:contact}; see
  Proposition~\ref{prop:transverse-divide}).
\end{remark}

\begin{proposition}
\label{prop:ybe-oriented}
  If oriented divides $\vec{D}_1$ and $\vec{D}_2$ are
  move equivalent, then the links $L(\vec{D}_1)$ and $L(\vec{D}_2)$ are 
  smoothly isotopic to each other.
\end{proposition}

To clarify: we care about smooth isotopy since we will later work in the category of transverse links.  

\begin{proof}
It suffices to show the existence of a $C^1$ isotopy. 
The existence of a $C^\infty$~isotopy would follow, since everything is compact and there is a polynomial approximation. 

  For the triangle move, there is a path of immersions of the branches in
  the plane connecting the two oriented divides, passing through a
  diagram that has a triple intersection point. We can lift this path
  of immersions to~$\mathbf{S}^3$ as in Definition~\ref{def:oriented-divide}.
  Since the (oriented) tangents never agree, we get an isotopy of
  links.

  The case of safe tangencies is similar.

It remains to treat the case of a U-turn move, which is a bit trickier. 
We note that the link of an oriented divide always avoids the
  equatorial circle in $\mathbf{S}^3$ given by
  $\{\,(x,y,0,0)\mid x^2+y^2 = 1\,\}$. 
  Loosely, a~U-turn move corresponds to
  letting $L(\vec{D})$ pass through that circle once. 
  To verify this, we use the following explicit construction. 
  Consider the 1-parameter family of oriented curves
  $\vec C_\varepsilon = \{(x_\varepsilon(t),y_\varepsilon(t))\} \subset \Disk$ 
  given by
\begin{align}
\label{eq:x_eps(t)}  x&=x_\varepsilon(t)=\varepsilon t +  t^3, \\
\label{eq:y_eps(t)}   y&=y_\varepsilon(t)=-(1-\tfrac12 \varepsilon^2)(1-\tfrac12 t^2), 
\end{align}
where $\varepsilon$ and~$t$ are small real parameters. 
(To be more precise, we consider $|\varepsilon|\le \frac12\delta^2$ and $|t|\le \delta$, for a small
positive~$\delta$.
Additional tweaking is required to have the two 
curve segments for $\varepsilon=\pm \frac12\delta^2$
match at the endpoints.) 
For $\varepsilon \ne 0$, 
the curves $\vec C_\varepsilon$ are segments of an oriented divide
in the interior of~$\Disk$, 
differing by a U-turn move near the boundary.
For $\varepsilon\!=\!0$, we get a curve~$\vec C_0$ with a cusp
at the boundary point~$(0,-1)$. 

For $\varepsilon \ne 0$, the corresponding (segment of the) 
link $L_\varepsilon=L(\vec C_\varepsilon)$ is given by 
\begin{align}
\label{eq:L-epsilon}
L_\varepsilon(t) 
         &= (x,y,\beta \dot x, \beta \dot y), \\
\intertext{where} 
\label{eq:dot-x}  \dot x&
=\varepsilon+3t^2, \\
\label{eq:dot-y}  \dot y&
=(1-\tfrac12 \varepsilon^2) \,t ,\\
\beta&=\beta_\varepsilon(t)=\Bigl(\dfrac{1-x^2-y^2}{\dot x^2+\dot y^2}\Bigr)^{1/2}. 
\label{eq:beta-varepsilon}
\end{align}
These formulas can be extended to the case $\varepsilon=0$,
with the convention $\beta_0(0)=1$. 

Each curve $L_\varepsilon$ does not intersect itself. 
We are going to show that the family~$L_\varepsilon$ gives a $C^1$-isotopy 
between the two sides of a U-turn move. 
More precisely, we will prove that each of the 4 coordinates of 
$L_\varepsilon(t)$ (cf.~\eqref{eq:L-epsilon}) is a differentiable function 
of two variables $\varepsilon$ and~$t$, 
and each of its partial derivatives is continuous in the vicinity of 
the point $\varepsilon=t=0$.
For the first two coordinates $x$ and~$y$, this statement is obvious,
cf.\ \eqref{eq:x_eps(t)}--\eqref{eq:y_eps(t)}. 
Let us treat the remaining coordinates $\beta\dot x$ and $\beta\dot y$. 

Straightforward calculations show that
\begin{align}
\label{eq:x^2+y^2}
x^2+y^2 &= 1-\varepsilon^2-t^2+\operatorname{poly}_{\ge 3}(\varepsilon,t),\\
\label{eq:dot-x^2+dot-y^2}
\dot x^2+\dot y^2 &= \varepsilon^2+t^2+\operatorname{poly}_{\ge 3}(\varepsilon,t),
\end{align}
where the notation $\operatorname{poly}_{\ge 3}(\varepsilon,t)$ stands for a polynomial
in $\mathbb{Q}[\varepsilon,t]$ in which each monomial has degree~$\ge 3$. 
Substituting \eqref{eq:x^2+y^2}--\eqref{eq:dot-x^2+dot-y^2} into~\eqref{eq:beta-varepsilon},
we see that
\begin{equation*}
\beta_\varepsilon(t)=\Bigl(\dfrac{\varepsilon^2+t^2+\operatorname{poly}_{\ge 3}(\varepsilon,t)}{\varepsilon^2+t^2+\operatorname{poly}_{\ge 3}(\varepsilon,t)}\Bigr)^{1/2}
=1+O(\sqrt{\varepsilon^2+t^2}). 
\end{equation*}
Since both~$|\varepsilon|$ and~$|t|$ 
do not exceed $\sqrt{\varepsilon^2+t^2}$, we conclude that 
\begin{equation}
\label{eq:beta=1+O}
\beta_\varepsilon(t)=1+O(\sqrt{\varepsilon^2+t^2})
\underset{(\varepsilon,t)\to(0,0)}{-\!\!\!-\!\!\!-\!\!\!-\!\!\!\longrightarrow}1, 
\end{equation}
so $\beta_\varepsilon(t)$ is a continuous function of $\varepsilon$ and~$t$. 

We next show that both $\varepsilon\beta_\varepsilon(t)$
and $t\beta_\varepsilon(t)$ are in~$C^1$. 
In view of \eqref{eq:dot-x}--\eqref{eq:dot-y},
this will imply that $\beta\dot x$ and $\beta\dot y$ 
 are in~$C^1$, as desired. 
The two cases are completely analogous, so let us consider 
$\varepsilon\beta_\varepsilon(t)$. 
This function is clearly smooth at every point other than $\varepsilon=t=0$,
so we only need to examine the latter point. 
Equation~\eqref{eq:beta=1+O} implies that
\[
\varepsilon\beta_\varepsilon(t)=\varepsilon+O(\varepsilon^2+t^2), 
\]
so $\varepsilon\beta_\varepsilon(t)$ is differentiable at
$\varepsilon\!=\!t\!=\!0$, with partial derivatives 
$\frac{\partial (\varepsilon\beta)}{\partial \varepsilon}(0,0)=1$
and $\frac{\partial (\varepsilon\beta)}{\partial t}(0,0)=0$. 
Away from this point, these derivatives can be 
computed using~\eqref{eq:beta-varepsilon}: 
\begin{align*}
\frac{\partial}{\partial \varepsilon}(\varepsilon\beta)
&=\beta+
\frac{\varepsilon}{2\beta}\frac{\operatorname{poly}_{\ge 4}(\varepsilon,t)}{(\varepsilon^2+t^2+\operatorname{poly}_{\ge 3}(\varepsilon,t))^2}, 
\\
\frac{\partial}{\partial t}(\varepsilon\beta)
&=\frac{\varepsilon}{2\beta}\frac{\operatorname{poly}_{\ge 4}(\varepsilon,t)}{(\varepsilon^2+t^2+\operatorname{poly}_{\ge 3}(\varepsilon,t))^2}.
\end{align*}
As $(\varepsilon,t)$ goes to~$(0,0)$, these expressions converge to 1 and~0,
respectively, establishing the continuity of the derivative. 
This completes the proof of Proposition~\ref{prop:ybe-oriented}. 
\end{proof}

\begin{remark}
\label{rem:same-direction tangency}
Passing through a \emph{same-direction tangency}, with the
two strands oriented in the same direction, 
does not produce an isotopy of the associated links, 
since in this case, $L(\vec D)$ crosses through itself.
\end{remark}


For an oriented divide $\vec{D}$, we denote by $-\vec{D}$ the same
  divide with all orientations reversed. 

\begin{lemma}
\label{lem:reverse-divide}
For any oriented divide $\vec{D}$, the links $L(\vec{D})$ and $L(-\vec{D})$ are isotopic.
\end{lemma}

\begin{proof}
  The two links are isotopic through the isotopy of $\mathbf{S}^3$ given by
  \[
    \phi_t(x,y,u,v) = (x,y,\cos(\pi t) u + \sin(\pi t) v, -\sin(\pi t)
    u + \cos(\pi t) v). \qedhere
  \]
\end{proof}

\begin{remark}
\label{rem:o-divide}
  The main difference between Arnold's theory and the Gibson-Ishi\-kawa
  theory described above is that we are working in the disk and
  lifting to $\mathbf{S}^3$, rather than working in the plane and
  lifting to the unit tangent bundle of the plane, which is
  topologically a solid torus. Every link is the link of an oriented
  divide, either in the solid torus (Arnold's theory; see S.~Chmutov, V.~Goryunov, and H.~Murakami~\cite{CGM}) or in
  $\mathbf{S}^3$ (W.~Gibson and M.~Ishikawa~\cite{gibson-ishikawa}). There is also a concrete set
  of moves that relate any two oriented divides whose links are
  isotopic (W.~Gibson~\cite{gibson}), analogous to the Reidemeister moves. This
  set of moves is slightly larger than the moves above; see
  Proposition~\ref{prop:transverse-divide} below for the explanation. 
\end{remark}

\begin{remark}
\label{rem:o-divides-vs-link-diagrams}
Since any link is a link of an oriented divide, one can think of oriented divides as a
  combinatorial representation of links, on the same level as
  the traditional link diagrams. It is a less intuitive representation, and
  some features of links are harder to discern from an oriented divide, 
  compared to link diagrams; 
  for instance, computing the linking number is more involved.
\end{remark}

A'Campo's construction of links of ordinary (i.e., unoriented) divides, 
reproduced in Definition~\ref{def:acampo-link}, 
is a special case of the Gibson-Ishikawa construction 
of links of oriented divides 
described in Definition~\ref{def:oriented-divide},
in the following precise sense. 

\begin{definition}
\label{def:doubling}
To any divide~$D$, we can associate an oriented divide~$o(D)$
obtained by the following ``doubling'' procedure:
\begin{itemize}[leftmargin=.3in]
\item  
replace each 
branch of~$D$ with two parallel oriented
branches of $o(D)$, with opposite orientations, following the
``rules of the road'' (driving on the right) illustrated in
Figure~\ref{fig:divide-to-o-divide};
\item
near each point where $D$ hits the boundary $\partial\Disk$,
connect the two oriented branches, as shown in Figure~\ref{fig:divide-to-o-divide}
on the right (at the top). 
\end{itemize}
\end{definition}

\begin{figure}[ht]
\begin{center}
\setlength{\unitlength}{2.5pt}
\begin{picture}(50,50)(-10,-5)
\put(15,21){\circle{55}}
\put(15,21){\circle{57}}
\thicklines
\qbezier(-3,43)(40,0)(15,0)
\qbezier(33,43)(-10,0)(15,0)
\put(15,23){\circle{20}}
\put(-5,21){\makebox(0,0){$P$}}
\end{picture}
\begin{picture}(20,50)(0,0)
\put(10,25){\makebox(0,0){$\longmapsto$}}
\end{picture}
\begin{picture}(50,50)(-10,-5)
\put(35,21){\makebox(0,0){$o(P)$}}
\put(15,21){\circle{55}}
\put(15,21){\circle{57}}
\thicklines
\qbezier(2,40)(41,-2)(15,-2)
\qbezier(0,39)(38,0)(15,0)
\qbezier(0,39)(-1,40)(-0.3,40.7)
\qbezier(-0.3,40.7)(0.5,41.5)(2,40)

\qbezier(28,40)(-11,-2)(15,-2)
\qbezier(30,39)(31,40)(30.3,40.7)
\qbezier(30.3,40.7)(29.5,41.5)(28,40)

\qbezier(30,39)(-8,0)(15,0)
\put(15,23){\circle{18}}
\put(15,23){\circle{22}}
\put(26,23){\vector(0,1){2}}
\put(24,23){\vector(0,-1){2}}
\put(6,23){\vector(0,1){2}}
\put(4,23){\vector(0,-1){2}}
\put(15,-2){\vector(1,0){2}}
\put(15,0){\vector(-1,0){2}}
\put(5.7,36){\vector(-39,43){2.7}}
\put(0,39){\vector(39,-40){4}}
\put(24.3,36){\vector(-39,-43){1}}
\put(28,37){\vector(39,40){1}}
\end{picture}
\end{center}
\caption{A divide $P$ gives rise to an oriented divide~$o(P)$ via
``doubling.''}
\label{fig:divide-to-o-divide}
\vspace{-.1in}
\end{figure}

\vspace{-.15in}

\begin{proposition}
\label{prop:L(D)=L(o(D))}
The links $L(D)$ and $L(o(D))$ are smoothly isotopic to each other.
\end{proposition}

\begin{proof}
Compare Definitions~\ref{def:acampo-link} 
and~\ref{def:oriented-divide}, and use the isotopy from the proof of Proposition~\ref{prop:ybe-oriented}. 
\end{proof}

\begin{remark}
While the construction of the link of an oriented divide is elementary, one may still want to construct a conventional link diagram for~$L(D)$ directly from the combinatorial topology of a divide~$D$. 
Several solutions of this problem were suggested by various authors. 
In particular, O.~Couture and B.~Perron~\cite{couture-perron} 
gave an algorithm producing a braid
representation for the link~$L(D)$ associated with \emph{any} divide~$D$. 
It~involves an extension of the basic construction to \emph{signed} divides, 
wherein each node is labeled by a sign, either $+$~or~$-$. 
(The case when all signs are positive corresponds to the usual notion.) 
In the special case of ``scannable'' divides, 
the Couture-Perron construction simplifies considerably, 
see Section~\ref{sec:scannable-divides} below. 

Other (related) constructions of braid representations of links of (oriented) divides
were given by S.~Chmutov~\cite{chmutov}, M.~Hirasawa~\cite{hirasawa}, and
W.~Gibson--M.~Ishikawa~\cite{gibson-ishikawa}. 
While those constructions are more direct than the one in~\cite{couture-perron},
and do not involve signs, 
they are not ``local'' as they require dragging the strands of the link to the 
boundary of the disk, and then back. All of these methods involve a
non-canonical choice of a preferred ``Morse direction'' within the
ambient disk of the divide.
\end{remark}

\newpage

\section{Links from plabic graphs}
\label{sec:links-from-plabic}

We next explain how a plabic graph gives rise to an oriented divide,
and therefore to a link. 

\begin{definition}
\label{def:L(P)=L(o(P))}
The oriented divide $o(P)$ associated to a plabic graph~$P$ is 
constructed as follows. 
Turn each edge in~$P$ into a pair of
oppositely-oriented strands as in Definition~\ref{def:doubling}. 
(Remember that we are ``driving on the right.'') 
At each white trivalent vertex of~$P$,
connect the strands by turning right; at each black vertex connect
the strands by turning left, see Figure~\ref{fig:plabic-to-o-divide}. 
Note that when we turn left, we introduce transversal crossings in the divide. 
At the univalent ends of~$P$ lying on~$\partial\Disk$, make a U-turn near the boundary by turning either left or
right depending on whether the end is white or black, respectively, see
Figure~\ref{fig:plabic-boundary-to-o-divide}. (This introduces a
crossing if the end is black.) 

We then construct a link $L(P)=L(o(P))$ from the resulting oriented divide~$o(P)$.
\end{definition}


\begin{figure}[ht]
\begin{center}
\setlength{\unitlength}{0.75pt}
\begin{picture}(75,80)(0,0)
\thicklines
\put(10,0){\line(1,2){20}}
\put(10,80){\line(1,-2){20}}
\put(30,40){\line(1,0){40}}
\put(30,40){\circle*{6}}
\end{picture}
\begin{picture}(30,80)(0,0)
\put(15,40){\makebox(0,0){$\longmapsto$}}
\end{picture}
\begin{picture}(70,80)(0,0)
  \thicklines
  \put(55,30){\line(1,0){15}}
  \put(55,50){\line(1,0){15}}
  \put(56,30){\vector(1,0){10}}
  \put(70,50){\vector(-1,0){15}}
  \qbezier(25,30)(35,30)(40,40)
  \qbezier(25,50)(35,50)(40,40)
  \qbezier(55,50)(45,50)(40,40)
  \qbezier(55,30)(45,30)(40,40)
  
  \qbezier(25,30)(15,30)(10,20)
  \qbezier(25,50)(15,50)(10,60)
  \put(4,8){\line(1,2){6}}
    \put(10,20){\vector(-1,-2){4}}
  \put(4,72){\vector(1,-2){6}}
  \put(4,72){\line(1,-2){6}}
  \qbezier(26,68)(30,60)(25,50)
  \qbezier(26,12)(30,20)(25,30)
  \qbezier(25,30)(20,40)(25,50)

  \put(20,0){\line(1,2){6}}
  \put(20,80){\line(1,-2){4}}
%
  \put(20,0){\vector(1,2){6}}
  \put(26,68){\vector(-1,2){4}}
\end{picture}
\qquad
\qquad
\begin{picture}(75,80)(0,0)
\thicklines
\put(10,0){\line(1,2){18.5}}
\put(10,80){\line(1,-2){18.5}}
\put(33,40){\line(1,0){37}}
\put(30,40){\circle{6}}
\end{picture}
\begin{picture}(30,80)(0,0)
\put(15,40){\makebox(0,0){$\longmapsto$}}
\end{picture}
\begin{picture}(70,80)(0,0)
\thicklines
\put(4,8){\line(1,2){12}}
\put(4,72){\line(1,-2){12}}
\qbezier(16,32)(20,40)(16,48)

\put(40,50){\line(1,0){30}}
\put(32,56){\line(-1,2){12}}
\put(40,30){\line(1,0){30}}
\put(32,24){\line(-1,-2){12}}
\qbezier(40,50)(35,50)(32,56)
\qbezier(40,30)(35,30)(32,24)

\put(70,50){\line(-1,0){30}}
\put(4,72){\vector(1,-2){10}}
\put(20,0){\vector(1,2){10}}
\put(70,50){\vector(-1,0){20}}
\put(16,32){\vector(-1,-2){8}}
\put(32,56){\vector(-1,2){7}}
\put(40,30){\vector(1,0){20}}
\end{picture}
\end{center}
\caption{Building an oriented divide around internal vertices of a plabic
graph.}
\label{fig:plabic-to-o-divide}
\vspace{-.1in}
\end{figure}

\vspace{-.05in}

\begin{figure}[ht]
\begin{center}
\setlength{\unitlength}{1.9pt}
\begin{picture}(33,26)(0,-6)
\thicklines
\put(16.5,-6){\line(0,1){30}}
\thinlines
\put(5,-6){\line(1,0){23}}
\put(5,-7){\line(1,0){23}}
\put(16.5,-6.5){\circle*{2.5}}
\end{picture}
\begin{picture}(10,26)(0,-6)
\put(5,12){\makebox(0,0){$\longmapsto$}}
\end{picture}
\begin{picture}(33,26)(0,-6)
\thicklines
\qbezier(14,4)(14,-2)(16.5,-2)
\qbezier(19,4)(19,-2)(16.5,-2)
\qbezier(19,4)(19,8)(16.5,12)
\qbezier(13,24)(13,18)(16.5,12)
\qbezier(14,4)(14,8)(16.5,12)
\qbezier(20,24)(20,18)(16.5,12)
\put(13.8,18){\vector(1,-4.5){0.1}}
\put(19.8,21){\vector(1,4.5){0.1}}
\put(19,3){\vector(0,-1){1}}
\put(14,5){\vector(0,1){1}}
\thinlines
\put(5,-6){\line(1,0){23}}
\put(5,-7){\line(1,0){23}}
\end{picture}
\qquad\quad
\begin{picture}(33,26)(0,-6)
\thicklines
\put(16.5,-5.2){\line(0,1){29.8}}
\thinlines
\put(5,-6){\line(1,0){10.3}}
\put(5,-7){\line(1,0){10.3}}
\put(28,-6){\line(-1,0){10.3}}
\put(28,-7){\line(-1,0){10.3}}
\put(16.5,-6.5){\circle{2.5}}
\end{picture}
\begin{picture}(10,26)(0,-6)
\put(5,12){\makebox(0,0){$\longmapsto$}}
\end{picture}
\begin{picture}(33,26)(0,-6)
\thicklines
\put(13,4){\line(0,1){21}}
\put(20,4){\line(0,1){21}}
\qbezier(13,4)(13,-2)(16.5,-2)
\qbezier(20,4)(20,-2)(16.5,-2)

\put(13,13){\vector(0,-1){1}}
\put(20,15){\vector(0,1){1}}
\thinlines
\put(5,-6){\line(1,0){23}}
\put(5,-7){\line(1,0){23}}
\end{picture}
\end{center}
\caption{Building an oriented divide near
the boundary of the ambient disk~$\Disk$.
}
\label{fig:plabic-boundary-to-o-divide}
\vspace{-.1in}
\end{figure}


\begin{remark}
\label{rem:graph-divides-preliminary}
The construction of the link of a plabic graph given in Definition~\ref{def:L(P)=L(o(P))} is a special case of the construction given by T.~Kawamura~\cite{kawamura} in her theory of \emph{graph divides}, which we will briefly review in Section~\ref{sec:transverse}; cf.\ in particular Definition~\ref{def:graph-divide}, Remark~\ref{rem:graph-divides-vs-plabic-graphs}, and Figure~\ref{fig:link-types}. 
\end{remark}

Since we view plabic graphs up to a global reversal of colors,
we need to check how such a reversal affects the notions 
introduced in Definition~\ref{def:L(P)=L(o(P))}:  

\begin{proposition}
\label{prop:plabic-reverse}
 Let $P$ be a plabic graph, 
 and let $-P$ denote the plabic graph obtained from~$P$ 
 by reversing the colors of all vertices. 
Then the oriented divide $o(-P)$ is move equivalent to~$-o(P)$. 
Furthermore, the links $L(P)$ and $L(-P)$ are isotopic to each other.
\end{proposition}

\begin{proof}
  The oriented divides $o(P)$ and $-o(-P)$ 
  differ by iso\-topy and a safe tangency move for each edge of~$P$
  connecting vertices of the same color, 
  as illustrated in Figure~\ref{fig:o(P)-vs-o(-P)}.
  The claim then follows by Proposition~\ref{prop:ybe-oriented} and
  Lemma~\ref{lem:reverse-divide}.
\end{proof}

\begin{figure}[ht]
\begin{center}
\begin{tabular}{ccc}
\setlength{\unitlength}{1.1pt}
\begin{picture}(100,135)(0,-17.5)
\thicklines
\put(10,-3){\circle{4}}
\put(10,50){\circle*{4}}
\put(10,103){\circle*{4}}
\put(50,-16.4){\circle*{4}}
\put(50,50){\circle{4}}
\put(90,-3){\circle*{4}}
\put(90,50){\circle{4}}
\put(90,103){\circle*{4}}
\put(12.5,50){\line(1,0){35}}
\put(52.5,50){\line(1,0){35}}
\put(10,47.5){\line(0,-1){48}}
\put(10,52.5){\line(0,1){48}}
\put(50,-14){\line(0,1){61.5}}
\put(90,47.5){\line(0,-1){48}}
\put(90,52.5){\line(0,1){48}}

\thinlines
\put(50,50){\circle{130.6}}
\put(50,50){\circle{135}}

\put(110,-10){\makebox(0,0){$P$}}
\end{picture}

&\qquad\qquad\qquad\qquad&

\setlength{\unitlength}{1.1pt}
\begin{picture}(100,135)(0,-17.5)
\thicklines
\put(13,27){\vector(0,1){6.5}}
\put(53,27){\vector(0,1){6.5}}
\put(93,27){\vector(0,1){6.5}}
\put(7,33){\vector(0,-1){6.5}}
\put(47,33){\vector(0,-1){6.5}}
\put(87,33){\vector(0,-1){6.5}}
\put(7,73){\vector(0,-1){6.5}}
\put(13,67){\vector(0,1){6.5}}
\put(87,73){\vector(0,-1){6.5}}
\put(93,67){\vector(0,1){6.5}}

\put(27,47){\vector(1,0){6.5}}
\put(33,53){\vector(-1,0){6.5}}
\put(67,47){\vector(1,0){6.5}}
\put(73,53){\vector(-1,0){6.5}}

\put(13,47){\line(0,-1){3}}
\put(13,47){\line(1,0){3}}
\put(13,53){\line(0,1){3}}
\put(13,53){\line(1,0){3}}
\put(16,47){\line(1,1){6}}
\put(16,53){\line(1,-1){6}}
\put(7,44){\line(1,-1){6}}
\put(13,44){\line(-1,-1){6}}
\put(7,56){\line(1,1){6}}
\put(13,56){\line(-1,1){6}}
\put(7,44){\line(0,1){12}}

\put(7,62){\line(0,1){19}}
\put(13,62){\line(0,1){19}}

\put(7,81){\line(1,1){6}}
\put(13,81){\line(-1,1){6}}
\put(7,87){\line(0,1){6}}
\put(13,87){\line(0,1){6}}
\qbezier(7,93)(7,96)(10,96)
\qbezier(13,93)(13,96)(10,96)

\put(7,38){\line(0,-1){19}}
\put(13,38){\line(0,-1){19}}

\put(7,19){\line(0,-1){6}}
\put(13,19){\line(0,-1){6}}
\put(7,13){\line(0,-1){6}}
\put(13,13){\line(0,-1){6}}
\qbezier(7,7)(7,4)(10,4)
\qbezier(13,7)(13,4)(10,4)

\put(22,47){\line(1,0){16}}
\put(22,53){\line(1,0){16}}

\put(38,47){\line(1,0){6}}
\put(38,53){\line(1,0){6}}
\put(56,47){\line(1,0){6}}
\put(56,53){\line(1,0){6}}
\put(44,53){\line(1,0){12}}
\put(44,47){\line(1,0){3}}
\put(53,47){\line(1,0){3}}
\put(53,47){\line(0,-1){3}}
\put(47,47){\line(0,-1){3}}
\put(47,44){\line(0,-1){6}}
\put(53,44){\line(0,-1){6}}

\put(47,38){\line(0,-1){35.4}}
\put(53,38){\line(0,-1){35.4}}

\put(47,2.6){\line(1,-1){6}}
\put(53,2.6){\line(-1,-1){6}}
\put(47,-3.4){\line(0,-1){6}}
\put(53,-3.4){\line(0,-1){6}}
\qbezier(47,-9.4)(47,-12.4)(50,-12.4)
\qbezier(53,-9.4)(53,-12.4)(50,-12.4)

\put(62,47){\line(1,0){16}}
\put(62,53){\line(1,0){16}}

\put(87,47){\line(0,-1){3}}
\put(87,47){\line(-1,0){3}}
\put(87,53){\line(0,1){3}}
\put(87,53){\line(-1,0){3}}
\put(84,47){\line(-1,0){6}}
\put(84,53){\line(-1,0){6}}
\put(87,44){\line(0,-1){6}}
\put(93,44){\line(0,-1){6}}
\put(87,56){\line(0,1){6}}
\put(93,56){\line(0,1){6}}
\put(93,44){\line(0,1){12}}

\put(87,62){\line(0,1){19}}
\put(93,62){\line(0,1){19}}

\put(87,81){\line(1,1){6}}
\put(93,81){\line(-1,1){6}}
\put(87,87){\line(0,1){6}}
\put(93,87){\line(0,1){6}}
\qbezier(87,93)(87,96)(90,96)
\qbezier(93,93)(93,96)(90,96)

\put(87,38){\line(0,-1){19}}
\put(93,38){\line(0,-1){19}}

\put(87,19){\line(1,-1){6}}
\put(93,19){\line(-1,-1){6}}
\put(87,13){\line(0,-1){6}}
\put(93,13){\line(0,-1){6}}
\qbezier(87,7)(87,4)(90,4)
\qbezier(93,7)(93,4)(90,4)

\thinlines
\put(50,50){\circle{130.6}}
\put(50,50){\circle{135}}

\put(120,-10){\makebox(0,0){$o(P)$}}
\end{picture}
\\[.3in]
\setlength{\unitlength}{1.1pt}
\begin{picture}(100,135)(0,-17.5)
\thicklines
\put(10,-3){\circle*{4}}
\put(10,50){\circle{4}}
\put(10,103){\circle{4}}
\put(50,-16.4){\circle{4}}
\put(50,50){\circle*{4}}
\put(90,-3){\circle{4}}
\put(90,50){\circle*{4}}
\put(90,103){\circle{4}}
\put(12.5,50){\line(1,0){35}}
\put(52.5,50){\line(1,0){35}}
\put(10,47.5){\line(0,-1){48}}
\put(10,52.5){\line(0,1){48}}
\put(50,-14){\line(0,1){61.5}}
\put(90,47.5){\line(0,-1){48}}
\put(90,52.5){\line(0,1){48}}

\thinlines
\put(50,50){\circle{130.6}}
\put(50,50){\circle{135}}

\put(110,-10){\makebox(0,0){$-P$}}
\end{picture}

&\qquad\qquad\qquad\qquad&

\setlength{\unitlength}{1.1pt}
\begin{picture}(100,135)(0,-17.5)
\thicklines
\put(7,27){\vector(0,1){6.5}}
\put(47,27){\vector(0,1){6.5}}
\put(87,27){\vector(0,1){6.5}}
\put(13,33){\vector(0,-1){6.5}}
\put(53,33){\vector(0,-1){6.5}}
\put(93,33){\vector(0,-1){6.5}}
\put(13,73){\vector(0,-1){6.5}}
\put(7,67){\vector(0,1){6.5}}
\put(93,73){\vector(0,-1){6.5}}
\put(87,67){\vector(0,1){6.5}}

\put(27,53){\vector(1,0){6.5}}
\put(33,47){\vector(-1,0){6.5}}
\put(67,53){\vector(1,0){6.5}}
\put(73,47){\vector(-1,0){6.5}}

\put(13,47){\line(0,-1){3}}
\put(13,47){\line(1,0){3}}
\put(13,53){\line(0,1){3}}
\put(13,53){\line(1,0){3}}
\put(16,47){\line(1,0){6}}
\put(16,53){\line(1,0){6}}
\put(7,44){\line(0,-1){6}}
\put(13,44){\line(0,-1){6}}
\put(7,56){\line(0,1){6}}
\put(13,56){\line(0,1){6}}
\put(7,44){\line(0,1){12}}

\put(7,62){\line(0,1){19}}
\put(13,62){\line(0,1){19}}

\put(7,81){\line(0,1){6}}
\put(13,81){\line(0,1){6}}
\put(7,87){\line(0,1){6}}
\put(13,87){\line(0,1){6}}
\qbezier(7,93)(7,96)(10,96)
\qbezier(13,93)(13,96)(10,96)

\put(7,38){\line(0,-1){19}}
\put(13,38){\line(0,-1){19}}

\put(7,19){\line(1,-1){6}}
\put(13,19){\line(-1,-1){6}}
\put(7,13){\line(0,-1){6}}
\put(13,13){\line(0,-1){6}}
\qbezier(7,7)(7,4)(10,4)
\qbezier(13,7)(13,4)(10,4)

\put(22,47){\line(1,0){16}}
\put(22,53){\line(1,0){16}}

\put(38,47){\line(1,1){6}}
\put(38,53){\line(1,-1){6}}
\put(56,47){\line(1,1){6}}
\put(56,53){\line(1,-1){6}}
\put(44,53){\line(1,0){12}}
\put(44,47){\line(1,0){3}}
\put(53,47){\line(1,0){3}}
\put(53,47){\line(0,-1){3}}
\put(47,47){\line(0,-1){3}}
\put(47,44){\line(1,-1){6}}
\put(53,44){\line(-1,-1){6}}

\put(47,38){\line(0,-1){35.4}}
\put(53,38){\line(0,-1){35.4}}

\put(47,2.6){\line(0,-1){6}}
\put(53,2.6){\line(0,-1){6}}
\put(47,-3.4){\line(0,-1){6}}
\put(53,-3.4){\line(0,-1){6}}
\qbezier(47,-9.4)(47,-12.4)(50,-12.4)
\qbezier(53,-9.4)(53,-12.4)(50,-12.4)

\put(62,47){\line(1,0){16}}
\put(62,53){\line(1,0){16}}

\put(87,47){\line(0,-1){3}}
\put(87,47){\line(-1,0){3}}
\put(87,53){\line(0,1){3}}
\put(87,53){\line(-1,0){3}}
\put(84,47){\line(-1,1){6}}
\put(84,53){\line(-1,-1){6}}
\put(87,44){\line(1,-1){6}}
\put(93,44){\line(-1,-1){6}}
\put(87,56){\line(1,1){6}}
\put(93,56){\line(-1,1){6}}
\put(93,44){\line(0,1){12}}

\put(87,62){\line(0,1){19}}
\put(93,62){\line(0,1){19}}

\put(87,81){\line(0,1){6}}
\put(93,81){\line(0,1){6}}
\put(87,87){\line(0,1){6}}
\put(93,87){\line(0,1){6}}
\qbezier(87,93)(87,96)(90,96)
\qbezier(93,93)(93,96)(90,96)

\put(87,38){\line(0,-1){19}}
\put(93,38){\line(0,-1){19}}

\put(87,19){\line(0,-1){6}}
\put(93,19){\line(0,-1){6}}
\put(87,13){\line(0,-1){6}}
\put(93,13){\line(0,-1){6}}
\qbezier(87,7)(87,4)(90,4)
\qbezier(93,7)(93,4)(90,4)

\thinlines
\put(50,50){\circle{130.6}}
\put(50,50){\circle{135}}

\put(120,-10){\makebox(0,0){$-o(-P)$}}
\end{picture}
\end{tabular}
\end{center}
\caption{Plabic graphs $P$ and $-P$, and oriented divides $o(P)$ and $-o(-P)$.}
\label{fig:o(P)-vs-o(-P)}
\end{figure}

Move equivalence of plabic graphs translates into move equivalence
of associated oriented divides: 

\begin{proposition}
\label{pr:plabic-oriented}
  If plabic graphs $P_1$ and~$P_2$ are move equivalent, 
  then the associated oriented divides $o(P_1)$ and $o(P_2)$ are move equivalent.
\end{proposition}

\begin{proof}
  The most complicated case is the square move, shown in
  Figure~\ref{fig:square-move-o-divide} (from left to right). The transition between the corresponding oriented divides 
  involves a total of four triangle moves and two safe tangency~moves.
  
  The other moves on plabic graphs are easier: a flip move between two white vertices changes $o(P)$ by an isotopy, and a flip move between two black
  vertices changes $o(P)$ by two safe tangency moves. A tail
  attachment/removal changes $o(P)$ by a safe tangency move (if the
  internal vertex is black), a U-turn move (if the boundary
  vertex is black), or both (if both are black).
 \end{proof}

\begin{figure}[ht]
\begin{center}
\begin{tabular}{ccc}
\setlength{\unitlength}{1.5pt}
\begin{picture}(40,43)(0,0)
\thicklines
\put(-3,-3){\line(1,1){12}}
\put(43,43){\line(-1,-1){12}}
\put(-3,43){\line(1,-1){12}}
\put(43,-3){\line(-1,1){12}}
\put(10,11.7){\line(0,1){16.6}}
\put(30,11.7){\line(0,1){16.6}}
\put(11.7,10){\line(1,0){16.6}}
\put(11.7,30){\line(1,0){16.6}}
\put(10,10){\circle*{3.4}}
\put(10,30){\circle{3.4}}
\put(30,10){\circle{3.4}}
\put(30,30){\circle*{3.4}}
\end{picture}
&
\setlength{\unitlength}{1.5pt}
\begin{picture}(80,43)(-20,0)
\thicklines
\put(-3,-3){\line(1,1){46}}
\put(-3,43){\line(1,-1){46}}
\end{picture}
&
\setlength{\unitlength}{1.5pt}
\begin{picture}(40,43)(0,0)
\thicklines
\put(-3,-3){\line(1,1){12}}
\put(43,43){\line(-1,-1){12}}
\put(-3,43){\line(1,-1){12}}
\put(43,-3){\line(-1,1){12}}
\put(10,11.7){\line(0,1){16.6}}
\put(30,11.7){\line(0,1){16.6}}
\put(11.7,10){\line(1,0){16.6}}
\put(11.7,30){\line(1,0){16.6}}
\put(10,10){\circle{3.4}}
\put(10,30){\circle*{3.4}}
\put(30,10){\circle*{3.4}}
\put(30,30){\circle{3.4}}
\end{picture}
\\[.4in]
\setlength{\unitlength}{1.5pt}
\begin{picture}(40,40)(0,0)
\thicklines
\put(8,10.8){\line(-1,-1){2.1}}
\put(10.8,8){\line(-1,-1){2.1}}
\put(5.9,8.7){\line(0,-1){5.6}}
\put(8.7,5.9){\line(-1,0){5.6}}
\put(-4.4,-1.6){\line(1,1){7.5}}
\put(-1.6,-4.4){\line(1,1){7.5}}
\put(-1.6,-4.4){\vector(1,1){7}}
\put(8,15){\line(0,-1){4.2}}
\put(12,15){\line(0,-1){3}}
\put(8,15){\line(1,1){4}}
\put(12,15){\line(-1,1){4}}
\put(15,8){\line(-1,0){4.2}}
\put(15,12){\line(-1,0){3}}
\put(15,8){\line(1,1){4}}
\put(15,12){\line(1,-1){4}}

\put(19,8){\line(1,0){10.2}}
\put(19,12){\line(1,0){9}}

\put(28,12){\line(0,1){9}}
\put(32,10.8){\line(0,1){10.2}}

\put(41.6,-4.4){\line(-1,1){12.4}}
\put(44.4,-1.6){\line(-1,1){12.4}}
\put(44.4,-1.6){\vector(-1,1){9}}
\put(32,29.2){\line(1,1){2.1}}
\put(29.2,32){\line(1,1){2.1}}
\put(34.1,31.3){\line(0,1){5.6}}
\put(31.3,34.1){\line(1,0){5.6}}
\put(44.4,41.6){\line(-1,-1){7.5}}
\put(41.6,44.4){\line(-1,-1){7.5}}
\put(41.6,44.4){\vector(-1,-1){7}}
\put(32,25){\line(0,1){4.2}}
\put(28,25){\line(0,1){3}}
\put(32,25){\line(-1,-1){4}}
\put(28,25){\line(1,-1){4}}
\put(25,32){\line(1,0){4.2}}
\put(25,28){\line(1,0){3}}
\put(25,32){\line(-1,-1){4}}
\put(25,28){\line(-1,1){4}}

\put(8,19){\line(0,1){10.2}}
\put(12,19){\line(0,1){9}}

\put(12,28){\line(1,0){9}}
\put(10.8,32){\line(1,0){10.2}}

\put(-4.4,41.6){\line(1,-1){12.4}}
\put(-1.6,44.4){\line(1,-1){12.4}}
\put(-4.4,41.6){\vector(1,-1){9}}
\end{picture}
&
\setlength{\unitlength}{1.5pt}
\begin{picture}(80,40)(-20,0)
\thicklines
\put(-4.4,-1.6){\line(1,1){46}}
\put(-1.6,-4.4){\line(1,1){46}}
\put(-4.4,41.6){\line(1,-1){46}}
\put(-1.6,44.4){\line(1,-1){46}}

\put(1.4,-1.4){\vector(1,1){8}}
\put(-1.4,38.6){\vector(1,-1){8}}
\put(41.4,1.4){\vector(-1,1){8}}
\put(38.6,41.4){\vector(-1,-1){8}}
\end{picture}
&
\setlength{\unitlength}{1.5pt}
\begin{picture}(40,40)(0,0)
\thicklines
\put(32,10.8){\line(1,-1){2.1}}
\put(29.2,8){\line(1,-1){2.1}}
\put(34.1,8.7){\line(0,-1){5.6}}
\put(31.3,5.9){\line(1,0){5.6}}
\put(44.4,-1.6){\line(-1,1){7.5}}
\put(41.6,-4.4){\line(-1,1){7.5}}
\put(-1.6,-4.4){\vector(1,1){9}}
\put(32,15){\line(0,-1){4.2}}
\put(28,15){\line(0,-1){3}}
\put(32,15){\line(-1,1){4}}
\put(28,15){\line(1,1){4}}
\put(25,8){\line(1,0){4.2}}
\put(25,12){\line(1,0){3}}
\put(25,8){\line(-1,1){4}}
\put(25,12){\line(-1,-1){4}}

\put(21,8){\line(-1,0){10.2}}
\put(21,12){\line(-1,0){9}}

\put(12,12){\line(0,1){9}}
\put(8,10.8){\line(0,1){10.2}}

\put(-1.6,-4.4){\line(1,1){12.4}}
\put(-4.4,-1.6){\line(1,1){12.4}}
\put(44.4,-1.6){\vector(-1,1){7}}
\put(8,29.2){\line(-1,1){2.1}}
\put(10.8,32){\line(-1,1){2.1}}
\put(5.9,31.3){\line(0,1){5.6}}
\put(8.7,34.1){\line(-1,0){5.6}}
\put(-4.4,41.6){\line(1,-1){7.5}}
\put(-1.6,44.4){\line(1,-1){7.5}}
\put(41.6,44.4){\vector(-1,-1){9}}
\put(8,25){\line(0,1){4.2}}
\put(12,25){\line(0,1){3}}
\put(8,25){\line(1,-1){4}}
\put(12,25){\line(-1,-1){4}}
\put(15,32){\line(-1,0){4.2}}
\put(15,28){\line(-1,0){3}}
\put(15,32){\line(1,-1){4}}
\put(15,28){\line(1,1){4}}

\put(32,19){\line(0,1){10.2}}
\put(28,19){\line(0,1){9}}

\put(28,28){\line(-1,0){9}}
\put(29.2,32){\line(-1,0){10.2}}

\put(44.4,41.6){\line(-1,-1){12.4}}
\put(41.6,44.4){\line(-1,-1){12.4}}
\put(-4.4,41.6){\vector(1,-1){7}}
\end{picture}
\end{tabular}

\end{center}
\caption{The square move and oriented divides.}
\label{fig:square-move-o-divide}
\end{figure}


Combining Propositions~\ref{prop:ybe-oriented} and~\ref{pr:plabic-oriented},
we obtain: 

\begin{corollary}
\label{cor:plabic-isotopic}
  If two plabic graphs $P_1$ and~$P_2$ are move equivalent, then
the links  $L(P_1)$ and $L(P_2)$ are isotopic.
\end{corollary}

\begin{remark}
Corollary~\ref{cor:plabic-isotopic} provides a powerful tool that can be used to show that a particular pair of plabic graphs 
are not move equivalent, by verifying that their respective links are not isotopic. (The links can be computed, e.g., using Hirasawa's algorithm~\cite{hirasawa}.) 
We note that using quiver mutations (i.e., a test based on Proposition~\ref{pr:plabic-vs-quivers}) for this purpose is problematic, since there is no known good algorithm for deciding whether two quivers are mutation equivalent or not, cf.\ Remark~\ref{rem:mut-equivalence-difficult}.
Besides, the isotopy class of $L(P)$ is a finer invariant of a plabic graph~$P$ than the quiver~$Q(P)$, cf.\ Example~\ref{ex:using-links-to show-move-inequivalence} below. 
\end{remark}

\begin{example}
\label{ex:using-links-to show-move-inequivalence}
The plabic graphs $P_1$ and $P_2$ in Figure~\ref{fig:isthmus-move}
have isomorphic quivers. 
In spite of that, $P_1$ and $P_2$ are \emph{not} move equivalent,
because the links $L(P_1)$ and $L(P_2)$ are \emph{not} isotopic. 
To be concrete, $L(P_1)$ (here $P_1$ is the graph on the left) has an unknotted component, 
whereas both components of $L(P_2)$ are trefoils. 
\end{example}

To any divide~$D$,
the ``roundabout'' construction in Definition~\ref{def:plabic-divide}
attaches a family of plabic graphs~$P\in\PP(D)$, 
all of them move equivalent to each other. 
Using the construction in Definition~\ref{def:L(P)=L(o(P))},
we then obtain a family of oriented divides~$o(P)$,
also move equivalent to each other, by virtue of Proposition~\ref{pr:plabic-oriented}. 
It~is natural to compare the oriented divides $o(P)$ 
to the divide $o(D)$
obtained by the doubling procedure of Definition~\ref{def:doubling}. 

\begin{proposition}
\label{pr:divide-plabic-oriented}
  Let $D$ be a divide, and $P \in \PP(D)$ a plabic
  graph attached to~$D$. 
  Then the oriented divides $o(D)$ and $o(P)$ are move equivalent to each other. 
\end{proposition}

\begin{proof}
Let us compare the oriented divides $o(D)$ and $o(P)$ near a
crossing of~$D$. This is shown in
Figure~\ref{fig:square-move-o-divide} (bottom row): $o(D)$ is in the middle
whereas $o(P)$ is either on the left or on the right (depending on the
choice made when constructing~$P$, cf.\ Definition~\ref{def:plabic-divide}). 
As explained in the proof of Proposition~\ref{pr:plabic-oriented}, these
oriented divides are move equivalent. 

The oriented divides $o(D)$ and $o(P)$ may also differ near the points of~$D$ lying on the boundary~$\partial\Disk$ (depending on the colors chosen for corresponding vertices of~$P$). These discrepancies can be straightened out using U-turn moves. 
\end{proof}

\begin{proposition}
\label{pr:divide-graph-divide}
  Let $D$ be a divide, and $P \in \PP(D)$ be a plabic
  graph attached to~$D$. 
  Then the links $L(D)$ and $L(P)$ are isotopic to each other.
\end{proposition}

\begin{proof}
We have $L(D)=L(o(D))$ by Proposition~\ref{prop:L(D)=L(o(D))},
and $L(P)=L(o(P))$ by Definition~\ref{def:L(P)=L(o(P))}. 
The claim follows by Propositions~\ref{prop:ybe-oriented}
and~\ref{pr:divide-plabic-oriented}. 
\end{proof}

\begin{corollary}
\label{cor:p=>d}
Let $P_1$ and $P_2$ be plabic graphs attached to divides $D_1$ and~$D_2$, respectively.
If $P_1$ and $P_2$ are move equivalent, then 
$D_1$ and $D_2$ are link equivalent. 
\end{corollary}

\begin{proof}
This is immediate from Corollary~\ref{cor:plabic-isotopic} and Proposition~\ref{pr:divide-graph-divide}. 
\end{proof}

Corollary~\ref{cor:p=>d} establishes the implication \textbf{(p)}$\Rightarrow$\textbf{(d)} of Conjecture~\ref{conj:morsif=mut-via-divides}, even without the assumption of algebraicity. 

The diagram in Figure~\ref{fig:divide-to-plabic-to-link} summarizes the correspondences between the various types of objects considered above. 
This diagram commutes (up to the appropriate equivalences), so that for example the link of a divide~$D$ (computed using the A'Campo construction, see Definition~\ref{def:acampo-link}) is isotopic to the link of an oriented divide corresponding to a plabic graph attached to~$D$. 

\begin{figure}[ht]
  \[
    \begin{tikzcd}[row sep=large] 
 && \text{divide}
\arrow[dd, rightarrow, yshift=0in, "\text{Def.~\ref{def:plabic-divide}}"] 
\arrow[lldddd, "\text{Def.~\ref{def:quiver-of-divide}}", rightarrow, bend left=-35] 
\arrow[dddd, rightarrow, yshift=0in, bend right=-70, "\text{Def.~\ref{def:doubling}}"] 
\arrow[dddddd, rightarrow, bend right=-85, looseness=1.5, "\text{Def.~\ref{def:acampo-link}}"] 
\\
& & \\
 && \text{plabic graph} 
 \arrow[dd, rightarrow, yshift=0in, "\text{Def.~\ref{def:L(P)=L(o(P))}}"] 
 \arrow[ddll, rightarrow, yshift=0in, "\text{Def.~\ref{def:Q(P)}}"] 
\arrow[dddd, rightarrow, bend right=-70, "\text{Rem.~\ref{rem:graph-divides-preliminary}}"] 
\\
& & \\
 \text{quiver} && \text{oriented divide} 
  \arrow[dd, rightarrow, "\text{Def.~\ref{def:oriented-divide}}"] 
\\
& & \\
  && \text{link}
    \end{tikzcd}
  \]
  \caption{Various constructions involving divides, plabic graphs (viewed up to move equivalence), quivers, oriented divides, and links.}
  \label{fig:divide-to-plabic-to-link}
\end{figure}

\newpage

\section{Quasipositive and transverse links}
\label{sec:transverse}

In Sections~\ref{sec:links-of-divides}--\ref{sec:links-from-plabic}, we discussed how to construct a link from a divide, a plabic graph, or an oriented divide. 
It~is natural to ask: what kind of links arise via these constructions? 
It~turns out that the links of plabic graphs (and even more generally,
of graph divides, see below) are special in two ways. 
On the one hand, these links are \emph{quasipositive}. On the other hand, they
naturally carry extra structure: they are \emph{transverse} links. 
This additional structure will lead us to formulate 
some more refined conjectures.

\begin{definition}
\label{def:positive}
  Let $\beta$ be a braid. We say that 
  \begin{itemize}[leftmargin=.3in]
  \item
  $\beta$ is \emph{positive} if it is
  a product of standard Artin generators~$\sigma_j$;
\item
  $\beta$ is  \emph{strongly quasipositive} if it is a product of the
  standard conjugates of the~$\sigma_j$: 
  \[
    \sigma_{i,j} = (\sigma_i \sigma_{i+1} \cdots \sigma_{j-1}) \sigma_j
        (\sigma_{j-1}^{-1}\cdots \sigma_{i+1}^{-1} \sigma_i^{-1});  
    \]
  \item
  $\beta$ is \emph{quasipositive} if it is
  a product of arbitrary conjugates of the~$\sigma_j$. 
  \end{itemize}
Now let $L$ be an oriented link. We say that 
  \begin{itemize}[leftmargin=.3in]
  \item
  $L$ is a \emph{positive braid link} if it can be obtained 
  as the closure of a positive braid; 
  \item
  $L$ is \emph{positive} if it can be represented by a diagram with all crossings positive;  
  \item
  $L$ is \emph{strongly quasipositive} if it is the closure of a 
  strongly quasipositive braid;
  \item
  $L$ is \emph{quasipositive} if it is the closure of a 
  quasipositive braid.
  \end{itemize}
In this listing,
each line describes a wider class of links than the previous one. (L.~Rudolph showed~\cite{rudolph-sqp} that every positive link is strongly quasipositive.) 
\end{definition}

In addition to braid closures and (more general) link diagrams,
links can be constructed using divides (Definition~\ref{def:acampo-link})
or, more generally, plabic graphs (Definition~\ref{def:L(P)=L(o(P))}) or oriented divides (Definition~\ref{def:oriented-divide}). 
Another closely related construction, 
already mentioned in Remark~\ref{rem:graph-divides-preliminary}, 
is the following, cf.~T.~Kawamura~\cite{kawamura}. 

\begin{definition}
\label{def:graph-divide}
  A \emph{graph divide} is a connected planar graph~$G=(V,E)$ 
    (multiple edges and loops are allowed) 
in which each vertex is colored black or white;  
we moreover fix a  
 proper embedding of~$G$ into the disk~$\Disk$ (viewed up to isotopy),
 and additionally assume that each vertex in 
  $V\cap \partial\Disk$ is univalent. 
  A~\emph{trivalent
    graph divide} is a graph divide in which all internal vertices
  (not lying on $\partial\Disk$) are trivalent. 
  A straightforward extension of Definition~\ref{def:L(P)=L(o(P))}
  associates to any graph divide~$G$ 
  the corresponding oriented divide $o(G)$, 
  and thus an oriented link $L(G) = L(o(G))$.
\end{definition}

\begin{remark}
\label{rem:graph-divides-vs-plabic-graphs}
Trivalent graph divides are very close to plabic graphs, but are not subject to some
technical restrictions. If a general graph divide~$G$ has no internal
univalent vertices, then there is a related trivalent graph divide
$G'$ obtained by (a) glueing the pairs of edges which meet at 2-valent vertices 
and (b) splitting each vertex of degree $d > 3$ into a tree made of 
$d-2$ trivalent vertices of the same color. 
It is easy to see that $L(G)$ is isotopic to $L(G')$, so the
key restriction that distinguishes plabic graphs from 
graph divides is the absence of internal univalent vertices.
It is unclear how this restriction affects the corresponding class of links. 
\end{remark}


\begin{definition}
\label{def:algebraic-link}
We call a link \emph{algebraic} if it can arise as the link $L(C,z)$ of an isolated plane curve singularity. 
(Readers beware: some authors ascribe a different meaning to the term ``algebraic link.'') 
By Theorem~\ref{th:L(D)=L(C,z)}, algebraic links are precisely the A'Campo links of algebraic divides. 
\end{definition}

\begin{definition}
\label{def:C-transverse}
  A link~$L\subset\mathbf{S}^3$ is called \emph{$\CC$-transverse} if there is an
  algebraic~plane curve $X \subset \CC^2$ such that $X \cap
  \mathbf{S}^3$ is isotopic to~$L$; here we require that $X$ is smooth along the sphere~$\mathbf{S}^3$ and intersects it transversally. 
This is a much larger class than algebraic links, since there may be  several singular points of~$X$ inside the ball bounded by~$\mathbf{S}^3$. 
\end{definition}

The known relationships between various classes of links mentioned above are shown in Figure~\ref{fig:link-types}. 
In particular: 
\begin{itemize}[leftmargin=.3in]
\item
Any algebraic link is a divide link, by A'Campo's Theorem~\ref{th:L(D)=L(C,z)}.  
\item
Any divide link is the link of a plabic graph, by~Proposition~\ref{pr:divide-graph-divide}. 
\item Any positive braid link can be represented by a plabic graph, 
see Section~\ref{sec:plabic-fences}.
\item Quasipositive links are the same as $\CC$-transverse links, as shown by M.~Boileau and S.~Orevkov~\cite{boileau-orevkov} and
  L.~Rudolph~\cite{rudolph}.
\item Any divide link is strongly quasipositive: 
  M.~Ishikawa \cite{ishikawa} showed that divide links are positive Hopf
  plumbings; L.~Rudolph \cite{rudolph-plumbing} proved that such
  plumbings are strongly quasipositive.
\item Any graph divide link is quasipositive (T.~Kawamura~\cite{kawamura}), 
hence $\CC$-transverse.
\item A'Campo~\cite{acampo-ihes} proved that the link of every divide is \emph{fibered} (i.e., the complement is a fiber bundle over a circle, with fiber a surface).
\item Any link is a link of an oriented divide, see Remark~\ref{rem:o-divide}. 
\end{itemize}

\begin{figure}[ht]
  \begin{center}
    \begin{tikzpicture}[x=5cm,y=1.4cm]
      \node[align=center] (alg) at (0,5) {Algebraic links};
      \node[align=center] (divide) at (0,4) {Divide links};
      \node[align=center] (tridiv) at (0,3) {Plabic graph links};
      \node[align=center] (graphdiv) at (0,2) {Graph divide links};
      \node[align=center] (Ctrans) at (0,1) {$\CC$-transverse links};
      \node[align=center] (o-div) at (0,0) {Oriented divide links};

      \node[align=center] (posbraid) at (1,4) {Positive braid links};
      \node[align=center] (pos) at (1,3) {Positive links};
      \node[align=center] (sqp) at (1,2) {Strongly quasipositive links};
      \node[align=center] (qp) at (1,1) {Quasipositive links};
      \node[align=center] (link) at (1,0) {All links};

      \node[align=center] (fibered) at (-1,3) {Fibered links};

      \draw[right hook->] (alg) to 
        (divide);
      \draw[right hook->] (divide) to 
        (tridiv);
      \draw[right hook->] (tridiv) to (graphdiv);
      \draw[right hook->] (graphdiv) to 
        (Ctrans);
      \draw[right hook->] (Ctrans) to 
        (o-div);

      \draw[right hook->] (posbraid) to 
        (pos);
      \draw[right hook->] (pos) to 
        (sqp);
      \draw[right hook->] (sqp) to 
        (qp);
      \draw[right hook->] (qp) to 
        (link);

      \draw[right hook->] (alg) to 
        (posbraid);
      \draw[left hook->] ([xshift=-18mm, yshift=-4mm] posbraid.100) to (tridiv.15);
      \node at (0.5,1) {$=$};
      \node at (0.5,0) {$=$};
      \draw[right hook->] ([yshift=-3mm] divide.10) to (sqp);
      \draw[left hook->] ([yshift=-3mm] divide.170) to 
        (fibered.30);
    \end{tikzpicture}
  \end{center}
  \caption{
  The known relationships between the classes of links obtained using different 
    constructions (middle column), or exhibiting various forms of positivity (right column). Cf.\ the Venn diagram in \cite[Figure~13]{kawamura-essential}.}
  \label{fig:link-types}
\end{figure}

In addition to being in a restricted class, the links arising from
plabic graphs also have extra structure coming from contact geometry.
(For more background on contact geometry, see~\cite{etnyre}.) 

\begin{definition}\label{def:contact}
  The \emph{standard contact structure} on 
  $\mathbf{S}^3=\{x^2+y^2+u^2+v^2=1\}$ is the
  2-plane field $\xi = \ker \omega$, where $\omega$ is the 1-form
  \begin{equation}\label{eq:std-contact-1form}
    \omega = -u\,dx - v\,dy + x\,du + y\,dv.
  \end{equation}
  (Up to isotopy, this is the unique contact structure on
  $\mathbf{S}^3$ with the additional property of being \emph{tight}.)
  A link~$L$ embedded in $\mathbf{S}^3$ is \emph{Legendrian} if its tangent
  vector~$\dot L$ always lies in~$\xi$, and is \emph{transverse} if
  $\dot L$ never lies in~$\xi$. 
By convention, a transverse link~$L$ is oriented so that $\langle \dot L,\omega\rangle>0$ everywhere on~$L$. 
Two Legendrian/transverse
  links are \emph{Legendrian/transverse isotopic} if they are isotopic
  through Legendrian/transverse embeddings of links. There is a
  natural construction that turns a
  Legendrian link into its \emph{transverse push-off}, see \cite[Section 2.9]{etnyre}.
\end{definition}

We will focus on transverse links. One basic invariant of a transverse
link is its classical link type, i.e., its isotopy type as an ordinary link. 
For each classical link type, there are many transverse link types. 
One of the invariants that distinguishes between some of them is the integer-valued
self-linking number: 

\begin{definition}
\label{def:self-link}
Let $T$ be a transverse link. Recall that
there is a connected orientable Seifert surface
$\Sigma \subset \mathbf{S}^3$ with $T = \partial\Sigma$. Consider
$\xi|_\Sigma$, the contact plane field restricted to~$\Sigma$. Since
$\Sigma$ has nonempty boundary, $\xi|_\Sigma$ is trivial; pick a
section of it, a vector field $v$ defined on~$\Sigma$.
Let $T'$ be a copy of $T$, pushed off in the direction of~$v$. 
The \emph{self-linking number} $\selflink(T)$ of~$T$ is the linking
number of $T$ and~$T'$.
\end{definition}

\begin{remark}
The self-linking number can always be decreased by~$2$---while preserving the ordinary isotopy type of the link---by a local operation called \emph{stabilization}.   For~any link type, there is a maximal realizable value of the self-linking number.
For~each link type and each value of the self-linking number, there are finitely many possible transverse link types. 
\end{remark}

The links arising from the various constructions discussed above are naturally transverse, as we will now explain. 

Braid closures naturally produce transverse links.
Let $U$ be a standard transversal unknot, say the equator $y = v = 0$. 
Given a braid~$\beta$, embed the closure of~$\beta$ in a small tubular
neighborhood of~$U$ to get a link $\beta(U)$. Since the tangent vectors
to~$\beta(U)$ are close to the tangents to~$U$, the link $\beta(U)$ is also
transverse.

\begin{proposition}
\label{pr:c-transverse-is-transverse}
  Every $\CC$-transverse link~$L$ (in particular, any algebraic link; see Definition~\ref{def:C-transverse}) has a natural transverse structure.
\end{proposition}

\begin{proof}
  Let $X\subset \CC^2$ be the algebraic curve giving rise to~$L$, as in  
  Definition~\ref{def:C-transverse}.
  Let $v \in TL = TX \cap T\mathbf{S}^3$ be a
  tangent vector to~$X$. Then $iv$ is also tangent to~$X$, and by
  transversality of the intersection is not in $T\mathbf{S}^3$. Thus
  $v \notin \xi$.
\end{proof}

We next show that the links of oriented divides are naturally transverse.\medskip

If $\Omega$ is a
smooth strictly pseudoconvex (cf., e.g., \cite[Chapter~3, Definition~1.9]{lieb-michel}) subset of~$\CC^2$, then $\partial\Omega$
inherits the structure of a contact manifold, by setting 
\[
  \xi = \{\,v \in T\partial\Omega \mid iv \in T\partial\Omega\,\}.
\]
Strict pseudoconvexity guarantees that locally $\xi=\ker \omega$
for a 1-form $\omega$ such that $\omega \wedge d\omega \ne 0$, as required
for a contact structure.
If $\Omega$ is the standard unit ball in~$\CC^2$, then this
is the contact structure in Definition~\ref{def:contact} above. More
generally, if $\Omega$ is any strictly pseudoconvex topological ball, then the
contact structure is tight and thus equivalent to the standard one.

\begin{definition}
Let $0<\lambda\le 1$.
Following A'Campo \cite[Theorem~3]{acampo-ihes}, consider the squashed ball
$\Ball_\lambda = \{\, (x,y,u,v) \in \CC^2 \mid x^2 + y^2 +
    \lambda^{-2}(u^2 + v^2) < 1\,\} \subset \CC^2$.
As $\mathbf{B}_\lambda$ is strictly pseudoconvex, there is a natural contact
structure on $\mathbf{S}^3_\lambda=\partial\mathbf{B}_\lambda$. (This structure is in fact equivalent to the standard contact structure, by uniqueness of the tight contact structure on $\mathbf{S}^3$.)
\end{definition}

\begin{definition}
\label{def:squashed-lifts}
For an oriented divide~$\vec D$, let $T_\lambda(\vec D)$ denote the lift of~$\vec D$ to~$\mathbf{S}^3_\lambda$ obtained via a straightforward extension of Definition~\ref{def:oriented-divide}.
In particular, for $\lambda=1$, we get $T_\lambda(\vec D)=L(\vec D)$. 
Furthermore, all lifts $T_\lambda(\vec D)$ are naturally isotopic to each other, and to~$L(\vec D)$. 
\end{definition}

\begin{proposition}
\label{prop:transverse-divide}
  Let $\vec D$ be an oriented divide. For any sufficiently
  small~$\lambda$, the link $T_\lambda(\vec D)$ is transverse to the
  contact structure on $\mathbf{S}^3_\lambda$. 
 
  Any move equivalence of
  oriented divides lifts, for sufficiently small~$\lambda$, to a
  transverse isotopy of their associated lifts~$T_\lambda(\vec D)$.
\end{proposition}

\begin{proof}
  Let $D$ be the unoriented divide underlying $\vec D$.
  In the proof of \cite[Theorem~3]{acampo-ihes}, A'Campo shows that
  there is a curve $X \subset \CC^2$ such that for $\lambda$
  sufficiently small, $X \cap \partial\Ball_\lambda$ is very close to
  $T_\lambda(D)$, which is thus $\CC$-transverse and hence transverse. The
  link $T_\lambda(\vec D)$ is a union of some of the components of
  $T_\lambda(D)$ and is therefore
  also transverse (but in general not $\CC$-transverse).%

  The construction of $X$ and $T_\lambda(\vec D)$ works just as well
  through triangle moves and safe tangency moves, which therefore give
  transverse isotopies. To see that a boundary U-turn move is also a
  transverse isotopy, it suffices to see that the $C^1$-isotopy from
  the proof of Proposition~\ref{prop:ybe-oriented} is a transverse
  isotopy. It suffices to check that the curve
  $L_\varepsilon$ is transverse to the contact structure on
  $\mathbf{S}^3$ at $(\varepsilon,t)=(0,0)$, since by continuity the
  curve will also be transverse for $(\varepsilon,t)$ near $(0,0)$. We
  see that
  $L_0(0) = (0,-1,0,0)$ and $\dot L_0(0) = (0,0,0,1)$. 
  Since the 1-form $\omega$ in Equation~\eqref{eq:std-contact-1form}
  has a term $y\,dv$, the link is transverse at this point as desired.
  (Similar computations also work if we lift to $\mathbf{S}^3_\lambda$
  instead.)
\end{proof}

For sufficiently small~$\lambda$, the transverse links $T_\lambda(\vec
D)$ are transversely isotopic to each other. We thus write simply $T(\vec D)$ for
their common transverse isotopy class. 
This isotopy class refines the ordinary isotopy class of the link $L(\vec D)$, cf.\ Definition~\ref{def:squashed-lifts}. 

We can similarly associate a transverse link to any plabic graph~$P$ by
setting $T(P)= T(o(P))$; or to a divide~$D$ by setting
$T(D)= T(o(D))$.


\begin{corollary}
\label{cor:plabic-transverse}
  If two plabic graphs $P_1$ and~$P_2$ are move equivalent, then
  $T(P_1)$ and $T(P_2)$ are transverse isotopic.
\end{corollary}

\begin{proof}
  Immediate from Propositions~\ref{pr:plabic-oriented}
  and~\ref{prop:transverse-divide}.
\end{proof}

It is natural to ask whether the converse is true: 

\begin{problem}
\label{prob:plabic-transverse}
  If two plabic graphs $P_1$ and $P_2$ are transverse-equivalent (i.e.,
  $T(P_1)$ is transverse isotopic to $T(P_2)$), does it follow that $P_1$ and $P_2$ are move equivalent?
(A~weaker version: do the quivers $Q(P_1)$ and $Q(P_2)$ have to be mutation equivalent?) 
\end{problem}

\begin{remark}
In order to pose a similar question for general graph divides (cf.\ Remark~\ref{rem:graph-divides-vs-plabic-graphs}), one would need a proper notion of move equivalence. More precisely, one would need to identify a set of moves (containing the local moves in Definition~\ref{def:moves}) relating any two graph divides whose links are transverse equivalent. 
\end{remark}

It is not hard to show that any transverse link can be realized as $T(\vec D)$ for some oriented divide~$D$, by adapting the argument in~\cite{gibson-ishikawa}; we omit the details. 
On~the other hand, not every transverse link can be realized as a transverse link of a plabic graph (resp., a divide, a graph divide). 
In particular, there is a restriction on the self-linking number: 

\begin{proposition}
\label{prop:plabic-sl}
  Let $P$ be a plabic graph. Then 
  $\selflink(T(P))=-\chi(P)$, where $\chi(P)$ denotes the Euler characteristic of~$P$ (viewed as a 1-dimensional simplicial complex).
\end{proposition}

The proof of Proposition~\ref {prop:plabic-sl} will require 
some preliminary lemmas. 

\begin{lemma}\label{lem:self-link-divide}
  Let $\vec D$ be an oriented divide, and let $\vec D_r$ be a small
  push-off of $\vec D$ in the direction of the right-handed normal
  of~$\vec D$, as shown in Figure~\ref{fig:hypotrochoid}. Then
  $\selflink(T(\vec D))$ is the linking number between
  $L(\vec D)$ and $L(\vec D_r))$.
\end{lemma}

\begin{figure}[ht]
\begin{center}
\setlength{\unitlength}{0.88pt}
\begin{picture}(150,130)(-15,-5)
\linethickness{1.4pt}
\qbezier(20,70)(5,15)(60,0)
\thinlines
\red{\qbezier(27,66)(12,19)(60,8)}
\linethickness{1.4pt}
\qbezier(100,70)(115,15)(60,0)
\thinlines
\red{\qbezier(93,66)(108,19)(60,8)}
\linethickness{1.4pt}
\qbezier(20,70)(60,110)(100,70)
\thinlines
\red{\qbezier(27,66)(60,101)(93,66)}
\linethickness{1.4pt}
\qbezier(100,70)(185,-25)(60,0)
\thinlines
\red{\qbezier(93,66)(171,-17)(60,8)}
\linethickness{1.4pt}
\qbezier(20,70)(-65,-25)(60,0)
\thinlines
\red{\qbezier(27,66)(-51,-17)(60,8)}
\linethickness{1.4pt}
\qbezier(20,70)(60,180)(100,70)
\thinlines
\red{\qbezier(27,66)(60,164)(93,66)}
\thicklines
\put(-7,34){\vector(3,5){1}}
\put(127,34){\vector(3,-5){1}}
\put(63,125){\vector(1,0){1}}
\thicklines
\put(1.5,34){\red{\vector(3,5){1}}}
\put(118.5,34){\red{\vector(3,-5){1}}}
\put(63,114.7){\red{\vector(1,0){1}}}
\put(-20,34){$\vec{D}$}
\put(-2,10){\red{$\vec{D}_r$}}
\end{picture}

\end{center}
\caption{An oriented divide $\vec D$ (thick, black) and its pushoff
  $\vec D_r$ in the direction of the rightward-pointing normal (thin, red).}
\label{fig:hypotrochoid}
\vspace{-.1in}
\end{figure}

\begin{proof}
  We work with the model
  $T_\lambda(\vec D) \subset \mathbf{S}^3_\lambda$, as in the
  definition of $T(\vec D)$. The first step is to find a section of
  the contact plane~$\xi$ on a Seifert surface for
  $T_\lambda(\vec D)$. In fact, the vector
  field on $\mathbf{S}^3_\lambda$ given by
  \[
    \vec a(x,y,u,v) =
      v \frac{\partial}{\partial x} - u \frac{\partial}{\partial y}
      +\lambda^2\Bigl(y \frac{\partial}{\partial u} - x \frac{\partial}{\partial v}\Bigr)
  \]
  is a global section of~$\xi$. Thus we want to find the linking
  number of $T_\lambda(\vec D)$ with its push-off in the direction
  of~$\vec a$. For small~$\lambda$, this push-off is very close to
  $T_\lambda(\vec D_r)$.
\end{proof}

We could use Lemma~\ref{lem:self-link-divide} to directly prove
Proposition~\ref{prop:plabic-sl}. But, as mentioned in
Remark~\ref{rem:o-divides-vs-link-diagrams}, 
it is not very convenient to compute linking numbers
from divide presentations of links---so we instead give some reductions.

\begin{lemma}\label{lem:self-link-circle}
  Let an oriented divide~$\vec D$ be a circle inside the disk~$\Disk$, with no
  crossings and an arbitrary orientation. Then $\selflink(T(\vec D)) = -1$.
More generally, if $\vec D$ is a union of $n$ disjoint non-nested circles, then $\selflink(T(\vec D)) = -n$.

  Similarly, if $\vec D$ is an oriented figure-eight curve with one
  crossing, as shown in Figure~\ref{fig:lemniscate}, then
  $\selflink(T(\vec D)) = -1$.
\end{lemma}

\begin{proof}
Let $\vec D$ be a circle. 
  With $\vec D_r$ the push-off as in Lemma~\ref{lem:self-link-divide},
  the link $L(\vec D \cup \vec D_r)$ is a Hopf link, oriented so that the
  linking number between the two components is~$-1$. (Thus $T(\vec D)$
  is the standard transverse unknot with self-linking $-1$, which is
  the maximum possible.)
If $\vec D$ is a union of $n$ disjoint non-nested circles, then $L(\vec D\cup \vec D_r)$ is a split union of $n$ Hopf links, each with linking number~$-1$. 

  The result for the figure-eight curve follows from
  Proposition~\ref{prop:transverse-divide} by applying a U-turn move
  to an oriented circle.
\end{proof}

\begin{lemma}
\label{lem:self-link-cross}
  Let $\vec D$ be an oriented divide with a crossing at~$x$, and let
  $\vec D'$ be the oriented divide with the crossing at~$x$ smoothed in an
  orientation-preserving way. Then $\selflink(\vec D') =
  \selflink(\vec D) - 1$.
\end{lemma}

\begin{proof}
  By Lemma~\ref{lem:self-link-divide}, we need to compare the linking
  numbers $\link(L(\vec D), L(\vec D_r))$ and $\link(L(\vec D'), L(\vec D'_r))$. To
  move from the oriented divide $\vec D \cup \vec D_r$ to $\vec D' \cup \vec
  D'_r$, we perform two operations, illustrated in Figure~\ref{fig:DDr}.

  First, we smooth a crossing of $\vec D$ with itself, and also a crossing of
  $\vec D_r$ with itself, both preserving orientations. The effect on
  the corresponding links is to do two surgeries, one on~$\vec D$ and one on~$\vec D_r$. This preserves the linking number, since each surgery is done on the same side of the linking number.

  Second, we perform a same-direction tangency move involving both $\vec D$
  and~$\vec D_r$. Since $L(\vec D)$ crosses once through~$L(\vec D_r)$ (cf.\ Remark~\ref{rem:same-direction tangency}), the
  linking number changes by $\pm 1$ (always one or the other). To
  establish the sign, we consider the case of a figure-eight
  divide~$\vec D$ being smoothed into a divide~$\vec D'$ consisting of
  two circles, as in Figure~\ref{fig:lemniscate}. By
  Lemma~\ref{lem:self-link-circle},
  $T(\vec D)$ is a standard transverse unknot with self-linking
  number~$-1$ and $T(\vec D')$ is two unlinked unknots with total
  self-linking number~$-2$.
\end{proof}

\begin{figure}[ht]
\begin{center}
\setlength{\unitlength}{1.5pt}
\begin{picture}(40,44)(0,-4)
\linethickness{1.4pt}
\put(0,0){\line(35,40){35}}
\put(35,0){\line(-35,40){35}}
\thicklines
\put(35,40){\vector(35,40){1}}
\put(0,40){\vector(-35,40){1}}
\thinlines
\put(5,0){\red{\line(35,40){35}}}
\put(40,0){\red{\line(-35,40){35}}}
\thicklines
\put(5,40){\red{\vector(-35,40){1}}}
\put(40,40){\red{\vector(35,40){1}}}
\put(-9,-8){$\vec{D}$}
\put(5,-8){\red{$\vec{D}_r$}}
\end{picture}
\qquad\qquad\qquad
\setlength{\unitlength}{1.5pt}
\begin{picture}(40,44)(0,-4)
\linethickness{1.4pt}
\qbezier(0,0)(25,20)(0.6,39.5)
\thicklines
\put(0,40){\vector(-1.2,1){1}}
\thinlines

\red{\qbezier(5,0)(40,20)(5,40)}

\thicklines
\put(5,40.1){\red{\vector(-1.6,1){1}}}

\thinlines
\red{\qbezier(40,0)(15,20)(40,40)}
\thicklines
\put(40,40){\red{\vector(1.2,1){1}}}

\linethickness{1.4pt}
\qbezier(35,0)(0,20)(34.1,39.4)

\thicklines
\put(35,40.1){\vector(1.6,1){1}}
\end{picture}
\qquad\qquad\qquad
\setlength{\unitlength}{1.5pt}
\begin{picture}(40,44)(0,-4)
\linethickness{1.4pt}
\qbezier(0,0)(20,20)(0,40)
\thicklines
\put(0,40){\vector(-1,1){1}}
\thinlines
\red{\qbezier(5,0)(24,20)(5,40)}
\thicklines
\put(5,40){\red{\vector(-1,1){1}}}
\put(-9,-8){$\vec{D}'$}
\put(5,-8){\red{$\vec{D}_r'$}}

\thinlines
\red{\qbezier(40,0)(20,20)(40,40)}
\thicklines
\put(40,40){\red{\vector(1,1){1}}}
\linethickness{1.4pt}
\qbezier(35,0)(16,20)(35,40)
\thicklines
\put(35,40){\vector(1,1){1}}
\end{picture}
\end{center}
\caption{An oriented divide with a crossing (left) and its oriented
  smoothing (right). Both are shown with their rightward push-offs.
  Shown in the middle is the intermediate step in the proof of Lemma~\ref{lem:self-link-cross}.}
\label{fig:DDr}
\vspace{-.1in}
\end{figure}

\begin{figure}[ht]
\vspace{-.1in}
\begin{center}
\setlength{\unitlength}{1.5pt}
\begin{picture}(80,30)(0,0)
\thicklines
\qbezier(0,15)(0,0)(15,0)
\qbezier(15,0)(30,0)(40,15)
\qbezier(0,15)(0,30)(15,30)
\qbezier(15,30)(30,30)(40,15)
\qbezier(80,15)(80,0)(65,0)
\qbezier(65,0)(50,0)(40,15)
\qbezier(80,15)(80,30)(65,30)
\qbezier(65,30)(50,30)(40,15)
\put(15,0){\vector(-1,0){2}}
\put(65,30){\vector(-1,0){2}}
\put(15,30){\vector(1,0){2}}
\put(65,0){\vector(1,0){2}}
\put(36,-3){$\vec{D}$}
\end{picture}
\qquad\qquad
\setlength{\unitlength}{1.5pt}
\begin{picture}(80,30)(0,0)
\thicklines
\put(15,15){\circle{30}}
\put(50,15){\circle{30}}
\put(15,0){\vector(-1,0){2}}
\put(50,30){\vector(-1,0){2}}
\put(15,30){\vector(1,0){2}}
\put(50,0){\vector(1,0){2}}
\put(27.5,-3){$\vec{D}'$}
\end{picture}
\end{center}
\caption{Smoothing a figure-eight oriented divide into two circles. 
This translates into smoothing an unknot with self-linking number~$-1$ 
into two disjoint unknots, each with self-linking number~$-1$.}
\label{fig:lemniscate}
\vspace{-.1in}
\end{figure}

\begin{proof}[Proof of Proposition~\ref{prop:plabic-sl}]
Let $\vec D_0 = \vec D(P)$ be the oriented divide associated with the plabic graph~$P$. The construction of the oriented divide $D_0$ is explained in Definition~\ref{def:L(P)=L(o(P))}, cf.\ especially Figure~\ref{fig:plabic-to-o-divide}. 
Each edge~$e$ of~$P$ contributes $0$, $1$, or~$2$ crossings to the oriented divide~$\vec D_0$, depending on whether $e$ connects two white vertices, a black and a white vertex, or two black vertices. To make things more uniform, let $\vec D_1$ be the result of doing a safe tangency move for each edge of~$P$ that connects two white vertices. Now, for each edge $e$ of~$P$, the oriented divide 
  $\vec D_1$ has $1$ or~$2$ crossings, depending on
  whether the endpoints of~$e$ have distinct colors or not. 
Let $\vec D_2$ be the result of smoothing each crossing of~$\vec D_1$
  in an orientation-preserving way. Let $V$, $E$, and $K$ denote the
  number of vertices in~$P$, edges in~$P$, and edges in~$P$ connecting
  vertices of the same color, respectively. Then, by
  Lemma~\ref{lem:self-link-cross},
$ \selflink(T(\vec D_2)) = \selflink(T(\vec D_1)) - K - E$.
  Since $\vec D_2$ consists of $V+K$ non-nested circles, we have
$\selflink(T(\vec D_2)) = -V-K$ by Lemma~\ref{lem:self-link-circle}. 
  Consequently 
  \begin{align*}
    \selflink(T(\vec D_0)) = \selflink(T(\vec D_1)) &=  K+E+(-V-K)= E - V = - \chi(P).\qedhere
  \end{align*}
\end{proof}

\begin{corollary}
\label{cor:divide-sl}
Let $D$ be a divide with $a$ nodes and $b$ interval branches.
Let $n$ denote the number of vertices in the quiver~$Q(D)$. 
Then 
\[
\selflink(T(D)) = 2a - b = n-1.
\] 
\end{corollary}

\begin{proof}
  Let $G$ be the $1$-skeleton of~$D$; it is a planar graph with $a$ vertices of
  degree~$4$ and $2b$ vertices of degree~$1$. 
  The graph $G$ has $v=a+2b$ vertices and $e=\frac12(4a+2b)=2a+b$ edges, 
so $\chi(D) = -e+ v= -a+b$.
  The plabic graph $P(D)$ has an extra cycle at each node of~$G$, 
  implying $\chi(P(D)) =\chi(G) - a=-2a+b$. 
  By Proposition~\ref{prop:plabic-sl}, we then have 
$\selflink(T(D)) = -\chi(P(D)) =  2a - b$.

Let $f$ be the number of (bounded) regions of~$D$.
By Euler's formula, $v-e+f=1$. Also, $n=a+f$.
Therefore  $n-1=a+f-1=a-v+e=2a-b$. 
\end{proof}

\begin{remark}
\label{rem:max-self-link}
 For any plabic graph~$P$, the transverse link
  $T(P)$ achieves the maximal self-linking number within its topological
  type, for the following reasons. First, there is a natural smooth
  surface~$\Sigma(P)$ embedded in the $4$-dimensional ball~$\Ball^4$ such that 
  $\partial \Sigma(P) = T(P)$ and $\Sigma(P)$ contains $P$ as a spine, so that
  $\chi(\Sigma(P)) = \chi(P)$. This is explained by V.~Shende, D.~Treumann,
  H.~Williams, and E.~Zaslow  in~\cite[Theorem~4.9]{stwz} 
  in a slightly different setting; we do not repeat the details. Second, the
  slice-Bennequin inequality~\cite{rudolph-slice} says that,
  for any transverse link~$T$ and smooth surface $\Sigma\subset \Ball^4$
  with $\partial \Sigma = T$, we have
  $\selflink(T) \le -\chi(\Sigma)$.
  Combining these two facts with Proposition~\ref{prop:plabic-sl}
  implies that the self-linking number is maximal.
\end{remark}

Given that the link of a (graph) divide is naturally 
transverse, one might think that Conjecture~\ref{conj:morsif=mut-via-divides}
and Problem~\ref{problem:divide-links} should be revised, with link equivalence replaced by a more refined notion based on transverse equivalence. 
This turns out to be unnecessary, 
see Corollary~\ref{cor:divide-link-transverse} below.

\begin{theorem}[J.~Etnyre and J.~van Horn-Morris {\cite[Corollary~5.3]{etnyre-vhm}}]
\label{th:etnyre-vhm}
  Let $L$ be a strongly quasipositive fibered link, with 
(quasipositive Seifert) fiber surface~$\Sigma$. 
Then $L$ has a unique transverse representative with self-linking number equal to $-\chi(\Sigma)$. (This is the maximal possible value of the self-linking number.)  
\end{theorem}

\begin{corollary}
\label{cor:divide-link-transverse}
For any divides $D_1$ and $D_2$, 
the following are equivalent:
\begin{itemize}[leftmargin=.3in]
\item[{\rm\bf (d)}] $D_1$ and $D_2$ are link equivalent 
(i.e., the links $L(D_1)$ and $L(D_2)$ are isotopic);
\item[{\rm\bf (t)}] the transverse links $T(D_1)$ and $T(D_2)$ are transverse isotopic. 
\end{itemize}
\end{corollary}


\begin{proof}
Clearly \textbf{(t)}$\Rightarrow$\textbf{(d)}. Let us prove the converse. 
The link $L(D)$ of any divide~$D$ is fibered and strongly
quasipositive, see Figure~\ref{fig:link-types}. 
Moreover $T(D)$ has the maximal self-linking number; to see this, either use
Remark~\ref{rem:max-self-link} or compare
Corollary~\ref{cor:divide-sl} with the Euler characteristic of the
fiber surface for $L(D)$, as computed in 
\hbox{\cite[Remark~1]{acampo-ihes}.} 
By Theorem~\ref{th:etnyre-vhm}, $L(D)$ has a unique transverse representative in the maximal self-linking number, so \textbf{(d)}$\Rightarrow$\textbf{(t)}.
\end{proof}

Corollary~\ref{cor:divide-link-transverse} implies the transverse version of A'Campo's Theorem~\ref{th:L(D)=L(C,z)}:

\begin{corollary}
The classical ($\CC$-transverse, cf.\  Proposition~\ref{pr:c-transverse-is-transverse}) link of an isolated plane curve singularity is transverse isotopic to the transverse link $T(D)$ associated to any algebraic divide~$D$ coming from a real morsification of this singularity. 
\end{corollary}

\begin{remark}
We can further extend Conjecture~\ref{conj:morsif=mut-via-divides} 
by adding the statement~\textbf{(t)} (transverse equivalence of links), 
which by Corollary~\ref{cor:divide-link-transverse} is equivalent to~\textbf{(d)}, 
even without the assumption of algebraicity. 
\end{remark}

\begin{remark}
\label{rem:transverse-non-simplicity}
For a general plabic graph~$P$ (not necessarily coming from a divide),
there is no reason to expect transverse simplicity;
that is, link equivalence is unlikely to imply transverse equivalence. 
\end{remark}

\begin{remark}
\label{rem:transverse-vs-legendrian}
  The appearance of transverse (rather than Legendrian) 
  links in the study of plane curve singularities 
  might be puzzling to a reader familiar with 
  Arnold's classical construction (cf.\ Remark~\ref{rem:o-divide}) 
  which associates to a plane curve a Legendrian link in the
  solid torus \cite{arnold-curves,stwz}. However, this is too much to
  hope for in the context of links in the 3-sphere (as opposed to the solid
  torus). In particular, the boundary U-turn move does not appear to
  extend to a Legendrian isotopy in any natural way. As an operation
  on braid closure, it preserves the transverse type of the link but
  not any natural Legendrian type, cf.\ Remark~\ref{rem:markov-braid}.
\end{remark}

\newpage

\section{Scannable divides
}
\label{sec:scannable-divides}

In Corollary~\ref{cor:p=>d},
we established the implication \textbf{(p)}$\Rightarrow$\textbf{(d)}
of Conjecture~\ref{conj:morsif=mut-via-divides}. 
Our strategy for tackling the converse implication
 \textbf{(d)}$\Rightarrow$\textbf{(p)} is closely aligned with 
the approach used by O.~Couture and B.~Perron~\cite{couture-perron, couture}.
The~main idea is to transform a given divide,
via local moves preserving the associated link, into a ``scannable'' divide  
(in their terminology, an ``ordered Morse divide"). 
The links of divides in this class are then described by a simple combinatorial rule. 

In this section, we review the basic results of~\cite{couture-perron} needed for
our purposes. 

\begin{definition}
\label{def:scannable}
A \emph{scannable divide} is a divide~$D$ drawn 
inside a rectangle of the form $[a_0,a]\times[b_0,b]\subset\RR^2$ so that the following conditions hold,
for some $a_0<a_1<a_2<a$.  
For every point $(x_0,y_0)$ on~$D$ 
such that the tangent line to a local
branch of~$D$ at $(x_0,y_0)$ 
is vertical (i.e., is given by the equation $x=x_0$), we require that 
\begin{itemize}[leftmargin=.3in]
\item
$(x_0,y_0)$ is a smooth point of~$D$ (i.e., not a node); 
\item
either $x_0=a_1$ or $x_0=a_2$; 
\item
if $x_0=a_1$, then the local branch of~$D$ lies
to the right of the tangent; 
\item
if $x_0=a_2$, then the local branch of~$D$ lies to the
left of the tangent. 
\end{itemize}
In other words, we can parametrize each branch of~$D$ so that, as we
move along it, the $x$-coordinate makes all of its U-turns at locations 
of the form $(a_1,y)$ (approaching them from the right) 
or $(a_2,y)$ (approaching from the left). 

Every vertical line $x=x_0$ with $a_1<x_0<a_2$ intersects a scannable
divide as above in the same number of points, the number of
\emph{strands} in the divide. 
We say that a scannable divide is of \emph{minimal index} if this
number is equal to the braid index of the corresponding link. 
(Recall that for algebraic divides, the braid index is equal
to the multiplicity of the corresponding singularity, 
see Remark~\ref{rem:libgober-williams}.) 

When considering a scannable divide~$D$, we ordinarily fix a particular representation 
of~$D$ of the form described above,
up to isotopy within the class of divides with the given number of strands
inside a fixed ambient rectangle. 
Such a representation is by no means unique, cf.\ Remark~\ref{rem:isotopic-scannable} below. 
\end{definition}

When drawing a scannable divide, we only show its part lying within the
rectangle $[a_1,a_2]\times [b_0,b]$, the rest of it being redundant. 

To illustrate, all divides shown in Figure~\ref{fig:divides-quasihom} are
manifestly scannable, with the exception of the divide at the bottom
of the last column, which is not scannable.  

\begin{remark}
\label{rem:isotopic-scannable}
Isotopic scannable divides may have rather different 
combinatorial types, and may even have a
different number of strands. 
See Figure~\ref{fig:isotopic-scanning}. 
\end{remark}

\begin{figure}[ht] 
\begin{center} 
\setlength{\unitlength}{1.5pt} 
\ \begin{picture}(40,20)(0,-10)
\put(-5,-10){\line(-1,0){5}} 
\put(25,-10){\line(1,0){5}} 
\put(0,5){\makebox(0,0){\usebox{\opening}}} 
\put(30,5){\makebox(0,0){\usebox{\closing}}} 
\put(0,0){\makebox(0,0){\usebox{\ssone}}} 
\put(10,0){\makebox(0,0){\usebox{\sstwo}}} 
\put(20,0){\makebox(0,0){\usebox{\ssone}}} 
\end{picture} 
\qquad
\setlength{\unitlength}{1.5pt} 
\begin{picture}(40,20)(0,-10)
\put(-5,10){\line(-1,0){5}} 
\put(25,10){\line(1,0){5}} 
\put(0,-5){\makebox(0,0){\usebox{\opening}}} 
\put(30,-5){\makebox(0,0){\usebox{\closing}}} 
\put(0,0){\makebox(0,0){\usebox{\sstwo}}} 
\put(10,0){\makebox(0,0){\usebox{\sstwo}}} 
\put(20,0){\makebox(0,0){\usebox{\sstwo}}} 
\end{picture} 
\qquad
\setlength{\unitlength}{1.5pt} 
\begin{picture}(40,20)(0,-10)
\put(-5,-10){\line(-1,0){5}} 
\put(25,10){\line(1,0){5}} 
\put(0,5){\makebox(0,0){\usebox{\opening}}} 
\put(30,-5){\makebox(0,0){\usebox{\closing}}} 
\put(0,0){\makebox(0,0){\usebox{\ssone}}} 
\put(10,0){\makebox(0,0){\usebox{\sstwo}}} 
\put(20,0){\makebox(0,0){\usebox{\sstwo}}} 
\end{picture} 
\qquad
\begin{picture}(30,30)(-5,-15)
\put(-5,-15){\line(-1,0){5}} 
\put(-5,15){\line(-1,0){5}} 
\put(0,0){\makebox(0,0){\usebox{\opening}}} 
\put(0,0){\makebox(0,0){\usebox{\sssonethree}}} 
\put(10,0){\makebox(0,0){\usebox{\ssstwo}}} 
\put(20,-10){\makebox(0,0){\usebox{\closing}}} 
\put(20,10){\makebox(0,0){\usebox{\closing}}} 
\end{picture}
\qquad
\begin{picture}(30,30)(-10,-15)
\put(-5,5){\line(-1,0){5}} 
\put(-5,15){\line(-1,0){5}} 
\put(0,-10){\makebox(0,0){\usebox{\opening}}} 
\put(0,0){\makebox(0,0){\usebox{\ssstwo}}} 
\put(10,0){\makebox(0,0){\usebox{\sssthree}}} 
\put(20,0){\makebox(0,0){\usebox{\ssstwo}}} 
\put(30,-10){\makebox(0,0){\usebox{\closing}}} 
\put(30,10){\makebox(0,0){\usebox{\closing}}} 
\end{picture}
\vspace{-.1in}
\end{center} 
\caption{Isotopic scannable divides. 
} 
\label{fig:isotopic-scanning} 
\end{figure} 


\pagebreak[3]

\begin{definition}
\label{def:braid-of-divide}
Let $D$ be a scannable divide with $k$ strands. 
The \emph{braid $\beta(D)$ asso\-ciated with~$D$}
is the (positive) braid in the $k$-strand braid group defined as follows. 
Let $\sigma_i$ denote the positive Artin generator that switches 
the $i$th and $(i\!+\!1)$st strands. 
(The strands are numbered bottom up.) 
Let $\texttt{left}(D)$ (respectively~$\texttt{right}(D)$)
be the product of the (commuting) generators $\sigma_i$
corresponding to the pairs of adjacent strands $(i,i\!+\!1)$ of the divide~$D$
which connect to each other at its left 
(respectively right) end, at the points with a vertical tangent. 
Let $\texttt{bulk}(D)$ be the product of Artin generators~$\sigma_i$, 
multiplied left to right, 
corresponding to the pairs of adjacent strands of the divide which get switched 
as we scan it left to right. 
Let~$\texttt{klub}(D)$ denote the product of the same generators,
multiplied right to left. 
We then set 
\[
\beta(D)=\texttt{left}(D)\texttt{bulk}(D)\texttt{right}(D)\texttt{klub}(D).
\] 
\end{definition}

See Figure~\ref{fig:left-bulk-right-klub} for an example. 
Additional examples appear further in the text. 

\begin{figure}[ht]
\begin{center}
\begin{picture}(60,20)(-10,-8)
\put(-5,-10){\line(-1,0){5}} 
\put(40,10){\line(-1,0){5}} 
\put(0,5){\makebox(0,0){\usebox{\opening}}} 
\put(0,0){\makebox(0,0){\usebox{\ssone}}} 
\put(10,0){\makebox(0,0){\usebox{\ssone}}} 
\put(20,0){\makebox(0,0){\usebox{\sstwo}}} 
\put(30,0){\makebox(0,0){\usebox{\ssone}}} 
\put(40,-5){\makebox(0,0){\usebox{\closing}}} 
\end{picture} 
\end{center}
\caption{For this scannable divide~$D$, we have 
$\texttt{left}(D)=\sigma_2$,
$\texttt{right}(D)=\sigma_1$, 
$\texttt{bulk}(D)=\sigma_1\sigma_1\sigma_2\sigma_1$, 
$\texttt{klub}(D)=\sigma_1\sigma_2\sigma_1\sigma_1$, 
and $\beta(D)=\sigma_2\sigma_1\sigma_1\sigma_2\sigma_1
\sigma_1\sigma_1\sigma_2\sigma_1\sigma_1$.
}
\vspace{-.1in}
\label{fig:left-bulk-right-klub}
\end{figure} 

We can now state the key result by Couture and Perron. 

\begin{theorem}[{\rm $\!\!$ \cite[Proposition~2.3]{couture-perron}}]
\label{th:palindromic}
The link~$L(D)$ of a scannable divide~$D$ is isotopic to 
the closure of the positive braid $\beta(D)$ given by Definition~\ref{def:braid-of-divide}. 
\end{theorem}

Theorem~\ref{th:palindromic} does not require the divide~$D$ to be algebraic. 

\begin{example}
\label{example:D1D2}
Let $D_1$ and $D_2$ be the scannable divides shown in
Figure~\ref{fig:divides-two-transversal cusps}.
Denote
\[
\Delta
=\sigma_1 \sigma_3 \sigma_2 \sigma_1 \sigma_3 \sigma_2
=\sigma_2 \sigma_1 \sigma_3 \sigma_2 \sigma_1 \sigma_3
=\sigma_1 \sigma_2 \sigma_1 \sigma_3 \sigma_2 \sigma_1\,. 
\]
Then
\begin{align*}
\beta(D_1)&=\sigma_1 \sigma_3 \sigma_2 \sigma_1 \sigma_3 \sigma_2 \sigma_1 \sigma_3 
\sigma_1 \sigma_3 \sigma_2 \sigma_1 \sigma_3 \sigma_2 
=\Delta \sigma_1 \sigma_3 \Delta =\Delta^2 \sigma_1 \sigma_3\,, 
\\
\beta(D_2)&=\sigma_1 \sigma_3 \sigma_2 \sigma_1 \sigma_3 \sigma_1 \sigma_2 
\sigma_1 \sigma_3 \sigma_2 \sigma_1 \sigma_3 \sigma_1 \sigma_2 
=\sigma_2^{-1} \Delta^2  \sigma_1 \sigma_3 \sigma_2 \,,
\end{align*}
which shows that the braids $\beta(D_1)$ and $\beta(D_2)$ are conjugate to each other. 
This was to be expected, since the divides $D_1$ and $D_2$ come from morsifications
of two different real forms of the same complex singularity. 
\end{example}

The description of the link $L(D)$ associated with a scannable divide~$D$
given in Theorem~\ref{th:palindromic} and Definition~\ref{def:braid-of-divide}
can be recast in the language of conventional link diagrams, as follows. 
Replace each strand of $D$ by a pair of parallel strands. 
Transform each crossing in~$D$, and each coupling of its strands 
that occurs at either of the two ends of~$D$, 
using the recipe shown in Figure~\ref{fig:double-crossing}. 
Finally, cap each of the remaining ``loose ends'' of~$D$ by connecting 
the corresponding two strands of the link to each other. 
The resulting (oriented) link is isotopic to~$L(D)$. 
(This claim is merely a restatement of Theorem~\ref{th:palindromic}.)
An example is shown in Figure~\ref{fig:link-scannable}. 

We note that our convention for drawing braids and links
is different from~\cite{couture-perron}, 
as we are using right-handed twists as positive Artin generators. 

\newsavebox{\pcrossing}
\setlength{\unitlength}{6pt} 
\savebox{\pcrossing}(10,10)[bl]{
\thicklines 


\put(0,0){\line(1,0){1.5}} 
\put(0,8){\line(1,0){0.5}} 
\put(8.5,0){\line(1,0){1.5}} 
\put(9.5,8){\line(1,0){0.5}} 

\put(1.5,0){\line(1,1){3}} 
\put(0.5,8){\line(1,-1){8}} 
\put(5.5,4){\line(1,1){4}} 

\put(1.5,0){\vector(1,1){2}} 
\put(1.5,7){\vector(1,-1){1}} 

\put(0,2){\red{\line(1,0){0.5}}} 
\put(0,10){\red{\line(1,0){1.5}}} 
\put(9.5,2){\red{\line(1,0){0.5}}} 
\put(8.5,10){\red{\line(1,0){1.5}}} 

\put(0.5,2){\red{\line(1,1){2.5}}} 
\put(1.5,10){\red{\line(1,-1){4.5}}} 
\put(7,4.5){\red{\line(1,-1){2.5}}} 
\put(4,5.5){\red{\line(1,1){0.5}}} 
\put(5.5,7){\red{\line(1,1){3}}} 

\put(8.5,10){\red{\vector(-1,-1){2}}} 
\put(9.5,2){\red{\vector(-1,1){2}}} 
}

\newsavebox{\pclosing}
\setlength{\unitlength}{6pt} 
\savebox{\pclosing}(10,10)[bl]{
\thicklines 
\put(0,0){\line(1,0){1.5}} 
\put(0,8){\line(1,0){0.5}} 
\put(8.5,0){\line(1,0){1.5}} 
\put(9.5,8){\line(1,0){0.5}} 

\put(1.5,0){\line(1,1){3}} 
\put(0.5,8){\line(1,-1){8}} 
\put(5.5,4){\line(1,1){4}} 

\put(1.5,0){\vector(1,1){2}} 
\put(1.5,7){\vector(1,-1){1}} 

\put(0,10){\red{\line(1,0){10}}} 
\put(0,2){\red{\line(1,0){2.5}}} 
\put(10,2){\red{\line(-1,0){2.5}}} 
\put(4.5,2){\red{\line(1,0){1}}} 
\put(7,10){\red{\vector(-1,0){2}}} 
\put(10,2){\red{\vector(-1,0){2}}} 
\qbezier(12,9)(12,10)(10,10)
\qbezier(12,9)(12,8)(10,8)
\qbezier(12,1)(12,2)(10,2)
\qbezier(12,1)(12,0)(10,0)
}

\newsavebox{\popening}
\setlength{\unitlength}{6pt} 
\savebox{\popening}(10,10)[bl]{
\thicklines 
\put(0,0){\line(1,0){1.5}} 
\put(0,8){\line(1,0){0.5}} 
\put(8.5,0){\line(1,0){1.5}} 
\put(9.5,8){\line(1,0){0.5}} 

\put(1.5,0){\line(1,1){3}} 
\put(0.5,8){\line(1,-1){8}} 
\put(5.5,4){\line(1,1){4}} 

\put(1.5,0){\vector(1,1){2}} 
\put(1.5,7){\vector(1,-1){1}} 

\put(0,10){\red{\line(1,0){10}}} 
\put(0,2){\red{\line(1,0){2.5}}} 
\put(10,2){\red{\line(-1,0){2.5}}} 
\put(4.5,2){\red{\line(1,0){1}}} 

\put(7,10){\red{\vector(-1,0){2}}} 
\put(10,2){\red{\vector(-1,0){2}}} 

\qbezier(-2,9)(-2,10)(0,10)
\qbezier(-2,9)(-2,8)(0,8)
\qbezier(-2,1)(-2,2)(0,2)
\qbezier(-2,1)(-2,0)(0,0)
}

\newsavebox{\ptrack}
\setlength{\unitlength}{6pt} 
\savebox{\ptrack}(10,10)[bl]{
\thicklines 
\put(0,0){\line(1,0){10}} 
\put(0,0){\vector(1,0){5}} 
\put(0,2){\red{\line(1,0){10}}} 
\put(10,2){\red{\vector(-1,0){5}}} 
}

\newsavebox{\uturnright}
\setlength{\unitlength}{6pt} 
\savebox{\uturnright}(10,10)[bl]{
\thicklines 
\qbezier(2,1)(2,2)(0,2)
\qbezier(2,1)(2,0)(0,0)
}

\newsavebox{\uturnleft}
\setlength{\unitlength}{6pt} 
\savebox{\uturnleft}(10,10)[bl]{
\thicklines 
\qbezier(-2,1)(-2,2)(0,2)
\qbezier(-2,1)(-2,0)(0,0)
}

\begin{figure}[ht] 
\begin{center} 
\begin{tabular}{c||c|c|c}
\setlength{\unitlength}{3pt} 
\begin{picture}(10,10)(0,-2)
\put(5,5){\makebox(0,0){divide}}
\end{picture}
& \setlength{\unitlength}{3pt} 
\begin{picture}(10,15)(0,-2)
\thicklines 
\qbezier(5,5)(7,10)(10,10)
\qbezier(5,5)(3,0)(0,0)
\qbezier(5,5)(3,10)(0,10)
\qbezier(5,5)(7,0)(10,0)
\end{picture}  
&
\setlength{\unitlength}{3pt} 
\begin{picture}(5,10)(0,-2)
\thicklines 
\qbezier(0,0)(5,0)(5,5)
\qbezier(0,10)(5,10)(5,5)
\end{picture} 
&
\setlength{\unitlength}{3pt} 
\begin{picture}(5,10)(-5,-2)
\thicklines 
\qbezier(0,0)(-5,0)(-5,5)
\qbezier(0,10)(-5,10)(-5,5)
\end{picture} 
\\
\hline
&&&\\
\setlength{\unitlength}{6pt} 
\begin{picture}(6,10)(0,0)
\put(3,5){\makebox(0,0){link}} 
\end{picture} 
&
{\ }
\setlength{\unitlength}{6pt} 
\begin{picture}(10,10)(0,0)
\put(5,5){\makebox(0,0){\usebox{\pcrossing}}} 
\end{picture} 
{\ \ }
&
{\ }
\setlength{\unitlength}{6pt} 
\begin{picture}(12,10)(0,0)
\put(5,5){\makebox(0,0){\usebox{\pclosing}}} 
\end{picture} 
{\ \ }
&
{\ }
\setlength{\unitlength}{6pt} 
\begin{picture}(12,10)(-2,0)
\put(5,5){\makebox(0,0){\usebox{\popening}}} 
\end{picture} 
{\ \ }
\end{tabular}
\end{center} 
\caption{Transforming a scannable divide into a link diagram. 
} 
\label{fig:double-crossing} 
\end{figure} 

\newsavebox{\artinone}
\setlength{\unitlength}{1.5pt} 
\savebox{\artinone}(10,20)[bl]{
\thicklines 
\qbezier(6,7)(7,10)(10,10)
\qbezier(4,3)(3,0)(0,0)
\qbezier(5,5)(3,10)(0,10)
\qbezier(5,5)(7,0)(10,0)
\put(0,20){\line(1,0){10}} 
}

\newsavebox{\artintwo}
\setlength{\unitlength}{1.5pt} 
\savebox{\artintwo}(10,20)[bl]{
\thicklines 
\qbezier(6,17)(7,20)(10,20)
\qbezier(4,13)(3,10)(0,10)
\qbezier(5,15)(3,20)(0,20)
\qbezier(5,15)(7,10)(10,10)
\put(0,0){\line(1,0){10}} 
}

\begin{figure}[ht] 
\begin{center} 
\begin{tabular}{c||c}
\setlength{\unitlength}{1.5pt} 
\begin{picture}(16,27)(2,-2)
\put(8,14){\makebox(0,0){divide}} 
\put(8,6){\makebox(0,0){$D$}} 
\end{picture}
& 
\setlength{\unitlength}{1.5pt} 
\begin{picture}(40,25)(-7,-13)
\put(-5,-10){\line(-1,0){5}} 
\put(25,-10){\line(1,0){5}} 
\put(0,5){\makebox(0,0){\usebox{\opening}}} 
\put(30,5){\makebox(0,0){\usebox{\closing}}} 
\put(0,0){\makebox(0,0){\usebox{\ssone}}} 
\put(10,0){\makebox(0,0){\usebox{\sstwo}}} 
\put(20,0){\makebox(0,0){\usebox{\ssone}}} 
\end{picture}  \\
\hline
\\
\setlength{\unitlength}{6pt} 
\begin{picture}(5,20)(0,0)
\put(2,11.2){\makebox(0,0){link}} 
\put(2,8.8){\makebox(0,0){$L(D)$}} 
\end{picture}
&
\setlength{\unitlength}{6pt} 
\begin{picture}(55,20)(-8,-6)
\put(10,0){\makebox(0,0){\usebox{\pcrossing}}} 
\put(20,8){\makebox(0,0){\usebox{\pcrossing}}} 
\put(30,0){\makebox(0,0){\usebox{\pcrossing}}} 
\put(40,8){\makebox(0,0){\usebox{\pclosing}}} 
\put(0,8){\makebox(0,0){\usebox{\popening}}} 

\put(0,0){\makebox(0,0){\usebox{\ptrack}}} 
\put(20,0){\makebox(0,0){\usebox{\ptrack}}} 
\put(40,0){\makebox(0,0){\usebox{\ptrack}}} 
\put(10,16){\makebox(0,0){\usebox{\ptrack}}} 
\put(30,16){\makebox(0,0){\usebox{\ptrack}}} 
\put(50,0){\makebox(0,0){\usebox{\uturnright}}} 
\put(0,0){\makebox(0,0){\usebox{\uturnleft}}} 
\end{picture} 
\\[.1in]
\hline
\setlength{\unitlength}{1.5pt} 
\begin{picture}(16,30)(2,0)
\put(8,10){\makebox(0,0){braid}} 
\put(8,1){\makebox(0,0){$\beta(D)$}} 
\end{picture}
& 
\setlength{\unitlength}{1.5pt} 
\begin{picture}(85,30)(-7,-10)
\put(0,0){\makebox(0,0){\usebox{\artintwo}}} 
\put(10,0){\makebox(0,0){\usebox{\artinone}}} 
\put(20,0){\makebox(0,0){\usebox{\artintwo}}} 
\put(30,0){\makebox(0,0){\usebox{\artinone}}} 
\put(40,0){\makebox(0,0){\usebox{\artintwo}}} 
\put(50,0){\makebox(0,0){\usebox{\artinone}}} 
\put(60,0){\makebox(0,0){\usebox{\artintwo}}} 
\put(70,0){\makebox(0,0){\usebox{\artinone}}} 
\end{picture} 
\\[.1in]
& $\sigma_2\,\sigma_1\,\sigma_2\,\sigma_1\,\sigma_2\,\sigma_1\,\sigma_2\,\sigma_1$
\\[.1in]
\hline
\setlength{\unitlength}{1.5pt} 
\begin{picture}(16,50)(2,0)
\put(8,24){\makebox(0,0){closed}} 
\put(8,16){\makebox(0,0){braid}} 
\end{picture}
& 
\setlength{\unitlength}{1.5pt} 
\begin{picture}(85,50)(-7,-10)
\put(0,0){\makebox(0,0){\usebox{\artintwo}}} 
\put(10,0){\makebox(0,0){\usebox{\artinone}}} 
\put(20,0){\makebox(0,0){\usebox{\artintwo}}} 
\put(30,0){\makebox(0,0){\usebox{\artinone}}} 
\put(40,0){\makebox(0,0){\usebox{\artintwo}}} 
\put(50,0){\makebox(0,0){\usebox{\artinone}}} 
\put(60,0){\makebox(0,0){\usebox{\artintwo}}} 
\put(70,0){\makebox(0,0){\usebox{\artinone}}} 
\put(-5,20){\line(1,0){80}} 
\put(-5,25){\line(1,0){80}} 
\put(-5,30){\line(1,0){80}} 
\qbezier(-5,10)(-10,10)(-10,15)
\qbezier(-5,20)(-10,20)(-10,15)
\qbezier(-5,0)(-15,0)(-15,12.5)
\qbezier(-5,25)(-15,25)(-15,12.5)
\qbezier(-5,-10)(-20,-10)(-20,15)
\qbezier(-5,30)(-20,30)(-20,15)

\qbezier(75,10)(80,10)(80,15)
\qbezier(75,20)(80,20)(80,15)
\qbezier(75,0)(85,0)(85,12.5)
\qbezier(75,25)(85,25)(85,12.5)
\qbezier(75,-10)(90,-10)(90,15)
\qbezier(75,30)(90,30)(90,15)
\end{picture} 
\end{tabular}
\end{center} 
\caption{Scannable divide $D$ coming from a 
morsification of a singularity of type~$E_6$;
the associated link $L(D)$ (a $(3,4)$-torus knot);  
and the corresponding positive braid~$\beta(D)$.
The closure of the braid recovers the link~$L(D)$. 
} 
\vspace{-.15in}
\label{fig:link-scannable} 
\end{figure} 

\clearpage

\newsavebox{\artin}
\setlength{\unitlength}{1.5pt} 
\savebox{\artin}(10,10)[bl]{
\thicklines 
\qbezier(6,7)(7,10)(10,10)
\qbezier(4,3)(3,0)(0,0)

\qbezier(5,5)(3,10)(0,10)
\qbezier(5,5)(7,0)(10,0)
}

\newpage

\section{Plabic fences}
\label{sec:plabic-fences}

For a scannable divide~$D$, there is a natural 
choice of a plabic graph attached to~$D$ which we call a 
``plabic fence''  
(borrowing the term \emph{fence} from  L.~Rudolph~\cite{rudolph-annuli}).  
The combinatorics of plabic fences is closely connected to 
positive braids. 

A quick note on pictorial conventions: from now on, to simplify the drawing~process,
we will not always picture the white vertices of a plabic graph as hollow circles. 
We will continue to depict black vertices as filled circles; all the remaining points 
in a drawing where three lines (representing edges of the graph) come together,
as well as all the remaining endpoints, will be understood to represent 
the white vertices. 

\begin{definition}
\label{def:plabic-fence}
Fix an integer~$k\ge 2$. 
Let $\ww$ be an arbitrary word in the alphabet 
\begin{equation}
\label{eq:double-alphabet}
\mathbf{A}_k=\{\sigma_1,\dots,\sigma_{k-1}\}\cup\{\tau_1,\dots,\tau_{k-1}\}.  
\end{equation}
The plabic graph $\Phi=\Phi(\ww)$, 
called the \emph{plabic fence} associated with~$\ww$, is constructed 
as follows.
Begin by stacking $k$ parallel horizontal line segments (``strands'') 
on top of each other, and numbering them $1,\dots,k$, bottom to top. 
Reading the entries of~$\ww$ left to right, 
place vertical connectors between pairs of adjacent strands of the fence, 
representing each entry~$\sigma_i$ as
\setlength{\unitlength}{1pt} 
\begin{picture}(10,15)(0,0)
\thicklines 
\put(0,0){\line(1,0){10}} 
\put(0,10){\line(1,0){10}} 
\put(5,0){\line(0,1){10}} 
\put(5,0){\circle*{3}}
\end{picture}~,
and each entry~$\tau_i$ as
\setlength{\unitlength}{1pt} 
\begin{picture}(10,15)(0,0)
\thicklines 
\put(0,0){\line(1,0){10}} 
\put(0,10){\line(1,0){10}} 
\put(5,0){\line(0,1){10}} 
\put(5,10){\circle*{3}}
\end{picture}~, each time connecting strands numbered $i$ and~$i+1$. 
See Figure~\ref{fig:fence-word}. 

Conversely, any plabic fence~$\Phi$ as above determines 
the associated word $\ww=\ww(\Phi)$
in the alphabet~$\mathbf{A}_k$. 
To be a bit more precise, since $\Phi$ is defined up to isotopy,
the corresponding word $\ww(\Phi)$ is defined up to transpositions of the form
$\sigma_i \sigma_j\leftrightarrow\sigma_j\sigma_i$, 
or $\sigma_i \tau_j\leftrightarrow\tau_j\sigma_i$, 
or $\tau_i \tau_j\leftrightarrow\tau_j\tau_i$, 
for $|i-j|\ge2$. 
\end{definition}

\begin{figure}[ht] 
\begin{center} 
\begin{tabular}{c||c}
word 
&
$\ \ \ \sigma_2\,\,\sigma_1\,\,\tau_1\,\,\sigma_2\,\,\tau_2\,\,\sigma_1\,\,\tau_1\,\,\sigma_2$
\\[.05in]
\hline
\\[-.1in]
\setlength{\unitlength}{3pt} 
\begin{picture}(20,12.5)(0,-1)
\put(10,6){\makebox(0,0){plabic fence}} 
\end{picture}
&
\setlength{\unitlength}{1.5pt} 
\begin{picture}(100,30)(-10,-5)
\put(-10,0){\makebox(0,0){\textbf{\Small 1}}} 
\put(-10,10){\makebox(0,0){\textbf{\Small 2}}} 
\put(-10,20){\makebox(0,0){\textbf{\Small 3}}} 
\thinlines
\put(-5,-5){\line(0,1){30}} 
\put(-6,-5){\line(0,1){30}} 
\put(85,-5){\line(0,1){30}} 
\put(86,-5){\line(0,1){30}} 
\multiput(85.5,0)(0,10){3}{\circle*{2.5}}
\thicklines 
\put(-5,0){\line(1,0){90}} 
\put(-5,10){\line(1,0){90}} 
\put(-5,20){\line(1,0){90}} 
\put(5,10){\line(0,1){10}} 
\put(15,0){\line(0,1){10}} 
\put(25,0){\line(0,1){10}} 
\put(35,10){\line(0,1){10}} 
\put(45,10){\line(0,1){10}} 
\put(55,0){\line(0,1){10}} 
\put(65,0){\line(0,1){10}} 
\put(75,10){\line(0,1){10}} 

\put(5,10){\circle*{2.5}}
\put(15,0){\circle*{2.5}}
\put(25,10){\circle*{2.5}}
\put(35,10){\circle*{2.5}}
\put(45,20){\circle*{2.5}}
\put(55,0){\circle*{2.5}}
\put(65,10){\circle*{2.5}}
\put(75,10){\circle*{2.5}}
\end{picture}  
\\
\hline
\\[-.15in]
braid
&
$\ \ \ \sigma_2\,\,\sigma_1\,\,\sigma_2\,\,\sigma_1\,\,\sigma_2\,\,\sigma_1\,\,\sigma_2\,\,\sigma_1$
\end{tabular}
\vspace{-.05in}
\end{center} 
\caption{A word~$\ww$, the associated plabic fence~$\Phi$, 
and the braid $\beta(\ww)=\beta(\Phi)$. 
} 
\label{fig:fence-word} 
\end{figure} 

\vspace{-.1in}


\begin{definition}
\label{def:Phi(D)}
Let $D$ be a scannable divide with $k$ strands. 
We define the associated plabic fence~$\Phi(D)$ as follows: 
\begin{itemize}[leftmargin=.3in]
\item[(i)]
stack $k$ parallel horizontal line segments (strands) 
on top of each other;
\item[(ii)]
for each node in~$D$, 
connect the corresponding strands in $\Phi(D)$ by a pair of vertical edges,
and color their four endpoints as follows:
\setlength{\unitlength}{1pt} \begin{picture}(20,15)(0,0)
\thicklines 
\put(0,0){\line(1,0){20}} 
\put(0,10){\line(1,0){20}} 
\put(5,0){\line(0,1){10}} 
\put(15,0){\line(0,1){10}} 
\put(5,0){\circle*{3}}
\put(15,10){\circle*{3}}
\end{picture}\,;
\item[(iii)]
for each instance of two adjacent strands in~$D$ connecting to each other at 
an end of~$D$, 
insert a connector \setlength{\unitlength}{1pt} 
\begin{picture}(10,15)(0,0)
\thicklines 
\put(0,0){\line(1,0){10}} 
\put(0,10){\line(1,0){10}} 
\put(5,0){\line(0,1){10}} 
\put(5,0){\circle*{3}}
\end{picture}\,
between the corresponding strands in~$\Phi(D)$; 
\item[(iv)]
color the left endpoints of~$\Phi(D)$ white, and color the right endpoints black. 
\end{itemize}
To illustrate, the scannable divide in
Figure~\ref{fig:link-scannable} gives rise to the plabic fence shown in Figure~\ref{fig:fence-word}. 
\end{definition}

\begin{remark}
\label{rem:P(D)==Phi(D)}
For a scannable divide~$D$, 
the plabic fence $\Phi(D)$ is a plabic graph
attached to~$D$, in the sense of Definition~\ref{def:plabic-divide}. 
In particular, the quiver $Q(\Phi(D))$ associated with the plabic fence~$\Phi(D)$ 
coincides with the quiver~$Q(D)$ defined by~$D$. 
\end{remark}

\begin{definition}
\label{def:fence-word-to-braid}
Let $\Phi$ be a plabic fence on $k$ strands, and let 
$\ww$ be the corresponding word in the alphabet 
$\mathbf{A}_k$ 
(see~\eqref{eq:double-alphabet}). 
We define the positive braid $\beta(\Phi)=\beta(\ww)$ 
in the $k$-strand braid group~$\mathbb{B}_k$ as follows.
Let $\bww$ be the word obtained from~$\ww$ by first recording all the entries 
of the form~$\sigma_i\,$, left to right,
then all the entries 
of the form~$\tau_i\,$, right to left,
then replacing each $\tau_i$ by~$\sigma_i\,$. \linebreak[3]
The braid $\beta(\Phi)$ is then obtained from~$\bww$
by interpreting each $\sigma_i$ as 
an Artin generator of~$\mathbb{B}_k$. 
\end{definition}

To illustrate, if $\ww=\sigma_1\tau_2\sigma_3\tau_4\,$,
then $\beta(\ww)=\sigma_1\sigma_3\sigma_4\sigma_2$. 
Also, see Figure~\ref{fig:fence-word}.


The following statement is easily confirmed by direct inspection. 

\begin{proposition}
\label{pr:divide-to-fence-to-braid}
Let $D$ be a scannable divide. Then $\beta(\Phi(D))=\beta(D)$. 
That~is, the braid  constructed from the plabic fence associated with~$D$ 
(see Definitions~\ref{def:Phi(D)} and~\ref{def:fence-word-to-braid}) 
coincides with the braid $\beta(D)$ described in Definition~\ref{def:braid-of-divide}. 
\end{proposition}

\begin{lemma}
\label{lem:ww-to-bww}
Let $\ww$ be a word in the alphabet $\mathbf{A}_k$. Let 
$\bww$ be the corresponding~braid word 
from Defini\-tion~\ref{def:fence-word-to-braid}.
Then the plabic fences $\Phi(\ww)$ and $\Phi(\bww)$ are move equivalent. 
\end{lemma}

\begin{proof}
The braid word $\bww$ can be obtained from $\ww$ by repeatedly applying transformations 
 of the form $\cdots\tau_j\sigma_i\cdots\leadsto\cdots\sigma_i\tau_j\cdots$
(pushing the $\tau$'s to the right of the~$\sigma$'s)
and/or $\cdots\tau_i\leadsto\cdots\sigma_i$
(replacing $\tau_i$ by~$\sigma_i$ at the end of the word). 
Each~of these transformations can be viewed as an instance of move equivalence:
\begin{itemize}[leftmargin=.3in]
\item
a switch $\tau_i\sigma_i\leftrightarrow\sigma_i\tau_i$
corresponds to a square~move:
\setlength{\unitlength}{1.2pt} 
\begin{tabular}{ccc}
\begin{picture}(20,13)(0,0)
\thicklines
\put(0,0){\line(1,0){20}}
\put(0,10){\line(1,0){20}}
\put(5,0){\line(0,1){10}}
\put(15,0){\line(0,1){10}}
\put(5,10){\circle*{2.5}}
\put(15,0){\circle*{2.5}}
\end{picture}
&
\begin{picture}(4,13)(0,0)
\thicklines
\put(2,5){\makebox(0,0){$\leftrightarrow$}}
\end{picture}
&
\begin{picture}(20,13)(0,0)
\thicklines
\put(0,0){\line(1,0){20}}
\put(0,10){\line(1,0){20}}
\put(5,0){\line(0,1){10}}
\put(15,0){\line(0,1){10}}
\put(15,10){\circle*{2.5}}
\put(5,0){\circle*{2.5}}
\end{picture}
\end{tabular}
;
\item
a~switch $\tau_{i\pm1}\sigma_i\leftrightarrow\sigma_i\tau_{i\pm1}$
corresponds to a flip move:
\begin{center}
\begin{tabular}{ccc}
\begin{picture}(20,23)(0,0)
\thicklines
\put(0,0){\line(1,0){20}}
\put(0,10){\line(1,0){20}}
\put(0,20){\line(1,0){20}}
\put(5,10){\line(0,1){10}}
\put(15,0){\line(0,1){10}}
\put(5,20){\circle*{2.5}}
\put(15,0){\circle*{2.5}}
\end{picture}
&
\begin{picture}(4,23)(0,0)
\thicklines
\put(2,10){\makebox(0,0){$\leftrightarrow$}}
\end{picture}
&
\begin{picture}(20,23)(0,0)
\thicklines
\put(0,0){\line(1,0){20}}
\put(0,10){\line(1,0){20}}
\put(0,20){\line(1,0){20}}
\put(5,0){\line(0,1){10}}
\put(15,10){\line(0,1){10}}
\put(15,20){\circle*{2.5}}
\put(5,0){\circle*{2.5}}
\end{picture}
\end{tabular}
\qquad or \qquad
\begin{tabular}{ccc}
\begin{picture}(20,23)(0,0)
\thicklines
\put(0,0){\line(1,0){20}}
\put(0,10){\line(1,0){20}}
\put(0,20){\line(1,0){20}}
\put(5,0){\line(0,1){10}}
\put(15,10){\line(0,1){10}}
\put(5,10){\circle*{2.5}}
\put(15,10){\circle*{2.5}}
\end{picture}
&
\begin{picture}(4,23)(0,0)
\thicklines
\put(2,10){\makebox(0,0){$\leftrightarrow$}}
\end{picture}
&
\begin{picture}(20,23)(0,0)
\thicklines
\put(0,0){\line(1,0){20}}
\put(0,10){\line(1,0){20}}
\put(0,20){\line(1,0){20}}
\put(5,10){\line(0,1){10}}
\put(15,0){\line(0,1){10}}
\put(15,10){\circle*{2.5}}
\put(5,10){\circle*{2.5}}
\end{picture}
\end{tabular}
\end{center} 
\item
a switch $\tau_j\sigma_i\!\leadsto\!\sigma_i\tau_j$ for $|i-j|\ge2$ 
translates into an isotopy of the plabic~fence;
\item
replacing $\tau_i$ by $\sigma_i$ at the end of a word
is emulated by a tail removal followed by a tail attachment:
\begin{equation}
\begin{tabular}{ccccc}
\begin{picture}(13,13)(1,-8)
\thicklines
\put(0,0){\line(1,0){10}}
\put(0,10){\line(1,0){10}}
\put(5,0){\line(0,1){10}}
\put(5,10){\circle*{2.5}}
\thinlines
\put(10,-3){\line(0,1){16}}
\put(11,-3){\line(0,1){16}}
\put(10.5,0){\circle*{2.5}}
\put(10.5,10){\circle*{2.5}}
\end{picture}
&
\begin{picture}(6,13)(0,-8)
\thicklines
\put(2,5){\makebox(0,0){$\leftrightarrow$}}
\end{picture}
&
\begin{picture}(13,13)(1,-8)
\thicklines
\qbezier(0,0)(5,0)(5,5)
\qbezier(0,10)(5,10)(5,5)
\put(5,5){\line(1,0){5}}
\put(5,5){\circle*{2.5}}
\thinlines
\put(10,-3){\line(0,1){16}}
\put(11,-3){\line(0,1){16}}
\put(10.5,5){\circle*{2.5}}

\end{picture}
&
\begin{picture}(6,13)(0,-8)
\thicklines
\put(2,5){\makebox(0,0){$\leftrightarrow$}}
\end{picture}
&
\begin{picture}(13,20)(1,-8)
\thicklines
\put(0,0){\line(1,0){10}}
\put(0,10){\line(1,0){10}}
\put(5,0){\line(0,1){10}}
\put(5,0){\circle*{2.5}}
\thinlines
\put(10,-3){\line(0,1){16}}
\put(10,-3){\line(0,1){16}}
\put(11,-3){\line(0,1){16}}
\put(10.5,0){\circle*{2.5}}
\put(10.5,10){\circle*{2.5}}
\end{picture}
\end{tabular}
\qedhere
\label{eq:replace-tau-by-sigma}
\end{equation}
\end{itemize}
\end{proof}

\vspace{-.1in}

\begin{proposition}
\label{pr:equal-braids=>move-equivalent}
Let $\Phi_1$ and $\Phi_2$ be plabic fences on $k$ strands. 
If the braids 
$\beta(\Phi_1),\beta(\Phi_2)\in\mathbb{B}_k$
are equal to each other, 
then $\Phi_1$ and~$\Phi_2$ are move equivalent. 
\end{proposition}

\begin{proof}
By Lemma~\ref{lem:ww-to-bww}, it is enough to treat the case of plabic fences
whose associated words only involve the $\sigma$'s but not the~$\tau$'s. 
In this case, we need to check that each relation in Artin's presentation of~$\mathbb{B}_k$ 
translates into an instance~of move equivalence for the corresponding plabic fences. 
Indeed, the switches $\sigma_j\sigma_i\leftrightarrow\sigma_i\sigma_j$ for $|i-j|\ge2$ 
translate into isotopies of the plabic graph, whereas  
the braid relations $\sigma_i \sigma_{i+1}\sigma_i
\leftrightarrow\sigma_{i+1}\sigma_i\sigma_{i+1}$
are emulated by flip and square moves, as follows:
\[
\setlength{\unitlength}{0.7pt}
\begin{tabular}{ccccccc}
\begin{picture}(60,40)(0,0)
\thicklines
\put(0,40){\line(1,0){60}}
\put(0,20){\line(1,0){60}}
\put(0,0){\line(1,0){60}}
\put(10,0){\line(0,1){20}}
\put(30,20){\line(0,1){20}}
\put(50,0){\line(0,1){20}}
\put(10,0){\circle*{4}}
\put(30,20){\circle*{4}}
\put(50,0){\circle*{4}}
\end{picture}
&
\begin{picture}(16,40)(0,0)
\thicklines
\put(8,20){\makebox(0,0){$\longleftrightarrow$}}
\end{picture}
&
\begin{picture}(60,40)(0,0)
\thicklines
\put(0,40){\line(1,0){60}}
\put(0,20){\line(1,0){20}}
\put(40,20){\line(1,0){20}}
\put(0,0){\line(1,0){60}}
\put(30,0){\line(0,1){10}}
\put(30,30){\line(0,1){10}}
\put(30,30){\line(1,-1){10}}
\put(30,30){\line(-1,-1){10}}
\put(30,10){\line(1,1){10}}
\put(30,10){\line(-1,1){10}}
\put(30,0){\circle*{4}}
\put(30,30){\circle*{4}}
\put(30,10){\circle*{4}}
\end{picture}
&
\begin{picture}(16,40)(0,0)
\thicklines
\put(8,20){\makebox(0,0){$\longleftrightarrow$}}
\end{picture}
&
\begin{picture}(60,40)(0,0)
\thicklines
\put(0,40){\line(1,0){60}}
\put(0,20){\line(1,0){20}}
\put(40,20){\line(1,0){20}}
\put(0,0){\line(1,0){60}}
\put(30,0){\line(0,1){10}}
\put(30,30){\line(0,1){10}}
\put(30,30){\line(1,-1){10}}
\put(30,30){\line(-1,-1){10}}
\put(30,10){\line(1,1){10}}
\put(30,10){\line(-1,1){10}}
\put(30,0){\circle*{4}}
\put(20,20){\circle*{4}}
\put(40,20){\circle*{4}}
\end{picture}
&
\begin{picture}(16,40)(0,0)
\thicklines
\put(8,20){\makebox(0,0){$\longleftrightarrow$}}
\end{picture}
&
\begin{picture}(60,45)(0,0)
\thicklines
\put(0,40){\line(1,0){60}}
\put(0,20){\line(1,0){60}}
\put(0,0){\line(1,0){60}}
\put(10,20){\line(0,1){20}}
\put(30,0){\line(0,1){20}}
\put(50,20){\line(0,1){20}}
\put(10,20){\circle*{4}}
\put(30,0){\circle*{4}}
\put(50,20){\circle*{4}}
\end{picture}
\end{tabular}
\qedhere
\]
\end{proof}

\newpage

\section{Positive braid isotopy}
\label{sec:positive-braid-isotopy}

As explained in Section~\ref{sec:scannable-divides}, 
(the diagram of) the link of a scannable divide is
described by a simple and explicit combinatorial recipe. 
Therefore it is natural 
to try to establish our main conjectures in the case of scannable divides. 
We start by noting that for two scannable divides $D$ and~$D'$, the following are equivalent:
\begin{itemize}[leftmargin=.3in]
\item
$D_1$ and~$D_2$ are link equivalent;
\item
the closures of the braids $\beta(D_1)$ and $\beta(D_2)$
are isotopic; 
\item
the braids $\beta(D_1)$ and $\beta(D_2)$ can be obtained from each other 
by a sequence of Markov moves combined with braid conjugation. 
\end{itemize}
The first two statements are equivalent by 
Theorem~\ref{th:palindromic}; 
the last two statements are equivalent by Markov's theorem. 


\begin{definition}
\label{def:positive-isotopy}
Two positive braids $\beta_1$ and~$\beta_2$, 
or more precisely positive braid words defining them,  
are called \emph{positive-isotopic} to each other 
if they are related by a sequence of the following transformations: 
\begin{itemize}[leftmargin=.3in]
\item[{\rm (i)}]
isotopy among positive braids 
(i.e., applying Artin's braid relations); 
\item[{\rm (ii)}]
cyclic shifts (i.e., moving the last entry in a braid word 
to the front);  
\item[{\rm (iii)}]
positive Markov moves and their inverses. 
\end{itemize}
(A positive Markov move adds a strand at the top of a $k$-strand
braid, and inserts the Artin generator~$\sigma_k$ into the braid word,
at a single arbitrarily chosen location.)
If $\beta_1$ and~$\beta_2$ can be related to each other 
using transformations (i)--(ii) only, 
then we say that they are \emph{positive-isotopic inside the solid torus}. 
\end{definition}

By definition, positive braid isotopy (resp., positive isotopy inside
the solid torus) 
corresponds to a particular subclass of isotopies of 
\emph{closed} positive braids inside~$\RR^3$ (resp., inside the solid torus). 

If two positive braids 
are positive-isotopic inside the solid torus,
then they 
have the same number of strands,
and moreover are conjugate to each other. 
In general, conjugate positive braids do not have to be
positive-isotopic inside the solid torus.

If one drops the adjective ``positive,'' the situation simplifies: 
two braids with the same number of strands are conjugate if and only if 
they are isotopic in the solid torus; 
see \cite{Morton-braids} or \cite[Theorem~2.1]{kassel-turaev}. 

\smallskip

The positive braids $\beta(D_1)$ and $\beta(D_2)$ in
Example~\ref{example:D1D2} 
are positive-isotopic to each other inside the solid torus. 
Another set of examples is provided by Figure~\ref{fig:isotopic-scanning}: all 
braids associated to the divides shown therein are
positive-isotopic to each other. 
Some of these braids have different number of strands,
and consequently are not positive-isotopic inside the solid torus. 

\begin{conjecture}
\label{conj:positive-isotopy}
Let $D_1$ and $D_2$ be link equivalent scannable algebraic divides. 
Then the braids $\beta(D_1)$ and $\beta(D_2)$ are positive-isotopic. 
\end{conjecture}

\begin{conjecture}
\label{conj:positive-isotopy-inside-torus}
Let $D_1$ and $D_2$ be link equivalent scannable algebraic divides of (the same)
minimal index. 
Then the braids $\beta(D_1)$ and $\beta(D_2)$ are positive-isotopic inside the solid torus. 
\end{conjecture}

Conjectures~\ref{conj:positive-isotopy} 
and~\ref{conj:positive-isotopy-inside-torus}---especially the latter one---are 
motivated by a couple of observations due to 
Stepan Orevkov, see Lemmas~\ref{lem:orevkov} and~\ref{lem:orevkov-blowup} below.
We thank Stepan for allowing us to cite them here.

The first observation is based on Garside's solution of the conjugacy
problem in the braid group. 

\begin{lemma}[{\cite[
Section~3.2]{fiedler}}]
\label{lem:orevkov}
Suppose that $\beta$ is a positive braid whose closure contains the
positive half-twist~$\Delta$, 
i.e., $\beta$~is positive isotopic inside a solid torus
to a braid of the form~$\gamma \Delta$,
with $\gamma$ positive. 
Then any braid conjugate to~$\beta$ is positive isotopic to~$\beta$
inside a solid torus. 
\end{lemma}

The second observation concerns a certain positive braid of minimal index 
associated with a given isolated singularity $f(x,y)=0$, as follows. 
Instead of a Milnor ball, consider a (small) bi-disk 
\[
\Disk^2=\{|x|\le\varepsilon,\ |y|\le\varepsilon\}\subset\CC^2.
\] 
Assume, without loss of generality, that inside~$\Disk^2$
our curve stays close to the plane $y=0$, so that its intersection
with the boundary $\partial(\Disk^2)$ is contained in the solid torus
$\mathbf{V}=\{|x|=\varepsilon,\ |y|\le\varepsilon\}\subset\CC^2$. 
Let $L_f$ denote this intersection:
\begin{equation}
\label{eq:Lf}
L_f=\{f(x,y)=0, \ |x|=\varepsilon,\ |y|\le\varepsilon\}.
\end{equation}
By construction, $L_f$~is a link inside the solid torus~$\mathbf{V}$.  
Furthermore, $L_f$ is the closure of a minimal-index positive braid~$\beta_f$ 
obtained by cutting $L_f$ by the plane
$\{x=x_\circ\}$, for some $x_\circ$ with $|x_\circ|=\varepsilon$. 
Note that all intersections of the complex line $x=x_\circ$ 
with our complex curve are positive,
and moreover no intersection point escapes the bi-disc as the line varies in the vertical pencil.
The braid $\beta_f$ depends on the choice of~$x_\circ$,
but its closure obviously doesn't.

\begin{lemma}[\cite{Orevkov} {\cite[Proposition~II, p.~699]{r.williams}}]
\label{lem:orevkov-blowup}
The braid $\beta_f$ 
contains the positive full twist~$\Delta^2$. 
\end{lemma}

\begin{proof} 
It is not hard to see that 
$\Delta^{-2}\beta_f$ is the (positive) braid corresponding 
to the blown-up singularity $f(x,xy)=0$. 
\end{proof}

\begin{remark}
\label{rem:Franks-Williams}
The fact that $\beta_f$ is a minimal-index braid also follows from 
\cite[Corollary~2.4]{Franks-Williams}:
a~positive braid on $n$ strands containing~$\Delta^2$
has braid index~$n$. 
\end{remark}

\begin{proposition}
\label{pr:conjugate-implies-positive-isotopy}
Conjecture~\ref{conj:positive-isotopy-inside-torus} holds for link 
equivalent scannable algebraic divides $D_1$ and~$D_2$ whose associated link
$L=L(D_1)=L(D_2)$ is isotopic inside the solid torus to the link $L_f$ 
of the corresponding singularity, cf.~\eqref{eq:Lf}. 
\end{proposition}

Note that Theorem~\ref{th:L(D)=L(C,z)}
only asserts that $L$ and $L_f$ are isotopic inside~$\mathbf{S}^3$. 

\begin{proof}
The fact that $L$ is isotopic to $L_f$ 
inside the solid torus means that the braids~$\beta(D)$ and $\beta(D')$
are conjugate to~$\beta_f$,
and consequently to each other. 
Moreover Lemma~\ref{lem:orevkov-blowup} implies that 
both braids contain~$\Delta^2$, so by Lemma~\ref{lem:orevkov}
they are positive isotopic to each other inside the solid torus. 
\end{proof}

\begin{remark}
\label{rem:isotopy-in-known-constructions}
The isotopy condition in Proposition~\ref{pr:conjugate-implies-positive-isotopy}
is satisfied for all divides of minimal index 
coming from real morsifications constructed in \cite{acampo}
(see \cite[Theorem~1]{acampo}) \linebreak[3]
and/or 
\cite[Section~2]{leviant-shustin}
(see \cite[Theorem~2]{leviant-shustin}). 
Consequently, whenever such constructions are used to produce 
different real morsifications of (different real forms of) the same complex singularity, 
the positive braids associated with the corresponding scannable
divides are positive isotopic to each other inside the solid torus. 
In particular, this statement holds for all examples of scannable
algebraic divides of minimal index discussed in this paper. 
\end{remark}


\begin{remark}
\label{rem:markov-braid}
As shown by S.~Orevkov--V.~Shevchishin~\cite{orevkov-shevchishin} 
and N.~C.~Wrinkle~\cite{Wrinkle}, two braids
  (positive or not) are related by braid isotopy, cyclic shifts, and
  positive Markov moves if and only if their closures are transverse isotopic
  (using a natural transverse structure on the closure of a braid). 
  By Theorem~\ref{th:etnyre-vhm}, transverse isotopy is the same as
  isotopy for many of the links we consider, including
  closures of positive braids like those from scannable divides, and
  also for (algebraic) links of singularities. 
  Thus the obstacle to a full proof of Conjecture~\ref{conj:positive-isotopy} 
  is the difference between
  positive  \emph{vs.}\  plain isotopy of positive braids inside the solid torus.
  (Using Proposition~\ref{pr:conjugate-implies-positive-isotopy} 
  to establish positive isotopy requires an additional topological condition.) 
\end{remark}

We next prove the implication \textbf{(d)}$\Rightarrow$\textbf{(p)}
of Conjecture~\ref{conj:morsif=mut-via-divides}
in the case of scan\-nable (not necessarily algebraic) divides, 
under the assumption of positive isotopy.  \linebreak[3]
This assumption is potentially redundant, cf.\ Conjecture~\ref{conj:positive-isotopy}
and Remark~\ref{rem:isotopy-in-known-constructions}.  

\begin{theorem}
\label{th:main-conjecture-scannable-1}
Let $D_1$ and $D_2$ be scannable divides 
whose respective braids $\beta(D_1)$ and $\beta(D_2)$ are positive-isotopic.  
(In particular, $D_1$ and $D_2$ are link equivalent.) 
\linebreak[3]
Then the plabic fences $\Phi(D_1)$ and $\Phi(D_2)$ are move equivalent. 
\end{theorem}

In view of Proposition~\ref{pr:divide-to-fence-to-braid}, 
Theorem~\ref{th:main-conjecture-scannable-1} follows from
Proposition~\ref{pr:positive-isotopic=>move-equivalent} below. 

\begin{proposition}
\label{pr:positive-isotopic=>move-equivalent}
Let $\Phi_1$ and $\Phi_2$ be plabic fences 
(possibly with a different number of strands)
whose associated braids $\beta(\Phi_1)$ and $\beta(\Phi_2)$
are positive-isotopic to each other.  
Then $\Phi_1$ and~$\Phi_2$ are move equivalent. 
\end{proposition}

\begin{proof}
We need to show that each of the transformations (i)--(iii) in Definition~\ref{def:positive-isotopy}
can be interpreted as an instance of move equivalence. 
Transformations~(i) are covered by Proposition~\ref{pr:equal-braids=>move-equivalent}. 
A~cyclic shift of the form $\ww\sigma_i\leftrightarrow\sigma_i\ww$ 
can be executed by first replacing $\sigma_i$ by~$\tau_i$ at the end of the word
(cf.~\eqref{eq:replace-tau-by-sigma}),
then moving $\tau_i$ all the way to the left (cf.\ the proof of Lemma~\ref{lem:ww-to-bww}), 
then replacing $\tau_i$ by~$\sigma_i$ at the beginning of the~word. 
Finally, a positive Markov move of the form $\ww_1\ww_2\leftrightarrow\ww_1 \sigma_k\ww_2$, 
with $\ww_1,\ww_2\in\mathbb{B}_k$ and $\ww_1 \sigma_k\ww_2\in\mathbb{B}_{k+1}$
(resp., the reverse of~it) 
is emulated by two tail attachments (resp., tail removals), 
see Figure~\ref{fig:positive-markov-as-tail-attachment}. 
\end{proof}

\begin{figure}[ht]
\begin{center}
\vspace{-.1in}
\begin{tabular}{ccc}
\begin{picture}(50,20)(1,-2)
\thicklines
\put(-5,0){\makebox(0,0){$k$}} 
\put(0,0){\line(1,0){50}}
\put(5,0){\line(0,-1){5}}
\put(15,0){\line(0,-1){5}}
\put(35,0){\line(0,-1){5}}
\put(45,0){\line(0,-1){5}}
\thinlines
\put(0,15){\line(1,0){50}}
\put(0,16){\line(1,0){50}}
\end{picture}
&
\begin{picture}(20,20)(0,-2)
\thicklines
\put(10,5){\makebox(0,0){$\leftrightarrow$}}
\end{picture}
&
\begin{picture}(50,20)(1,-2)
\thicklines
\put(55,0){\makebox(0,0){$k$}} 
\put(0,0){\line(1,0){50}}
\put(55,10){\makebox(0,0){\ \quad $k\!+\!1$}} 
\put(0,10){\line(1,0){50}}
\put(5,0){\line(0,-1){5}}
\put(15,0){\line(0,-1){5}}
\put(25,10){\line(0,-1){10}}
\put(35,0){\line(0,-1){5}}
\put(45,0){\line(0,-1){5}}
\put(25,0){\circle*{2.5}}
\put(50,10){\circle*{2.5}}
\thinlines
\put(0,15){\line(1,0){50}}
\put(0,16){\line(1,0){50}}
\end{picture}
\\[.1in]
$\ww_1\ww_2$ & $\leftrightarrow$ & $\ww_1 \sigma_k\ww_2$
\end{tabular}
\end{center}
\vspace{-.1in}
\caption{
Positive Markov move via tail attachments. 
}
\label{fig:positive-markov-as-tail-attachment}
\end{figure}

\vspace{-.15in}

\begin{corollary}
\label{cor:scannable+positive-isotopic}
Scannable divides whose associated braids are positive-isotopic
have mutation equivalent quivers. 
\end{corollary}

\begin{proof}
By Propositions~\ref{pr:plabic-vs-quivers}~and~\ref{pr:Q(D)=Q(P(D))}, 
move equivalence of the plabic fences $\Phi(D_1)$ and~$\Phi(D_2)$
implies mutation equivalence of the quivers $Q(D_1)$ and~$Q(D_2)$. 
\end{proof}

\begin{example}
Consider the two divides shown in Figure~\ref{fig:link-equivalent=>mutation-equivalent}. 
Their associated positive braids are both equal to~$\Delta^4$, 
where 
$\Delta=\sigma_1\sigma_3\sigma_2\sigma_1\sigma_3\sigma_2
=\sigma_2\sigma_1\sigma_3\sigma_2\sigma_1\sigma_3$.
(The link equivalence of these two divides 
is explained by the fact that they arise from morsifications of different real forms
of the quasihomogeneous singularity $x^8+y^4=0$,  
cf.\ Figure~\ref{fig:non-partitions}(b).) 
By Corollary~\ref{cor:scannable+positive-isotopic}, 
the quivers associated to these divides 
must be mutation equivalent to each other. 
This can be also verified directly using any of the widely available 
software packages for quiver mutations. 
\end{example}

\begin{figure}[ht] 
\begin{center} 
\begin{tabular}{ccc}
\setlength{\unitlength}{1.5pt} 
\begin{picture}(80,30)(-5,-15)
\put(0,0){\makebox(0,0){\usebox{\sssonethree}}} 
\put(10,0){\makebox(0,0){\usebox{\ssstwo}}} 
\put(20,0){\makebox(0,0){\usebox{\sssonethree}}} 
\put(30,0){\makebox(0,0){\usebox{\ssstwo}}} 
\put(40,0){\makebox(0,0){\usebox{\ssstwo}}} 
\put(50,0){\makebox(0,0){\usebox{\sssonethree}}} 
\put(60,0){\makebox(0,0){\usebox{\ssstwo}}} 
\put(70,0){\makebox(0,0){\usebox{\sssonethree}}} 
\end{picture} 
&&
\begin{picture}(80,30)(-1.5,-15)
\put(10,-10){\makebox(0,0){\usebox{\opening}}} 
\put(10,10){\makebox(0,0){\usebox{\opening}}} 
\put(10,0){\makebox(0,0){\usebox{\ssstwo}}} 
\put(20,0){\makebox(0,0){\usebox{\sssonethree}}} 
\put(30,0){\makebox(0,0){\usebox{\ssstwo}}} 
\put(40,0){\makebox(0,0){\usebox{\sssonethree}}} 
\put(50,0){\makebox(0,0){\usebox{\ssstwo}}} 
\put(60,0){\makebox(0,0){\usebox{\sssonethree}}} 
\put(70,0){\makebox(0,0){\usebox{\ssstwo}}} 
\put(80,-10){\makebox(0,0){\usebox{\closing}}} 
\put(80,10){\makebox(0,0){\usebox{\closing}}} 
\end{picture} 
\\[.2in]
\setlength{\unitlength}{2.4pt} 
\begin{picture}(80,20)(0,0)
\thicklines
\multiput(8,0)(0,20){2}{\red{\vector(-1,0){6}}} 
\multiput(12,0)(0,10){3}{\red{\vector(1,0){6}}} 
\put(28,10){\red{\vector(-1,0){6}}} 
\put(32,10){\red{\vector(1,0){6}}} 
\put(48,10){\red{\vector(-1,0){6}}} 
\put(52,10){\red{\vector(1,0){6}}} 
\multiput(38,0)(0,20){2}{\red{\vector(-1,0){16}}} 
\multiput(42,0)(0,20){2}{\red{\vector(1,0){16}}} 
\multiput(72,0)(0,20){2}{\red{\vector(1,0){6}}} 
\multiput(68,0)(0,10){3}{\red{\vector(-1,0){6}}} 

\put(10,18){\red{\vector(0,-1){6}}} 
\put(20,18){\red{\vector(0,-1){6}}} 
\put(60,18){\red{\vector(0,-1){6}}} 
\put(70,18){\red{\vector(0,-1){6}}} 
\put(40,8){\red{\vector(0,-1){6}}} 
\put(10,2){\red{\vector(0,1){6}}} 
\put(20,2){\red{\vector(0,1){6}}} 
\put(60,2){\red{\vector(0,1){6}}} 
\put(70,2){\red{\vector(0,1){6}}} 
\put(40,12){\red{\vector(0,1){6}}} 

\put(18.5,11.5){\red{\vector(-1,1){7}}} 
\put(18.5,8.5){\red{\vector(-1,-1){7}}} 
\put(61.5,11.5){\red{\vector(1,1){7}}} 
\put(61.5,8.5){\red{\vector(1,-1){7}}} 
\put(38.5,18.5){\red{\vector(-1,-1){7}}} 
\put(41.5,18.5){\red{\vector(1,-1){7}}} 
\put(38.5,1.5){\red{\vector(-1,1){7}}} 
\put(41.5,1.5){\red{\vector(1,1){7}}} 
\put(22,11){\red{\vector(2,1){15.5}}} 
\put(22,9){\red{\vector(2,-1){15.5}}} 
\put(58,11){\red{\vector(-2,1){15.5}}} 
\put(58,9){\red{\vector(-2,-1){15.5}}}

\put(40,20){\red{\circle*{2}}} 
\put(40,0){\red{\circle*{2}}} 
\multiput(0,20)(10,0){3}{\red{\circle*{2}}} 
\multiput(60,20)(10,0){3}{\red{\circle*{2}}} 
\multiput(10,10)(10,0){7}{\red{\circle*{2}}} 
\multiput(0,0)(10,0){3}{\red{\circle*{2}}} 
\multiput(60,0)(10,0){3}{\red{\circle*{2}}} 
\end{picture} 
&\qquad\qquad&
\setlength{\unitlength}{2.4pt} 
\begin{picture}(60,20)(0,0)
\thicklines
\multiput(2,0)(0,10){3}{\red{\vector(1,0){6}}} 
\multiput(22,0)(0,10){3}{\red{\vector(1,0){6}}} 
\multiput(42,0)(0,10){3}{\red{\vector(1,0){6}}} 
\multiput(18,0)(0,10){3}{\red{\vector(-1,0){6}}} 
\multiput(38,0)(0,10){3}{\red{\vector(-1,0){6}}} 
\multiput(58,0)(0,10){3}{\red{\vector(-1,0){6}}} 

\multiput(0,2)(10,0){7}{\red{\vector(0,1){6}}} 
\multiput(0,18)(10,0){7}{\red{\vector(0,-1){6}}} 
\multiput(8.5,8.5)(20,0){3}{\red{\vector(-1,-1){7}}} 
\multiput(11.5,8.5)(20,0){3}{\red{\vector(1,-1){7}}} 
\multiput(8.5,11.5)(20,0){3}{\red{\vector(-1,1){7}}} 
\multiput(11.5,11.5)(20,0){3}{\red{\vector(1,1){7}}}

\multiput(0,0)(10,0){7}{\red{\circle*{2}}} 
\multiput(0,10)(10,0){7}{\red{\circle*{2}}} 
\multiput(0,20)(10,0){7}{\red{\circle*{2}}} 
\end{picture} 
\end{tabular}
\end{center} 
\caption{Scannable divides associated with 
two different morsifications of the quasihomogeneous singularity $x^8+y^4=0$,
and the corresponding quivers.
} 
\vspace{-.15in}
\label{fig:link-equivalent=>mutation-equivalent} 
\end{figure} 

We conclude this section by a discussion of one simple instance
of Conjecture~\ref{conj:positive-isotopy}. 

\begin{remark}
\label{rem:klein-group}
The \emph{Klein four-group}~$\KK$ naturally acts on (isomorphism classes~of)
scan\-nable divides. 
For $D$ a scannable divide rendered as in Definition~\ref{def:scannable},  
the images of~$D$ under the action of~$\KK$ are:
\begin{itemize}[leftmargin=.3in]
\item
$D$ itself; 
\item
the reflection of $D$ with respect to a vertical line,
denoted~$D^{\leftrightarrow}$; 
\item
the reflection of $D$ with respect to a horizontal line,
denoted~$D^{\updownarrow}$; 
\item
the result of rotating $D$ by a $180^\circ$ turn,
denoted~$D^\curvearrowright$. 
\end{itemize}
Each of the scannable divides 
$D^{\leftrightarrow}, D^{\updownarrow}, D^\curvearrowright$
is link equivalent~to~$D$. 
If $D$ is algebraic, then so are $D^{\leftrightarrow}$, 
$D^{\updownarrow}$, and~$D^\curvearrowright$.  
According to Conjecture~\ref{conj:positive-isotopy},
the four positive~braids $\beta(D)$, $\beta(D^{\leftrightarrow})$,
$\beta(D^{\updownarrow})$ and~$\beta(D^\curvearrowright)$
must be positive-isotopic to each other. 
It is~easy to see that $\beta(D^{\leftrightarrow})$ and~$\beta(D)$
are related via cyclic shifts, so these braids are positive-isotopic. 
The challenge is to prove that $\beta(D)$ is positive-isotopic to 
$\beta(D^{\updownarrow})$ and~$\beta(D^\curvearrowright)$, 
in the case of algebraic divides (and possibly beyond). 
It would suffice to establish this claim for~$\beta(D^{\updownarrow})$. 
While $\beta(D^{\updownarrow})$ is conjugate to~$\beta(D)$
by the half-twist~$\Delta$, this in itself does not guarantee the existence
of positive isotopy.
On the other hand, the latter property
holds whenever $\beta(D)$ contains~$\Delta$, 
by Lemma~\ref{lem:orevkov}. 
In view of Lemma~\ref{lem:orevkov-blowup} and
Remark~\ref{rem:isotopy-in-known-constructions},
this condition is satisfied for the scannable algebraic divides of minimal index
arising from all common constructions of real morsifications. 
\end{remark}

\newpage

\section{Triangle moves}
\label{sec:yang-baxter-transformations}

\begin{definition}
A \emph{triangle move}
(or a ``Yang-Baxter move'')  
is a local deformation of a divide that can be applied for any
triangular region, 
i.e., a region whose boundary consists of three $1$-cells
and three nodes.  
A~triangle move pushes a branch containing one of these three $1$-cells
across the opposing node, see Figure~\ref{fig:yb-move}. 
\end{definition} 

\begin{figure}[ht] 
\begin{center} 
\setlength{\unitlength}{1.5pt} 
\begin{picture}(33,20)(0,0)
\thicklines
\qbezier(2,12)(16.5,5)(31,12)
\qbezier(10,0)(12,14)(23,24)
\qbezier(10,24)(21,14)(23,0)
\end{picture} 
\begin{picture}(25,24)(0,0)
\put(12.5,12){\makebox(0,0){$\longleftrightarrow$}}
\end{picture}
\begin{picture}(33,20)(0,0)
\thicklines
\qbezier(2,12)(16.5,19)(31,12)
\qbezier(10,0)(21,10)(23,24)
\qbezier(10,24)(12,10)(23,0)
\end{picture} 
\vspace{-.1in}
\end{center} 
\caption{A triangle move. 
} 
\label{fig:yb-move} 
\end{figure} 

\vspace{-.25in}

\begin{definition}
Two divides are called \emph{\YB-equivalent} 
(``triangle-move equivalent") 
if they are related 
via a sequence of triangle moves and isotopies. 
An example is shown in Figure~\ref{fig:transforming-into-scannable}. 
\end{definition}

\begin{figure}[ht] 
\begin{center} 
\setlength{\unitlength}{1.5pt} 
\begin{picture}(60,30)(-5,-15)
\put(0,0){\makebox(0,0){\usebox{\sssonethree}}} 
\put(10,0){\makebox(0,0){\usebox{\sssonethree}}} 
\put(20,0){\makebox(0,0){\usebox{\ssstwo}}} 
\put(30,0){\makebox(0,0){\usebox{\sssonethree}}} 
\put(40,0){\makebox(0,0){\usebox{\ssstwo}}} 
\put(50,0){\makebox(0,0){\usebox{\closing}}} 
\put(50,0){\makebox(0,0){\usebox{\cclosing}}} 
\end{picture} 
\begin{picture}(20,30)(-5,-15)
\put(5,0){\makebox(0,0){$\leadsto$}}
\end{picture}
\begin{picture}(80,30)(-5,-15)
\put(0,0){\makebox(0,0){\usebox{\sssonethree}}} 
\put(10,0){\makebox(0,0){\usebox{\sssone}}} 
\put(20,0){\makebox(0,0){\usebox{\ssstwo}}} 
\put(30,0){\makebox(0,0){\usebox{\sssthree}}} 
\put(40,0){\makebox(0,0){\usebox{\ssstwo}}} 
\put(50,0){\makebox(0,0){\usebox{\sssone}}} 
\put(60,0){\makebox(0,0){\usebox{\ssstwo}}} 
\put(70,0){\makebox(0,0){\usebox{\closing}}} 
\put(70,0){\makebox(0,0){\usebox{\cclosing}}} 
\end{picture} 
\begin{picture}(20,30)(-5,-15)
\put(5,0){\makebox(0,0){$\leadsto$}}
\end{picture}
\begin{picture}(70,30)(-5,-15)
\put(0,0){\makebox(0,0){\usebox{\sssonethree}}} 
\put(10,0){\makebox(0,0){\usebox{\sssone}}} 
\put(20,0){\makebox(0,0){\usebox{\ssstwo}}} 
\put(30,0){\makebox(0,0){\usebox{\sssonethree}}} 
\put(40,0){\makebox(0,0){\usebox{\ssstwo}}} 
\put(50,0){\makebox(0,0){\usebox{\sssone}}} 
\put(60,0){\makebox(0,0){\usebox{\closing}}} 
\put(60,0){\makebox(0,0){\usebox{\cclosing}}} 
\end{picture} 
\\[.3in] \hspace{-.1in}
\begin{picture}(20,10)(-5,-15)
\put(5,0){\makebox(0,0){$\sim$}}
\end{picture}
\hspace{-.1in}
\begin{picture}(65,30)(-5,-15)
\put(0,0){\makebox(0,0){\usebox{\sssonethree}}} 
\put(10,0){\makebox(0,0){\usebox{\sssone}}} 
\put(20,0){\makebox(0,0){\usebox{\ssstwo}}} 
\put(30,0){\makebox(0,0){\usebox{\sssonethree}}} 
\put(40,0){\makebox(0,0){\usebox{\ssstwo}}} 
\put(50,0){\makebox(0,0){\usebox{\ssstwo}}} 
\put(60,-10){\makebox(0,0){\usebox{\closing}}} 
\put(60,10){\makebox(0,0){\usebox{\closing}}} 
\end{picture} 
\hspace{-.1in}
\begin{picture}(20,30)(-5,-15)
\put(5,0){\makebox(0,0){$\leadsto$}}
\end{picture}
\hspace{-.1in}
\begin{picture}(75,30)(-5,-15)
\put(0,0){\makebox(0,0){\usebox{\sssonethree}}} 
\put(10,0){\makebox(0,0){\usebox{\ssstwo}}} 
\put(20,0){\makebox(0,0){\usebox{\sssone}}} 
\put(30,0){\makebox(0,0){\usebox{\ssstwo}}} 
\put(40,0){\makebox(0,0){\usebox{\sssthree}}} 
\put(50,0){\makebox(0,0){\usebox{\ssstwo}}} 
\put(60,0){\makebox(0,0){\usebox{\ssstwo}}} 
\put(70,-10){\makebox(0,0){\usebox{\closing}}} 
\put(70,10){\makebox(0,0){\usebox{\closing}}} 
\end{picture} 
\hspace{-.1in}
\begin{picture}(20,30)(-5,-15)
\put(5,0){\makebox(0,0){$\leadsto$}}
\end{picture}
\hspace{-.1in}
\begin{picture}(70,30)(-5,-15)
\put(0,0){\makebox(0,0){\usebox{\sssonethree}}} 
\put(10,0){\makebox(0,0){\usebox{\ssstwo}}} 
\put(20,0){\makebox(0,0){\usebox{\sssonethree}}} 
\put(30,0){\makebox(0,0){\usebox{\ssstwo}}} 
\put(40,0){\makebox(0,0){\usebox{\sssthree}}} 
\put(50,0){\makebox(0,0){\usebox{\ssstwo}}} 
\put(60,-10){\makebox(0,0){\usebox{\closing}}} 
\put(60,10){\makebox(0,0){\usebox{\closing}}} 
\end{picture} 
\end{center} 
\caption{A sequence of triangle moves (denoted $\leadsto$) and isotopies (denoted~$\sim$). 
All these divides are \hbox{\YB-equivalent} to each other. 
The first two divides are not scannable. 
While the shown drawing of the third divide is not scannable, this divide is isotopic to the first divide in the bottom row, which is scannable (as are all divides in this~row). 
} 
\vspace{-.2in}
\label{fig:transforming-into-scannable} 
\end{figure} 

\begin{problem}
\label{problem:triangle-moves-and-algebraicity}
Can an algebraic divide be \hbox{\YB-equivalent} to a non-algebraic~one?
\end{problem}

The following result, in different guises, 
has been a part of the ``cluster folklore'' for at least a decade;  
we~do not claim any originality for it. 
See, e.g., \cite[Section~2.3]{kenyon-pemantle}.

\begin{proposition} 
\label{prop:ybe-via-mutations}
\YB-equivalent divides have mutation equivalent quivers. 
\end{proposition}

We sketch two (implicitly related) proofs of Proposition~\ref{prop:ybe-via-mutations}. 

\begin{proof} 
Let us examine what happens to the quiver in the vicinity
of a triangle move. The case where all neighboring 
connected components of the complement of a divide are bounded
is shown in Figure~\ref{fig:yb-move-quiver}. 
There are many other cases, cf.\ e.g.\ Figure~\ref{fig:E6-morsifications}. 
\end{proof}


\begin{figure}[ht] 
\begin{center} 
\setlength{\unitlength}{6.0pt} 
\begin{picture}(33,24)(0,0)
\thicklines
\qbezier(2,12)(16.5,5)(31,12)
\qbezier(9,0)(12,14)(24,24)
\qbezier(9,24)(21,14)(24,0)

\put(16.5,12){\circle*{0.7}}
\put(12,8.8){\circle*{0.7}}
\put(21,8.8){\circle*{0.7}}
\put(16.5,16.3){\circle*{0.7}}

\put(16.5,20){\circle*{0.7}}
\put(16.5,4){\circle*{0.7}}

\put(9,7.2){\circle*{0.7}}
\put(9,16.8){\circle*{0.7}}
\put(24,7.2){\circle*{0.7}}
\put(24,16.8){\circle*{0.7}}

\put(16,10.7){\makebox(0,0){\text{\SMALL \textbf{0}}}}
\put(15.1,15.4){\makebox(0,0){\text{\SMALL \textbf{1}}}}
\put(20.9,7.5){\makebox(0,0){\text{\SMALL \textbf{2}}}}
\put(12.1,7.5){\makebox(0,0){\text{\SMALL \textbf{3}}}}
\put(16.5,21){\makebox(0,0){\text{\SMALL \textbf{4}}}}
\put(24.8,17.4){\makebox(0,0){\text{\SMALL \textbf{5}}}}
\put(24.8,6.6){\makebox(0,0){\text{\SMALL \textbf{6}}}}
\put(16.5,3){\makebox(0,0){\text{\SMALL \textbf{7}}}}
\put(8.2,6.6){\makebox(0,0){\text{\SMALL \textbf{8}}}}
\put(8.2,17.4){\makebox(0,0){\text{\SMALL \textbf{9}}}}

\put(16.5,11.3){\red{\vector(0,-1){6.6}}}
\put(16.5,15.6){\red{\vector(0,-1){3.1}}}
\put(16.5,17){\red{\vector(0,1){2.4}}}

\put(12.7,9.2){\red{\vector(45,34){3.1}}}

\put(20.3,9.2){\red{\vector(-45,34){3.1}}}

\put(9.6,16.78){\red{\vector(12,-1){6.2}}}
\put(23.4,16.78){\red{\vector(-12,-1){6.2}}}
\put(9.2,16.2){\red{\vector(1,-2.7){2.5}}}
\put(23.8,16.2){\red{\vector(-1,-2.7){2.5}}}
\put(9,7.9){\red{\vector(0,1){8.2}}}
\put(24,7.9){\red{\vector(0,1){8.2}}}
\put(9.5,7){\red{\vector(2.3,-1){6.5}}}
\put(23.5,7){\red{\vector(-2.3,-1){6.5}}}
\put(11.4,8.5){\red{\vector(-2,-1){2.1}}}
\put(21.6,8.5){\red{\vector(2,-1){2.1}}}
\put(16.1,4.4){\red{\vector(-2,2.15){3.7}}}
\put(16.9,4.4){\red{\vector(2,2.15){3.7}}}
\put(17,19.8){\red{\vector(2.3,-1){6.5}}}
\put(16,19.8){\red{\vector(-2.3,-1){6.5}}}
\put(16,12.2){\red{\vector(-2,1.3){6.5}}}
\put(17,12.2){\red{\vector(2,1.3){6.5}}}

\end{picture} 
\begin{picture}(33,24)(0,0)
\thicklines
\qbezier(2,12)(16.5,19)(31,12)
\qbezier(9,0)(21,10)(24,24)
\qbezier(9,24)(12,10)(24,0)

\put(16.5,12){\circle*{0.7}}
\put(12,15.2){\circle*{0.7}}
\put(21,15.2){\circle*{0.7}}
\put(16.5,7.7){\circle*{0.7}}

\put(16.5,20){\circle*{0.7}}
\put(16.5,4){\circle*{0.7}}

\put(9,7.2){\circle*{0.7}}
\put(9,16.8){\circle*{0.7}}
\put(24,7.2){\circle*{0.7}}
\put(24,16.8){\circle*{0.7}}

\put(15.8,13.4){\makebox(0,0){\text{\SMALL \textbf{0}}}}
\put(15.1,8.6){\makebox(0,0){\text{\SMALL \textbf{1}}}}
\put(20.9,16.5){\makebox(0,0){\text{\SMALL \textbf{3}}}}
\put(12.1,16.5){\makebox(0,0){\text{\SMALL \textbf{2}}}}
\put(16.5,21){\makebox(0,0){\text{\SMALL \textbf{4}}}}
\put(24.8,17.4){\makebox(0,0){\text{\SMALL \textbf{5}}}}
\put(24.8,6.6){\makebox(0,0){\text{\SMALL \textbf{6}}}}
\put(16.5,3){\makebox(0,0){\text{\SMALL \textbf{7}}}}
\put(8.2,6.6){\makebox(0,0){\text{\SMALL \textbf{8}}}}
\put(8.2,17.4){\makebox(0,0){\text{\SMALL \textbf{9}}}}

\put(16.5,19.3){\red{\vector(0,-1){6.6}}}
\put(16.5,11.5){\red{\vector(0,-1){3.1}}}
\put(16.5,4.6){\red{\vector(0,1){2.4}}}
\put(17,19.8){\red{\vector(2.3,-1){6.5}}}
\put(16,19.8){\red{\vector(-2.3,-1){6.5}}}

\put(9.6,7.22){\red{\vector(12,1){6.2}}}
\put(23.4,7.22){\red{\vector(-12,1){6.2}}}

\put(9,7.9){\red{\vector(0,1){8.2}}}
\put(24,7.9){\red{\vector(0,1){8.2}}}
\put(9.5,7){\red{\vector(2.3,-1){6.5}}}
\put(23.5,7){\red{\vector(-2.3,-1){6.5}}}
\put(9.5,16.5){\red{\vector(2,-1){2}}}
\put(23.5,16.5){\red{\vector(-2,-1){2}}}

\put(16.9,12.3){\red{\vector(2,1.45){3.6}}}
\put(16.1,12.3){\red{\vector(-2,1.45){3.6}}}
\put(11.7,14.6){\red{\vector(-1,-2.7){2.5}}}
\put(21.3,14.6){\red{\vector(1,-2.7){2.5}}}
\put(9.5,7.5){\red{\vector(2,1.3){6.5}}}
\put(23.5,7.5){\red{\vector(-2,1.3){6.5}}}
\put(12.4,15.6){\red{\vector(2,2.15){3.7}}}
\put(20.6,15.6){\red{\vector(-2,2.15){3.7}}}

\end{picture} 
\end{center} 
\caption{Quivers of two divides related by a triangle move. 
Only 
the arrows connecting pairs of vertices labeled $0,1,2,\dots,9$ 
are shown. 
The two quivers are related via the composition of $5$ mutations 
\hbox{$\mu_0\!\circ\! \mu_1\!\circ\! \mu_2\!\circ\!\mu_3\!\circ\!\mu_0$}. 
} 
\label{fig:yb-move-quiver} 
\vspace{-.15in}
\end{figure} 

An alternative proof of Proposition~\ref{prop:ybe-via-mutations} 
uses the machinery of plabic graphs. 
In view of Propositions~\ref{pr:plabic-vs-quivers} and~\ref{pr:Q(D)=Q(P(D))},
it suffices to establish the following claim. 

\begin{proposition}
\label{prop:braid-equivalence}
Let $D_1$ and~$D_2$ be two \YB-equivalent divides. 
Then any plabic graphs 
$P_1\in\PP(D_1)$ and $P_2\in\PP(D_2)$ are move equivalent. 
\end{proposition}

\begin{proof}
For $D_1$ and~$D_2$ related by a triangle move,
Figure~\ref{fig:plabic-hexagon-3} shows 
a sequence of local moves relating a plabic graph $P_1\!\in\!\PP(D_1)$ 
to a plabic graph $P_2\!\in\!\PP(D_2)$. 
\end{proof}

\begin{figure}[ht]
\begin{center}
\setlength{\unitlength}{0.41pt}
\begin{picture}(160,160)(0,0)
\thicklines
\put(0,80){\line(1,0){20}}
\put(140,80){\line(1,0){20}}
\put(50,20){\line(1,0){60}}
\put(50,140){\line(1,0){60}}
\put(20,80){\line(1,2){30}}
\put(110,20){\line(1,2){30}}
\put(20,80){\line(1,-2){30}}
\put(110,140){\line(1,-2){30}}
\put(40,0){\line(1,2){10}}
\put(110,140){\line(1,2){10}}
\put(40,160){\line(1,-2){10}}
\put(110,20){\line(1,-2){10}}
\put(50,20){\circle*{8}}
\put(110,140){\circle*{8}}
\put(140,80){\circle*{8}}

\put(30,100){\line(1,-2){40}}
\put(30,100){\circle*{8}}

\put(90,20){\line(1,2){40}}
\put(90,20){\circle*{8}}

\put(40,120){\line(1,0){80}}
\put(40,120){\circle*{8}}
\end{picture}
\qquad
\begin{picture}(160,160)(0,0)
\thicklines
\put(0,80){\line(1,0){20}}
\put(140,80){\line(1,0){20}}
\put(50,20){\line(1,0){60}}
\put(50,140){\line(1,0){60}}
\put(20,80){\line(1,2){30}}
\put(110,20){\line(1,2){30}}
\put(20,80){\line(1,-2){30}}
\put(110,140){\line(1,-2){30}}
\put(40,0){\line(1,2){10}}
\put(110,140){\line(1,2){10}}
\put(40,160){\line(1,-2){10}}
\put(110,20){\line(1,-2){10}}
\put(50,20){\circle*{8}}
\put(110,140){\circle*{8}}
\put(140,80){\circle*{8}}

\put(30,100){\line(1,-2){40}}
\put(30,100){\circle*{8}}

\put(90,20){\line(1,2){40}}
\put(90,20){\circle*{8}}

\put(40,80){\line(1,0){80}}
\put(40,80){\circle*{8}}
\end{picture}
\qquad
\begin{picture}(160,160)(0,0)
\thicklines
\put(0,80){\line(1,0){20}}
\put(140,80){\line(1,0){20}}
\put(50,20){\line(1,0){60}}
\put(50,140){\line(1,0){60}}
\put(20,80){\line(1,2){30}}
\put(110,20){\line(1,2){30}}
\put(20,80){\line(1,-2){30}}
\put(110,140){\line(1,-2){30}}
\put(40,0){\line(1,2){10}}
\put(110,140){\line(1,2){10}}
\put(40,160){\line(1,-2){10}}
\put(110,20){\line(1,-2){10}}
\put(50,20){\circle*{8}}
\put(110,140){\circle*{8}}
\put(140,80){\circle*{8}}

\put(30,100){\line(1,-2){40}}
\put(30,100){\circle*{8}}

\put(90,20){\line(1,2){40}}
\put(70,20){\circle*{8}}

\put(40,80){\line(1,0){80}}
\put(120,80){\circle*{8}}
\end{picture}
\qquad
\begin{picture}(160,160)(0,0)
\thicklines
\put(0,80){\line(1,0){20}}
\put(140,80){\line(1,0){20}}
\put(50,20){\line(1,0){60}}
\put(50,140){\line(1,0){60}}
\put(20,80){\line(1,2){30}}
\put(110,20){\line(1,2){30}}
\put(20,80){\line(1,-2){30}}
\put(110,140){\line(1,-2){30}}
\put(40,0){\line(1,2){10}}
\put(110,140){\line(1,2){10}}
\put(40,160){\line(1,-2){10}}
\put(110,20){\line(1,-2){10}}
\put(50,20){\circle*{8}}
\put(110,140){\circle*{8}}
\put(140,80){\circle*{8}}

\put(30,100){\line(1,-2){10}}
\put(30,100){\circle*{8}}

\put(120,80){\line(1,2){10}}
\put(120,80){\circle*{8}}

\put(30,60){\line(1,2){10}}
\put(30,60){\circle*{8}}

\put(40,80){\line(1,0){80}}
\put(120,80){\line(1,-2){10}}

\end{picture}
\\[.2in]
\begin{picture}(160,160)(0,0)
\thicklines
\put(0,80){\line(1,0){20}}
\put(140,80){\line(1,0){20}}
\put(50,20){\line(1,0){60}}
\put(50,140){\line(1,0){60}}
\put(20,80){\line(1,2){30}}
\put(110,20){\line(1,2){30}}
\put(20,80){\line(1,-2){30}}
\put(110,140){\line(1,-2){30}}
\put(40,0){\line(1,2){10}}
\put(110,140){\line(1,2){10}}
\put(40,160){\line(1,-2){10}}
\put(110,20){\line(1,-2){10}}
\put(50,20){\circle*{8}}
\put(110,140){\circle*{8}}
\put(130,60){\circle*{8}}

\put(30,100){\line(1,-2){10}}
\put(40,80){\circle*{8}}

\put(120,80){\line(1,2){10}}
\put(130,100){\circle*{8}}

\put(30,60){\line(1,2){10}}
\put(20,80){\circle*{8}}

\put(40,80){\line(1,0){80}}
\put(120,80){\line(1,-2){10}}
\end{picture}
\qquad
\begin{picture}(160,160)(0,0)
\thicklines
\put(0,80){\line(1,0){20}}
\put(140,80){\line(1,0){20}}
\put(50,20){\line(1,0){60}}
\put(50,140){\line(1,0){60}}
\put(20,80){\line(1,2){30}}
\put(110,20){\line(1,2){30}}
\put(20,80){\line(1,-2){30}}
\put(110,140){\line(1,-2){30}}
\put(40,0){\line(1,2){10}}
\put(110,140){\line(1,2){10}}
\put(40,160){\line(1,-2){10}}
\put(110,20){\line(1,-2){10}}
\put(50,20){\circle*{8}}
\put(110,140){\circle*{8}}
\put(130,60){\circle*{8}}

\put(40,80){\line(1,2){30}}
\put(40,80){\circle*{8}}

\put(120,80){\line(-1,2){30}}
\put(90,140){\circle*{8}}

\put(30,60){\line(1,2){10}}
\put(20,80){\circle*{8}}

\put(40,80){\line(1,0){80}}
\put(120,80){\line(1,-2){10}}
\end{picture}
\qquad
\begin{picture}(160,160)(0,0)
\thicklines
\put(0,80){\line(1,0){20}}
\put(140,80){\line(1,0){20}}
\put(50,20){\line(1,0){60}}
\put(50,140){\line(1,0){60}}
\put(20,80){\line(1,2){30}}
\put(110,20){\line(1,2){30}}
\put(20,80){\line(1,-2){30}}
\put(110,140){\line(1,-2){30}}
\put(40,0){\line(1,2){10}}
\put(110,140){\line(1,2){10}}
\put(40,160){\line(1,-2){10}}
\put(110,20){\line(1,-2){10}}
\put(50,20){\circle*{8}}
\put(110,140){\circle*{8}}
\put(130,60){\circle*{8}}

\put(40,80){\line(1,2){30}}
\put(120,80){\circle*{8}}

\put(120,80){\line(-1,2){30}}
\put(70,140){\circle*{8}}

\put(30,60){\line(1,2){10}}
\put(20,80){\circle*{8}}

\put(40,80){\line(1,0){80}}
\put(120,80){\line(1,-2){10}}
\end{picture}
\qquad
\begin{picture}(160,160)(0,0)
\thicklines
\put(0,80){\line(1,0){20}}
\put(140,80){\line(1,0){20}}
\put(50,20){\line(1,0){60}}
\put(50,140){\line(1,0){60}}
\put(20,80){\line(1,2){30}}
\put(110,20){\line(1,2){30}}
\put(20,80){\line(1,-2){30}}
\put(110,140){\line(1,-2){30}}
\put(40,0){\line(1,2){10}}
\put(110,140){\line(1,2){10}}
\put(40,160){\line(1,-2){10}}
\put(110,20){\line(1,-2){10}}
\put(50,20){\circle*{8}}
\put(70,140){\circle*{8}}
\put(130,60){\circle*{8}}

\put(30,60){\line(1,2){40}}
\put(20,80){\circle*{8}}

\put(90,140){\line(1,-2){40}}
\put(110,140){\circle*{8}}

\put(40,40){\line(1,0){80}}
\put(120,40){\circle*{8}}
\end{picture}
\end{center}
\caption{Viewing a triangle move as a sequence of local moves.}
\label{fig:plabic-hexagon-3}
\vspace{-.15in}
\end{figure}

\begin{remark}
\label{rem:YB-via-plabic-fences}
The above argument can be recast in the language of plabic fences and associated 
words, cf.\ Section~\ref{sec:plabic-fences}. 
Look at the fragments of plabic graphs shown in Figure~\ref{fig:plabic-hexagon-3}
in the upper-left and lower-right corners. 
These fragments can be drawn as plabic fences on $3$~strands, 
see Figure~\ref{fig:YB--via-braids}. Their associated words are
$\sigma_1\tau_1\sigma_2\tau_2\sigma_1\tau_1$
and $\sigma_2\tau_2\sigma_1\tau_1\sigma_2\tau_2$, respectively. 
These are related to each other via switches 
of the form $\tau_j\sigma_i\leftrightarrow\sigma_i\tau_j$
combined with the braid relations 
$\sigma_1\sigma_2\sigma_1\leftrightarrow\sigma_2\sigma_1\sigma_2$ and 
$\tau_1\tau_2\tau_1\leftrightarrow\tau_2\tau_1\tau_2$. 
\end{remark}

\begin{figure}[ht]
\begin{center}
\setlength{\unitlength}{1.5pt} 
\begin{picture}(70,30)(0,-5)
\thicklines 
\put(5,0){\line(1,0){70}} 
\put(5,10){\line(1,0){70}} 
\put(5,20){\line(1,0){70}} 
\put(15,0){\line(0,1){10}} 
\put(25,0){\line(0,1){10}} 
\put(35,10){\line(0,1){10}} 
\put(45,10){\line(0,1){10}} 
\put(55,0){\line(0,1){10}} 
\put(65,0){\line(0,1){10}} 

\put(15,0){\circle*{2.5}}
\put(25,10){\circle*{2.5}}
\put(35,10){\circle*{2.5}}
\put(45,20){\circle*{2.5}}
\put(55,0){\circle*{2.5}}
\put(65,10){\circle*{2.5}}
\end{picture}  
\qquad\qquad
\begin{picture}(70,30)(0,-5)
\thicklines 
\put(5,0){\line(1,0){70}} 
\put(5,10){\line(1,0){70}} 
\put(5,20){\line(1,0){70}} 
\put(15,10){\line(0,1){10}} 
\put(25,10){\line(0,1){10}} 
\put(35,0){\line(0,1){10}} 
\put(45,0){\line(0,1){10}} 
\put(55,10){\line(0,1){10}} 
\put(65,10){\line(0,1){10}} 

\put(15,10){\circle*{2.5}}
\put(25,20){\circle*{2.5}}
\put(35,0){\circle*{2.5}}
\put(45,10){\circle*{2.5}}
\put(55,10){\circle*{2.5}}
\put(65,20){\circle*{2.5}}
\end{picture}  
\end{center}
\caption{Interpreting a triangle move in the language of fences.}
\label{fig:YB--via-braids}
\vspace{-.15in}
\end{figure}

Proposition~\ref{prop:braid-equivalence} and
Corollary~\ref{cor:plabic-isotopic} imply the following result. 


\begin{proposition}[{\rm \cite[Lemma~1.3]{couture-perron}}]
\label{prop:ybe-preserves-links}
\YB-equivalent divides are link equivalent. 
\end{proposition}

\begin{problem}
\label{problem:yb-equivalence-of morsif}
Let $\mathcal{D}$ denote the set of all (algebraic) divides 
coming from morsifications of the same real singularity. 
Are any two divides in this set \YB-equivalent?
If $D'$ is a divide \YB-equivalent to $D\in\mathcal{D}$,
does it follow that $D'\in\mathcal{D}$?
(Cf.\ Problem~\ref{problem:triangle-moves-and-algebraicity}.) 
\end{problem}

By Propositions~\ref{prop:ybe-via-mutations}, \ref{prop:braid-equivalence},
and~\ref{prop:ybe-preserves-links}, 
triangle moves preserve:
\begin{itemize}[leftmargin=.3in]
\item[{(d)}]
the isotopy class of the A'Campo link of a divide;
\item[{(q)}] 
the mutation class of the associated quiver; and
\item[{(p)}]
the move equivalence class of the associated plabic graphs. 
\end{itemize}
This means that once a connection between 
the statements \textbf{(d)}, \textbf{(q)}, and \textbf{(p)} 
appearing in Conjecture~\ref{conj:morsif=mut-via-divides} 
has been established for a particular class of divides,
it can be immediately extended to all divides \YB-equivalent to a divide in this class. 
With this in mind, we make the following definition. 

\begin{definition}
A divide is \emph{malleable} if it is \YB-equivalent to a scannable divide. 
\end{definition}

To illustrate, all divides in Figure~\ref{fig:transforming-into-scannable} 
are malleable. 

\begin{conjecture}
\label{conj:alg-is-malleable}
Every algebraic divide is malleable. 
\end{conjecture}

Theorem~\ref{th:malleable+positive-isotopic} below establishes
the implication \textbf{(d)}$\Rightarrow$\textbf{(p)} 
of Conjecture~\ref{conj:morsif=mut-via-divides}
under the assumptions of positive isotopy and malleability.  
Each of these assumptions is potentially redundant, 
cf.\ Conjectures~\ref{conj:positive-isotopy} and~\ref{conj:alg-is-malleable},
respectively. 

\begin{theorem}
\label{th:malleable+positive-isotopic}
Let $D_1$ and $D_2$ be (link equivalent) malleable divides
which are \hbox{\YB-equivalent} to scannable divides whose respective 
braids are positive-isotopic. \linebreak[3]
Then the associated plabic graphs 
$P_1\in\PP(D_1)$ and $P_2\in\PP(D_2)$ are move equivalent. 
Consequently, the quivers $Q(D_1)$ and~$Q(D_2)$ are mutation equivalent. 
\end{theorem}

\begin{proof}
Follows from Theorem~\ref{th:main-conjecture-scannable-1}
and Propositions~\ref{prop:braid-equivalence} and~\ref{prop:ybe-preserves-links}. 
\end{proof}

\begin{example}
Consider the three divides $D_1,D_2,D_3$ 
in the lower-right corner of Figure~\ref{fig:divides-quasihom},  
representing the morsifications of three different real forms 
of the quasi-homogeneous singularity $x^6+y^4=0$. 
The first two divides are scannable, with the same associated braid 
$\beta(D_1)=\beta(D_2)=\Delta^3$,
where $\Delta=(\sigma_1 \sigma_3 \sigma_2)^2=(\sigma_2\sigma_1 \sigma_3)^2$
is the positive half-twist. 
The divide~$D_3$ is not scannable but malleable:
it is \hbox{\YB-equivalent} to the scannable divide~$D$ 
shown in Figure~\ref{fig:transforming-into-scannable} 
at the right end of the bottom row. 
The braid associated with~$D$ is given by 
\[
\beta(D)=\sigma_1 \sigma_3 \sigma_2\sigma_1 \sigma_3 \sigma_2
\cdot \sigma_3\sigma_2\sigma_1  \sigma_3\sigma_2 \sigma_3
\cdot \sigma_2\sigma_1 \sigma_3\sigma_2\sigma_1 \sigma_3
=\Delta^3=\beta(D_1)=\beta(D_2). 
\]
This means (see Proposition~\ref{pr:equal-braids=>move-equivalent})
that the plabic graphs associated with $D_1$, $D_2$, and~$D$
(and hence~$D_3$) are pairwise move equivalent,
as asserted by Theorem~\ref{th:malleable+positive-isotopic}.
Furthermore 
the quivers $Q(D_1), Q(D_2), Q(D_3)$ are mutation equivalent to each other. 
\end{example}

\newpage

\section{Transversal overlays}
\label{sec:overlays}

In this section, we discuss operations which combine 
scannable divides to produce new divides, also scannable. 
We show that by changing the order in which these operations are applied,
one can obtain different scannable divides which are
link equivalent to each other, with positive-isotopic braids. 
By Corollary~\ref{cor:scannable+positive-isotopic}, 
it follows that the corresponding quivers are mutation equivalent. 

\begin{definition} 
\label{def:overlays}
Let $D_1$ and $D_2$ be two scannable divides, 
with $k_1$ and $k_2$ strands respectively, oriented on the coordinate plane 
as in Section~\ref{sec:scannable-divides}. 
Let us place $D_2$ above~$D_1$, so that their respective ambient rectangles have the same width. 
Stretch their strands horizontally near the right ends,
then bend these extensions up (for~$D_1$) and down (for~$D_2$), 
as shown in Figure~\ref{fig:overlay},
thereby creating $k_1 k_2$ new nodes in the form of a grid. 
The resulting divide $D_1\sharp D_2$, which is scannable by construction, 
is called the \emph{transversal overlay} of $D_1$ and~$D_2$. 

We note that this operation depends on a choice of ``scanning directions'' for the input divides. 
If two divides $D_1$ and $D_1'$ are isotopic to each other but have different scanning directions
(in particular, they might have a different number of strands,
cf.\ Figure~\ref{fig:isotopic-scanning}),
then the overlays $D_1\sharp D_2$ and $D_1'\sharp D_2$ do not have to be isotopic. 
\end{definition}

\begin{figure}[ht]
\begin{center}
\setlength{\unitlength}{1.5pt} 
\begin{picture}(50,50)(0,-5)
\thicklines 
\qbezier(5,2)(0,2)(0,4)
\qbezier(5,6)(0,6)(0,4)
\multiput(0,10)(0,4){3}{\line(1,0){5}} 
\multiput(0,32)(0,12){2}{\line(1,0){5}} 
\qbezier(5,36)(0,36)(0,38)
\qbezier(5,40)(0,40)(0,38)
\put(5,0){\line(1,0){40}} 
\put(5,20){\line(1,0){40}} 
\put(5,30){\line(1,0){40}} 
\put(5,46){\line(1,0){40}} 
\put(5,0){\line(0,1){20}} 
\put(45,0){\line(0,1){20}} 
\put(5,30){\line(0,1){16}} 
\put(45,30){\line(0,1){16}} 
\put(26,10){\makebox(0,0){$D_1$}}
\put(26,38){\makebox(0,0){$D_2$}}

\multiput(45,10)(0,4){1}{\line(1,0){4}} 
\multiput(45,40)(0,4){2}{\line(1,0){4}} 

\qbezier(45,2)(50,2)(50,4)
\qbezier(45,6)(50,6)(50,4)
\qbezier(45,14)(50,14)(50,16)
\qbezier(45,18)(50,18)(50,16)

\qbezier(45,32)(50,32)(50,34)
\qbezier(45,36)(50,36)(50,34)

\end{picture}  
\qquad\qquad\qquad
\begin{picture}(85,50)(0,-5)
\thicklines 
\qbezier(5,2)(0,2)(0,4)
\qbezier(5,6)(0,6)(0,4)
\multiput(0,10)(0,4){3}{\line(1,0){5}} 
\multiput(0,32)(0,12){2}{\line(1,0){5}} 
\qbezier(5,36)(0,36)(0,38)
\qbezier(5,40)(0,40)(0,38)

\put(5,0){\line(1,0){40}} 
\put(5,20){\line(1,0){40}} 
\put(5,30){\line(1,0){40}} 
\put(5,46){\line(1,0){40}} 
\put(5,0){\line(0,1){20}} 
\put(45,0){\line(0,1){20}} 
\put(5,30){\line(0,1){16}} 
\put(45,30){\line(0,1){16}} 
\put(26,10){\makebox(0,0){$D_1$}}
\put(26,38){\makebox(0,0){$D_2$}}
\multiput(45,2)(0,4){5}{\line(1,0){4}} 
\multiput(45,32)(0,4){4}{\line(1,0){4}} 
\multiput(49,2)(0,4){5}{\line(1,1){26}} 
\multiput(49,32)(0,4){4}{\line(1,-1.15){26}} 
\multiput(75,10)(0,4){2}{\line(1,0){5}} 
\qbezier(75,2)(80,2)(80,4)
\qbezier(75,6)(80,6)(80,4)
\put(75,36){\line(1,0){5}} 
\qbezier(75,28)(80,28)(80,30)
\qbezier(75,32)(80,32)(80,30)
\qbezier(75,40)(80,40)(80,42)
\qbezier(75,44)(80,44)(80,42)

\end{picture}  
\vspace{-.15in}
\end{center}
\caption{Two scannable divides, and their transversal overlay.}
\label{fig:overlay}
\end{figure}


\begin{remark}
\label{rem:overlay-YB}
The construction of Definition~\ref{def:overlays} allows for a multitude of modifications, obtained as follows. 
Stretch both divides $D_1$ and~$D_2$ very wide, then rotate one of them and overlay
on top of the other, making sure that each strand of~$D_1$ transversally intersects
each strand of~$D_2$ exactly once. 
All such overlays are \YB-equivalent to each other,
and consequently have equivalent (i.e., isotopic or move/mutation equivalent) 
links, braids, quivers, and plabic graphs. 
\end{remark}

\begin{remark}
\label{rem:overlay-sing}
Transversal overlays have a natural interpretation
in the context of plane curve singularities. 
To explain this, we shall make some simplifying assumptions which can in principle be relaxed. 
Let $(C,z)$ and~$(C',z)$ be two complex isolated plane curve singularities
with a common singular point~$z$. 
Suppose $C$ and~$C'$ are in general position at~$z$ with respect to each other,
i.e., $C$ and~$C'$ have no common tangents. Then the topological type of 
the overlay $(C\cup C',z)$ is canonically defined. 

Let $(C,z)$ be a real singularity, and $L$ a line transversal to it. 
Suppose that $(C_t)_{0\le t\le\zeta}$ is a real morsification of $(C,z)$ 
such that each curve $\RR C_t$,
for $0< t\le\zeta$, is scannable (in an appropriate neighborhood of~$z$) 
by a pencil of lines parallel~to~$L$. 
Each of these lines intersects $\RR C_t$ in the same number of points
equal to the multiplicity of~$(C,z)$. 
The direction of~$L$ corresponds to the vertical direction in Definition~\ref{def:scannable}. 

Now let $(C,z)$ and $(C',z)$ be real singularities, 
say of multiplicities $k$ and~$k'$, respectively. 
Suppose that each of them has a scannable morsification as above. 
Then the same is true for the (generic) transversal overlay $(C\cup C',z)$. 
In order to obtain a scannable morsification for $(C\cup C',z)$, 
let us shrink each of the two input morsifications along their respective transversal directions,
then overlay them at an angle, 
so that they intersect in $kk'$ points. 
The divide corresponding to the resulting morsification is obtained 
from the two input divides via the procedure outlined in Remark~\ref{rem:overlay-YB}. 
\end{remark}

\begin{definition}
We define the equivalence relation $\simplus$ on scannable divides as \linebreak[3]
follows. 
Let $D$ and $D'$ be two scannable divides with the same number of strands.
The notation $D\simplus D'$ means that the associated braids
$\beta(D)$ and $\beta(D')$ are positive-isotopic to each other inside the solid torus,
see Definition~\ref{def:positive-isotopy}. 
\end{definition}

The operation of transversal overlay of scannable divides descends to the level of
equivalence classes with respect to the equivalence relation~$\simplus$: 

\begin{lemma}
\label{lem:overlay-equivalence}
Let $D_1, D_2, D_1', D_2'$ be scannable divides. 
If $D_1\simplus D_1'$ and $D_2\simplus D_2'$, 
then $D_1\sharp D_2\simplus D_1'\sharp D_2'$.
\end{lemma}

\begin{proof}
Direct inspection shows that 
the positive braid $\beta(D_1\sharp D_2)$ associated with the
transversal overlay of two scannable divides $D_1$ and~$D_2$
is obtained by ``linking'' the braids $\beta(D_1)$ and~$\beta(D_2)$ as
shown in Figure~\ref{fig:link-overlay}. \linebreak[3]
The closure of $\beta(D_1\sharp D_2)$ 
is thus obtained by placing the closures of $\beta(D_1)$ and~$\beta(D_2)$ 
in the vicinities of the two components of a Hopf link. 
The claim follows. 
\end{proof}

\begin{figure}[ht]
\begin{center}
\setlength{\unitlength}{1.4pt} 
\begin{picture}(90,50)(0,-5)
\thicklines 

\put(0,6){\line(5,-4){5}} 
\put(0,40){\line(5,-4){5}} 
\multiput(0,2)(3.5,2.8){2}{\line(5,4){1.5}} 
\multiput(0,36)(3.5,2.8){2}{\line(5,4){1.5}} 

\multiput(0,10)(0,4){3}{\line(1,0){5}} 
\multiput(0,32)(0,12){2}{\line(1,0){5}} 

\put(5,0){\line(1,0){40}} 
\put(5,20){\line(1,0){40}} 
\put(5,30){\line(1,0){40}} 
\put(5,46){\line(1,0){40}} 
\put(5,0){\line(0,1){20}} 
\put(45,0){\line(0,1){20}} 
\put(5,30){\line(0,1){16}} 
\put(45,30){\line(0,1){16}} 
\put(26,10){\makebox(0,0){$D_1$}}
\put(26,38){\makebox(0,0){$D_2$}}

\put(50,0){\line(1,0){40}} 
\put(50,20){\line(1,0){40}} 
\put(50,30){\line(1,0){40}} 
\put(50,46){\line(1,0){40}} 
\put(50,0){\line(0,1){20}} 
\put(90,0){\line(0,1){20}} 
\put(50,30){\line(0,1){16}} 
\put(90,30){\line(0,1){16}} 

\multiput(45,10)(0,4){1}{\line(1,0){5}} 
\multiput(45,40)(0,4){2}{\line(1,0){5}} 

\put(45,6){\line(5,-4){5}} 
\multiput(45,2)(3.5,2.8){2}{\line(5,4){1.5}} 
\put(45,18){\line(5,-4){5}} 
\multiput(45,14)(3.5,2.8){2}{\line(5,4){1.5}} 

\put(45,36){\line(5,-4){5}} 
\multiput(45,32)(3.5,2.8){2}{\line(5,4){1.5}} 

\multiput(90,2)(0,4){5}{\line(1,0){4}} 
\multiput(90,32)(0,4){4}{\line(1,0){4}} 

\put(70,10){\makebox(0,0){\reflectbox{$D_1$}}}
\put(70,38){\makebox(0,0){\reflectbox{$D_2$}}}

\end{picture}  
\qquad\quad
\begin{picture}(158,50)(0,-5)
\thicklines 
\put(0,6){\line(5,-4){5}} 
\multiput(0,2)(3.5,2.8){2}{\line(5,4){1.5}} 
\put(0,40){\line(5,-4){5}} 
\multiput(0,36)(3.5,2.8){2}{\line(5,4){1.5}} 

\multiput(0,10)(0,4){3}{\line(1,0){5}} 
\multiput(0,32)(0,12){2}{\line(1,0){5}} 

\put(5,0){\line(1,0){40}} 
\put(5,20){\line(1,0){40}} 
\put(5,30){\line(1,0){40}} 
\put(5,46){\line(1,0){40}} 
\put(5,0){\line(0,1){20}} 
\put(45,0){\line(0,1){20}} 
\put(5,30){\line(0,1){16}} 
\put(45,30){\line(0,1){16}} 
\put(26,10){\makebox(0,0){$D_1$}}
\put(26,38){\makebox(0,0){$D_2$}}
\multiput(45,2)(0,4){5}{\line(1,0){4}} 
\multiput(45,32)(0,4){4}{\line(1,0){4}} 

\put(49,2){\line(1,1){13}} 
\put(49,6){\line(1,1){11.1}} 
\put(49,10){\line(1,1){9.2}} 
\put(49,14){\line(1,1){7.3}} 
\put(49,18){\line(1,1){5.4}} 

\put(75,28){\line(-1,-1){5.4}} 
\put(75,32){\line(-1,-1){7.3}} 
\put(75,36){\line(-1,-1){9.2}} 
\put(75,40){\line(-1,-1){11.1}} 
\put(75,44){\line(-1,-1){13}} 

\multiput(49,32)(0,4){4}{\line(1,-1.15){26}} 

\multiput(75,2)(0,4){4}{\line(1,0){4}} 
\multiput(75,28)(0,4){5}{\line(1,0){4}} 
\multiput(84,2)(0,4){4}{\line(1,0){4}} 
\multiput(84,28)(0,4){5}{\line(1,0){4}} 

\put(75,36){\line(1,0){5}} 

\put(79,6){\line(5,-4){5}} 
\multiput(79,2)(3.5,2.8){2}{\line(5,4){1.5}} 
\put(79,32){\line(5,-4){5}} 
\multiput(79,28)(3.5,2.8){2}{\line(5,4){1.5}} 
\put(79,44){\line(5,-4){5}} 
\multiput(79,40)(3.5,2.8){2}{\line(5,4){1.5}} 

\multiput(114,2)(0,4){5}{\line(-1,1){26}} 

\put(88,2){\line(1,1.15){11.2}} 
\put(88,6){\line(1,1.15){9.3}} 
\put(88,10){\line(1,1.15){7.4}} 
\put(88,14){\line(1,1.15){5.5}} 
\put(114,44){\line(-1,-1.15){11.2}} 
\put(114,40){\line(-1,-1.15){9.3}} 
\put(114,36){\line(-1,-1.15){7.4}} 
\put(114,32){\line(-1,-1.15){5.5}} 

\multiput(75,2)(0,4){4}{\line(1,0){4}} 
\multiput(75,28)(0,4){5}{\line(1,0){4}} 
\multiput(114,2)(0,4){5}{\line(1,0){4}} 
\multiput(114,32)(0,4){4}{\line(1,0){4}} 

\put(79,10){\line(1,0){5}} 
\put(79,14){\line(1,0){5}} 
\put(79,36){\line(1,0){5}} 

\put(118,0){\line(1,0){40}} 
\put(118,20){\line(1,0){40}} 
\put(118,30){\line(1,0){40}} 
\put(118,46){\line(1,0){40}} 
\put(118,0){\line(0,1){20}} 
\put(158,0){\line(0,1){20}} 
\put(118,30){\line(0,1){16}} 
\put(158,30){\line(0,1){16}} 

\multiput(158,2)(0,4){5}{\line(1,0){4}} 
\multiput(158,32)(0,4){4}{\line(1,0){4}} 

\put(138,10){\makebox(0,0){\reflectbox{$D_1$}}}
\put(138,38){\makebox(0,0){\reflectbox{$D_2$}}}
\end{picture}  
\end{center}
\caption{Positive braids associated with scannable divides $D_1$ and~$D_2$
(on the left) and with their transversal overlay (on the right).
}
\label{fig:link-overlay}
\end{figure}


\begin{lemma}
\label{lem:overlay-associative-commutative}
Transversal overlay of scannable divides 
is associative and commutative modulo the equivalence~$\simplus$. 
\end{lemma}

\begin{proof}
It is easy to see that transversal overlay 
is associative modulo \YB-equi\-valence. 
More precisely, if $D_1,D_2,D_3$ are scannable divides,
then the divides \linebreak[3] 
$(D_1\sharp D_2)\sharp D_3$ and $D_1\sharp (D_2\sharp D_3)$
are \YB-equivalent to each other. 
Moreover the corresponding braids are 
positive-isotopic inside the solid torus, cf.\ Remark~\ref{rem:YB-via-plabic-fences}. 

To prove commutativity, we need to show that the positive braid $\beta(D_1\sharp D_2)$
shown in Figure~\ref{fig:link-overlay} and the analogous 
braid~$\beta(D_2\sharp D_1)$ 
are positive-isotopic to each other inside the solid torus.
To see that, pull the strands of $\beta(D_1\sharp D_2)$ coming from~$\beta(D_1)$
(resp., from~$\beta(D_2)$) along the corresponding components of the Hopf link,
so that the fragments marked $D_1$ and \reflectbox{$D_1$} in Figure~\ref{fig:link-overlay}
get repositioned above the fragments marked $D_2$ and \reflectbox{$D_2$},
matching---up to a cyclic shift---the braid~$\beta(D_2\sharp D_1)$. 
\end{proof}

\begin{remark}
\label{rem:m-overlay}
Let $D_1,\dots,D_m$ be scannable divides. 
In view of Remark~\ref{rem:overlay-YB} and 
Lemmas~\ref{lem:overlay-equivalence}
and~\ref{lem:overlay-associative-commutative}, 
we can construct different 
versions of the transversal overlay 
of these $m$ divides, as follows. 
Take a generic planar arrangement of $m$ straight lines. 
Pick an arbitrary bijection between these lines and the given divides. 
Stretch each divide~$D_i$ along its scanning direction (horizontally,
in our usual rendering), 
then place $D_i$ near the corresponding line. 
For each~$i$, there are, generally speaking, four ways to do this, 
related by the action of the Klein four-group.
We assume that all these four versions yield the same result
modulo the equivalence~$\simplus$. 
This assumption is satisfied for all scannable divides arising from commonly used constructions, 
see Remark~\ref{rem:klein-group}. 

Different choices of a line arrangement, 
of an assignment of the divides $D_1,\dots,D_m$ to the lines in the arrangement,
and of a placement of each divide~$D_i$ along the corresponding line 
will produce different overlays of $D_1,\dots,D_m$. 
All of them will be $\simplus$-equi\-valent to each other. 
\end{remark}

\begin{remark}
\label{rem:evidence-via-overlays}
Iterated transversal overlays, as described in Remark~\ref{rem:m-overlay},
can be used to construct numerous examples in support of our main conjectures. 
As inputs, we take two $m$-tuples of scannable divides $D_1,\dots,D_m$
and $D'_1,\dots,D'_m$ such that, for $i=1,\dots,m$,
\begin{itemize}[leftmargin=.3in]
\item
the divides $D_i$ and $D_i'$ come from scannable morsifications of 
(potentially different real forms of) the same complex singularity,
as in Remark~\ref{rem:overlay-sing}; 
\item
the divides $D_i$ and $D_i'$ are $\simplus$-equivalent. 
\end{itemize}
(For example, one can take $D_i=D_i'$, 
or let $D_i$ and $D_i'$ be related 
by the action of a Klein group element, cf.\ Remark~\ref{rem:klein-group}.
In any case, the number of strands in both $D_i$ and $D_i'$
should be equal to the multiplicity of the singularity.) 
For each of these two $m$-tuples of divides,
we then construct some version of their transversal overlay,   
see Remark~\ref{rem:m-overlay}.  
This will produce a pair of divides $D$ and~$D'$ such that
\begin{itemize}[leftmargin=.3in]
\item
$D$ and $D'$ come from morsifications of (real forms of) the same complex singularity, 
see Remark~\ref{rem:overlay-sing}; 
consequently, $D$ and $D'$ are link equivalent; 
\item
$D$ and $D'$ are $\simplus$-equivalent, by Remark~\ref{rem:overlay-YB} and 
Lemmas~\ref{lem:overlay-equivalence}
and~\ref{lem:overlay-associative-commutative};
hence the quivers $Q(D)$ and $Q(D')$ are mutation equivalent,
by Corollary~\ref{cor:scannable+positive-isotopic}. 
\end{itemize}
Summing up, morsifications of the same complex singularity
obtained via different versions of transversal overlays, as described above,
give rise to mutation equivalent quivers, 
thereby providing evidence in support of Conjecture~\ref{conj:morsif=mut}. 
\end{remark} 

A number of concrete examples are given in Section~\ref{sec:lissajous}. 

\newpage

\section{Lissajous divides and wiring diagrams}
\label{sec:lissajous}

In this section,
we illustrate the construction of Section~\ref{sec:overlays}
using a class of divides coming from quasihomogeneous singularities. 
These examples show that mutation equivalence of
quivers obtained via different iterated transversal overlays
(as in Remark~\ref{rem:evidence-via-overlays})
may be far from obvious from a combinatorial standpoint. 

\pagebreak[3]

\begin{definition}
\label{def:lissajous}
A \emph{Lissajous divide} of type~$(a,b)$
(here $a\ge b\ge 2$)
is one of the two scannable divides on $b$ strands constructed as follows. 
Begin by drawing $b$ horizontal and $a$ vertical line segments 
in the form of a $(b-1)\times(a-1)$ grid.
Pick one of the two proper black-and-white (``checkerboard'') 
colorings of the $(b-1)(a-1)$ squares of the grid. 
Finally, replace each black square by a crossing~$\times$. 
See Figure~\ref{fig:divides-lissajous}. 
\end{definition}

\vspace{-.1in}

\begin{figure}[htbp] 
\begin{center} 
\begin{tabular}{c|c|c}
$x^5+y^3=0$
&
\setlength{\unitlength}{1.5pt} 
\begin{picture}(50,20)(-10,-5)
\put(0,5){\makebox(0,0){\usebox{\opening}}} 
\put(0,0){\makebox(0,0){\usebox{\ssone}}} 
\put(10,0){\makebox(0,0){\usebox{\sstwo}}} 
\put(20,0){\makebox(0,0){\usebox{\ssone}}} 
\put(30,0){\makebox(0,0){\usebox{\sstwo}}} 
\put(40,-5){\makebox(0,0){\usebox{\closing}}} 
\end{picture} 
&
\setlength{\unitlength}{1.5pt} 
\begin{picture}(50,20)(-10,-5)
\put(0,-5){\makebox(0,0){\usebox{\opening}}} 
\put(0,0){\makebox(0,0){\usebox{\sstwo}}} 
\put(10,0){\makebox(0,0){\usebox{\ssone}}} 
\put(20,0){\makebox(0,0){\usebox{\sstwo}}} 
\put(30,0){\makebox(0,0){\usebox{\ssone}}} 
\put(40,5){\makebox(0,0){\usebox{\closing}}} 
\end{picture} 
\\[.2in]
\hline
$x^5+y^4=0$
&
\setlength{\unitlength}{1.5pt} 
\begin{picture}(50,30)(-10,-5)
\put(0,0){\makebox(0,0){\usebox{\opening}}} 
\put(0,0){\makebox(0,0){\usebox{\sssonethree}}} 
\put(10,0){\makebox(0,0){\usebox{\ssstwo}}} 
\put(20,0){\makebox(0,0){\usebox{\sssonethree}}} 
\put(30,0){\makebox(0,0){\usebox{\ssstwo}}} 
\put(40,-10){\makebox(0,0){\usebox{\closing}}} 
\put(40,10){\makebox(0,0){\usebox{\closing}}} 
\end{picture}
&
\setlength{\unitlength}{1.5pt} 
\begin{picture}(50,30)(-10,-5)
\put(0,-10){\makebox(0,0){\usebox{\opening}}} 
\put(0,10){\makebox(0,0){\usebox{\opening}}} 
\put(10,0){\makebox(0,0){\usebox{\sssonethree}}} 
\put(0,0){\makebox(0,0){\usebox{\ssstwo}}} 
\put(30,0){\makebox(0,0){\usebox{\sssonethree}}} 
\put(20,0){\makebox(0,0){\usebox{\ssstwo}}} 
\put(40,0){\makebox(0,0){\usebox{\closing}}} 
\end{picture}
\\[.3in]
\hline
$x^6+y^3=0$
&
\setlength{\unitlength}{1.5pt} 
\begin{picture}(60,20)(-10,0)
\put(0,5){\makebox(0,0){\usebox{\opening}}} 
\put(0,0){\makebox(0,0){\usebox{\ssone}}} 
\put(10,0){\makebox(0,0){\usebox{\sstwo}}} 
\put(20,0){\makebox(0,0){\usebox{\ssone}}} 
\put(30,0){\makebox(0,0){\usebox{\sstwo}}} 
\put(40,0){\makebox(0,0){\usebox{\ssone}}} 
\put(50,5){\makebox(0,0){\usebox{\closing}}} 
\end{picture} 
&
\setlength{\unitlength}{1.5pt} 
\begin{picture}(60,20)(-10,0)
\put(0,-5){\makebox(0,0){\usebox{\opening}}} 
\put(10,0){\makebox(0,0){\usebox{\ssone}}} 
\put(0,0){\makebox(0,0){\usebox{\sstwo}}} 
\put(30,0){\makebox(0,0){\usebox{\ssone}}} 
\put(20,0){\makebox(0,0){\usebox{\sstwo}}} 
\put(40,0){\makebox(0,0){\usebox{\sstwo}}} 
\put(50,-5){\makebox(0,0){\usebox{\closing}}} 
\end{picture} 
\\[.3in]
\hline
$x^6+y^4=0$
&
\setlength{\unitlength}{1.5pt} 
\begin{picture}(60,30)(-10,-5)
\put(0,-10){\makebox(0,0){\usebox{\opening}}} 
\put(0,10){\makebox(0,0){\usebox{\opening}}} 
\put(0,0){\makebox(0,0){\usebox{\ssstwo}}} 
\put(10,0){\makebox(0,0){\usebox{\sssonethree}}} 
\put(20,0){\makebox(0,0){\usebox{\ssstwo}}} 
\put(30,0){\makebox(0,0){\usebox{\sssonethree}}} 
\put(40,0){\makebox(0,0){\usebox{\ssstwo}}} 
\put(50,-10){\makebox(0,0){\usebox{\closing}}} 
\put(50,10){\makebox(0,0){\usebox{\closing}}} 
\end{picture}
&
\setlength{\unitlength}{1.5pt} 
\begin{picture}(60,30)(-10,-5)
\put(0,0){\makebox(0,0){\usebox{\opening}}} 
\put(0,0){\makebox(0,0){\usebox{\sssonethree}}} 
\put(10,0){\makebox(0,0){\usebox{\ssstwo}}} 
\put(20,0){\makebox(0,0){\usebox{\sssonethree}}} 
\put(30,0){\makebox(0,0){\usebox{\ssstwo}}} 
\put(40,0){\makebox(0,0){\usebox{\sssonethree}}} 
\put(50,0){\makebox(0,0){\usebox{\closing}}} 
\end{picture}
\\[.2in]
\end{tabular}
\end{center} 
\caption{Lissajous divides. Cf.\ Figure~\ref{fig:divides-quasihom}.} 
\label{fig:divides-lissajous} 
\end{figure} 

\vspace{-.15in}

It is straightforward to check that the two Lissajous divides of type $(a,b)$ 
are $\simplus$-equivalent to each other. 

It is well known, and not hard to see, that a Lissajous divide of type~$(a,b)$
represents a morsification of (an appropriate real form of) 
the quasihomogeneous singularity $x^a+y^b=0$, 
cf.\ Definition~\ref{def:quasi-homogeneous}. 
The conditions of Remark~\ref{rem:evidence-via-overlays} are readily checked,
providing a large class of examples of pairs of divides 
(\emph{viz.}, overlays of Lissajous divides) whose associated quivers are---provably---mutation equivalent. 

This phenomenon is already combinatorially nontrivial 
for transversal overlays of singularities of types~$A_1$ and~$A_2$
(i.e., nodes and/or cusps).  
Let $D_1,\dots,D_m$ and $D'_1,\dots,D'_m$ be such that for each $i=1,\dots,m$, 
we have one of the following:
\begin{itemize}[leftmargin=.3in]
\item
each of $D_i$ and $D_i'$ is either an ellipse \,\rotatebox[origin=c]{90}{\texttt{O}}\, or a pair of lines
\rotatebox[origin=c]{90}{\texttt{X}}\,; or 
\item
each of $D_i$ and $D_i'$ is a nodal cubic $\propto\,$. 
\end{itemize}
When we construct transversal overlays for each of the two input $m$-tuples, 
we have the freedom of rotating each morsified cusp ~$\propto\,$ by~$180^\circ$,
and more importantly,
the freedom of choosing the cyclic ordering of the ingredient divides; 
this ordering determines the placement of their stretched versions near the $m$~lines 
of a planar line arrangement. 
All the resulting divides will have mutation equivalent quivers. 

\begin{remark}
\label{rem:lissajous-grassmannian}
The readers familiar with cluster algebras will recognize 
the quiver associated to a Lissajous divide of type~$(a,b)$
as the quiver defining the standard cluster structure on 
the corresponding \emph{Pl\"ucker ring}, 
the homogeneous coordinate ring of the Grassmannian
$\operatorname{Gr}_{a,a+b}(\CC)$, see~\cite{scott}. 
This suggests the existence of an intrinsic connection
between the quasihomogeneous complex singularity $x^a+y^b=0$ 
and the standard cluster structure on the affine cone over $\operatorname{Gr}_{a,a+b}(\CC)$. 
One algebraic interpretation of this connection (via additive categorification) 
was proposed in~\cite{jensen-king-su}. 
It would be very interesting to find an explanation of this connection
that directly relates Pl\"ucker rings, viewed as cluster algebras, 
to quasihomogeneous singularities (viewed up to topological equivalence). 
\end{remark}

\begin{definition}
\label{def:wiring-a-ac}
An \emph{alternating wiring diagram} of type~$(a,b)$ 
(here $a,b\ge 2$) is a scannable divide
constructed via the following modification of Definition~\ref{def:lissajous} \linebreak[3]
Draw $b$ horizontal and $a+1$ vertical line segments 
forming a grid of size $(b-1)\times a$.
Pick a checkerboard coloring of the squares of the grid. 
Replace each black square by a crossing~$\times$. 
Finally, remove all the remaining vertical segments. 
See Figure~\ref{fig:divides-wiring}. 
\end{definition}

\begin{figure}[htbp] 
\begin{center} 
\begin{tabular}{c|c}
$x^4+y^4=0$
&
\setlength{\unitlength}{1.4pt} 
\begin{picture}(50,25)(-10,-5)
\put(0,0){\makebox(0,0){\usebox{\sssonethree}}} 
\put(10,0){\makebox(0,0){\usebox{\ssstwo}}} 
\put(20,0){\makebox(0,0){\usebox{\sssonethree}}} 
\put(30,0){\makebox(0,0){\usebox{\ssstwo}}} 
\end{picture}
\\[.3in]
\hline
$x^6+y^3=0$
&
\setlength{\unitlength}{1.4pt} 
\begin{picture}(60,20)(-5,0)
\put(0,0){\makebox(0,0){\usebox{\ssone}}} 
\put(10,0){\makebox(0,0){\usebox{\sstwo}}} 
\put(20,0){\makebox(0,0){\usebox{\ssone}}} 
\put(30,0){\makebox(0,0){\usebox{\sstwo}}} 
\put(40,0){\makebox(0,0){\usebox{\ssone}}} 
\put(50,0){\makebox(0,0){\usebox{\sstwo}}} 
\end{picture} 
\\[.2in]
\end{tabular}
\end{center} 
\caption{Alternating wiring diagrams of types $(4,4)$, and $(6,3)$. 
Cf.\ Figure~\ref{fig:divides-quasihom}.
} 
\label{fig:divides-wiring} 
\end{figure} 

\vspace{-.15in}

In~the special case $a=b$, we get a wiring diagram with
$b$ branches (``pseudolines'') \linebreak[3]
each pair of which intersect each other exactly once. 
This wiring diagram represents a morsification of 
the real singularity $\prod_{i=1}^b (y-r_i x)=0$, with distinct real slopes~$r_i$. 
More generally, 
when $b$ divides~$a$, an alternating wiring diagram of type~$(a,b)$ 
represents a quasihomogeneous singularity of the same type.
(For arbitrary $a$ and~$b$, this can be false.) 
To be more precise, one can show that an alternating wiring diagram of type~$(bc,b)$
represents a morsification of a quasihomogeneous singularity of type~$(bc,b)$,
specifically its real form involving $b$~smooth branches
each pair of which have a tangency of order~$c$. 
(For $c=1$, each pair is transversal.) 
It is also straightforward to check that this wiring diagram is 
$\simplus$-equivalent to a Lissajous divide of type~$(bc,b)$. 
We~can consequently use this wiring diagram as a replacement for 
a Lissajous divide of type~$(bc,b)$ in any transversal overlay.


\newpage

\section{Oriented plabic graphs and their links}
\label{sec:plabic-graphs-and-links}

In this section, we present an alternative approach to the theory of divide links. \linebreak[3]
Unlike the Gibson-Ishikawa/Kawamura constructions 
(cf.\ Definition~\ref{def:L(P)=L(o(P))}),
this approach utilizes conventional link diagrams, so it is 
more explicit and combinatorial. \linebreak[3]
Unfortunately, diagrammatic constructions based on uniform local rules appear to 
only exist for scannable divides (resp., the corresponding class of plabic graphs),
hence the limitations of the method. 
On the other hand, many morsifications arising via commonly used constructions
are scannable (cf.\ Sections~\ref{sec:overlays}--\ref{sec:lissajous}), 
so the technique described in this section has a fairly wide applicability. 

Recall that in 
Theorem~\ref{th:main-conjecture-scannable-1}, 
we showed, modulo some technical assumptions, 
that link equivalence of scannable divides
implies move equivalence of associated plabic graphs (fences). 
In this section, we provide a combinatorial proof of the converse implication, 
again modulo some (conjecturally redundant) assumptions. 
The end result (cf.\ Proposition~\ref{pr:main-conjecture-scannable-2})
is thus substantially weaker than Corollary~\ref{cor:plabic-isotopic}, 
which required neither those assumptions nor scannability. 
Nevertheless, we opted to include this section in the paper
since the machinery developed herein is much simpler 
(though less powerful) than the one employed in 
Section~\ref{sec:links-from-plabic};
it also suggests connections with 
Postnikov's theory of perfect orientations of plabic graphs,
cf.\ Remark~\ref{rem:perfect-orientations}.  

\begin{definition}
\label{def:admissible-orientation}
Let $P$ be a plabic graph. 
An orientation of the edges of $P$ is called \emph{admissible} 
if it satisfies the following requirements:
\begin{itemize}[leftmargin=.3in]
\item[(B)]
at each trivalent black vertex, two edges are incoming, and one outgoing;\\
at each univalent black vertex, one edge is incoming; 
\item[(W)]
at each trivalent white vertex,  two edges are outgoing, and one is incoming;\\
at each univalent white vertex, one edge is outgoing; 
\item[(F)]
the edges at the boundary of each internal face of $P$ form a directed graph (an
orientation of a cycle) with exactly one source and one sink. 
\end{itemize}
\end{definition}

\begin{remark}
\label{rem:perfect-orientations}
In Postnikov's original treatment~\cite{postnikov}, 
the orientations satisfying conditions (B) and~(W) 
in Definition~\ref{def:admissible-orientation}
(for the interior trivalent vertices) 
were called \emph{perfect}. 
For our purposes~however, ``perfect'' is not enough, as we do need condition~(F), 
or~some variant thereof. 
(One acceptable way to relax condition~(F) is to only forbid two types of
orientations of the boundaries of internal faces, namely (a) oriented cycles
and (b)~orientations with two sources and two sinks.) 
\end{remark}


\begin{lemma}
\label{lem:admissible-acyclic}
An admissible orientation of a plabic graph is acyclic. 
\end{lemma}


\begin{proof}
Suppose not. Let $C$ be an oriented cycle in an admissible orientation.
We~may assume that $C$ is simple 
(i.e., it does not visit the same vertex more than once)
and moreover $C$~does not enclose another oriented cycle. 
If $C$ encloses a single face, then we are in contradiction with condition~(F). 
Otherwise $C$ contains a vertex~$v$ incident to an edge~$e$ located inside~$C$. 
Assume that $e$ is oriented away from~$v$, as the other case is similar. 
Starting with~$e$, keep moving in the direction of the orientation.
When arriving at a black vertex, make the unique choice of the outgoing edge; 
at a white vertex, choose any of the two outgoing edges. 
Eventually, this walk will either hit itself or hit~$C$, 
thereby creating an oriented cycle enclosed by~$C$, a contradiction. 
\end{proof}

While an admissible orientation does not have to exist, it is always unique: 

\begin{proposition}
\label{pr:admissible-unique}
A plabic graph has at most one admissible orientation. 
\end{proposition}

(This statement relies on the convention that we do color boundary vertices.) 

\begin{proof} 
Let $O_1$ and $O_2$ be distinct admissible orientations~of 
the same plabic graph~$P$. 
At each univalent boundary vertex, the two orientations co\-incide by
conditions (B) and~(W). 
At a trivalent vertex~$v$ in the interior~of~$P$, they either coincide
for all three edges, 
or else they coincide at one edge, and are opposite 
at the remaining two edges $e_1$ and~$e_2$;
moreover $e_1$ and~$e_2$ form an oriented two-edge path in both $O_1$ and~$O_2$. 
It~follows that the edges whose orientations in $O_1$ and~$O_2$ differ from each other form 
a collection of disjoint oriented cycles 
(with different orientations in $O_1$ and in~$O_2$). 
This contradicts Lemma~\ref{lem:admissible-acyclic}. 
\end{proof}

A plabic graph is 
\emph{balanced} if it contains equally many black and white vertices.

To illustrate, any plabic fence (cf.\ Definition~\ref{def:plabic-fence}) 
is balanced, assuming its boundary vertices located at the left
(resp., right) end are colored white (resp., black), 
as in Figure~\ref{fig:fence-word}. 

\begin{lemma}
\label{lem:admissible-is-balanced}
Any plabic graph possessing an admissible orientation is balanced. 
\end{lemma}

\begin{proof}
In an admissible orientation, 
at each white (resp., black) vertex, whether internal or located on the boundary,
the number of outgoing edges minus 
the number of incoming edges is equal to~$1$ (resp.,~$-1$).
Summing over all the half-edges, we obtain the claim. 
\end{proof}

\begin{remark}
The converse to Lemma~\ref{lem:admissible-is-balanced} is false: 
a balanced plabic graph does not have to allow an admissible orientation. 
A~counterexample is shown in Figure~\ref{fig:no-admissible orientation}. 
\end{remark}

\begin{figure}[ht]
\begin{center}
\setlength{\unitlength}{1pt}
\begin{picture}(60,30)(0,5)
\thinlines
\put(-1,0){\line(0,1){27.5}}
\put(1,0){\line(0,1){27.5}}
\put(-1,32.5){\line(0,1){7.5}}
\put(1,32.5){\line(0,1){7.5}}
\put(59,0){\line(0,1){7.5}}
\put(61,0){\line(0,1){7.5}}
\put(59,12.5){\line(0,1){27.5}}
\put(61,12.5){\line(0,1){27.5}}
\thicklines
\put(2.5,10){\line(1,0){15}}
\put(2.5,30){\line(1,0){15}}
\put(22.5,10){\line(1,0){15}}
\put(22.5,30){\line(1,0){15}}
\put(20,12.5){\line(0,1){15}}
\put(40,12.5){\line(0,1){15}}
\put(42.5,10){\line(1,0){15}}
\put(42.5,30){\line(1,0){15}}
\put(0,10){\circle*{5}}
\put(0,30){\circle{5}}
\put(60,10){\circle{5}}
\put(60,30){\circle*{5}}
\put(20,10){\circle{5}}
\put(20,30){\circle*{5}}
\put(40,10){\circle*{5}}
\put(40,30){\circle{5}}
\end{picture}
\end{center}
\caption{A balanced plabic graph that does not have an admissible orientation.
(It does have perfect orientations, but they do not satisfy condition~(F).)}
\label{fig:no-admissible orientation}
\end{figure}


\begin{remark}
\label{rem:staying-balanced}
The property of being balanced is preserved by both flip and square moves, 
cf.\ Figure~\ref{fig:plabic-moves}(a,b).  
For it to be preserved by a tail attachment/removal move, 
one needs to require that the two vertices involved in the move 
have opposite colors, as in Figure~\ref{fig:plabic-moves}(c). 
Similarly, a switch transformation (see Definition~\ref{def:switch})
preserves balanceness provided the portion of the plabic graph 
being flipped over
(cf.\ Figure~\ref{fig:isthmus-move} inside the dotted line) contains
an equal number of black and white vertices. 
\end{remark}

\begin{definition}
Two balanced plabic graphs are 
\emph{move equivalent through balanced plabic graphs}
if they are related to each other by a sequence of local moves
which only involve balanced plabic graphs. 
As noted in Remark~\ref{rem:staying-balanced},
Postnikov's flip and square moves preserve the property of being balanced. 
\end{definition}

\begin{proposition}
\label{pr:move-admissible-orientation}
If two balanced plabic graphs are move equivalent through 
balanced plabic graphs, and one of them has a
(necessarily unique) admissible orientation,
then so does the other. 
\end{proposition}

\begin{proof}
All we need to do is check 
that for each type of local move transforming a balanced
plabic graph~$P_1$ into another balanced plabic graph~$P_2$ 
(see Figure~\ref{fig:plabic-moves}), 
we can transport an admissible orientation of~$P_1$ 
into an admissible orientation of~$P_2$.
This is demonstrated in Figure~\ref{fig:orientations-via-moves}. 
It~is straightforward to verify that each of these local transformations preserves the conditions 
(B), (W), and~(F) of Definition~\ref{def:admissible-orientation}. 
We note that Figure~\ref{fig:orientations-via-moves}
shows all possible edge orientations 
(up to isotopy, rotation, and/or reflection)
which are consistent with these conditions. 
\end{proof}

\begin{figure}[ht]
\begin{tabular}{ll}
\setlength{\unitlength}{1pt}
\begin{picture}(60,40)(0,0)
\put(20,20){\makebox(0,0){(a)}}
\end{picture}
&
\setlength{\unitlength}{1pt}
\begin{picture}(40,40)(0,0)
\thicklines
\put(0,0){\line(2,1){20}}
\put(20,30){\line(2,1){20}}
\put(0,40){\line(2,-1){20}}
\put(20,10){\line(2,-1){20}}
\put(20,10){\line(0,1){20}}
\put(0,0){\vector(2,1){15}}
\put(40,0){\vector(-2,1){15}}
\put(0,40){\vector(2,-1){15}}
\put(20,30){\vector(2,1){15}}
\put(20,10){\vector(0,1){15}}
\put(20,10){\circle*{5}}
\put(20,30){\circle*{5}}
\end{picture}
\begin{picture}(40,40)(0,0)
\put(20,20){\makebox(0,0){$\longleftrightarrow$}}
\end{picture}
\begin{picture}(40,40)(0,0)
\thicklines
\put(0,0){\line(1,2){10}}
\put(30,20){\line(1,2){10}}
\put(40,0){\line(-1,2){10}}
\put(10,20){\line(-1,2){10}}
\put(10,20){\line(1,0){20}}
\put(0,0){\vector(1,2){7.5}}
\put(0,40){\vector(1,-2){7.5}}
\put(10,20){\vector(1,0){15}}
\put(30,20){\vector(1,2){7.5}}
\put(40,0){\vector(-1,2){7.5}}

\put(10,20){\circle*{5}}
\put(30,20){\circle*{5}}
\end{picture}
\begin{picture}(60,40)(0,0)
\end{picture}
\begin{picture}(40,40)(0,0)
\thicklines
\put(0,0){\line(2,1){18}}
\put(40,40){\line(-2,-1){18}}
\put(0,40){\line(2,-1){18}}
\put(40,0){\line(-2,1){18}}
\put(20,12.5){\line(0,1){15}}
\put(20,25){\vector(0,-1){10}}
\put(16,8){\vector(-2,-1){10}}
\put(24,8){\vector(2,-1){10}}
\put(36,38){\vector(-2,-1){10}}
\put(16,32){\vector(-2,1){10}}

\put(20,10){\circle{5}}
\put(20,30){\circle{5}}
\end{picture}
\begin{picture}(40,40)(0,0)
\thicklines
\put(20,20){\makebox(0,0){$\longleftrightarrow$}}
\end{picture}
\begin{picture}(40,40)(0,0)
\thicklines
\put(0,0){\line(1,2){9}}
\put(40,40){\line(-1,-2){9}}
\put(40,0){\line(-1,2){9}}
\put(0,40){\line(1,-2){9}}
\put(12.5,20){\line(1,0){15}}
\put(25,20){\vector(-1,0){10}}
\put(8,24){\vector(-1,2){5}}
\put(8,16){\vector(-1,-2){5}}
\put(32,16){\vector(1,-2){5}}
\put(38,36){\vector(-1,-2){5}}

\put(10,20){\circle{5}}
\put(30,20){\circle{5}}
\end{picture}
\\[.2in]
\setlength{\unitlength}{1pt}
\begin{picture}(60,40)(0,0)
\put(20,20){\makebox(0,0){(b)}}
\end{picture}
&
\setlength{\unitlength}{1pt}
\begin{picture}(40,40)(0,0)
\thicklines
\put(0,0){\line(1,1){8.5}}
\put(40,40){\line(-1,-1){8.5}}
\put(0,40){\line(1,-1){8.5}}
\put(40,0){\line(-1,1){8.5}}
\put(10,12.5){\line(0,1){15}}
\put(30,12.5){\line(0,1){15}}
\put(12.5,10){\line(1,0){15}}
\put(12.5,30){\line(1,0){15}}
\put(0,0){\vector(1,1){7}}
\put(0,40){\vector(1,-1){7}}
\put(33,7){\vector(1,-1){6}}
\put(33,33){\vector(1,1){6}}
\put(10,25){\vector(0,-1){9}}
\put(15,10){\vector(1,0){9}}
\put(15,30){\vector(1,0){9}}
\put(30,15){\vector(0,1){9}}

\put(10,10){\circle*{5}}
\put(10,30){\circle{5}}
\put(30,10){\circle{5}}
\put(30,30){\circle*{5}}
\end{picture}
\begin{picture}(40,40)(0,0)
\thicklines
\put(20,20){\makebox(0,0){$\longleftrightarrow$}}
\end{picture}
\begin{picture}(40,40)(0,0)
\thicklines
\put(0,0){\line(1,1){8.5}}
\put(40,40){\line(-1,-1){8.5}}
\put(0,40){\line(1,-1){8.5}}
\put(40,0){\line(-1,1){8.5}}
\put(10,12.5){\line(0,1){15}}
\put(30,12.5){\line(0,1){15}}
\put(12.5,10){\line(1,0){15}}
\put(12.5,30){\line(1,0){15}}
\put(0,0){\vector(1,1){7}}
\put(0,40){\vector(1,-1){7}}
\put(33,7){\vector(1,-1){6}}
\put(33,33){\vector(1,1){6}}
\put(15,10){\vector(1,0){9}}
\put(15,30){\vector(1,0){9}}
\put(10,15){\vector(0,1){9}}
\put(30,25){\vector(0,-1){9}}

\put(10,10){\circle{5}}
\put(10,30){\circle*{5}}
\put(30,10){\circle*{5}}
\put(30,30){\circle{5}}
\end{picture}
\\[0in]
\setlength{\unitlength}{1pt}
\begin{picture}(60,40)(0,0)
\put(20,20){\makebox(0,0){(c)}}
\end{picture}
&
\setlength{\unitlength}{1pt}
\begin{picture}(40,30)(0,0)
\thinlines
\put(0,6){\line(1,0){18}}
\put(0,4){\line(1,0){18}}
\put(40,6){\line(-1,0){18}}
\put(40,4){\line(-1,0){18}}
\thicklines

\put(0,25){\line(1,0){40}}
\put(20,25){\line(0,-1){17.5}}
\put(20,8){\vector(0,1){12}}
\put(5,25){\vector(1,0){10}}
\put(25,25){\vector(1,0){10}}

\put(20,25){\circle*{5}}
\put(20,5){\circle{5}}
\put(-12,5){\makebox(0,0){$\partial\Disk$}}
\end{picture}
\begin{picture}(40,30)(0,0)
\thicklines
\put(20,15){\makebox(0,0){$\longleftrightarrow$}}
\end{picture}
\begin{picture}(40,30)(0,0)
\thinlines
\put(0,6){\line(1,0){40}}
\put(0,4){\line(1,0){40}}
\thicklines
\put(0,25){\line(1,0){40}}
\put(15,25){\vector(1,0){10}}
\end{picture}
\begin{picture}(40,30)(0,0)
\thicklines
\put(20,15){\makebox(0,0){$\longleftrightarrow$}}
\end{picture}
\begin{picture}(40,30)(0,0)
\thinlines
\put(0,6){\line(1,0){40}}
\put(0,4){\line(1,0){40}}
\put(20,5){\circle*{5}}
\thicklines
\put(0,25){\line(1,0){17.5}}
\put(40,25){\line(-1,0){17.5}}
\put(20,5){\line(0,1){17.5}}
\put(20,22){\vector(0,-1){12}}
\put(5,25){\vector(1,0){10}}
\put(25,25){\vector(1,0){10}}
\put(20,25){\circle{5}}
\put(52,5){\makebox(0,0){$\partial\Disk$}}
\end{picture}
\end{tabular}
\caption{Transporting admissible orientations
via moves in plabic graphs.}
\label{fig:orientations-via-moves}
\end{figure}

\pagebreak[3]

\vspace{-.15in}

\begin{proposition}
\label{pr:fence-is-orientable}
Any plabic fence has an admissible orientation. 
\end{proposition}

\begin{proof}
Orient all horizontal edges of a plabic fence left to right, and orient each vertical edge
from the white vertex to the black one. 
See Figure~\ref{fig:fence-orientation}. 
\end{proof}

\begin{figure}[ht] 
\begin{center} 
\setlength{\unitlength}{2.5pt} 
\begin{picture}(100,30)(-10,-5)
\thinlines
\put(-5,-5){\line(0,1){30}} 
\put(-5.5,-5){\line(0,1){30}} 
\put(85,-5){\line(0,1){30}} 
\put(85.5,-5){\line(0,1){30}} 
\multiput(85.25,0)(0,10){3}{\circle*{1.5}}
\thicklines 
\put(-5,0){\line(1,0){90}} 
\put(-5,10){\line(1,0){90}} 
\put(-5,20){\line(1,0){90}} 
\put(5,10){\line(0,1){10}} 
\put(15,0){\line(0,1){10}} 
\put(25,0){\line(0,1){10}} 
\put(35,10){\line(0,1){10}} 
\put(45,10){\line(0,1){10}} 
\put(55,0){\line(0,1){10}} 
\put(65,0){\line(0,1){10}} 
\put(75,10){\line(0,1){10}} 

\put(5,0){\vector(1,0){3}} 
\put(20,0){\vector(1,0){3}} 
\put(40,0){\vector(1,0){3}} 
\put(60,0){\vector(1,0){3}} 
\put(75,0){\vector(1,0){3}} 
\multiput(0,10)(10,0){9}{\vector(1,0){3}} 
\multiput(0,20)(40,0){3}{\vector(1,0){3}} 
\multiput(20,20)(40,0){2}{\vector(1,0){3}} 
\put(15,6){\vector(0,-1){3}} 
\put(55,6){\vector(0,-1){3}} 
\put(25,4){\vector(0,1){3}} 
\put(65,4){\vector(0,1){3}} 
\put(5,16){\vector(0,-1){3}} 
\put(45,16){\vector(0,-1){3}} 
\put(75,16){\vector(0,-1){3}} 
\put(35,14){\vector(0,1){3}}

\put(5,10){\circle*{1.5}}
\put(15,0){\circle*{1.5}}
\put(25,10){\circle*{1.5}}
\put(35,20){\circle*{1.5}}
\put(45,10){\circle*{1.5}}
\put(55,0){\circle*{1.5}}
\put(65,10){\circle*{1.5}}
\put(75,10){\circle*{1.5}}
\end{picture}  
\end{center} 
\vspace{-.05in}
\caption{An admissible orientation of a plabic fence. 
} 
\label{fig:fence-orientation} 
\end{figure} 

\vspace{-.15in}

\begin{definition}
\label{def:link-of-oriented-plabic-graph}
Let $P$ be a (necessarily balanced) plabic graph allowing 
a (necessarily unique) 
admissible orientation. 
The (oriented) \emph{link $L^\circ(P)$ associated with~$P$} is defined 
in terms of a link diagram constructed as follows. 
Replace each edge of~$P$ by a pair of parallel strands, oriented according to the
``drive on the right side'' rule. 
Connect these strands at each internal trivalent vertex of~$P$ according to the recipe shown in 
Figure~\ref{fig:black-to-link}. 
Finally, at each univalent boundary vertex of~$P$, connect the two strands to each other. 
\end{definition}

Definition~\ref{def:link-of-oriented-plabic-graph} is illustrated, for the special case of plabic fences, 
in Figures~\ref{fig:fence-to-link}--\ref{fig:link-via-fence}. 

\begin{figure}[ht]
\vspace{.1in}
\begin{center}
\setlength{\unitlength}{0.8pt}
\begin{picture}(70,80)(0,0)
\thicklines
\put(10,0){\line(1,2){20}}
\put(10,80){\line(1,-2){20}}
\put(30,40){\line(1,0){40}}
\put(10,0){\vector(1,2){15}}
\put(10,80){\vector(1,-2){15}}
\put(30,40){\vector(1,0){25}}
\put(30,40){\circle*{6}}
\end{picture}
\qquad
\begin{picture}(70,80)(0,0)
\thicklines
\put(18,36){\line(-1,-2){14}}
\put(4,72){\line(1,-2){21}}
\put(25,30){\line(1,0){45}}
\put(20,0){\line(1,2){13}}
\put(40,40){\line(-1,2){20}}
\put(40,40){\line(-1,-2){3}}
\put(70,50){\line(-1,0){31}}
\put(25,50){\line(1,0){6}}
\put(25,50){\line(-1,-2){3}}
\put(4,72){\vector(1,-2){10}}
\put(20,0){\vector(1,2){10}}
\put(70,50){\vector(-1,0){20}}
\put(18,36){\vector(-1,-2){10}}
\put(40,40){\vector(-1,2){15}}
\put(25,30){\vector(1,0){35}}
\end{picture}
\qquad
\qquad\quad
\begin{picture}(70,80)(0,0)
\thicklines
\put(10,0){\line(1,2){18.5}}
\put(10,80){\line(1,-2){18.5}}
\put(33,40){\line(1,0){37}}
\put(25,30){\vector(-1,-2){5}}
\put(25,50){\vector(-1,2){5}}
\put(48,40){\vector(-1,0){5}}
\put(30,40){\circle{6}}
\end{picture}
\qquad
\begin{picture}(70,80)(5,0)
\thicklines
\put(20,40){\line(-1,-2){16}}
\put(20,40){\line(-1,2){16}}
\put(35,50){\line(1,0){35}}
\put(35,50){\line(-1,2){15}}
\put(35,30){\line(1,0){35}}
\put(35,30){\line(-1,-2){15}}
\put(70,50){\line(-1,0){31}}

\put(4,72){\vector(1,-2){10}}
\put(20,0){\vector(1,2){10}}
\put(70,50){\vector(-1,0){20}}
\put(18,36){\vector(-1,-2){10}}
\put(35,50){\vector(-1,2){10}}
\put(35,30){\vector(1,0){25}}
\end{picture}
\end{center}
\caption{Building a link diagram around a black, resp.\ white, vertex
of a plabic graph carrying an admissible orientation.
Cf.\ Figure~\ref{fig:plabic-to-o-divide}.}
\label{fig:black-to-link}
\end{figure}

\newsavebox{\slink}
\setlength{\unitlength}{6pt} 
\savebox{\slink}(5,10)[bl]{
\thicklines 
\put(0,0){\line(1,0){1.5}} 
\put(0,8){\line(1,0){1.5}} 
\put(1.5,8){\line(0,-1){7}} 
\put(1.5,8){\vector(0,-1){4}} 
\put(3.5,1){\vector(0,1){4}} 
\put(1.5,1){\line(2,-1){2}} 
\put(1.5,0){\line(2,1){0.5}} 
\put(3.5,1){\line(-2,-1){0.5}} 
\put(3.5,1){\line(0,1){7}} 
\put(3.5,8){\line(1,0){1.5}} 
\put(3.5,0){\line(1,0){1.5}} 

\put(1,2){\red{\line(-1,0){1.5}}}
\put(1,2){\red{\vector(-1,0){1.5}}}
\put(3,2){\red{\line(-1,0){1}}}
\put(5,2){\red{\line(-1,0){1}}}
\put(5,10){\red{\line(-1,0){5}}}
\put(5,10){\red{\vector(-1,0){3}}}
}

\newsavebox{\tlink}
\setlength{\unitlength}{6pt} 
\savebox{\tlink}(5,10)[bl]{
\thicklines 
\put(0,0){\line(1,0){5}} 
\put(0,0){\vector(1,0){3}} 
\put(0,8){\line(1,0){5}} 
\put(0,8){\vector(1,0){3.2}} 

\put(0,2){\red{\line(1,0){1.5}}}
\put(3.5,2){\red{\line(1,0){1.5}}}
\put(0,10){\red{\line(1,0){1.5}}}
\put(3.5,10){\red{\line(1,0){1.5}}}
\put(1.5,7.5){\red{\line(0,-1){5.5}}}
\put(1.5,7.5){\red{\vector(0,-1){3.5}}}
\put(1.5,8.5){\red{\line(0,1){0.5}}}
\put(3.5,8.5){\red{\line(0,1){0.5}}}
\put(1.5,9){\red{\line(2,1){0.5}}}
\put(3.5,10){\red{\line(-2,-1){0.5}}}
\put(1.5,10){\red{\line(2,-1){2}}}
\put(3.5,2){\red{\line(0,1){5.5}}}
\put(3.5,2){\red{\vector(0,1){3.5}}}
}

\newsavebox{\strack}
\setlength{\unitlength}{6pt} 
\savebox{\strack}(5,10)[bl]{
\thicklines 
\put(0,0){\line(1,0){5}} 

\put(5,2){\red{\line(-1,0){5}}}
}

\newsavebox{\tslink}
\setlength{\unitlength}{6pt} 
\savebox{\tslink}(10,10)[bl]{
\thicklines 
\put(0,0){\line(1,0){7}} 
\put(0,0){\vector(1,0){5}} 
\put(0,8){\line(1,0){7}} 
\put(0,8){\vector(1,0){5}} 
\put(7,8){\line(0,-1){7}} 
\put(7,8){\vector(0,-1){4}} 
\put(7,1){\line(2,-1){2}} 
\put(7,0){\line(2,1){0.5}} 
\put(9,1){\line(-2,-1){0.5}} 
\put(9,1){\line(0,1){7}} 
\put(9,8){\line(1,0){1}} 
\put(9,0){\line(1,0){1}} 

\put(0,2){\red{\line(1,0){1}}}
\put(0,10){\red{\line(1,0){1}}}
\put(1,7.5){\red{\line(0,-1){5.5}}}
\put(1,7.5){\red{\vector(0,-1){3.5}}}
\put(1,8.5){\red{\line(0,1){0.5}}}
\put(3,8.5){\red{\line(0,1){0.5}}}
\put(1,9){\red{\line(2,1){0.5}}}
\put(3,10){\red{\line(-2,-1){0.5}}}
\put(1,10){\red{\line(2,-1){2}}}
\put(3,2){\red{\line(0,1){5.5}}}
\put(3,2){\red{\vector(0,1){3.5}}}
\put(10,10){\red{\line(-1,0){7}}}
\put(10,10){\red{\vector(-1,0){4}}}
\put(6.5,2){\red{\line(-1,0){3.5}}}
\put(6.5,2){\red{\vector(-1,0){2.5}}}
\put(8.5,2){\red{\line(-1,0){1}}}
\put(10,2){\red{\line(-1,0){0.5}}}
}

\begin{figure}[ht] 
\vspace{-.1in}
\begin{center} 
\begin{tabular}{c||c|c|c|c}
\setlength{\unitlength}{3pt} 
\begin{picture}(10,10)(0,-2)
\put(5,7){\makebox(0,0){plabic}}
\put(5,3){\makebox(0,0){fence}}
\end{picture}
& \setlength{\unitlength}{3pt} 
\begin{picture}(10,15)(0,-2)
\thicklines 
\put(0,0){\line(1,0){10}} 
\put(0,10){\line(1,0){10}} 
\put(5,0){\line(0,1){10}} 
\put(5,0){\circle*{1.5}}
\end{picture}  
&
\setlength{\unitlength}{3pt} 
\begin{picture}(10,10)(0,-2)
\thicklines 
\put(0,0){\line(1,0){10}} 
\put(0,10){\line(1,0){10}} 
\put(5,0){\line(0,1){10}} 
\put(5,10){\circle*{1.5}}
\end{picture} 
&
\setlength{\unitlength}{3pt} 
\begin{picture}(15,10)(0,-2)
\thicklines 
\put(0,0){\line(1,0){15}} 
\put(0,10){\line(1,0){15}} 
\put(4,0){\line(0,1){10}} 
\put(11,0){\line(0,1){10}} 
\put(4,0){\circle*{1.5}}
\put(11,10){\circle*{1.5}}
\end{picture} 
&
\setlength{\unitlength}{3pt} 
\begin{picture}(15,10)(0,-2)
\thicklines 
\put(0,0){\line(1,0){15}} 
\put(0,10){\line(1,0){15}} 
\put(4,0){\line(0,1){10}} 
\put(11,0){\line(0,1){10}} 
\put(4,10){\circle*{1.5}}
\put(11,0){\circle*{1.5}}
\end{picture} 
\\
\hline
&&&& \\
\setlength{\unitlength}{6pt} 
\begin{picture}(4,10)(0,0)
\put(2,5){\makebox(0,0){link}} 
\end{picture} 
&
\setlength{\unitlength}{6pt} 
\begin{picture}(8,10)(0,0)
\put(4,5){\makebox(0,0){\usebox{\slink}}} 
\end{picture} 
&
\setlength{\unitlength}{6pt} 
\begin{picture}(8,10)(0,0)
\put(4,5){\makebox(0,0){\usebox{\tlink}}} 
\end{picture} 
&
\setlength{\unitlength}{6pt} 
\begin{picture}(13,10)(0,0)
\put(4,5){\makebox(0,0){\usebox{\slink}}} 
\put(9,5){\makebox(0,0){\usebox{\tlink}}} 
\end{picture} 
&
\setlength{\unitlength}{6pt} 
\begin{picture}(13,10)(0,0)
\put(4,5){\makebox(0,0){\usebox{\tlink}}} 
\put(9,5){\makebox(0,0){\usebox{\slink}}} 
\end{picture} 
\end{tabular}
\end{center} 
\caption{Transforming a plabic fence $\Phi$ into a link diagram for $L^\circ(\Phi)$. 
} 
\label{fig:fence-to-link} 
\end{figure} 

\begin{figure}[ht] 
\vspace{-.1in}
\begin{center} 
\setlength{\unitlength}{6pt} 
\begin{picture}(40,19)(-3,0)
\put(0,13){\makebox(0,0){\usebox{\slink}}} 
\put(5,5){\makebox(0,0){\usebox{\slink}}} 
\put(10,5){\makebox(0,0){\usebox{\tlink}}} 
\put(15,13){\makebox(0,0){\usebox{\tlink}}} 
\put(20,13){\makebox(0,0){\usebox{\slink}}} 
\put(25,5){\makebox(0,0){\usebox{\slink}}} 
\put(30,5){\makebox(0,0){\usebox{\tlink}}} 
\put(35,13){\makebox(0,0){\usebox{\slink}}} 

\put(0,5){\makebox(0,0){\usebox{\strack}}} 
\put(15,5){\makebox(0,0){\usebox{\strack}}} 
\put(20,5){\makebox(0,0){\usebox{\strack}}} 
\put(35,5){\makebox(0,0){\usebox{\strack}}} 

\put(5,21){\makebox(0,0){\usebox{\strack}}} 
\put(10,21){\makebox(0,0){\usebox{\strack}}} 
\put(25,21){\makebox(0,0){\usebox{\strack}}} 
\put(30,21){\makebox(0,0){\usebox{\strack}}}

\put(2.5,21){\makebox(0,0){\usebox{\uturnleft}}} 
\put(2.5,13){\makebox(0,0){\usebox{\uturnleft}}} 
\put(2.5,5){\makebox(0,0){\usebox{\uturnleft}}} 

\put(42.5,21){\makebox(0,0){\usebox{\uturnright}}} 
\put(42.5,13){\makebox(0,0){\usebox{\uturnright}}} 
\put(42.5,5){\makebox(0,0){\usebox{\uturnright}}} 
\end{picture} 
\end{center} 
\caption{The link $L^\circ(\Phi)$ for the plabic fence~$\Phi$ 
in Figure~\ref{fig:fence-orientation};  
cf.\ Figure~\ref{fig:link-scannable}. 
} 
\label{fig:link-via-fence} 
\end{figure} 


\pagebreak[3]

For plabic fences associated to scannable divides, we recover the A'Campo links:  

\begin{proposition}
\label{pr:L(D)=L(Phi)}
Let $D$ be a scannable divide, and $\Phi=\Phi(D)$ the corresponding plabic fence,
cf.\ Definition~\ref{def:Phi(D)}.
Then the links $L(D)=L(\Phi)$ and $L^\circ(\Phi)$ are isotopic to each other. 
\end{proposition}

\begin{proof}
Compare the constructions of the links $L(D)$ and $L^\circ(\Phi)$
given in Theorem~\ref{th:palindromic} (due to O.~Couture and B.~Perron)
and Definition~\ref{def:link-of-oriented-plabic-graph}, respectively. 
\end{proof}

\begin{remark}
\label{rem:L^o=L}
More generally, for any plabic graph~$P$ supporting an admissible orientation,
the link $L^\circ(P)$ constructed in Definition~\ref{def:link-of-oriented-plabic-graph} 
is isotopic to the Kawamura link~$L(P)$ from Definition~\ref{def:L(P)=L(o(P))}. 
Since we will not rely on this statement, we shall only sketch an argument
justifying it. 

  The construction of Definition~\ref{def:link-of-oriented-plabic-graph} 
 naturally produces a link in the solid torus $UT\Disk$, the unit tangent
  bundle of the disk~$\Disk$. (Collapsing the boundary of $UT\Disk$
  yields a link in $\mathbf{S}^3$.) Links in the solid torus are a bit 
  difficult to visualize, especially since rotating the tangent vector
  of an oriented divide~$\vec D$ moves a short distance in the disk while the
  resulting link makes a full revolution around the solid torus. 
  Now suppose that $\vec v$ 
  is a nonzero vector field on~$\Disk$ such that $\vec D$ is
  never tangent to~$\vec v$, with the same orientation. Then $\vec v$ gives a section of
  $UT\Disk$, a meridional disk~$M_{\vec v}$ cutting across the solid
  torus. 
  The  complement  $UT\Disk \setminus M_{\vec v}$ is homeomorphic to $\RR^3$, 
  and it is much easier to draw pictures there. To concretely draw
  $L(\vec D)$ in this space, given a  point $(x, w) \in L(\vec D)$, two of
  the three coordinates of its image in $\RR^3\cong UT\Disk$
  are the coordinates of~$x$ in~$\Disk$, \linebreak[3]
  and the third coordinate
  (which in our pictures would be perpendicular to the page) measures the
  clockwise rotation from $\vec v(x)$ to~$w$. 
  Thus, to convert a crossing in~$\vec D$ into a crossing in
  the traditional link diagram for~$L(\vec D)$, 
  we designate as the bottom strand the one whose
  tangent you reach first when rotating clockwise from~$\vec v(x)$.
  (Hirasawa's technique for visualizing divides \cite{hirasawa} can be
  viewed as the special case of this construction, with $\vec v$ the constant
  vector field pointing to the left, and with additional elaboration for
  cases when the link passes through the forbidden set~$M_{\vec v}$.)

From an admissible orientation of a plabic graph~$P$, we can
  construct a non-vanishing vector field $\vec w$, arranged to be
  parallel to each oriented edge (with the same orientation) except
  near the vertices, and turning
  a minimal amount at each vertex.
  Condition~(F) ensures that $\vec w$ extends continuously across the
  faces. (To extend a non-zero vector field across a disk, it suffices
  to check that the index on the boundary is~$0$, which is exactly
  what Condition~(F) ensures.)
  New let $\vec v$ be the clockwise $90^\circ$ rotation of~$\vec w$. 
  By construction, away from univalent vertices,
  $\vec v$ is never parallel with the same orientation to the oriented
  divide~$o(P)$ from
  Definition~\ref{def:L(P)=L(o(P))}. 

  We need to see that the link diagram in Figure~\ref{fig:black-to-link}
  represents $L(o(P))$. 
  Since $o(P)$ is not tangent to~$\vec v$ away from the boundary, we can use the rule above
  to establish which strand is on top at the crossings of
  $o(P)$ near black vertices. The result differs by a Reidemeister III move from
  Figure~\ref{fig:black-to-link}.
  Finally, near the univalent vertices lying on $\partial\Disk$, the conventions from
  Figure~\ref{fig:plabic-boundary-to-o-divide} produce an oriented divide
  $o(P)$ that does pass through the forbidden set~$M_{\vec v}$. We can fix
  this by applying a boundary U-turn move to $o(P)$ near each boundary
  univalent
  vertex. This introduces extra crossings near the univalent
  vertices; these crossings can be removed with Reidemeister~I~moves.
\end{remark}

In light of Remark~\ref{rem:L^o=L}, it comes as no surprise that 
local moves on oriented plabic graphs preserve the associated links,
cf.\ Proposition~\ref{cor:plabic-isotopic}: 

\begin{proposition}
\label{pr:moves-preserve-links}
Let $P_1$ and $P_2$ be balanced plabic graphs which are move equivalent through balanced plabic graphs. 
Suppose that one of them 
(hence the other, see Proposition~\ref{pr:move-admissible-orientation}) 
has an admissible orientation. 
Then the links $L^\circ(P_1)$ and $L^\circ(P_2)$ are isotopic to each other. 
\end{proposition}

\begin{proof}
We need to check that each type of local move 
shown in Figure~\ref{fig:orientations-via-moves} preserves the isotopy type~of
the link $L^\circ(P)$ associated with a plabic graph~$P$ 
carrying an admissible orientation. 
For a flip move involving two white vertices, the statement is clear 
(no strands cross each other).
The~case of a flip move involving two black vertices is shown in Figure~\ref{fig:black-flip-link}.
For the square move, examine the last two columns of Figure~\ref{fig:fence-to-link} 
and verify that the isotopy type of the link does not change. 
Finally, the case of tail attachment/removal is treated in Figure~\ref{fig:tail-link}. 
\end{proof}

\begin{figure}[ht]
\begin{center}
\setlength{\unitlength}{0.8pt}
\begin{picture}(100,90)(0,5)
\thicklines
\put(30,56){\line(0,-1){17}}
\put(30,56){\vector(0,-1){14}}

\put(30,70){\line(0,-1){6}}
\put(30,70){\line(2,1){6}}
\put(70,90){\line(-2,-1){26}}
\put(70,90){\vector(-2,-1){21}}

\put(30,25){\line(0,1){6}}
\put(30,25){\line(2,-1){6}}
\put(44,18){\line(2,-1){26}}
\put(44,18){\vector(2,-1){21}}

\put(40,40){\line(-2,-1){40}}
\put(40,40){\vector(-2,-1){27}}

\put(40,40){\line(2,-1){6}}

\put(80,20){\line(-2,1){26}}
\put(80,20){\vector(-2,1){16}}

\put(50,25){\line(-2,-1){40}}
\put(10,5){\vector(2,1){21}}

\put(50,25){\line(0,1){31}}
\put(50,25){\vector(0,1){26}}

\put(50,70){\line(-2,1){40}}
\put(50,70){\vector(-2,1){27}}
\put(50,70){\line(0,-1){6}}

\put(0,75){\line(2,-1){40}}
\put(0,75){\vector(2,-1){24}}
\put(40,55){\line(2,1){40}}
\put(40,55){\vector(2,1){27}}

\end{picture}
\qquad\qquad
\setlength{\unitlength}{0.8pt}
\begin{picture}(100,80)(0,0)
\thicklines
\put(18,36){\line(-1,-2){14}}
\put(4,72){\line(1,-2){21}}
\put(25,30){\line(1,0){45}}
\put(20,0){\line(1,2){13}}
\put(40,40){\line(-1,2){20}}
\put(40,40){\line(-1,-2){3}}
\put(70,50){\line(-1,0){31}}
\put(25,50){\line(1,0){6}}
\put(25,50){\line(-1,-2){3}}
\put(4,72){\vector(1,-2){10}}
\put(20,0){\vector(1,2){10}}
\put(70,50){\vector(-1,0){24}}
\put(18,36){\vector(-1,-2){10}}
\put(40,40){\vector(-1,2){15}}
\put(25,30){\vector(1,0){31}}

\put(69,30){\line(1,2){21}}
\put(69,30){\vector(1,2){15}}
\put(91,8){\line(-1,2){15}}
\put(91,8){\vector(-1,2){10}}
\put(70,50){\line(1,-2){3}}
\put(75,80){\line(-1,-2){13}}
\put(75,80){\vector(-1,-2){10}}
\put(55,40){\line(1,2){3}}
\put(55,40){\line(1,-2){3}}
\put(62,26){\line(1,-2){13}}
\put(62,26){\vector(1,-2){8}}
\end{picture}
\end{center}
\caption{Link diagrams for two oriented plabic graphs
related via a flip move involving two black vertices.
The~two links are isotopic to each other.}
\label{fig:black-flip-link}
\end{figure}

\begin{figure}[ht]
\begin{center}
\setlength{\unitlength}{1.2pt}
\begin{picture}(50,40)(0,0)
\thicklines
\put(0,30){\line(1,0){20}}
\put(0,30){\vector(1,0){15}}
\put(50,40){\vector(-1,0){28}}
\put(30,30){\line(1,0){20}}
\put(30,30){\vector(1,0){15}}
\put(0,40){\line(1,0){50}}
\put(20,30){\line(0,-1){20}}
\put(20,30){\vector(0,-1){15}}
\put(30,10){\vector(0,1){13}}
\put(30,30){\line(0,-1){20}}
\qbezier(20,10)(20,5)(25,5)
\qbezier(30,10)(30,5)(25,5)
\end{picture}
\qquad\qquad
\begin{picture}(50,40)(0,0)
\thicklines
\put(0,30){\line(1,0){50}}
\put(0,30){\vector(1,0){30}}
\put(50,40){\vector(-1,0){30}}
\put(0,40){\line(1,0){50}}
\end{picture}
\qquad\qquad
\begin{picture}(50,40)(0,0)
\thicklines
\put(0,30){\line(1,0){50}}
\put(0,30){\vector(1,0){13}}
\put(30,30){\vector(1,0){13}}
\put(0,40){\line(1,0){20}}

\put(50,40){\line(-1,0){20}}
\put(50,40){\vector(-1,0){15}}
\put(20,40){\vector(-1,0){15}}

\put(20,40){\line(2,-1){10}}
\put(30,40){\line(-2,-1){3}}
\put(20,35){\line(2,1){3}}
\put(20,35){\line(0,-1){3}}
\put(30,35){\line(0,-1){3}}
\put(20,28){\line(0,-1){18}}
\put(20,28){\vector(0,-1){13}}
\qbezier(20,10)(20,5)(25,5)
\qbezier(30,10)(30,5)(25,5)
\put(30,10){\vector(0,1){13}}
\put(30,10){\line(0,1){18}}

\end{picture}
\end{center}
\caption{Transforming the link diagram under a tail removal/attachment.}
\label{fig:tail-link}
\end{figure}

\begin{proposition}
\label{pr:main-conjecture-scannable-2}
Let $D_1$ and $D_2$ be scannable divides 
whose respective plabic fences $\Phi(D_1)$ and $\Phi(D_2)$ 
are move equivalent through balanced plabic graphs. 
Then $D_1$ and $D_2$ are link equivalent. 
\end{proposition}

\begin{proof} 
By Proposition~\ref{pr:fence-is-orientable},
the plabic fences $\Phi(D_1)$ and $\Phi(D_2)$ have admissible orientations. 
Since they are move equivalent through balanced plabic graphs, 
their links are isotopic, by Proposition~\ref{pr:moves-preserve-links}. 
Moreover by Proposition~\ref{pr:L(D)=L(Phi)} 
these links are isotopic to the respective A'Campo links $L(D_1)$ and~$L(D_2)$,
and we are done. 
\end{proof}

\begin{remark}
The following example illustrates the difficulties involved in extending the arguments
presented in this section to arbitrary (not necessarily
scannable) algebraic divides. 
Let $D$ be the non-scannable divide shown in Figure~\ref{fig:transforming-into-scannable} on the upper left (or in Figure~\ref{fig:non-partitions}(a)). 
It is not hard to verify that any plabic graph $P\in\PP(D)$ (see Definition~\ref{def:plabic-divide}) 
does not have an admissible orientation. 
(To see this, examine the portion of~$P$ corresponding to the vicinity of the rightmost node in~$D$.) 
Consequently $P$ cannot be obtained from a plabic graph 
attached to a scannable algebraic divide of the same type 
(see e.g.\ Figure~\ref{fig:transforming-into-scannable} on the lower right) 
via local moves through balanced plabic graphs.
This means that some of the moves in Figure~\ref{fig:transforming-into-scannable},
when translated into the language of plabic graphs, must pass through non-balanced graphs. 
(Note that the plabic graph for any divide can be chosen to be balanced.) 
\end{remark}



\section{Simple singularities}
\label{sec:simple-singularities}

In this section, we verify that the main construction of this paper
provides a direct link between two \textit{ADE}
classifications: (i)~Arnold's classification of simple sin\-gu\-lar\-ities and
(ii)~the classification of skew-symmetric cluster algebras of finite type. 
We~begin by reviewing these two classifications, starting with the latter. 

\smallskip

Any quiver~$Q$ 
gives rise to a \emph{cluster algebra} $\mathcal{A}(Q)$ 
(with trivial coefficients and a skew-symmetric exchange matrix). 
The cluster algebra $\mathcal{A}(Q)$ only depends 
on the mutation equivalence class of~$Q$.  
It is generated inside a field of rational functions
in several variables by a distinguished set of generators 
called \emph{cluster variables};
see, e.g.,~\cite[Chapter~3]{fwz}. \linebreak[3]
The~cluster algebra~$\mathcal{A}(Q)$ (or the quiver~$Q$) 
is said to have \emph{finite type} if this generating set is finite. 
While cluster algebras as such do not play a role in our arguments,
we would like to recall the following classification result. 

\begin{theorem}[{\rm \cite{ca2}}]
A connected quiver~$Q$ has finite type 
if and only if $Q$ is mutation equivalent to an orientation of a
simply-laced Dynkin diagram. 
\end{theorem}

Orientations of two Dynkin diagrams are mutation equivalent
if and only if the (unoriented) diagrams are isomorphic. 
Thus each connected quiver of finite type has a well defined \emph{type} 
in the standard $ADE$ nomenclature of the simply-laced Dynkin diagrams. 
These types are $A_n$ ($n\ge 1$), 
$D_n$ ($n\ge 4$), and $E_n$ ($6\le n\le 8$). 

A singularity is called \emph{simple} if any two curves in its
topological equivalence class are locally diffeomorphic. 
Thus, simple singularities are those for which 
the topological classification coincides with the
contact analytic one,
i.e., the classification up to local diffeomorphism in the source
and multiplication by a unit of the local ring. 

A classical result by V.~Arnold 
asserts that simple singularities are classified by 
the simply-laced Dynkin diagrams.
In the language of Section~\ref{sec:ag-diagrams}, Arnold's classification 
can be stated as follows. 

\begin{theorem}[{\rm \cite{arnold-classification}}]
A plane curve singularity is simple if and only if it has a real morsification 
whose unlabeled \AG-diagram is a simply-laced Dynkin diagram. 
\end{theorem}

Furthermore, every \textit{ADE} type comes up as a type of a plane curve singularity,  
cf.\ Figures~\ref{fig:divides-quasihom} and~\ref{fig:divides-D5-D6-E7}. 
See \cite[Section~I.2.4]{GLS} 
for a detailed treatment. 

\pagebreak[3]

A quiver corresponding to an 
arbitrary real morsification of a simple singularity does not have to 
be an orientation of an \textit{ADE} Dynkin diagram, in the traditional Lie-theoretic 
sense of the term; see for example Figure~\ref{fig:E6-morsifications} on the left. 
What, then, distinguishes these quivers among those arising from 
real morsifications of general singularities? 
The answer is suggested by Conjecture~\ref{conj:morsif=mut}:

\begin{theorem}
\label{th:simple-singularities}
Let $(C,z)$ be a plane curve singularity, 
as in Section~\ref{sec:singularities-and-morsifications}.
Let $D$ be a divide associated with its real morsification,
and $Q=Q(D)$ the corresponding quiver.
Then the following are equivalent:
\begin{itemize}[leftmargin=.3in]
\item
$(C,z)$ is a simple singularity; 
\item
$Q$ is a quiver of finite type. 
\end{itemize}
Furthermore, if these statements hold, then the type of the singularity $(C,z)$
matches the type of the quiver~$Q$. 
\end{theorem}

To rephrase, Theorem~\ref{th:simple-singularities} (proved below in this section)
asserts that a plane curve singularity is simple if and only if 
a quiver arising from some (equivalently,~any) real morsification 
of this singularity gives rise to a cluster algebra of finite type. 
Moreover, the type of this cluster algebra matches the type of the singularity. 

\begin{example}
A quasihomogeneous singularity of type~$(a,b)$, $a\ge b\ge 2$, 
is simple if and only if $b=2$ or else $b=3$ and $a\le 5$.
These are precisely the cases in which the $(a-1)\times(b-1)$ 
grid quiver defines a cluster algebra of finite type.
Cf.\ Remark~\ref{rem:lissajous-grassmannian}.
\end{example}

\begin{remark}
Most quivers of finite type do not arise from any morsification
(nor from any divide, for that matter). 
The complete list of the quivers of type \textit{ADE} which do arise in this way 
can be obtained from 
Proposition~\ref{pr:ade-morsifications} below. 
\end{remark}

Theorem~\ref{th:simple-singularities}
implies that Conjecture~\ref{conj:morsif=mut} holds for simple singularities:

\begin{corollary}
Given two real morsifications of isolated plane curve
singularities, one of which is known to be simple,
the following are equivalent:
\begin{itemize}[leftmargin=.3in]
\item
the two singularities 
have the same complex topological type; 
\item
the quivers associated with the two morsifications are
mutation equivalent. 
\end{itemize}
\end{corollary}

Our proof of Theorem~\ref{th:simple-singularities} will rely on the 
classification of morsifications of \textit{ADE} singularities, 
see 
J.~Callahan \cite[Section~4]{callahan} and J.~Chislenko \cite[Section~4]{chislenko}. 

\begin{proposition}[{\rm \cite{callahan, chislenko}}]
\label{pr:ade-morsifications}
Divides arising from real morsifications of simple singularities
are classified, up to a homeomorphism of an ambient disk, as follows: 
\begin{itemize}[leftmargin=.3in]
\item
for each type $A_n$ ($n$ odd), there are two possible divides,
see Figure~\ref{fig:morsifications-An}, {\rm (i)--(ii)}; 
\item
for each type $A_n$ ($n$ even), there is one possible divide,
see Figure~\ref{fig:morsifications-An}, {\rm (iii)}; 
\item
for each type~$D_n$, there are $\lfloor \frac{n}{2}\rfloor$ possible divides,
see Figure~\ref{fig:morsifications-Dn};  
\item
for type~$E_6$, there are two possible divides, see Figure~\ref{fig:E6-morsifications};
\item
for type~$E_7$, there are two possible divides, see Figure~\ref{fig:divides-D5-D6-E7};
\item
for type~$E_8$, there are three possible divides, see Figure~\ref{fig:divides-quasihom}. 
\end{itemize}
\end{proposition}

\begin{figure}[ht]
\begin{center}
\hspace{-.5in}
\setlength{\unitlength}{1pt}
\begin{picture}(100,20)(0,-5)
\thinlines
\put(15,5){\line(1,1){5}}
\put(15,15){\line(1,-1){5}}
\put(85,5){\line(-1,1){5}}
\put(85,15){\line(-1,-1){5}}
\qbezier(20,10)(30,20)(40,10)
\qbezier(40,10)(50,20)(60,10)
\qbezier(60,10)(70,20)(80,10)
\qbezier(20,10)(30,0)(40,10)
\qbezier(40,10)(50,0)(60,10)
\qbezier(60,10)(70,0)(80,10)
\put(50,-8){\makebox(0,0){(i)}}
\end{picture}
\begin{picture}(100,20)(0,-5)
\thinlines
\qbezier(40,10)(50,20)(60,10)
\qbezier(60,10)(70,20)(80,10)
\qbezier(80,10)(85,15)(90,15)
\qbezier(95,10)(95,15)(90,15)
\qbezier(40,10)(35,15)(30,15)
\qbezier(25,10)(25,15)(30,15)

\qbezier(40,10)(50,0)(60,10)
\qbezier(60,10)(70,0)(80,10)
\qbezier(80,10)(85,5)(90,5)
\qbezier(95,10)(95,5)(90,5)
\qbezier(40,10)(35,5)(30,5)
\qbezier(25,10)(25,5)(30,5)
\put(60,-8){\makebox(0,0){(ii)}}
\end{picture}
\begin{picture}(100,20)(-20,-5)
\thinlines
\put(15,5){\line(1,1){5}}
\put(15,15){\line(1,-1){5}}
\qbezier(20,10)(30,20)(40,10)
\qbezier(40,10)(50,20)(60,10)
\qbezier(60,10)(70,20)(80,10)
\qbezier(80,10)(85,15)(90,15)
\qbezier(95,10)(95,15)(90,15)
\qbezier(20,10)(30,0)(40,10)
\qbezier(40,10)(50,0)(60,10)
\qbezier(60,10)(70,0)(80,10)
\qbezier(80,10)(85,5)(90,5)
\qbezier(95,10)(95,5)(90,5)
\put(55,-8){\makebox(0,0){(iii)}}
\end{picture}
\end{center}
\caption{Divides arising from real morsifications of type~$A_n$ singularities. 
For $n$ odd, there are two isotopy classes, shown in (i)--(ii), for $n=7$. 
For $n$ even, there is one isotopy class, shown in~(iii), for $n=8$. 
}
\label{fig:morsifications-An}
\end{figure}
\nopagebreak


\begin{figure}[ht]
\begin{center}
\hspace{-.5in}
\setlength{\unitlength}{1pt}
\begin{picture}(100,15)(0,3)
\thinlines
\put(15,5){\line(1,1){5}}
\put(15,15){\line(1,-1){5}}
\put(85,5){\line(-1,1){5}}
\put(85,15){\line(-1,-1){5}}
\qbezier(20,10)(30,20)(40,10)
\qbezier(40,10)(50,20)(60,10)
\qbezier(60,10)(70,20)(80,10)
\qbezier(20,10)(30,0)(40,10)
\qbezier(40,10)(50,0)(60,10)
\qbezier(60,10)(70,0)(80,10)
\put(30,0){\line(0,1){20}}
\end{picture}
\begin{picture}(100,15)(0,3)
\thinlines
\qbezier(40,10)(50,20)(60,10)
\qbezier(60,10)(70,20)(80,10)
\qbezier(80,10)(85,15)(90,15)
\qbezier(95,10)(95,15)(90,15)
\qbezier(40,10)(35,15)(30,15)
\qbezier(25,10)(25,15)(30,15)

\qbezier(40,10)(50,0)(60,10)
\qbezier(60,10)(70,0)(80,10)
\qbezier(80,10)(85,5)(90,5)
\qbezier(95,10)(95,5)(90,5)
\qbezier(40,10)(35,5)(30,5)
\qbezier(25,10)(25,5)(30,5)
\put(50,0){\line(0,1){20}}
\end{picture}
\begin{picture}(100,15)(-20,3)
\thinlines
\put(15,5){\line(1,1){5}}
\put(15,15){\line(1,-1){5}}
\qbezier(20,10)(30,20)(40,10)
\qbezier(40,10)(50,20)(60,10)
\qbezier(60,10)(70,20)(80,10)
\qbezier(80,10)(85,15)(90,15)
\qbezier(95,10)(95,15)(90,15)
\qbezier(20,10)(30,0)(40,10)
\qbezier(40,10)(50,0)(60,10)
\qbezier(60,10)(70,0)(80,10)
\qbezier(80,10)(85,5)(90,5)
\qbezier(95,10)(95,5)(90,5)
\put(30,0){\line(0,1){20}}
\end{picture}
\end{center}
\caption{Divides arising from real morsifications of type~$D_n$ singularities. 
Each of them is obtained as a transversal overlay of 
a real morsification of a type $A_{n-3}$ singularity 
and a single smooth branch 
(shown as a~vertical line). 
}
\label{fig:morsifications-Dn}
\end{figure}

\pagebreak[3]

We will also need the following lemma, 
listing several properties of the quivers which come from divides. 

\begin{lemma}
\label{lem:divide-quivers-restrictions}
Let $D$ be a divide, and $Q=Q(D)$ the associated quiver. Then:
\begin{itemize}[leftmargin=.3in]
\item[{\rm (1)}]
$Q$ does not contain an improperly oriented $3$-cycle.
That is, 
if $Q$ contains arrows $a\to b$ and $b\to c$, then it does not contain an arrow $a\to c$. 
\item[{\rm (2)}]
If $Q$ contains arrows $a\to b$ and $b\to c$, then it contains a $3$-cycle that
includes at least one of these arrows. 
\item[{\rm (3)}]
Let $\Delta$ be an oriented $3$-cycle in~$Q$.
Then $Q$ contains another oriented $3$-cycle 
which is different from~$\Delta$ but shares an arrow with~$\Delta$. 
\item[{\rm (4)}]
If $Q$ contains arrows $a\to b\to c\to d\to e$, then it contains a $3$-cycle that
includes two consecutive arrows among these four. 
\item[{\rm (5)}]
If $v$ is a vertex in~$Q$ with exactly two incoming arrows $a\to v$ and $c\to v$,
and exactly two outgoing arrows $v\to b$ and $v\to d$,
with $a,b,c,d$ distinct, 
then $Q$ contains arrows $a\leftarrow b\to c\leftarrow d\to a$. 
\item[{\rm (6)}]
If a vertex $b$ is incident in~$Q$ to exactly three arrows $a\to b$, $b\to c$, $d\to b$, 
and $c$ is incident in~$Q$ to exactly three arrows $c\to a$, $b\to c$, $c\to d$,
then \linebreak[3]
$Q$~is the $4$-vertex quiver with vertices $a,b,c,d$ and the five arrows listed above. 
\end{itemize}
\end{lemma}

\begin{proof}
(1) This statement is immediate from the orientation rule 
of Definition~\ref{def:quiver-of-divide}. 

(2) Note that by the construction of~$Q(D)$, 
each arrow $\circledplus\to\circleddash$ is contained in a $3$-cycle, and 
each path $\circleddash\to\bullet\to\circledplus$ is contained in a $3$-cycle. 

(3) If $Q$ contains an arrow parallel to one of the arrows in~$\Delta$
(i.e., if two vertices of~$\Delta$ are connected in~$Q$ by two or more arrows), 
then the claim is obvious. 
Otherwise, $\Delta$ contains an arrow $a\to b$ 
corresponding to two adjacent (inner) regions of~$D$
separated by a curve segment connecting two distinct nodes $u$ and~$v$. 
This yields two oriented $3$-cycles $a\to b\to u\to a$ and $a\to b\to v\to a$,  
one of which is~$\Delta$. 

(4) The oriented path 
$a\to b\to c\to d\to e$ must include three consecutive vertices labeled 
$\circleddash\to\bullet\to\circledplus$.  
This two-arrow path has to be contained in a $3$-cycle. 

\pagebreak[3]

(5) 
If $v$ is labeled~$\bullet$ (i.e., comes from a node
in~$D$), then $a,b,c,d$ correspond to the four incident regions, 
and the claim follows by the construction of~$Q(D)$. 
If $v$ is labeled $\circledplus$ or~$\circleddash$, 
then it comes from a region of~$D$ bounded by two $1$-cells (i.e., a digon), 
with the vertices $a,b,c,d$ corresponding to the two vertices of the digon and 
the two adjacent regions; the claim follows. 

(6) If $b$ is labeled~$\bullet$ (i.e., comes from a node), 
then none of $a,c,d$ is labeled~$\bullet\,$. 
Hence $c$ comes from a region adjacent to a single node (i.e., a monogon),
and consequently $c$ cannot be incident to three arrows. 
One similarly shows that $c$ cannot be labeled~$\bullet\,$.
Hence both $a$ and~$d$ are labeled~$\bullet\,$. Furthermore $b$ and~$c$
come from adjacent digons, each having two vertices $a$ and~$d$. 
Moreover these digons are not adjacent to any other regions of~$D$.
The statement now follows from the connectedness of~$Q(D)$. 
\end{proof}

\begin{proof}[Proof of Theorem~\ref{th:simple-singularities}]
If $(C,z)$ is a simple singularity, then $D$ is one of the divides 
listed in Proposition~\ref{pr:ade-morsifications}. 
It is straightforward to check that for each of these divides, the quiver $Q(D)$
is mutation equivalent to an orientation of the Dynkin diagram of the corresponding type. 

It remains to show that if $Q(D)$ is a quiver of finite type, then $(C,z)$ is a simple singularity. 
If $n\le 8$, then the claim follows from the well known fact 
(see, e.g., \cite[Preamble to Part~II]{agzv1})
that all singularities with the Milnor number
$n\le 8$ are simple. 
Thus, we may assume that $Q(D)$ is a quiver of type
$A_n$ or~$D_n$, with~$n\ge 9$; i.e., $D$~is mutation equivalent to an orientation
of the corresponding Dynkin diagram. 
We next show that this assumption implies 
that $D$ is one of the divides catalogued in Proposition~\ref{pr:ade-morsifications}
(cf.\ Figures \ref{fig:morsifications-An}--\ref{fig:morsifications-Dn}),
and therefore the underlying singularity is simple, of the corresponding type
($A_n$ or~$D_n$), 
cf.\ Theorem~\ref{th:ag-diagram-determines-topology}. 
We will treat the types $A_n$ and~$D_n$ separately,
assuming the reader's familiarity with the basic combinatorics of
triangulations giving rise to quivers of these types; 
see, e.g., \cite[Chapter~5]{fwz45}. 

\textbf{Type~$A_n$}. 
Each quiver~$Q$ of type~$A_n$ describes the signed adjacencies of diagonals 
in a triangulation~$T$ of a convex $(n+3)$-gon. 
If $T$ contains three diagonals forming a triangle,
then $Q$ contains an oriented $3$-cycle whose sides are not 
contained in any other $3$-cycles. 
By Lemma~\ref{lem:divide-quivers-restrictions}(3),
such a quiver cannot come from a divide. 
If~$T$ does not include a triple of diagonals forming a triangle, 
then $Q$ is an orientation of the Dynkin diagram of type~$A_n$. 
If such a quiver $Q$ includes two co-oriented arrows $a\to b\to c$,
then $Q$ cannot come from a divide, by Lemma~\ref{lem:divide-quivers-restrictions}(2). 
We conclude that $Q$ has an alternating orientation, 
i.e., every vertex is either a source or a sink. 
The only divides producing such quivers are the ones appearing in 
Figure~\ref{fig:morsifications-An}. 

\textbf{Type~$D_n$}. 
Each quiver~$Q$ of type~$D_n$
comes from a (tagged) triangulation~$T$
of an $n$-gon with a single puncture~$p$. 
Removing the arcs of~$T$ incident to~$p$,
we obtain a partial triangulation~$T'$,
with $p$ lying inside a punctured
$k$-gon $\mathbf{P}$, for some $k\ge 2$. 

Suppose $k\ge 4$.
Then $Q$ contains a chordless oriented $k$-cycle,
and therefore cannot come from a divide, 
by Lemma~\ref{lem:divide-quivers-restrictions}(4). 

Suppose $k=3$. 
If exactly one side of the triangle $\mathbf{P}$ is a diagonal of the $n$-gon,
then Lemma~\ref{lem:divide-quivers-restrictions}(6) applies, 
and $Q$ is a quiver of type~$D_4$ which comes from a divide
obtained by crossing a circle with a line. 
If two or three sides of $\mathbf{P}$ are diagonals, 
then we run into a contradiction with Lemma~\ref{lem:divide-quivers-restrictions}(5). 

Suppose $k=2$. Then $p$ is incident to two arcs in~$T$. 
If these two arcs connect $p$ to two distinct vertices of the digon~$\mathbf{P}$
(with the same tagging), then $Q$ contains a chordless oriented $4$-cycle,
contradicting Lemma~\ref{lem:divide-quivers-restrictions}(4). 
Thus, let us assume that the two aforementioned arcs connect~$p$ to the same vertex
(and are tagged differently). 
Arguing as in the type~$A_n$ case above, we conclude that
the partial triangulation~$T'$ cannot include three diagonals forming a triangle.
It follows that the associated quiver~$Q$ is an orientation of
a graph of the form 
\begin{center}
\setlength{\unitlength}{1pt}
\begin{picture}(180,15)(0,5)
\thicklines
\put(0,10){\line(1,0){180}}

\put(60,10){\line(1,1){10}}
\put(60,10){\line(1,-1){10}}
\put(70,0){\line(1,1){10}}
\put(70,20){\line(1,-1){10}}

\put(70,20){\circle*{4}}
\put(70,0){\circle*{4}}
\multiput(0,10)(20,0){4}{\circle*{4}}
\multiput(80,10)(20,0){6}{\circle*{4}}
\end{picture}
\end{center}
(with both triangles properly oriented). 
Moreover, each of the chains on either side of this graph 
must have an alternating orientation, by Lemma~\ref{lem:divide-quivers-restrictions}(2)
(cf.\ the argument in type~$A_n$).
Finally, none of the degree~$4$ vertices in~$Q$ can have two incoming and two outgoing arrows,
by Lemma~\ref{lem:divide-quivers-restrictions}(5). 
These considerations completely fix the orientation of all the arrows in
the quiver~$Q$, 
up to a global reversal. 
We conclude that this quiver must come from one of the divides shown in 
Figure~\ref{fig:morsifications-Dn}. 
%
\end{proof}


\end{document}